\documentclass[a4paper, 11pt, dvipsnames, svgnames]{article}

\usepackage[utf8]{inputenc}
\usepackage[top=1.1in, bottom=1.1in, left=1.0in, right=1.0in]{geometry}
\usepackage[all]{xy}
\usepackage{amsfonts}
\usepackage{amsmath}
\usepackage{amssymb}
\usepackage{amsthm}
\usepackage{calrsfs}
\usepackage{enumerate}
\usepackage{enumitem}
\usepackage{xparse, etoolbox}
\usepackage{mathtools, nccmath, textcomp}
\usepackage{relsize}
\usepackage{rsfso}
\usepackage{stmaryrd}
\usepackage{sectsty}
\usepackage{setspace}
\usepackage[nottoc]{tocbibind}
\usepackage{tensor}
\usepackage{tikz}
\usepackage{tikz-cd}
\usepackage{titlesec}	
\usepackage[titles]{tocloft}
\usepackage{xcolor}
\usepackage[backend=biber, style=alphabetic, sorting=nyt, doi=false, url=false, maxbibnames=99, maxalphanames=99]{biblatex}
\usepackage[T1]{fontenc}
\usepackage{lmodern}

\usepackage{hyperref}
\hypersetup{
	colorlinks=true,
	linkcolor=WildStrawberry,
	citecolor=Teal,
	urlcolor=black,
}

\usepackage{imakeidx}
\makeindex

\usepackage{fancyhdr}

\newcommand{\addperiod}[1]{#1.}
\titleformat{\section}[block]{\scshape\Large\filcenter}{\thesection.}{1em}{}
\titleformat{\subsection}[runin]{\normalfont\large\bfseries}{\thesubsection.}{1em}{\addperiod}
\titleformat{\subsubsection}[runin]{\normalfont\bfseries}{\thesubsubsection.}{1em}{\addperiod}

\numberwithin{equation}{section}

\let \savenumberline \numberline
\def \numberline#1{\savenumberline{#1.}}

\makeatletter
\def\keywords{\xdef\@thefnmark{}\@footnotetext}
\makeatother

\theoremstyle{plain}
\newtheorem{thm}{Theorem}[section]
\newtheorem{cor}[thm]{Corollary}
\newtheorem{lem}[thm]{Lemma}
\newtheorem{prop}[thm]{Proposition}

\theoremstyle{definition}
\newtheorem{defi}[thm]{Definition}

\newtheorem{assum}[thm]{Assumption}
\newtheorem{rem}[thm]{Remark}
\newtheorem{exm}[thm]{Example}
\AtEndEnvironment{const}{\qed}

\theoremstyle{remark}
\newtheorem*{conv}{Convention}

\newtheorem*{nota}{Notation}

\newenvironment{enumromanup}
{\begin{enumerate}[font=\upshape, labelindent=\parindent, label=(\roman*)]}
{\end{enumerate}}

\DeclareMathOperator*{\bmoplus}{\text{\raisebox{0.25ex}{\scalebox{0.7}{$\bigoplus$}}}}

\DeclareMathOperator*{\bmwedge}{\text{\raisebox{0.25ex}{\scalebox{0.7}{$\bigwedge$}}}}

\DeclareMathOperator*{\bmstar}{\text{\raisebox{0.25ex}{\scalebox{0.8}{$\bigstar$}}}}
\DeclareMathOperator*{\smstar}{\text{\raisebox{0.25ex}{\scalebox{0.65}{$\bigstar$}}}}

\DeclareMathAlphabet{\pazocal}{OMS}{zplm}{m}{n}
\DeclareMathAlphabet{\dutchcal}{U}{dutchcal}{m}{n}

\newcommand{\isomorphic}{\xrightarrow{\hspace{0.5mm} \sim \hspace{0.5mm}}}
\newcommand{\lisomorphic}{\xleftarrow{\hspace{0.5mm} \sim \hspace{0.5mm}}}

\newcommand{\cald}{\mathcal{D}}
\newcommand{\cale}{\mathcal{E}}
\newcommand{\calf}{\mathcal{F}}

\newcommand{\call}{\mathcal{L}}

\newcommand{\cals}{\mathcal{S}}

\newcommand{\calu}{\mathcal{U}}

\newcommand{\pazb}{\pazocal{B}}

\newcommand{\pazf}{\pazocal{F}}
\newcommand{\pazg}{\pazocal{G}}
\newcommand{\pazh}{\pazocal{H}}

\newcommand{\pazk}{\pazocal{K}}
\newcommand{\pazl}{\pazocal{L}}
\newcommand{\pazm}{\pazocal{M}}
\newcommand{\pazn}{\pazocal{N}}
\newcommand{\pazo}{\pazocal{O}}

\def\CC{\mathbb{C}}

\def\LL{\mathbb{L}}
\def\NN{\mathbb{N}}
\def\QQ{\mathbb{Q}}
\def\RR{\mathbb{R}}
\def\ZZ{\mathbb{Z}}

\newcommand{\mbfa}{\mathbf{A}}
\newcommand{\mbfb}{\mathbf{B}}
\newcommand{\mbfd}{\mathbf{D}}
\newcommand{\mbfe}{\mathbf{E}}

\newcommand{\mbfn}{\mathbf{N}}

\newcommand{\smbff}{\mathbf{f}}
\newcommand{\smbfg}{\mathbf{g}}
\newcommand{\smbfi}{\mathbf{i}}

\newcommand{\smbfk}{\mathbf{k}}

\newcommand{\smbfy}{\mathbf{y}}

\newcommand{\frakU}{\mathfrak{U}}

\newcommand{\frakX}{\mathfrak{X}}
\newcommand{\frakY}{\mathfrak{Y}}
\newcommand{\frakZ}{\mathfrak{Z}}

\newcommand{\Ainf}{\mbfa_{\inf}}

\newcommand{\Acrys}{\mbfa_{\textup{cris}}}
\newcommand{\OAcrys}{\pazo\mbfa_{\textup{cris}}}

\newcommand{\AFpi}{\mbfa_{F,\varpi}}

\newcommand{\AR}{\mbfa_R}
\newcommand{\BR}{\mbfb_R}

\newcommand{\ARpi}{\mbfa_{R,\varpi}}
\newcommand{\OARpi}{\pazo\mbfa_{R,\varpi}}
\newcommand{\ARbar}{\mbfa_{\overline{R}}}
\newcommand{\ASbar}{\mbfa_{\overline{S}}}
\newcommand{\ARinfty}{\mbfa_{R_{\infty}}}
\newcommand{\ASinfty}{\mbfa_{S_{\infty}}}
\newcommand{\BdR}{\mbfb_{\textup{dR}}}
\newcommand{\OBdR}{\pazo\mbfb_{\textup{dR}}}
\newcommand{\Bcrys}{\mbfb_{\textup{cris}}}
\newcommand{\OBcrys}{\pazo\mbfb_{\textup{cris}}}

\newcommand{\cont}{\textup{cont}}
\newcommand{\converge}{\textup{conv}}

\newcommand{\CR}{\textup{CR}}
\newcommand{\crys}{\textup{cris}}

\newcommand{\CRYS}{\textup{CRIS}}
\newcommand{\cycl}{\textup{cycl}}

\newcommand{\ODdR}{\pazo\mbfd_{\textup{dR}}}

\newcommand{\ODcrys}{\pazo\mbfd_{\textup{cris}}}
\newcommand{\dlog}{\hspace{1mm}d\textup{\hspace{0.5mm}log\hspace{0.3mm}}}

\newcommand{\Dpi}{D_{\varpi}}

\newcommand{\dr}{\textup{dR}}

\newcommand{\etale}{\textup{\'et}}
\newcommand{\ERpi}{\mbfe_{R, \varpi}}
\newcommand{\fERpi}{E_{R, \varpi}}
\newcommand{\fERbar}{E_{\overline{R}}}

\newcommand{\fESbar}{E_{\overline{S}}}
\newcommand{\fESbarn}{E_{\overline{S}, n}}
\newcommand{\Fil}{\textup{Fil}}

\newcommand{\Ftor}{\textup{F}^r}

\newcommand{\FM}{\textup{FM}}
\newcommand{\Fr}{\textup{Fr}\hspace{0.5mm}}

\newcommand{\free}{\hspace{0.5mm}\textup{free}}
\newcommand{\GL}{\textup{GL}}
\newcommand{\Gal}{\textup{Gal}}
\newcommand{\gr}{\textup{gr}}

\newcommand{\holim}{\textup{holim}_n\hspace{1mm}}
\newcommand{\Hom}{\textup{Hom}}

\newcommand{\Kphi}{\textup{K}_{\varphi}}
\newcommand{\Kdphi}{\textup{K}_{\partial, \varphi}}
\newcommand{\Kdphid}{\textup{K}_{\partial, \varphi, \partial_A}}
\newcommand{\Kphid}{\textup{K}_{\varphi, \partial_A}}

\newcommand{\kert}{\textup{Ker }}
\newcommand{\Kos}{\textup{Kos}}

\newcommand{\Laz}{\pazl\textrm{az}}

\newcommand{\Lie}{\textup{Lie }}

\newcommand{\MIC}{\textup{MIC}}
\newcommand{\MF}{\textup{MF}}

\newcommand{\Mat}{\textup{Mat}}

\newcommand{\Mpi}{M_{\varpi}}
\newcommand{\Npi}{N_{\varpi}}

\newcommand{\unrami}{\textup{ur}}

\newcommand{\padic}{p\textrm{-adic}}
\newcommand{\PD}{\textup{PD}}

\newcommand{\prm}{^{\prime}}

\newcommand{\Rbar}{\overline{R}}
\newcommand{\Rep}{\textup{Rep}}
\newcommand{\Rinfty}{R_{\infty}}
\newcommand{\Rup}{\textup{R}}
\newcommand{\RGamma}{\textup{R}\Gamma}
\newcommand{\Rpi}{R_{\varpi}}
\newcommand{\Rpin}{R_{\varpi, n}}

\newcommand{\Sbar}{\overline{S}}

\newcommand{\Spec}{\textup{Spec}\hspace{0.5mm}}
\newcommand{\Spf}{\textup{Spf}\hspace{0.6mm}}
\newcommand{\Sp}{\textup{Sp}\hspace{0.6mm}}

\newcommand{\syn}{\textup{syn}}
\newcommand{\Syn}{\textup{Syn}}

\newcommand{\textpd}{\textup{PD}}

\title{\vspace{-5mm}\textsc{Syntomic complex and $\padic$ nearby cycles}}
\author{\textsc{Abhinandan}}

\newcommand{\Addresses}{{
  \footnotesize

  \rule{2cm}{0.4pt}\vspace{2mm}

  \textsc{Abhinandan}\par\nopagebreak
  \textsc{IMJ-PRG, Sorbonne Universit\'e, 4 Place Jussieu, Paris, France}\par\nopagebreak\vspace{-0.7mm}
  \textit{E-mail}: \footnotesize{\href{abhinandan@imj-prg.fr}{abhinandan@imj-prg.fr}}, \textit{Web}: \footnotesize{\href{https://abhinandan.perso.math.cnrs.fr/}{https://abhinandan.perso.math.cnrs.fr/}}
}}

\date{ }

\begin{document}

\pagenumbering{arabic}

\sloppy

\goodbreak


\maketitle
{
	\vspace{-3mm}
	\textsc{Abstract.} In local relative $\padic$ Hodge theory, we show that the Galois cohomology of a finite height crystalline representation (up to a twist) is essentially computed via the (Fontaine--Messing) syntomic complex with coefficients in the associated $F\textrm{-isocrystal}$.
	In global applications, for smooth ($\padic$ formal) schemes, we establish a comparison between the syntomic complex with coefficients in a locally free Fontaine--Laffaille module and the $\padic$ nearby cycles of the associated étale local system on the (rigid) generic fibre.
}


\section{Introduction}

Let $p$ denote a fixed prime and $\kappa$ a perfect field of characteristic $p$.
Let $K$ be a mixed characteristic complete discrete valuation field with ring of integers $O_K$ and residue field $\kappa$ and $F := W(\kappa)[1/p]$ the fraction field of the ring of $p\textrm{-typical}$ Witt vectors with coefficients in $\kappa$.
Fontaine's \textit{crystalline conjecture} for a proper and smooth $O_K\textrm{-scheme}$ relates the $p\textrm{-adic}$ \'etale cohomology of its generic fibre to the crystalline cohomology of its special fibre.
In \cite{fontaine-messing-padic}, Fontaine and Messing initiated a program for proving the crystalline conjecture via \textit{syntomic} methods.
By subsequent works of \cite[Kato--Messing]{kato-messing-syntomic}, \cite[Kato]{kato-semistable-etale} and with the remarkable work of \cite[Tsuji]{tsuji-semistable-comparison}, the crystalline conjecture was shown to be true.
There have been several proofs and generalisations of the crystalline comparison theorem: \cite{tsuji-semistable-comparison, faltings-crystalline, faltings-almost-etale, niziol-crystalline-K-theory, beilinson-derived-derham, scholze-rigid, yamashita-yasuda-open, andreatta-iovita-comparison-smooth, colmez-niziol-nearby-cycles, bhatt-morrow-scholze-integral, diao-lan-liu-zhu, guo-reinecke-analytic}.

\subsection{\texorpdfstring{$\padic$}{-} nearby cycles}

Let $\frakX$ be a smooth ($\padic$ formal) $O_K\textrm{-scheme}$ with (rigid) generic fibre $X$ and special fibre $\frakX_{\kappa}$.
Let $j : X_{\etale} \rightarrow \frakX_{\etale}$ and $i : \frakX_{\kappa, \etale} \rightarrow \frakX_{\etale}$ denote natural morphisms of sites.
For $r \geqslant 0$, let $\cals_n(r)_{\frakX}$ denote the syntomic sheaf modulo $p^n$ on $\frakX_{\kappa, \etale}$ (see \S \ref{sec:crystals_syntomic_cohomology} and \S \ref{sec:padic_nearby_cycles} for the definition of the syntomic complex).
In \cite{fontaine-messing-padic}, Fontaine and Messing constructed a period morphism from the syntomic complex to the complex of $p\textrm{-adic}$ nearby cycles,
\begin{equation}\label{intro_eq:fm_period_map_triv}
	\alpha_{r, n}^{\FM} : \cals_n(r)_{\frakX} \longrightarrow i^{\ast} \textup{R} j_{\ast} \ZZ / p^n (r)\prm_X,
\end{equation}
where $\ZZ_p(r)\prm := \frac{1}{a(r)!p^{a(r)}} \ZZ_p(r)$, for $r = (p-1)a(r) + b(r)$ with $0 \leqslant b(r) < p-1$.
For $\frakX$ a smooth and proper $O_K\textrm{-scheme}$ and $0 \leqslant r \leqslant p-1$, by truncating \eqref{intro_eq:fm_period_map_triv} in degree $\leqslant r$, the map $\alpha_{r, n}^{\FM}$ is known to be a quasi-isomorphism by \cite[Kato]{kato-vanishing-cycles, kato-semistable-etale}, \cite[Kurihara]{kurihara-padic-etale} and \cite[Tsuji]{tsuji-semistable-comparison}.
In \cite{tsuji-syntomic-complex}, Tsuji generalised this result to proper and semistable schemes and non-trivial \'etale local systems arising from (the pullback of) Fontaine--Laffaille modules over $O_F$ (see \cite{fontaine-laffaille}).
Moreover, in \cite{colmez-niziol-nearby-cycles}, Colmez and Nizio{\l} proved a similar result for semistable ($\padic$ formal) schemes and constant coefficients case, without any restrictions on $r$.
In particular, for a smooth ($\padic$ formal) scheme we have the following:
\begin{thm}[{\cite[Theorem 1.1]{colmez-niziol-nearby-cycles}}]\label{intro_thm:nearby_cycles_triv}
	For $0 \leqslant k \leqslant r$, the natural map
	\begin{equation*}
		\alpha_{r, n}^{\textup{FM}} : \pazh^k(\cals_n(r)_{\frakX}) \longrightarrow i^{\ast} \textup{R}^k j_{\ast} \ZZ/p^n(r)\prm_{X},
	\end{equation*}
	is a $p^N\textrm{-isomorphism}$, i.e.\ its kernel and cokernel are killed by $p^N$, where $N = N(e, p, r) \in \NN$ depends on the absolute ramification index $e$ of $K$, prime $p$ and twist $r$ but not on $X$ or $n$.
\end{thm}

The proof of Theorem \ref{intro_thm:nearby_cycles_triv} in \cite{colmez-niziol-nearby-cycles} works by reducing the problem to the local setting, i.e.\ one works over the $\padic$ completion of an \'etale algebra over $O_K[X_1^{\pm 1}, \ldots, X_d^{\pm 1}]$, for some indeterminates $X_1, \ldots, X_d$.
Locally, Colmez and Nizio{\l} also show that it is enough to work with $\padic$ formal schemes and deduce the result for schemes by invoking Elkik's approximation theorem and a form of rigid GAGA (see \cite[\S 5.1]{colmez-niziol-nearby-cycles}).

For simplicity in the introduction, let $R$ be the $\padic$ completion of $O_F[X_1^{\pm 1}, \ldots, X_d^{\pm 1}]$ and $S := O_K \otimes_{O_F} R$ (see Assumption \ref{assum:small_algebra} for a more general setup).
Let $G_S := \pi_1^{\etale}(S[1/p], \overline{\eta})$, for a fixed geometric generic point of $\Sp (S[1/p])$.
Denote by $\Syn(S, r)$ the $r\textrm{-th}$ Tate twist of the (log-) syntomic complex (see \cite[\S 3.3]{colmez-niziol-nearby-cycles} for details).

\begin{thm}[{\cite[Theorem 1.6]{colmez-niziol-nearby-cycles}}]\label{intro_thm:colmez_niziol_lazard_iso}
	If $K$ contains enough roots of unity, then the maps
	\begin{align*}\label{eq:intro_colmez_niziol_lazard_iso}
		\begin{split}
			\alpha_r^{\Laz} &: \tau_{\leqslant r} \Syn(S, r) \longrightarrow \tau_{\leqslant r} \RGamma_{\cont}(G_S, \ZZ_p(r)),\\
			\alpha_{r, n}^{\Laz} &: \tau_{\leqslant r} \Syn(S, r)_n \longrightarrow \tau_{\leqslant r} \RGamma_{\cont}(G_S, \ZZ/p^n(r)) \longrightarrow \tau_{\leqslant r} \RGamma\big(\big(\textup{Sp } S[1/p])_{\etale}, \ZZ/p^n(r)\big),
		\end{split}
	\end{align*}
	are $p^{Nr}\textrm{-quasi-isomorphisms}$ for a universal constant $N$, i.e.\ $N$ does not depend on $p$, $X$, $K$, $n$ or $r$.
\end{thm}

One of our main goals in this article is to generalise Theorem \ref{intro_thm:colmez_niziol_lazard_iso} by studying syntomic complexes with coefficients.
Subsequently, by ``gluing'' the local results for relative Fontaine--Laffaille modules, we will obtain a global generalisation of Theorem \ref{intro_thm:nearby_cycles_triv}.
Note that in the local setting, on the \'etale side, by using a $K(\pi, 1)\textrm{-Lemma}$ (see \cite[Theorem 4.9]{scholze-rigid}), we can reduce to the setting of $\ZZ_p\textrm{-representations}$ of $G_R$.
Then, due to the ``crystalline'' nature of our goal, we will consider $G_R\textrm{-stable}$ $\ZZ_p\textrm{-lattices}$ inside ``finite height'' crystalline representations of $G_R$ and certain natural invariants attached to such representations as in \cite[\S 4]{abhinandan-crystalline-wach}.

\subsection{Finite height representations}

Fix $p \geqslant 3$, $m \in \NN_{\geqslant 2}$, $K = F(\zeta_{p^m})$ and $\varpi = \zeta_{p^m}-1$ (see Remark \ref{rem:cyclotomic_restriction} on the rationale behind our assumptions).
Fix an algebraically closed field $\overline{\Fr(R)}$ containing $\overline{F}$ an algebraic closure of $F$ and set $F_{\infty} := F(\mu_{p^{\infty}}) \subset \overline{F}$.
Let $\Rbar$ denote the union of finite $R\textrm{-subalgebras}$ $R' \subset \overline{\Fr(R)}$ such that $R'[1/p]$ is \'etale over $R[1/p]$.
Set $R_{\infty} := \cup_{n \in \NN} R\big[\mu_{p^n}, X_1^{1/p^n}, \ldots, X_d^{1/p^n}\big]$, $G_R := \Gal\big(\Rbar[1/p]/R[1/p]\big)$, $\Gamma_R := \Gal\big(\Rinfty[1/p]/R[1/p]\big)$, $H_R := \kert(G_R \twoheadrightarrow \Gamma_R)$ and note that we have $\Gamma_R = \Gamma_R' \rtimes \Gamma_F$, where $\Gamma_R' := \Gal\big(\Rinfty[1/p]/F_{\infty}R[1/p]\big) \isomorphic \ZZ_p(1)^d$ and $\Gamma_F := \Gal(F_{\infty}/F) \isomorphic \ZZ_p^{\times}$.

Recall that \cite{fontaine-festschrift} showed a categorical equivalence between $\ZZ_p\textrm{-representations}$ of $G_F$ and \'etale $(\varphi, \Gamma_F)\textrm{-modules}$ over a certain period ring $\mbfa_F$.
These results were generalised to the relative setting in \cite{andreatta-generalized-phiGamma}, to establish a categorical equivalence between $\ZZ_p\textrm{-representations}$ of $G_R$ and \'etale $(\varphi, \Gamma_R)\textrm{-modules}$ over a certain period ring $\AR$ (see \S \ref{subsec:relative_phi_gamma_mod}).
Moreover, Fontaine's work on crystalline representations of $G_F$, in \cite{fontaine-annals, fontaine-corps-des-periodes, fontaine-semistables}, was generalised to the relative case, in \cite{brinon-padicrep-relatif}, via the construction of a fully faithful functor $\ODcrys$ from the category of crystalline representations of $G_R$ to the category of filtered $(\varphi, \partial)\textrm{-modules}$ over $R[1/p]$ (see \S \ref{subsec:relative_padic_reps}).

Let $q = \varphi(\pi)/\pi$ in $\AR$, where $\pi$ is the usual element of Fontaine (see \S \ref{subsec:period_rings}).
In \cite{abhinandan-crystalline-wach}, we studied finite $q\textrm{-height}$ representations of $G_R$, a notion parallel to the arithmetic case, i.e.\ $R = O_F$ in \cite{wach-pot-crys, wach-cristallines-torsion, colmez-finite-height, berger-limites-cristallines} (see \cite[Remark 1.4]{abhinandan-crystalline-wach}).
A representation $T$ in $\Rep_{\ZZ_p, \free}(G_R)$ is of finite $q\textrm{-height}$ if it admits a unique $(\varphi, \Gamma_R)\textrm{-module}$ over a certain subring $\AR^+ \subset \AR$ satisfying certain conditions on the $(\varphi, \Gamma_R)\textrm{-action}$ (see Definition \ref{defi:wach_reps}); the aforementioned $\AR^+\textrm{-module}$ is called the \textit{Wach module} associated to $T$ and denoted as $\mbfn(T)$.
Moreover, we showed that finite $q\textrm{-height}$ representations are closely related to crystalline representations via a certain period ring $\OARpi^{\textpd} \subset \OAcrys(\Rbar)$, where the former is equipped with structures induced from the latter (see \cite[\S 4.3]{abhinandan-crystalline-wach}).

\begin{thm}[{\cite[Theorem 4.24, Proposition 4.27]{abhinandan-crystalline-wach}}]
	Let $T$ be a $\ZZ_p\textrm{-representation}$ of $G_{R}$ and assume that $T$ is of positive finite $q\textrm{-height}$.
	Then $V := T[1/p]$ is a positive crystalline representation and we have an isomorphism of $R[1/p]\textrm{-modules}$ $\ODcrys(V) \lisomorphic \big(\OARpi^{\textpd} \otimes_{\AR^+} \mbfn(T)\big)^{\Gamma_R}[1/p]$ compatible with the respective Frobenii, filtrations and connections.
\end{thm}

\subsection{Syntomic coefficients and \texorpdfstring{$(\varphi, \Gamma)$}{-}-modules}

In this subsection, we will assume the following:
Let $T$ be a $\ZZ_p\textrm{-representation}$ of $G_R$ of positive finite $q\textrm{-height}$ $s \in \NN$ and set $V := T[1/p]$ (see Definition \ref{defi:wach_reps}).
Assume that $\mbfn(T)$ is free of rank $= \textrm{rk}_{\ZZ_p} T$ over $\AR^+$ and $M \subset \ODcrys(V)$ is a finite free $R\textrm{-submodule}$ of rank $=\textrm{rk}_{\ZZ_p} T$, such that $M[1/p] \isomorphic \ODcrys(V)$ and satisfies Assumption \ref{assum:relative_crystalline_wach_free} (see Example \ref{exm:choice_odcrist} for obtaining $M$ from $\mbfn(T)$).

Our objective is to compute the continuous $G_R\textrm{-cohomology}$ of $T(r)$ using the syntomic complex for $R$ with coefficients in $M \subset \ODcrys(V)$.
Set $S = R[\varpi]$ and note that we have a divided power thickening $\Rpi^{\textpd} \twoheadrightarrow S$ (using an ``arithmetic'' varaibale $X_0$, see \S \ref{subsec:pd_envelope}) and the ring $\Rpi^{\textpd}$ is equipped with a Frobenius endomorphism $\varphi$; let $\Omega^1_{\Rpi^{\textpd}}$ denote the $\padic$ completion of the module of differentials of $\Rpi^{\textpd}$ with respect to $\ZZ$.
Set $\Mpi^{\textpd} := \Rpi^{\textpd} \otimes_R M$ equipped with the induced supplementary structures to obtain a filtered de Rham complex (see \S \ref{subsec:main_result}),
\begin{equation*}
	\Fil^r \cald_{S, M}^{\bullet} := \Fil^r \Mpi^{\textpd} \longrightarrow \Fil^{r-1} \Mpi^{\textpd} \otimes_{\Rpi^{\textpd}} \Omega^1_{\Rpi^{\textpd}} \longrightarrow \Fil^{r-2} \Mpi^{\textpd} \otimes_{\Rpi^{\textpd}} \Omega^2_{\Rpi^{\textpd}} \longrightarrow \cdots.
\end{equation*}

\begin{defi}\label{intro_defi:syntomic_complex_coeff}
	Define the \textit{syntomic complex} of $S$ with coefficients in $M$ and its modulo $p^n\textrm{-version}$ as $\Syn(S, M, r) := \big[\hspace{1mm} \Fil^r \cald_{S, M}^{\bullet} \xrightarrow{\hspace{1mm}p^r - p^{\bullet} \varphi\hspace{1mm}} \cald_{S, M}^{\bullet}\hspace{1mm}\big]$ and $\Syn(S, M, r)_n := \Syn(S, M, r) \otimes \ZZ/p^n$, for $n \geqslant 1$.
\end{defi}

\begin{thm}[{Theorem \ref{thm:syntomic_complex_galois_cohomology}}]\label{intro_thm:syntomic_complex_galois_cohomology}
	Let $T$ be a positive finite $q\textrm{-height}$ $\ZZ_p\textrm{-representation}$ of $G_R$ of height $s$ as above and take $r \in \NN$ such that $r \geqslant s + 1$.
	Then, there exist $p^N\textrm{-quasi-isomorphisms}$,
	\begin{align*}
		\alpha_r^{\Laz} : \tau_{\leqslant r-s-1} \Syn(S, M, r) &\simeq \tau_{\leqslant r-s-1} \RGamma_{\cont}(G_S, T(r)),\\
		\alpha_{r, n}^{\Laz} : \tau_{\leqslant r-s-1} \Syn(S, M, r)_n &\simeq \tau_{\leqslant r-s-1} \RGamma_{\cont}(G_S, T/p^n(r)),
	\end{align*}
	where $N = N(T, e, r) \in \NN$ depends on the representation $T$, $e = [K:F]$ and the twist $r$.
\end{thm}

Similarly, we have a filtered de Rham complex with coefficients in $M$ and one can also define the \textit{syntomic complex} of $R$ with coefficients in $M$.
Using Theorem \ref{intro_thm:syntomic_complex_galois_cohomology} for $\varpi = \zeta_{p^2}-1$ and Galois descent (see Lemma \ref{lem:syn_galois_descent}), we obtain the following:
\begin{cor}[{Corollary \ref{cor:syntomic_complex_galois_cohomology}}]\label{intro_cor:syntomic_complex_galois_cohomology}
	Let $T$ be a positive finite $q\textrm{-height}$ $\ZZ_p\textrm{-representation}$ of $G_R$ of height $s$ as above and take $r \in \NN$ such that $r \geqslant s + 1$.
	Then, there exist $p^N\textrm{-quasi-isomorphisms}$,
	\begin{align*}
		\alpha_r^{\Laz} : \tau_{\leqslant r-s-1} \Syn(R, M, r) &\simeq \tau_{\leqslant r-s-1} \RGamma_{\cont}(G_R, T(r)),\\
		\alpha_{r, n}^{\Laz} : \tau_{\leqslant r-s-1} \Syn(R, M, r)_n &\simeq \tau_{\leqslant r-s-1} \RGamma_{\cont}(G_R, T/p^n(r)),
	\end{align*}
	where $N = N(p, r, s) \in \NN$ depends on the prime $p$, twist $r$ and height $s$ of $T$.
\end{cor}

The proof of Theorem \ref{thm:syntomic_complex_galois_cohomology} is broadly divided in two main steps.
First, we modify the syntomic complex with coefficients in $M$ and relate it to a ``differential'' Koszul complex with coefficients in $\mbfn(T)$ (see Proposition \ref{prop:syntomic_to_phi_gamma}).
Next, we modify the Koszul complex from the first step to obtain a Koszul complex computing the continuous $G_S\textrm{-cohomology}$ of $T(r)$ (see Theorem \ref{thm:syntomic_complex_galois_cohomology} and Proposition \ref{prop:differential_koszul_complex_galois_cohomology}).
The key idea behind relating these two steps is the comparison isomorphism in \cite[Theorem 4.24]{abhinandan-crystalline-wach} and a Poincar\'e Lemma (see \S \ref{subsec:poincare_lemma_app}).
Our proof of Theorem \ref{thm:syntomic_complex_galois_cohomology} is inspired by \cite{colmez-niziol-nearby-cycles}, however our setting demands several non-trivial generalisations of the ideas in loc.\ cit.

\begin{rem}
	Setting $T = \ZZ_p$ in Theorem \ref{intro_thm:syntomic_complex_galois_cohomology}, we obtain a statement similar to Theorem \ref{intro_thm:nearby_cycles_triv} (note that we truncate in degree $\leqslant r-1$ as we are working with the syntomic complex instead of the log-syntomic complex as in \cite{colmez-niziol-nearby-cycles}).
\end{rem}

\begin{rem}\label{rem:cyclotomic_restriction}
	In Theorem \ref{intro_thm:syntomic_complex_galois_cohomology} we restrict to a finite cyclotomic $K/F$ because we are used the cyclotomic Frobenius ($X_0 \mapsto (1+X_0)^p-1$) in Definition \ref{intro_defi:syntomic_complex_coeff}, instead of the Kummer Frobenius ($X_0 \mapsto X_0^p$) as in \cite{colmez-niziol-nearby-cycles}.
	For $K/F$ finite, one should use Kummer Frobenius to define a log-syntomic complex (log-structure with respect to $X_0$).
	Then it should be possible to obtain Theorem \ref{intro_thm:syntomic_complex_galois_cohomology} for all finite extensions $K/F$ (with truncation in degree $\leqslant r-s$ as in \cite{colmez-niziol-nearby-cycles}).
	Furthermore, to obtain the statement over $\overline{F}$ one could pass to the limit over all finite extensions $K/F$.
	Alternatively, one could directly work over $\CC_p = \widehat{\overline{F}}$ as in \cite{gilles-morphismes} to avoid complications arising from Frobenius on $X_0$.
	In the latter case, our proofs can be adapted to obtain Theorem \ref{intro_thm:syntomic_complex_galois_cohomology} for $S = R \widehat{\otimes}_{O_F} O_{\CC_p}$ (with truncation in degrees $\leqslant r-s-1$).
\end{rem}

\begin{rem}
	The case $p=2$ is different from $p \geqslant 3$, as for $p=2$, the constant $N$ in Theorem \ref{intro_thm:syntomic_complex_galois_cohomology} also depends on the relative dimension of $R/O_F$ (see \cite[Lemma 3.11]{colmez-niziol-nearby-cycles}).
\end{rem}

Next, using the fundamental exact sequence in $\padic$ Hodge theory \eqref{eq:fes_acrys}, one can define a local Fontaine--Messing period map for $T$ as in Theorem \ref{intro_thm:syntomic_complex_galois_cohomology} (see \S \ref{subsec:fm_maps_comparison}).
Then, we show the following:
\begin{thm}[{Theorem \ref{thm:lazard_fmlocal_comparison}}]\label{intro_thm:lazard_fmlocal_comparison}
	The period map $\tilde{\alpha}_{r, n, S}^{\FM}$ is $p^{N(T, e, r)}\textrm{-equal}$ to $\alpha_{r, n}^{\Laz}$ from Theorem \ref{intro_thm:syntomic_complex_galois_cohomology}.
\end{thm}

\subsection{Fontaine--Laffaille modules and \texorpdfstring{$\padic$}{ } nearby cycles}

In this subsection, we will specialise Theorem \ref{intro_thm:syntomic_complex_galois_cohomology} to the case of global relative Fontaine--Laffaille modules introduced by Faltings in \cite[\S II]{faltings-crystalline}.
Let $R$ denote the $\padic$ completion of an \'etale algebra over $O_F[X_1^{\pm 1}, \ldots, X_d^{\pm 1}]$ with non-empty geometrically integral special fibre (see \S \ref{subsec:setup_nota} for details).
Note that Theorem \ref{intro_thm:syntomic_complex_galois_cohomology} and Corollary \ref{intro_cor:syntomic_complex_galois_cohomology} are true in this setting as well.
In \cite[\S 5]{abhinandan-crystalline-wach}, we considered the category $\MF_{[0, s], \free}(R, \Phi, \partial)$ of free relative Fontaine--Laffaille modules of level $[0, s]$ (see Remark \ref{rem:fl_wach_comparison} (i)) as a full subcategory of $\mathfrak{MF}_{[0, s]}^{\nabla}(R)$ in \cite[\S II]{faltings-crystalline}.
To any $M$ in $\MF_{[0, s], \free}(R, \Phi, \partial)$, one can functorially attach a representation $T_{\crys}(M)$ in $\Rep_{\ZZ_p, \free}(G_R)$, which admits a Wach module $\mbfn(T)$ (see \cite[Theorem 5.4]{abhinandan-crystalline-wach}) and satisfies Assumption \ref{assum:relative_crystalline_wach_free} (see Example \ref{exm:choice_odcrist} (iii)).
Next, let $\frakX$ be a smooth ($\padic$ formal) scheme defined over $O_F$ and cover it by affine ($\padic$ formal) schemes $\{\frakU_i\}_{i \in I}$, where for all $i \in I$, we have that $\frakU_i = \Spec A_i$ (resp.\ $\frakU_i = \Spf A_i$) such that its $\padic$ completion $\widehat{A}_i$ is an \'etale algebra as above; we also fix compatible Frobenius lifts $\varphi_i : \widehat{A}_i \rightarrow \widehat{A}_i$.
Take $\MF_{[0, s], \free}(\frakX, \Phi, \partial)$ to be the category of finite locally free filtered $\pazo_{\frakX}\textrm{-modules}$ $\pazm$ equipped with a quasi-nilpotent integrable connection satisfying Griffiths transversality such that there exists a covering $\{\frakU_i\}_{i \in I}$ of $\frakX$ as above with $\pazm_{\frakU_i} \in \MF_{[0, s], \free}(\widehat{A}_i, \Phi, \partial)$, for all $i \in I$ (see \S \ref{subsec:global_fm_mods}).

To state the main global result, let $\frakX$ be a smooth ($\padic$ formal) scheme defined over $O_F$ (for $\frakX$ a scheme, assume that it is proper or an open subscheme of a proper semistable scheme defined over $O_F$).
Let $\pazm$ be an object of $\MF_{[0, s], \free}(\frakX, \Phi, \partial)$ with $0 \leqslant s \leqslant p-2$ (for $\frakX$ an open scheme, further assume that $\pazm$ extends to the compatification of $\frakX$, see Remark \ref{rem:locsys_from_fl}).
Let $\LL$ denote the associated $\ZZ_p\textrm{-local system}$ on the (rigid) generic fibre $X$ of $\frakX$.
Then, we show the following:
\begin{thm}[{Theorem \ref{thm:syntomic_nearby_comparison}}]\label{intro_thm:syntomic_nearby_comparison}
	For $r \geqslant s+1$ and $0 \leqslant k \leqslant r-s-1$ the Fontaine--Messing period map
	\begin{equation*}
		\alpha_{r, n, \frakX}^{\FM} : \pazh^k\big(\cals_n(\pazm, r)_\frakX\big) \longrightarrow i^{\ast} \Rup^k j_{\ast} \LL/p^n(r)\prm_X,
	\end{equation*}
	is a $p^N\textrm{-isomorphism}$, where $N = N(p, r, s) \in \NN$ depends on $p$, $r$ and $s$ but not on $\frakX$ or $n$.
\end{thm}

The proof of Theorem \ref{intro_thm:syntomic_nearby_comparison} proceeds by reducing to the local setting, whence we may directly apply Theorem \ref{intro_thm:syntomic_complex_galois_cohomology}.

\begin{rem}
	In personal communications with Takeshi Tsuji, I learnt that in some unpublished work he obtained similar results over $\overline{F}$ and large enough $p$.
	However, our respective approaches are different and this article includes more general local results and the arithmetic case as well.
\end{rem}

\begin{rem}
	Note that from \cite[\S 10]{bhatt-morrow-scholze-topological}, we have a prismatic syntomic complex and it is known to compute $\padic$ nearby cycles in the case of constant coefficients.
	Using the results of \cite{morrow-tsuji-coeff} on coefficients in integral $\padic$ Hodge theory and prismatic cohomology, it should be possible to obtain an integral version of our results (in the geometric case, i.e.\ over $\overline{F}$).
	Moreover, using the theory of analytic prismatic $F\textrm{-crystals}$ on the absolute prismatic site from \cite{du-liu-moon-shimizu-completed, guo-reinecke-analytic}, we should be able to generalise those results to the arithmetic case as well.
	We will report on these ideas in future.
\end{rem}

\subsection{Outline of the paper}

Sections \ref{sec:relative_padic_Hodge_theory}-\ref{sec:syntomic_galcoh} comprise the local part of the paper, while sections \ref{sec:crystals_syntomic_cohomology}-\ref{sec:padic_nearby_cycles} consist of global applications.
In \S \ref{subsec:setup_nota} we describe our local setup, notations and some conventions.
In \S \ref{subsec:period_rings}, \S \ref{subsec:relative_padic_reps} and \S \ref{subsec:relative_phi_gamma_mod} we quickly recall basics of period rings, crystalline representations and relative \'etale $(\varphi, \Gamma)\textrm{-modules}$.
Subsection \ref{subsec:pd_envelope} introduces ``good'' crystalline coordinates and we define certain rings of analytic functions convergent on some annulus following \cite[\S 2]{colmez-niziol-nearby-cycles}; these rings are denoted as $\Rpi^{\bmstar}$, for $\smstar \in \{+, \textpd, [u], [u, v], (0, v]+\}$, where we can take $u= p/(p-1)$ and $v = p-1$.
In \S \ref{subsec:cyclotomic_frob}, we equip these rings with a Frobenius endomorphism and in \S \ref{subsec:cyclotomic_embeddings}, we consider their Frobenius-equivariant ``cyclotomic'' embedding $\iota_{\cycl}$ into period rings and define $\ARpi^{\bmstar}$ as the image of $\Rpi^{\bmstar}$ under $\iota_{\cycl}$.
The latter enables us to relate differential operators on the ring $\Rpi^{[u, v]}$ to the infinitesimal action of $\Gamma_S := \Gal(\Rinfty[1/p]/S[1/p])$ on its ``cyclotomic'' image, i.e.\ $\ARpi^{[u, v]}$.
Finally, in \S \ref{subsec:fat_period_rings}, we introduce certain big period rings, in particular, $\fERpi^{\bmstar}$ and $\fERbar^{\bmstar}$, study a natural filtration on the scalar extension of $M$ to these rings and prove a version of the filtered Poincar\'e Lemma.
The latter, together with the results of \S \ref{subsec:wachmod_poincare_lem}, are key ingredients in relating syntomic complexes with coefficients in $M$ to Koszul complexes with coefficients in $\mbfn(T)$.
The motivation for our approach comes from the computations of \cite[\S 2.6]{colmez-niziol-nearby-cycles}.

In \S \ref{subsec:relative_wach_modules} and \S \ref{subsec:wach_crystalline}, we recall the notion of finite height representations and their relationship to crystalline representations from \cite{abhinandan-crystalline-wach}, as well as, prove some useful technical lemmas.
In \S \ref{subsec:wachmod_poincare_lem}, we study a filtration on scalar extensions of Wach modules and prove another filtered Poincar\'e Lemma.
The local theory of relative Fontaine--Laffaille modules is recalled in \S \ref{subsec:fontaine_laffaile_to_wach}.
Section \S \ref{sec:galois_cohomology} recalls the definition of Koszul complexes computing continuous $\Gamma_S\textrm{-cohomology}$ (see \S \ref{subsec:gal_coho_kos_complex}) and $\Lie\Gamma_S\textrm{-cohomology}$ (see \S \ref{subsec:lie_algebra_coh}).

In \S \ref{sec:syntomic_complex_finite_height}, we formulate our main local result, Theorem \ref{intro_thm:syntomic_complex_galois_cohomology}, and carry out the local syntomic computations for its proof.
The aim of \S \ref{sec:syntomic_galcoh} is to carry out the $(\varphi, \Gamma)\textrm{-module}$ side computations for the proof of Theorem \ref{intro_thm:syntomic_complex_galois_cohomology}.
To explain the content of these two sections to the reader, we introduce the following commutative diagram of complexes (see the discussion after Theorem \ref{thm:lazard_fmlocal_comparison} for a more complete picture and explanations), where all isomorphisms are $p\textrm{-power}$ quasi-isomorphisms, i.e.\ the kernel and the cokernel of the induced map on cohomolgy are killed by a fixed bounded power of $p$.
\begin{center}
	\begin{tikzcd}[column sep=small]
		\Kdphi(\Ftor \Mpi^{\textpd}) \arrow[r] \arrow[d, "\tau_{\leqslant r}", "\wr"'] & C_G(\Kdphi(\Ftor \Delta^{\textpd})) & C_G(\Kphi(\Ftor \Delta^{\textpd, \partial})) \arrow[l, "\sim", "\textup{PL}"'] \arrow[r] & C_G(\Kphi(\Ftor TA_{\crys}))\\
		\Kdphi(\Ftor \Mpi^{[u, v]}) \arrow[d, "\wr"', "\textup{PL}"] & & & C_G(T(r)) \arrow[u, "\wr", "\textup{FES}"']\\
		\Kdphid(\Ftor \Delta_{\varpi}^{[u, v]}) & & & C_G(\Kphi(TA_{\Sbar}(r))) \arrow[u, "\wr", "\textup{AS}"']\\
		\Kphid(\Ftor \Npi^{[u, v]}) \arrow[u, "\wr", "\textup{PL}"'] \arrow[d, "t^{\bullet}", "\tau_{\leqslant r} \hspace{0.2mm} \wr"'] & & & C_{\Gamma}(\Kphi(D_{R_{\infty}}(r))) \arrow[u, "\wr"]\\
		\pazk_{\varphi, \Lie \Gamma}(\Ftor \Npi^{[u, v]}) & & & C_{\Gamma}(\Kphi(\Dpi(r))) \arrow[u, "\wr"]\\
		\pazk_{\varphi, \Lie \Gamma}(\Npi^{[u, v]}(r)) \arrow[u, "\wr", "t^r"'] & \pazk_{\varphi, \Gamma}(\Npi^{[u, v]}(r)) \arrow[l, "\sim"', "\Laz"] & \pazk_{\varphi, \Gamma}(\Npi^{(0, v]+}(r)) \arrow[l, "\sim"', "\textup{can}"] \arrow[r, "\sim"] & \textup{K}_{\varphi, \Gamma}(D_{\varpi}(r)) \arrow[u, "\wr"].
	\end{tikzcd}
\end{center}
In the diagram, we set $\Mpi^{\bmstar} = \Rpi^{\bmstar} \otimes_R M$, $\Npi^{\bmstar} = \ARpi^{\bmstar} \otimes_{\AR^+} \mbfn(T)$, $\Npi^{\bmstar}(r) = \ARpi^{\bmstar} \otimes_{\AR^+} \mbfn(T(r))$, $\Delta^{\textpd} = \fERbar^{\textpd} \otimes_R M$, $\Delta^{\textpd, \partial} = (\Delta^{\textpd})^{\partial=0}$, $\Delta_{\varpi}^{[u, v]} = \fERpi^{[u, v]} \otimes_R M$ and $TA_{\crys} = \Acrys(\Rbar) \otimes_{\ZZ_p} T$.
Moreover, using the rings from the theory of $(\varphi, \Gamma)\textrm{-modules}$ (see \S \ref{subsec:relative_phi_gamma_mod}), we set $TA^{[u, v]} = \ARbar^{[u, v]} \otimes_{\ZZ_p} T$, $TA_{\Rbar}(r) = \ARbar \otimes_{\ZZ_p} T(r)$, $\Dpi(r) = \ARpi \otimes_{\AR^+} \mbfn(T(r))$ (see \S \ref{subsec:cyclotomic_embeddings} for $\ARpi$), and $D_{R_{\infty}}(r) = \ARinfty \otimes_{\ARpi} \Dpi(r)$.
Furthermore, we have $G = G_S$, $\Gamma = \Gamma_S$ with $C_G$ and $C_{\Gamma}$ denoting the complex of continuous cochains for $G$ and $\Gamma$, respectively.
The letter ``K'' denotes the Koszul complex with subscripts: $\partial$ denotes the operators $((1+X_0) \frac{\partial}{\partial X_0}, \ldots, X_d \frac{\partial}{\partial X_d})$, the subscript $\Gamma$ denotes the operators $(\gamma_0-1, \ldots, \gamma_d-1)$ for our choice of topological generators of $\Gamma$, the subscript $\Lie \Gamma$ denotes the operators $(\nabla_0, \ldots, \nabla_d)$, with $\nabla_i = \log \gamma_i$ and the subscript $\partial_A$ denotes $((1+X_0)\frac{\partial}{\partial X_0}, X_1\frac{\partial}{\partial X_1}, \ldots, X_d \frac{\partial}{\partial X_d})$ as operators on $\AR^{[u, v]}$ and $E_R^{[u, v]}$ via the isomorphism $\iota_{\cycl} : \Rpi^{[u, v]} \isomorphic \ARpi^{[u, v]}$.
The letter ``$\pazk$'' denotes a certain subcomplex of the Koszul complex (see \S \ref{subsec:diff_to_lie}, \S \ref{subsec:lie_to_gamma}, \S \ref{subsec:change_annulus_1}, \S \ref{subsec:change_annulus_2}).

Let us now describe the maps in the diagram.
FES denotes a map coming from the fundamental exact sequences in \eqref{eq:fes_acrys} and \eqref{eq:fes_aruv}.
AS denotes a map originating from the Artin-Schreier theory in \eqref{eq:artin_schreier_arbar}.
PL denotes the maps coming from the filtered Poincar\'e Lemma of \S \ref{subsec:fat_period_rings}.
In the first column, the map from the first to the second row is induced by the inclusion $\Rpi^{\textpd} \subset \Rpi^{[u, v]}$ (the $p\textrm{-power}$ quasi-isomorphism is shown by using the operator $\psi$ - left inverse of $\varphi$ - and $p\textrm{-power}$ acyclicity of the $\psi=0$ eigencomplexes similar to \cite[\S 3]{colmez-niziol-nearby-cycles}, see \S \ref{subsec:change_disk_convergence} and \S \ref{subsec:change_annulus_converge}); the maps from the second to the third row and from the fourth to the third row are applications of the filtered Poincar\'e Lemma (see \S \ref{subsec:differential_koszul} and \S \ref{subsec:poincare_lemma_app}, in particular, Proposition \ref{prop:syntomic_to_phi_gamma}); the map from the fourth to the fifth row is given by multplication by suitable powers of $t$, exploiting the relation $\partial_i = (\log \gamma_i)/t$, and the map from the sixth to the fifth row is multiplication by $t^r$ (see \S \ref{subsec:diff_to_lie}).
In the fourth column, the map from the fourth to the third row is the inflation map from $\Gamma_S$ to $G_S$, using the inclusion $\ARinfty \subset \ARbar$ (one could use almost \'etale descent to obtain the quasi-isomorphism); the map from the fifth to the fourth row uses the inclusion $\ARpi \subset \ARinfty$ (the quasi-isomorphism is obtained by decompletion techniques); the map from the sixth to the fifth row is the comparison between the complex computing the continuous cohomology of $\Gamma_S$ and the Koszul complex as in \S \ref{subsec:gal_coho_kos_complex}.
The top two maps from the first to the second column are induced by the respective inclusions $\Rpi^{\textpd} \subset \fESbar^{\textpd}$ and $\Rpi^{[u, v]} \subset \fESbar^{[u, v]}$.
The bottom map $\Laz$ between the first and the second column is the Lazard isomorphism discussed in \S \ref{subsec:lie_to_gamma}.
The bottom map from the third to the second column is induced canonically from the inclusion $\ARpi^{(0, v]+} \subset \ARpi^{[u, v]}$ (see \S \ref{subsec:change_annulus_1}).
From the third to the fourth column, the top horizontal map is induced similar to \eqref{eq:pdpartil0_acrist} and the bottom horizontal map is induced by the inclusion $\ARpi^{(0, v]+} \subset \ARpi$ (the $p\textrm{-power}$ quasi-isomorphism is proven by using the operator $\psi$ - left inverse of $\varphi$ - and $p\textrm{-power}$ acyclicity of the $\psi=0$ eigencomplexes, a standard technique in the theory of $(\varphi, \Gamma)\textrm{-modules}$, see \S \ref{subsec:change_annulus_2} and \S \ref{subsec:change_disk}).

Composition of the left vertical, bottom horizontal and right vertical arrows produces the $p\textrm{-power}$ quasi-isomorphism $\alpha_r^{\Laz}$ of Theorem \ref{intro_thm:syntomic_complex_galois_cohomology}; composition of the top horizontal arrows gives the $\padic$ version of the map $\tilde{\alpha}_{r, n, S}^{\FM}$ of Theorem \ref{intro_thm:lazard_fmlocal_comparison}.
The proof of Theorem \ref{intro_thm:syntomic_complex_galois_cohomology} follows from the discussion above and the proof of Theorem \ref{intro_thm:lazard_fmlocal_comparison} is the content of \S \ref{subsec:fm_maps_comparison}.

In \S \ref{sec:crystals_syntomic_cohomology} we describe our global setup and define the syntomic complex with coefficients globally.
In \S \ref{subsec:global_fm_mods} and \S \ref{subsec:fm_period_map}, we describe global relative Fontaine--Laffaille modules and the global Fontaine--Messing period map as in \cite[\S 5]{tsuji-syntomic-complex} and \cite[\S 3.1]{tsuji-semistable-comparison}.
Finally, in \S \ref{subsec:global_result} we state and prove Theorem \ref{intro_thm:syntomic_nearby_comparison}, by first reducing the problem to the local setting via cohomological descent \cite{tsuji-syntomic-complex, tsuji-semistable-comparison} and then to the computation of Galois cohomology by a $K(\pi, 1)\textrm{-Lemma}$ \cite{scholze-rigid}, whence the claim follows from Corollary \ref{intro_cor:syntomic_complex_galois_cohomology}.

\begin{nota}
	Let $ f : C_1 \rightarrow C_2 $ be a morphism of complexes.
	The \textit{mapping cone} of $f$ is the complex $\textup{Cone}(f)$ whose degree $n$ part is given as $C_1^{n+1} \bmoplus C_2^{n}$ and the differential is given by $d(c_1, c_2) = (-d(c_1), d(c_2)-f(c_1))$.
	Furthermore, we denote the \textit{mapping fibre} of $f$ by $\big[C_1 \xrightarrow{\hspace{1mm} f \hspace{1mm}} C_2\big] := \textrm{Cone}(f)[-1]$.
	We also set
	\begin{displaymath}
		\left[
			\vcenter
			{
				\xymatrix
				{
					C_1 \ar[r]^{f} \ar[d] & C_2 \ar[d] \\
					C_3 \ar[r]^{g} & C_4
				}
			}
		\right] := \big[\big[C_1 \xrightarrow{\hspace{1mm} f \hspace{1mm}} C_2\big] \longrightarrow \big[C_3 \xrightarrow{\hspace{1mm} g \hspace{1mm}} C_4\big]\big].
	\end{displaymath}
	In other words, this amounts to taking the total complex of the associated double complex.
\end{nota}

\subsubsection*{Acknowledgements}

Local results of the paper were part of my PhD thesis at Université de Bordeaux.
I would like to thank my advisor Denis Benois for several discussions related to the project as well as guidance and motivation during my time in Bordeaux.
I would also like to thank Nicola Mazzari for helpful discussions related to crystalline cohomology and syntomic coefficients.
Ideas employed in this paper have been heavily influenced by the article \cite{colmez-niziol-nearby-cycles} of Colmez and Nizio{\l} and I would like to thank them for their work.
I would also like to thank Takeshi Tsuji for discussions concerning relative Fontaine--Laffaille modules.
Finally, I would like to thank the referee for their feedback and many valuable remarks and suggestions for improvements.
The last part of the project was carried out at Université de Lille where I was supported by ANR project GALF (ANR-18-CE40-0029) and I-SITE ULNE project PAFAGEO (ANR-16-IDEX-0004).


\section{Relative \texorpdfstring{$p$}{-}-adic Hodge theory}\label{sec:relative_padic_Hodge_theory}

In this section we will recall some constructions and results in local relative $\padic$ Hodge theory from \cite{andreatta-generalized-phiGamma, brinon-padicrep-relatif, andreatta-brinon-surconvergence} and describe some properties of the objects to be considered in sections \S \ref{sec:relative_finite_height} -- \S \ref{sec:syntomic_galcoh}.

\subsection{Setup and notations}\label{subsec:setup_nota}

Let $p \geqslant 3$ be a fixed prime, $\kappa$ a perfect field of characteristic $p$, set $O_F := W(\kappa)$ the ring of $p\textrm{-typical}$ Witt vectors with coefficients in $\kappa$ and $F := O_F[1/p]$.
Let $\overline{F}$ be a fixed algebraic closure of $F$ so that its residue field $\overline{\kappa}$ is an algebraic closure of $\kappa$ and set $G_F := \Gal(\overline{F}/F)$.

\begin{conv}
	We will work under the convention that $0 \in \NN$, the set of natural numbers.
\end{conv}

Let $Z = (Z_1, \ldots, Z_s)$ denote a set of indeterminates and for $\smbfk = (k_1, \ldots, k_s) \in \NN^s$ a multi-index, we will write $Z^{\smbfk} := Z_1^{k_1} \cdots Z_s^{k_s}$.
For a topological algebra $\Lambda$, we set $\Lambda\{Z\} := \big\{\sum_{\smbfk \in \NN^s} a_{\smbfk} Z^{\smbfk}, \hspace{1mm} \textrm{where} \hspace{1mm} a_{\smbfk} \in \Lambda \hspace{1mm} \textrm{and} \hspace{1mm} p\textrm{-adically} \hspace{1mm} a_{\smbfk} \rightarrow 0 \hspace{1mm} \textrm{as} \hspace{1mm} |\smbfk| = \sum k_i \rightarrow +\infty \big\}$.

\begin{assum}\label{assum:small_algebra}
	Fix $d \in \NN$ and $X = (X_1, X_2, \ldots, X_d)$ a set of indeterminates.
	Let $R$ be the $\padic$ completion of an \'etale algebra over $O_F\{X, X^{-1}\}$ with non-empty geometrically integral special fibre.
	In particular, we assume that $R = O_F\{X, X^{-1}\}\{Z_1, \ldots, Z_s\} / \hspace{0.5mm} (Q_1, \ldots, Q_s)$, where $Q_i(Z_1, \ldots, Z_s)$ is in $O_F\{X, X^{-1}\}[Z_1, \ldots, Z_s]$ for $1 \leqslant i \leqslant s$, are multivariate polynomials such that $\det \big(\frac{\partial Q_i}{\partial Z_j}\big)_{1 \leqslant i,j \leqslant s}$ is invertible in $R$.
\end{assum}

Fix an algebraic closure $\overline{\Fr(R)}$ of $\Fr(R)$ containing $\overline{F}$.
Let $\Rbar$ denote the union of finite $R\textrm{-subalgebras}$ $S \subset \overline{\Fr(R)}$, such that $S[1/p]$ is \'etale over $R[1/p]$.
Let $\overline{\eta}$ denote a geometric point of the generic fibre $\Sp(R[1/p])$ (corresponding to $\overline{\Fr(R)}$) and let $G_R := \pi_1^{\etale}\big(\Sp(R[1/p]), \overline{\eta}\big) = \Gal\big(\Rbar[1/p] / R[1/p]\big)$ denote the \'etale fundamental group.

For $n \in \NN$, let $F_n := F(\mu_{p^n})$.
Fix some $m \in \NN_{\geqslant 1}$ and set $K := F_m$, with ring of integers $O_K$.
The element $\varpi = \zeta_{p^m}-1$ in $O_K$ is a uniformiser of $K$ and its minimal polynomial $P_{\varpi}(X) := \big((1+X)^{p^m}-1\big)/\big((1+X)^{p^{m-1}}-1\big)$ is an Eisenstein polynomial in $O_F[X]$ of degree $e := [K:F] = p^{m-1}(p-1)$.
Moreover, $S = R[\varpi] = O_K \otimes_{O_F} R$ is totally ramified over the prime $(p) \subset R$.
Observe that, we have Galois groups $G_K \triangleleft G_F$ and $G_S \triangleleft G_R$ respectively, such that $G_R / G_S = G_F / G_K = \Gal(K/F)$.
Moreover, $R$ and $R[\varpi]$ are \textit{small} algebras in the sense of Faltings (see \cite[\S II 1(a)]{faltings-padic-hodge-theory}).

For $k \in \NN$, let $\Omega^k_R$ denote the $\padic$ completion of the module of $k\textrm{-differentials}$ of $R$ relative to $\ZZ$.
Then, we have $\Omega^1_R = \oplus_{i=1}^d R \dlog X_i$ and $\Omega^k_R = \bmwedge^k_R \Omega^1_R$.
More explicitly, for $1 \leqslant i \leqslant d$, let us set $\partial_i := X_i \frac{d}{dX_i}$ as an operator on $R$.
Then, for any $f$ in $R$, its differential can be written as $df = \sum_{i=1}^d \partial_i(f) \dlog X_i$ in $\Omega^1_R$.
Furthermore, note that $R/pR \isomorphic S/\varpi S$ and for any $n \in \NN$, $R / p^nR$ is a smooth $\ZZ / p^n\ZZ\textrm{-algebra}$.
Finally, we fix a lift $\varphi : R \rightarrow R$ of the absolute Frobenius $x \mapsto x^p$ over $R/p R$ such that $\varphi(X_i) = X_i^p$ for $1 \leqslant i \leqslant d$.

Note that to carry out some computations in later sections, we will need to extend our base field (hence the base ring) by adjoining a $p\textrm{-power}$ root of unity (see $K$ and $S = R[\varpi]$ above).
As a consequence, we will also require period rings defined for such rings.
However, we will only recall the results by fixing our base as $R$, because the period rings that we consider will only depend on $\Rbar$ and we have $\overline{S} = \Rbar \subset \overline{\Fr(R)} = \overline{\Fr(S)}$ (see \cite{andreatta-generalized-phiGamma, brinon-padicrep-relatif, andreatta-brinon-surconvergence} for general constructions).

\begin{conv}
	Let $A$ be a ring and $I \subsetneq A$ an ideal.
	An $A\textrm{-module}$ $M$ is said to be $I\textrm{-adically}$ complete if $M \isomorphic \lim_n M / I^n M$.
\end{conv}

\begin{nota}
	Let $A$ be a $\ZZ_p\textrm{-algebra}$.
	A morphism $f : M \rightarrow N$ of two $A\textrm{-modules}$ is said to be a \textit{$p^n\textrm{-isomorphism}$}, for some $n \in \NN$, if the kernel and cokernel of $f$ are killed by $p^n$.
\end{nota}

\subsection{Period rings}\label{subsec:period_rings}

Let $\CC_p$ denote the $\padic$ completion of $\overline{F}$.
Recall that $\Rbar$ is the union of finite $R\textrm{-subalgebras}$ $S \subset \overline{\Fr(R)} = \overline{\Fr(R[\varpi])}$, such that $S[1/p]$ is \'etale over $R[1/p]$.
Let $\CC^+(\Rbar)$ denote the $\padic$ completion of $\Rbar$ and $\CC(\Rbar) = \CC^+(\Rbar)[1/p]$.
We define the tilt $\CC^+(\Rbar)$ as $\CC^+(\Rbar)^{\flat} := \lim_{x \mapsto x^p} \CC^+(\Rbar) / p = \lim_{x \mapsto x^p} \Rbar/p$ and equip it with the inverse limit topology (where we equip $\Rbar/p$ with the discrete topology), and let $\CC(\Rbar)^{\flat} := \CC^+(\Rbar)^{\flat}[1/p^{\flat}]$, for $p^{\flat} := (p, p^{1/p}, p^{1/p^2}, \ldots) \in \CC^+(\Rbar)^{\flat}$, and equipped with the coarsest ring topology such that $\CC^+(\Rbar)^{\flat}$ is an open subring.
These rings admit a continuous action of $G_R$.

Let us fix $\varepsilon := (1, \zeta_p, \zeta_{p^2}, \ldots)$ in $\CC_p^{\flat}$ and $X_i^{\flat} := \big(X_i, X_i^{1/p}, X_i^{1/p^2}, \ldots\big)$ in $\CC^+(\Rbar)^{\flat}$, for $1 \leqslant i \leqslant d$.
Set $\Ainf(\Rbar) := W(\CC^+(\Rbar)^{\flat})$, the ring of $p\textrm{-typical}$ Witt vectors with coefficients in $\CC^+(\Rbar)^{\flat}$.
The absolute Frobenius on $\CC^+(\Rbar)^{\flat}$ lifts to an endomorphism $\varphi : \Ainf(\Rbar) \rightarrow \Ainf(\Rbar)$ and the $G_R\textrm{-action}$ extends to a continuous (for the weak topology, see \cite[\S 2.10]{andreatta-iovita-relative-phiGamma}) action on $\Ainf(\Rbar)$.
For $x \in \CC^+(\Rbar)^{\flat}$, let $[x] = (x, 0, 0, \ldots)$ in $\Ainf(\Rbar)$ denote its Teichm\"uller representative.
So we have $[\varepsilon]$ in $\Ainf(\Rbar)$ with $\varphi([\varepsilon]) = [\varepsilon]^p$ and $g[\varepsilon] = [\varepsilon]^{\chi(g)}$, for $g$ in $G_R$ and $\chi : G_R \rightarrow \ZZ_p^{\times}$ the $\padic$ cyclotomic character.
Furthermore, let $\pi := [\varepsilon] - 1, \pi_1 := \varphi^{-1}(\pi) = [\varepsilon^{1/p}] - 1$, and $\xi := \pi/\pi_1$.
Clearly, we have that $g(\pi) = (1 + \pi)^{\chi(g)} - 1$ for $g \in G_R$ and $\varphi(\pi) = (1+\pi)^p - 1$.

We will use the de Rham period rings $\BdR^+(\Rbar)$ and $\BdR(\Rbar)$ defined in \cite[Chapitre 5]{brinon-padicrep-relatif} and \cite[\S 2.1]{abhinandan-crystalline-wach}.
These are $F\textrm{-algebras}$ equipped with a natural action of $G_R$ and a $G_R\textrm{-stable}$ filtration.
We have that $t := \log[\varepsilon] = \log(1+\pi) = \sum_{k \in \NN} (-1)^k \frac{\pi^{k+1}}{k+1}$ converges in $\BdR^+(\Rbar)$ and the action on $t$ of any $g$ in $G_R$ can be described by the formula $g(t) = \chi(g) t$.
Moreover, we will use fat period rings $\OBdR^+(\Rbar)$ and $\OBdR(\Rbar)$ defined in \cite[Chapitre 5]{brinon-padicrep-relatif} and \cite[\S 2.1]{abhinandan-crystalline-wach}.
These are $R[1/p]\textrm{-algebras}$ equipped with a natural action of $G_R$, a $G_R\textrm{-stable}$ filtration and a $G_R\textrm{-equivariant}$ connection satisfying Griffiths transversality with respect to the filtration.
Furthermore, we have $\big(\OBdR^+(\Rbar)\big)^{\partial=0} = \BdR^+(\Rbar)$, $\big(\OBdR(\Rbar)\big)^{\partial=0} = \BdR(\Rbar)$ and $\big(\OBdR(\Rbar)\big)^{G_R} = R[1/p]$.

We will also use the crystalline period rings $\Acrys(\Rbar)$, $\Bcrys^+(\Rbar)$ and $\Bcrys(\Rbar)$, from \cite[Chapitre 6]{brinon-padicrep-relatif} and \cite[\S 2.2]{abhinandan-crystalline-wach}, as subrings of $\BdR(\Rbar)$.
The ring $\Acrys(\Rbar)$ is an $O_F\textrm{-algebra}$ and $\Bcrys^+(\Rbar)$ and $\Bcrys(\Rbar)$ are $F\textrm{-algebras}$.
These rings are equipped with a natural action of $G_R$, a $G_R\textrm{-stable}$ filtration (induced from the filtration on $\BdR(\Rbar)$) and a $G_R\textrm{-equivariant}$ Frobenius endomorphism $\varphi$.
Note that $t$ converges in $\Acrys(\Rbar)$ and $\varphi(t) = pt$.
Moreover, we will use fat period rings $\OAcrys(\Rbar)$, $\OBcrys^+(\Rbar)$ and $\OBcrys(\Rbar)$ defined in \cite[Chapitre 6]{brinon-padicrep-relatif} and \cite[\S 2.2]{abhinandan-crystalline-wach} as subrings of $\OBdR(\Rbar)$.
The ring $\OAcrys(\Rbar)$ is an $R\textrm{-algebra}$ and $\OBcrys^+(\Rbar)$ and $\OBcrys(\Rbar)$ are $R[1/p]\textrm{-algebras}$.
These rings are equipped with a natural action of $G_R$, a $G_R\textrm{-stable}$ induced filtration (from $\OBdR(\Rbar)$), a $G_R\textrm{-equivariant}$ Frobenius endomorphism $\varphi$ and a $G_R\textrm{-equivariant}$ induced connection (from $\OBdR(\Rbar)$), satisfying Griffiths transversality with respect to the filtration and commuting with $\varphi$.
Taking the horizontal sections for the connection, we get that $\big(\OAcrys(\Rbar)\big)^{\partial=0} = \Acrys(\Rbar)$, $\big(\OBcrys^+(\Rbar)\big)^{\partial=0} = \Bcrys^+(\Rbar)$, $\big(\OBcrys(\Rbar)\big)^{\partial=0} = \Bcrys(\Rbar)$, and by taking $G_R\textrm{-invariants}$ we get that $\big(\OAcrys(\Rbar)\big)^{G_R} = R$ and $\big(\OBcrys^+(\Rbar)\big)^{G_R} = \big(\OBcrys(\Rbar)\big)^{G_R} = R[1/p]$.

\subsubsection{Fundamental exact sequence}

From the Artin-Schrier theory in \cite[\S 8.1.1]{andreatta-iovita-relative-phiGamma}, we have an exact sequence
\begin{equation}\label{eq:artin_schreier_ainf}
	0 \longrightarrow \ZZ_p \longrightarrow \Ainf(\Rbar) \xrightarrow{\hspace{1mm} 1-\varphi \hspace{1mm}} \Ainf(\Rbar) \longrightarrow 0.
\end{equation}
Let $r \in \NN$ and write $r = (p-1)a(r) + b(r)$, with $0 \leqslant b(r) < p-1$, and set $\ZZ_p(r)\prm = \frac{1}{p^{a(r)}} \ZZ_p(r)$.
From \cite[Theorem A3.26]{tsuji-semistable-comparison} and \cite[Lemma 2.23]{colmez-niziol-nearby-cycles}, we have a $p^r\textrm{-exact}$ sequence called the fundamental exact sequence in $\padic$ Hodge theory:
\begin{equation}\label{eq:fes_acrys}
	0 \longrightarrow \ZZ_p(r)\prm \longrightarrow \Fil^r \Acrys(\Rbar) \xrightarrow{\hspace{1mm} p^r-\varphi \hspace{1mm}} \Acrys(\Rbar) \longrightarrow 0.
\end{equation}

\subsection{\texorpdfstring{$p$}{-}-adic Galois representations}\label{subsec:relative_padic_reps}

For the ring $B = \OBdR(\Rbar)$ and $\OBcrys(\Rbar)$, we will consider \textit{$B\textrm{-admissible}$} $\padic$ representations in the sense of \cite[Chapitre 8]{brinon-padicrep-relatif} and \cite[\S 2.3]{abhinandan-crystalline-wach}.
Note that $\OBdR(\Rbar)$ is a $G_R\textrm{-regular}$ $R[1/p]\textrm{-algebra}$.
Let $V$ be a $\padic$ representation of $G_R$ and we set $\ODdR(V) := \big(\OBdR(\Rbar) \otimes_{\QQ_p} V\big)^{G_R}$.
We say that $V$ is de Rham if it is $\OBdR(\Rbar)\textrm{-admissible}$.
The $R[1/p]\textrm{-module}$ $\ODdR(V)$ is equipped with a decreasing, separated and exhaustive filtration and an integrable connection satisfying Griffiths transversality with respect to the filtration (all induced from the corresponding structures on $\OBdR(\Rbar) \otimes_{\QQ_p} V$).
Furthermore, $\ODdR(V)$ is projective over $R[1/p]$ and of rank $\leqslant \dim(V)$.
If $V$ is de Rham, then for all $r \in \ZZ$, the $R[1/p]\textrm{-modules}$ $\Fil^r \ODdR(V)$ and $\gr^r \ODdR(V)$ are projective of finite type and the collection of integers $r_i$, for $1 \leqslant i \leqslant \dim_{\QQ_p}(V)$, such that $\gr^{-r_i} \ODdR(V) \neq 0$ are called the \textit{Hodge--Tate weights} of $V$ (see \cite[\S 8.3]{brinon-padicrep-relatif}).
Moreover, we say that $V$ is \textit{positive} if and only if $r_i \leqslant 0$, for all $1 \leqslant i \leqslant \dim_{\QQ_p}(V)$.

Next, we note that $\OBcrys(\Rbar)$ is also a $G_R\textrm{-regular}$ $R[1/p]\textrm{-algebra}$.
Let $V$ be a $\padic$ representation of $G_R$ and we set $\ODcrys(V) := \big(\OBcrys(\Rbar) \otimes_{\QQ_p} V\big)^{G_R}$.
We say that $V$ is crystalline if it is $\OBcrys(\Rbar)\textrm{-admissible}$.
The $R[1/p]\textrm{-module}$ $\ODcrys(V)$ is equipped with a Frobenius-semilinear operator $\varphi$ induced from the Frobenius on $\OBcrys(\Rbar) \otimes_{\QQ_p} V$, where we consider the $G_R\textrm{-equivariant}$ Frobenius on $\OBcrys(\Rbar)$.
Moreover, the inclusion $\OBcrys(\Rbar) \subset \OBdR(\Rbar)$ induces an $R[1/p]\textrm{-linear}$ inclusion $\ODcrys(V) \subset \ODdR(V)$ (see \cite[\S 8.2 and \S 8.3]{brinon-padicrep-relatif}), and we equip $\ODcrys(V)$ with (induced from $\ODdR(V)$) filtration and connection satisfying Griffiths transversality with respect to the filtration.
Additionally, we have $\partial \varphi = \varphi\partial$ over $\ODcrys(V)$.
The module $\ODcrys(V)$ is finite projective over $R[1/p]$ of rank $\leqslant \dim(V)$.
If $V$ is crystalline, then the $R[1/p]\textrm{-linear}$ homomorphism $1 \otimes \varphi : R[1/p] \otimes_{\varphi, R[1/p]} \ODcrys(V) \rightarrow \ODcrys(V)$ is an isomorphism and $\ODcrys(V)$ is called a filtered $(\varphi, \partial)\textrm{-module}$.

\subsection{\texorpdfstring{$(\varphi, \Gamma)$}{-}-modules}\label{subsec:relative_phi_gamma_mod}

In this subsection, we will briefly recall some results from the theory of relative \'etale $(\varphi, \Gamma)\textrm{-modules}$ (see \cite{andreatta-generalized-phiGamma, andreatta-brinon-surconvergence, andreatta-iovita-relative-phiGamma} for details).

\subsubsection{The Galois group \texorpdfstring{$\Gamma_R$}{-}}

Let $F_n = F(\mu_{p^n})$, for $n \in \NN$, and $F_{\infty} = \cup_n F_n$.
We take $R_n$ to be the integral closure of $R \otimes_{O_F[X^{\pm 1}]} O_{F_n}\big[X_1^{p^{-n}}, \ldots X_d^{p^{-n}}\big]$ inside $\Rbar[1/p]$ and set $R_{\infty} := \cup_{n \geqslant m} R_n$, noting that $F_{\infty} \subset R_{\infty}[1/p]$.
From \S \ref{subsec:period_rings} recall that $\CC(\Rbar) = \CC^+(\Rbar)[1/p]$ and $\CC(\Rbar)^{\flat}$ denotes its tilt.
The ring $\CC(\Rbar)^{\flat}$ is perfect of characteristic $p$ and we set $\ARbar := W(\CC(\Rbar)^{\flat})$, the ring of $p\textrm{-typical}$ Witt vectors with coefficients in $\CC(\Rbar)^{\flat}$, and endow it with the weak topology (see \cite[\S 2.10]{andreatta-iovita-relative-phiGamma}).
The absolute Frobenius over $\CC(\Rbar)^{\flat}$ lifts to an endomorphism $\varphi : \ARbar \rightarrow \ARbar$, which we again call the Frobenius.
The continuous action of $G_{R}$ on $\CC(\Rbar)^{\flat}$ extends to a continuous action on $\ARbar$ commuting with Frobenius.
The inclusion $\overline{F} \subset \Rbar[1/p]$ induces inclusions $\CC_p^{\flat} \subset \CC(\Rbar)^{\flat}$ and $\mbfa_{\overline{F}} \subset \ARbar$, and the inclusion $O_{\overline{F}} \subset \Rbar$ induces inclusions $O_{\CC_p}^{\flat} \subset \CC^+(\Rbar)^{\flat} \hspace{2mm} \textrm{and} \hspace{2mm} \Ainf(O_{\overline{F}}) \subset \Ainf(\Rbar)$.

The ring $R_{\infty}[1/p]$ is Galois over $R[1/p]$ with Galois group $\Gamma_{R} := \Gal\big(R_{\infty}[1/p] / R[1/p]\big)$.
Let $\Gamma_F := \Gal(F_{\infty}/F) \isomorphic \ZZ_p^{\times}$ and $\Gamma_{R}\prm := \Gal\big(R_{\infty}[1/p]/F_{\infty}R[1/p]\big) \isomorphic \ZZ_p(1)^d$ (see \cite[p.\ 9]{brinon-padicrep-relatif} and \cite[\S 2.4]{andreatta-generalized-phiGamma}), and note that we have an exact sequence,
\begin{equation}\label{eq:gammar_semidirect_product}
	1 \longrightarrow \Gamma_R\prm \longrightarrow \Gamma_R \longrightarrow \Gamma_F \longrightarrow 1.
\end{equation}
The group $\Gamma_F$ can be viewed as a subgroup of $\Gamma_{R}$, i.e.\ we can take a section of the projection map in \eqref{eq:gammar_semidirect_product} such that for $\gamma \in \Gamma_F$ and $g \in \Gamma_{R}\prm$, we have $\gamma g \gamma^{-1} = g^{\chi(\gamma)}$.
So we can choose topological generators $\{\gamma, \gamma_1, \ldots, \gamma_d\}$ of $\Gamma_R$, such that $\gamma_0 = \gamma^e$, with $\chi(\gamma_0) = \exp(p^m)$, is a topological generator of $\Gamma_K = \Gal(K_{\infty} / K)$, where $K_{\infty} = F_{\infty}$ and $e = [K:F]$.
It follows that $\{\gamma_1, \ldots, \gamma_d\}$ are topological generators of $ \Gamma_{R}\prm $ and $\gamma$ is a topological generator of $\Gamma_F$.
In particular, we have $\chi : \Gamma_K = \Gal(F_{\infty}/K) \isomorphic 1 + p^{m}\ZZ_p$.
The action of these generators on the elements of $\CC(\Rbar)^{\flat}$, fixed in \S \ref{subsec:period_rings}, is given as $\gamma(\varepsilon) = \varepsilon^{\chi(\gamma)}$ and $\gamma_i(\varepsilon) = \varepsilon$, for $1 \leqslant i \leqslant d$; $\gamma_i(X_i^{\flat}) = \varepsilon X_i^{\flat}$ and $\gamma_i(X_j^{\flat}) = X_j^{\flat}$, for $i \neq j$ and $1 \leqslant j \leqslant d$.

\subsubsection{\'Etale \texorpdfstring{$(\varphi, \Gamma_R)$}{-}-modules}

In \cite{andreatta-generalized-phiGamma}, Andreatta introduced the theory of \'etale $(\varphi, \Gamma_R)\textrm{-modules}$ for $\padic$ representations of $G_R$ (see \cite[\S 3.1]{abhinandan-crystalline-wach} for a quick recollection).
From loc.\ cit., let us recall that we have characteristic $p$ period rings $\mbfe^+ \subset \mbfe \subset \CC(\Rbar)^{\flat}$.
Let $\overline{\pi}$ denote the reduction modulo $p$ of $\pi \in \Ainf(O_{F_{\infty}})$.
Then, the characteristic $p$ period rings above are $\overline{\pi}\textrm{-adically}$ complete and equipped with a continuous $G_R\textrm{-action}$.
Furthermore, we have rings $\mbfe^+_R \subset \mbfe_R \subset \widehat{R}_{\infty}^{\flat}[1/p^{\flat}]$, complete for the $\overline{\pi}\textrm{-adic}$ topology and equipped with a continuous $G_R\textrm{-action}$.
Moreover, we have $\big(\CC^+(\Rbar)\big)^{H_R} = \widehat{R}_{\infty}$, $\big(\CC^+(\Rbar)^{\flat}\big)^{H_R} = \widehat{R}_{\infty}^{\flat}$, $\big(\CC(\Rbar)^{\flat}\big)^{H_R} = \widehat{R}_{\infty}^{\flat}[1/p^{\flat}]$, $(\mbfe^+)^{H_R} = \mbfe_R^+$ and $\mbfe^{H_R} = \mbfe_R$.

In mixed characteristic, we have period rings $\mbfa^+ \subset \mbfa \subset W(\CC(\Rbar)^{\flat})$ equipped with an induced weak topology, an induced Frobenius endomorphism $\varphi$ and a continuous $G_R\textrm{-action}$.
Furthermore, we have $\AR^+ = \AR \subset W\big(\widehat{R}_{\infty}^{\flat}[1/p^{\flat}]\big)$, complete for the induced weak topology and equipped with an induced Frobenius and a continuous $\Gamma_R\textrm{-action}$.
Additionally, from \cite{andreatta-iovita-relative-phiGamma} we have that $\mbfa^{H_R} = \AR$, $(\mbfa^+)^{H_R} = \AR^+$ and $\mbfa/p\mbfa = \mbfe$, and from \cite[Remark 3.7]{abhinandan-crystalline-wach} we have that $\mbfa^+/p\mbfa^+ = \mbfe^+$.

Let $D$ be a finitely generated $\AR\textrm{-module}$ equipped with a continuous (for the weak topology) and semilinear action of $\Gamma_R$ and a Frobenius-semilinear and $\Gamma_R\textrm{-equivariant}$ endomorphism $\varphi$.
\begin{defi}\label{defi:phigamma_modules}
	The $\AR\textrm{-module}$ $D$ is said to be \textit{\'etale} if the linearisation of Frobenius, i.e.\ the natural map $1 \otimes \varphi : \AR \otimes_{\varphi, \AR} D \rightarrow D$, is an isomorphism.
\end{defi}

Denote by $(\varphi, \Gamma_R)\textup{-Mod}_{\AR}^{\etale}$ the category of \'etale $(\varphi, \Gamma_R)\textrm{-modules}$ over $\AR$ with morphisms between objects being continuous and $(\varphi, \Gamma_R)\textrm{-equivariant}$ morphisms of $\AR\textrm{-modules}$.
Furthermore, denote by $\Rep_{\ZZ_p}(G_R)$ the category of finitely generated $\ZZ_p\textrm{-modules}$ equipped with a linear and continuous $G_R\textrm{-action}$ and morphisms between objects being continuous and $G_R\textrm{-equivariant}$ morphisms of $\ZZ_p\textrm{-modules}$.
Let $T$ denote a $\ZZ_p\textrm{-representation}$ of $G_R$, then $\mbfd(T) := (\mbfa \otimes_{\ZZ_p} T)^{H_R}$ is an \'etale $(\varphi, \Gamma_R)\textrm{-module}$ over $\AR$.
Furthermore, if $T$ is finite free over $\ZZ_p$, then $\mbfd(T)$ is finite projective over $\AR$, of rank $= \textup{rk}_{\ZZ_p} T$ (see \cite[Theorem 7.11]{andreatta-generalized-phiGamma}).
Finally, the functor $\mbfd : \Rep_{\ZZ_p}(G_R) \rightarrow (\varphi, \Gamma_R)\textup{-Mod}_{\AR}^{\etale}$, induces an equivalence of categories (see \cite[Theorem 7.11]{andreatta-generalized-phiGamma}).

\subsubsection{Overconvergent \'etale \texorpdfstring{$(\varphi, \Gamma_R)$}{-}-modules}\label{subsubsec:overconvergence}

In this subsection, we will quickly recall the theory of overconvergent relative \'etale $(\varphi, \Gamma)\textrm{-modules}$ from \cite{andreatta-brinon-surconvergence}, which generalises the classical results of \cite{cherbonnier-colmez-surconvergentes}.
Denote the natural valuation on $O_{\CC_p}^{\flat}$ by $\upsilon^{\flat}$ and extend it to a map $\upsilon^{\flat} : \CC^+(\Rbar)^{\flat} \rightarrow \RR \cup \{+\infty\}$ by setting $\upsilon^{\flat}(x) = \frac{p}{p-1}\max\{n \in \QQ, x \in \overline{\pi}^{-n} \CC^+(\Rbar)^{\flat}\}$.
Let $v >0$ and let $\alpha \in O_{\CC_p}^{\flat}$ such that $\upsilon^{\flat}(\alpha) = 1/v$.
Set
\begin{align*}
	\ARbar^{(0, v]} &:= \big\{\textstyle\sum_{k \in \NN} p^k [x_k] \in \ARbar, \hspace{1mm} v\upsilon^{\flat}(x_k) + k \rightarrow +\infty \hspace{1mm} \textrm{when} \hspace{1mm} k \rightarrow +\infty \big\},\\
	\ARbar^{(0, v]+} &:= \big\{\textstyle\sum_{k \in \NN} p^k [x_k] \in \ARbar^{(0, v]} \textrm{ with } v\upsilon^{\flat}(x_k) + k \geqslant 0 \big\} = p\textrm{-adic completion of} \hspace{1mm} \Ainf(\Rbar)\big[p/[\alpha]\big].
\end{align*}
Note that we have $\ARbar^{(0, v]} = \ARbar^{(0, v]+}[1/[p^{\flat}]]$.
The $G_R\textrm{-action}$ on $\Ainf(\Rbar)$ extends to these rings and it commutes with the induced Frobenius $\varphi$, where $\varphi\big(\ARbar^{(0, v]+}\big) = \ARbar^{(0, v/p]+}$ and $\varphi\big(\ARbar^{(0, v]}\big) = \ARbar^{(0, v/p]}$.
Moreover, we have that $\ARbar^{(0, v]+} \subset \BdR^+(\Rbar)$ and $\ARbar^{(0, v]} \subset \BdR(\Rbar)$ for $v \geqslant 1$ (see \cite[\S 2.4.2]{colmez-niziol-nearby-cycles}).
We use these embeddings to induce filtrations on $\ARbar^{(0, v]+}$ and $\ARbar^{(0, v]}$.

\begin{defi}
	Define the ring of \textit{overconvergent coefficients} as $\ARbar^{\dagger} := \cup_{v \in \QQ_{>0}} \ARbar^{(0, v]}$.
	Moreover, inside $\ARbar$, we set $\AR^{(0, v]} := \AR \cap \ARbar^{(0, v]}$ and $\mbfa^{(0, v]} := \mbfa \cap \ARbar^{(0, v]}$.
	Define $\AR^{\dagger} := \AR \cap \ARbar^{\dagger} = \cup_{v \in \QQ_{>0}} \AR^{(0, v]}$ and $\mbfa^{\dagger} := \mbfa \cap \ARbar^{\dagger} = \cup_{v \in \QQ_{>0}} \mbfa^{(0, v]}$.
\end{defi}

The rings defined above are equipped with a topology described in \cite[\S 4]{andreatta-brinon-surconvergence}.
We have an embedding $\ARbar^{\dagger} \subset \ARbar$ compatible with the weak topology on $\ARbar$.
Furthermore, $\ARbar^{\dagger}$ is stable under the induced Frobenius $\varphi$ and the $G_R\textrm{-action}$ which commutes with $\varphi$ (see \cite[Proposition 7.2]{andreatta-generalized-phiGamma}).
Finally, all rings appearing above are equipped with a $(\varphi, G_R)\textrm{-action}$ (induced from $\ARbar$) and from \cite[Lemma 2.11]{andreatta-iovita-relative-phiGamma} we have that $\big(\mbfa^{(0, v]}\big)^{H_R} = \AR^{(0, v]}$, $(\mbfa^{\dagger})^{H_R} = \AR^{\dagger}$ and $\AR^{\dagger}/p\AR^{\dagger} = \mbfe_R$.

Define $(\varphi, \Gamma_R)\textup{-Mod}_{\AR^{\dagger}}^{\etale}$ to be the category of \'etale $(\varphi, \Gamma_R)\textrm{-modules}$ over $\AR^{\dagger}$, similar to Definition \ref{defi:phigamma_modules}.
Let $T \in \Rep_{\ZZ_p}(G_R)$, then $\mbfd^{\dagger}(T) := (\mbfa^{\dagger} \otimes_{\ZZ_p} T)^{H_R}$ is an \'etale $(\varphi, \Gamma_R)\textrm{-module}$ over $\AR^{\dagger}$.
Moreover, if $T$ is finite free over $\ZZ_p$, then $\mbfd^{\dagger}(T)$ is finite projective over $\AR^{\dagger}$ of rank $= \textup{rk}_{\ZZ_p} T$.
The functor $\mbfd^{\dagger} : \Rep_{\ZZ_p}(G_R) \rightarrow (\varphi, \Gamma_R)\textup{-Mod}_{\AR^{\dagger}}^{\etale}$ induces an equivalence of categories (see \cite[Th\'eor\`eme 4.35]{andreatta-brinon-surconvergence}).
Moreover, extension of scalars along $\AR^{\dagger} \rightarrow \AR$ gives an isomorphism of \'etale $(\varphi, \Gamma_R)\textrm{-modules}$ over $\AR$ as $\AR \otimes_{\AR^{\dagger}} \mbfd^{\dagger}(T) \isomorphic \mbfd(T)$.

Finally, we will introduce the analytic rings to be used in \S \ref{sec:syntomic_complex_finite_height}.
Let $0 < u \leqslant v$ and $\alpha, \beta \in O_{\CC_p}^{\flat}$, such that $\upsilon^{\flat}(\alpha) = 1/v$ and $\upsilon^{\flat}(\beta) = 1/u$.
Set $\ARbar^{[u]} := p\textrm{-adic completion of} \hspace{1mm} \Ainf(\Rbar)\big[[\beta]/p]$ and $\ARbar^{[u, v]} := p\textrm{-adic completion of} \hspace{1mm} \Ainf(\Rbar)\big[p/[\alpha], [\beta]/p\big]$.
The $G_R\textrm{-action}$ on $\Ainf(\Rbar)$ extends to these rings and commutes with the extension of Frobenius to these rings, denoted again by $\varphi$.
For the homomorphism $\varphi$, we have that $\varphi\big(\ARbar^{[u]}\big) = \ARbar^{[u/p]}$ and $\varphi\big(\ARbar^{[u, v]}\big) = \ARbar^{[u/p, v/p]}$.
Moreover, we have inclusions $\ARbar^{[u]} \subset \BdR^+(\Rbar)$ for $u \leqslant 1$ and $\ARbar^{[u, v]} \subset \BdR^+(\Rbar)$ for $u \leqslant 1 \leqslant v$ (see \cite[\S 2.4.2]{colmez-niziol-nearby-cycles}).
We use these embeddings to induce filtrations on $\ARbar^{[u]}$ and $\ARbar^{[u, v]}$.

\subsubsection{Fundamental exact sequences}

The Artin-Schreier exact sequence in \eqref{eq:artin_schreier_ainf} can be upgraded to the following exact sequences (see \cite[\S 8.1]{andreatta-iovita-relative-phiGamma} and \cite[Lemma 2.23]{colmez-niziol-nearby-cycles}):
\begin{align}\label{eq:artin_schreier_arbar}
	\begin{split}
		0 &\longrightarrow \ZZ_p \longrightarrow \ARbar \xrightarrow{\hspace{1mm} 1-\varphi \hspace{1mm}} \ARbar \longrightarrow 0,\\
		0 &\longrightarrow \ZZ_p \longrightarrow \ARbar^{(0, v]+} \xrightarrow{\hspace{1mm} 1-\varphi \hspace{1mm}} \ARbar^{(0, v/p]+} \longrightarrow 0, \hspace{2mm} \textrm{for} \hspace{1mm} v > 0.
	\end{split}
\end{align}
Furthermore, for $0 < u \leqslant 1 \leqslant v$, the $p^r\textrm{-exact}$ sequence in \eqref{eq:fes_acrys} can be upgraded to a $p^{4r}\textrm{-exact}$ sequence (see \cite[Lemma 2.23]{colmez-niziol-nearby-cycles}):
\begin{equation}\label{eq:fes_aruv}
	0 \longrightarrow \ZZ_p(r) \longrightarrow \Fil^r \ARbar^{[u, v]} \xrightarrow{\hspace{1mm} p^r-\varphi \hspace{1mm}} \ARbar^{[u, v/p]} \longrightarrow 0.
\end{equation}

\subsubsection{The operator \texorpdfstring{$\psi$}{-}}\label{sec:operator_psi_phi_gamma}

Let us define a left inverse $\psi$ of the Frobenius operator $\varphi$ on the ring $\mbfa$.
From \cite[Corollaire 4.10]{andreatta-brinon-surconvergence} note that the $\mbfa\textrm{-module}$ $\varphi^{-1}(\mbfa)$ is free with a basis given as $u_{\alpha/p} = (1+\pi)^{\alpha_0/p} [X_1^{\flat}]^{\alpha_1/p} \cdots [X_d^{\flat}]^{\alpha_d/p}$, where $\alpha = (\alpha_0, \ldots, \alpha_d)$ is a $(d+1)\textrm{-tuple}$ with $\alpha_i \in \{0, 1, \ldots, p-1\}$ for each $0 \leqslant i \leqslant d$ (note that to get this statement from loc.cit., one should replace $\varphi^{-1}(\mbfa)$ by $\mbfa$ there and take $p\textrm{-th}$ root of the basis elements).
Define an operator (a left inverse of $\varphi$), denoted as $\psi : \mbfa \rightarrow \mbfa$ and given by the formula $x \mapsto \tfrac{1}{p^{d+1}} \circ \textup{Tr}_{\varphi^{-1}(\mbfa)/\mbfa}\circ \varphi^{-1}(x)$.

\begin{prop}[{\cite[\S 4.8]{andreatta-brinon-surconvergence}}]\label{prop:psi_oper}
	Let $x \in \mbfa$ and write $\varphi^{-1}(x) = \sum_{\alpha} x_{\alpha} u_{\alpha/p}$, then we have $\psi(x) = x_0$.
	Moreover, for the operator $\psi$ we have $\psi \circ \varphi = id$.
	Furthermore, $\psi$ commutes with the action of $G_R$, $\psi(\mbfa^+) \subset \mbfa^+$ and $\psi(\mbfa^{\dagger}) \subset \mbfa^{\dagger}$.
\end{prop}

\subsection{Crystalline coordinates}\label{subsec:pd_envelope}

In this subsection, we will introduce good ``crystalline'' coordinates (see \cite[\S 3.2]{abhinandan-crystalline-wach}).
Let $r_{\varpi}^+ = O_F \llbracket X_0 \rrbracket$ and $r_{\varpi} = O_F \llbracket X_0 \rrbracket\{X_0^{-1}\}$.
Sending $X_0$ to $\varpi = \zeta_{p^m}-1$ induces a surjective ring homomorphism $r_{\varpi}^+ \twoheadrightarrow O_K$, whose kernel is generated by a degree $e = [K:F] = p^{m-1}(p-1)$ Eisenstein polynomial $P_{\varpi} = P_{\varpi}(X_0)$.
Let $R^+_{\varpi, \square}$ denote the completion of $O_F[X_0, X, X^{-1}]$ for the $(p, X_0)\textrm{-adic}$ topology.
Sending $X_0$ to $\varpi$ induces a surjective ring homomorphism $R_{\varpi, \square}^+ \twoheadrightarrow O_K\{X, X^{-1}\}$, whose kernel is again generated by $P_{\varpi}$.
Recall that $R$ is \'etale over $O_F\{X, X^{-1}\}$ and we have multivariate polynomials $Q_i(Z_1, \ldots, Z_s) \in O_F\{X, X^{-1}\}[Z_1, \ldots, Z_s]$, for $1 \leqslant i \leqslant s$, such that $\det \big(\frac{\partial Q_i}{\partial Z_j}\big)$ is invertible in $R$.
Set $\Rpi^+$ to be the quotient of the $(p, X_0)\textrm{-adic}$ completion of $R_{\varpi, \square}^+[Z_1, \ldots, Z_s]$ by the ideal $(Q_1, \ldots, Q_s)$.
Again, we have that $\det \big(\frac{\partial Q_i}{\partial Z_j}\big)$ is invertible in $\Rpi^+$ (since $R \hookrightarrow \Rpi^+$).
Hence, $\Rpi^+$ is \'etale over $R_{\varpi, \square}^+$ and smooth over $O_F$.
Sending $X_0$ to $\varpi$ induces a surjective ring homomorphism $\Rpi^+ \twoheadrightarrow R[\varpi]$, whose kernel is again generated by $P_{\varpi}$.
Since $P_{\varpi} \equiv X_0^e \mod p$, we have that $\Rpi^+[P_{\varpi}^k/k!]_{k \in \NN} = \Rpi^+[X_0^k/[k/e]!]_{k \in \NN}$.
Set $\Rpi^{\PD} := p\textrm{-adic}$ completion of $\Rpi^+[P_{\varpi}^k/k!]_{k \in \NN}$.

Recall that $\Omega^1_R$ denotes the $p\textrm{-adic}$ completion of the module of differentials of $R$ relative to $\ZZ$ and we have $\Omega^1_R = \oplus_{i=1}^d R \dlog X_i$ and $\Omega^k_R = \wedge^k_R \Omega^1_R$.
Moreover, since $\Rpi^+$ is \'etale over $R_{\varpi, \square}^+$, therefore, for $S = R_{\varpi, \square}^+$ or $\Rpi^+$, we have that $\Omega^1_{S} = S \tfrac{dX_0}{1+X_0} \oplus \big(\oplus_{i=1}^d S \dlog X_i\big)$.

\begin{defi}\label{def:coordinate_rings}
	For $0 < u \leqslant v$, define $\Rpi^{(0, v]+}$ to be the $p\textrm{-adic}$ completion of $\Rpi^+[p^{\lceil vk/e \rceil}/X_0^k]_{k \in \NN}$ and set $\Rpi^{(0, v]} := \Rpi^{(0, v]+}[1/X_0]$.
	Furthermore, define $\Rpi^{[u]}$ to be the $p\textrm{-adic}$ completion of $\Rpi^+[X_0^k/p^{\lfloor uk/e \rfloor}]_{k \in \NN}$, define $\Rpi^{[u, v]}$ to be the $p\textrm{-adic}$ completion of $\Rpi^+[X_0^k/p^{\lfloor uk/e \rfloor},p^{\lceil vk/e \rceil}/X_0^k]_{k \in \NN}$ and set $\Rpi$ as the $p\textrm{-adic}$ completion of $\Rpi^+[1/X_0]$.
	We will write $\Rpi^{\bmstar}$ for $\smstar \in \{\hspace{1mm}, +, \textpd, [u], (0, v]+, [u, v]\}$ and for the arithmetic case, i.e.\ $R = O_F$, we will write $r_{\varpi}^{\bmstar}$ instead.
	Going from $\Rpi^+$ to $\Rpi^{\bmstar}$ only involves the arithmetic variable $X_0$, so we have $\Rpi^{\bmstar} = r_{\varpi}^{\bmstar} \widehat{\otimes}_{r_{\varpi}^+} \Rpi^+$, where $\widehat{\otimes}$ denotes the $p\textrm{-adic}$ completion of the usual tensor product.
\end{defi}

\begin{rem}
	Unless otherwise stated, we will assume $(p-1)/p \leqslant u \leqslant v/p < 1 < v < p$, for example, we can take $u = (p-1)/p$ and $v = p-1$.
\end{rem}

\begin{defi}\label{defi:filtration_vanishing_varpi}
	Define a filtration on the rings in Definition \ref{def:coordinate_rings} as follows:
	\begin{enumromanup}
		\item Let $S = \Rpi^{(0, v]+}$ ($v < 1$), $\Rpi^{(0, v]}$ ($v < 1$), $\Rpi^{[u, v]}$ ($1 \not\in [u, v]$) or $\Rpi$.
			As $P_{\varpi}$ is invertible in $S[1/p]$, we put the trivial filtration on $S$.

		\item Let $S$ be the placeholder for all the remaining rings in Definition \ref{def:coordinate_rings}, in particular, we have that $P_{\varpi}$ is not invertible in $S[1/p]$.
			Then, there is a natural embedding $S \rightarrow R[\varpi, 1/p] \llbracket P_{\varpi} \rrbracket = R[\varpi, 1/p] \llbracket X_0 - \varpi \rrbracket$, obtained by completing $S[1/p]$ for the $P_{\varpi}\textrm{-adic}$ topology and where we note that $P_{\varpi}$ and $X_0-\varpi$ generate the same ideal in $R[\varpi, 1/p] \llbracket P_{\varpi} \rrbracket$.
			We use this embedding to endow $S$ with a natural filtration $\Fil^k S := S \cap P_{\varpi}^k R[\varpi, 1/p] \llbracket P_{\varpi} \rrbracket$, for all $k \in \ZZ$.
	\end{enumromanup}
\end{defi}

\begin{rem}\label{rem:fil_pd_u}
	Let us describe the filtration on the rings of Definition \ref{defi:filtration_vanishing_varpi} (ii), more concretely.
	Note that $\Fil^k S = S$, for $k \leqslant 0$.
	For any $k \in \NN$, the ideal $\Fil^k \Rpi^{\textpd} \subset \Rpi^{\textpd}$ is topologically generated by the elements $P_{\varpi}^n/n!$, for $n \geqslant k$, i.e.\ $\Fil^k \Rpi^{\PD}$ is the closure of the ideal generated by such elements.
	Similarly, the ideal $\Fil^k \Rpi^{[u]} \subset \Rpi^{[u]}$ is topologically generated by the elements $P_{\varpi}^n/p^{\lfloor nu \rfloor}$, for $n \geqslant k$.
	Using this description, an easy computation shows that $\Fil^k \Rpi^{[u]} \subset (P_{\varpi}/p)^k \Rpi^{[u]}$.
	On the other hand, we have that $\Fil^k \Rpi^{(0, v]+} = P_{\varpi}^k \Rpi^{(0, v]+}$.
	By definition, note that $\Rpi^{[u, v]} = \Rpi^{[u]} + \Rpi^{(0, v]+}$, so we get that the ideal $\Fil^k \Rpi^{[u, v]} \subset \Rpi^{[u, v]}$ is topologically generated by $(\Fil^k \Rpi^{[u]} + \Fil^k \Rpi^{(0, v]+})\Rpi^{[u, v]}$.
\end{rem}

The following claim easily follows from Remark \ref{rem:fil_pd_u}:
\begin{lem}[{\cite[Lemma 2.6]{colmez-niziol-nearby-cycles}}]\label{lem:split_pd_elements}
	For any $k \in \NN$ and $f \in \Rpi^{[u]}$, we can write $f = f_1 + f_2$ with $f_1 \in \Fil^k \Rpi^{[u]}$ and $f_2 \in \frac{1}{p^{\lfloor ku \rfloor}} \Rpi^+$.
\end{lem}

\subsection{Cyclotomic Frobenius}\label{subsec:cyclotomic_frob}

In this subsection, we will define a (cyclotomic) Frobenius endomorphism and its left inverse on the rings studied in the previous section (see \cite[\S 3.3]{abhinandan-crystalline-wach}).
\begin{defi}
	Over $R_{\varpi, \square}^+$ define the (cyclotomic) Frobenius as a lift of the absolute Frobenius modulo $p$, denoted as $\varphi : R_{\varpi, \square}^+ \rightarrow R_{\varpi, \square}^+$ and sending $X_0 \mapsto (1+X_0)^p - 1$ and $X_i \mapsto X_i^p$, for $1 \leqslant i \leqslant d$.
	Clearly, we have that $\varphi(x) - x^p$ is in $p R_{\varpi, \square}^+$ for any $x$ in $R_{\varpi, \square}^+$.
	Using \cite[Proposition 2.1]{colmez-niziol-nearby-cycles}, the Frobenius extends to an endomorphism $\varphi : \Rpi^+ \rightarrow \Rpi^+$.
	Finally, by continuity, the Frobenius admits unique extensions $\Rpi^{\textpd} \rightarrow \Rpi^{\textpd}$, $\Rpi^{[u]} \rightarrow \Rpi^{[u]}$, $\Rpi^{(0, v]+} \rightarrow \Rpi^{(0, v/p]+}$, $\Rpi^{[u, v]} \rightarrow \Rpi^{[u, v/p]}$ and $\Rpi \rightarrow \Rpi$.
\end{defi}

Recall that $r_{\varpi}^{[u]} = \big\{\sum_{k \in \NN} a_k p^{-\lfloor \frac{ku}{e} \rfloor}X_0^k, \hspace{1mm} \textrm{such that} \hspace{1mm} a_k \in O_F \hspace{1mm} \textrm{goes to} \hspace{1mm} 0 \hspace{1mm} \textrm{as} \hspace{1mm} i \rightarrow +\infty\big\}$.
Denote by $\upsilon_{X_0} : r_{\varpi}^{[u]} \rightarrow \NN \cup \{+\infty\}$, the valuation relative to $X_0$, i.e.\ if $f = \sum b_k X_0^k$, then $\upsilon_{X_0}(f) = \inf \hspace{1mm} \{k \in \NN, b_k \neq 0\}$.
For $N \in \NN$, we set $r_{\varpi, N}^{[u]} := \{f \in r_{\varpi}^{[u]}, \hspace{1mm} \upsilon_{X_0}(f) \geqslant N\}$ and define $R_{\varpi, N}^{[u]}$ to be the topological closure of $r_{\varpi, N}^{[u]} \otimes_{r_{\varpi}^+} \Rpi^+ \subset \Rpi^{[u]}$.

\begin{lem}\label{lem:id_minus_frob_bijective}
	Let $s \in \ZZ$ and $N \in \NN_{\geqslant 1}$ such that $N \geqslant se/u(p-1)$, then $1-p^{-s}\varphi$ is bijective on $R_{\varpi,N}^{[u]}$.
\end{lem}
\begin{proof}
	The claim follows from \cite[Lemma 3.1]{colmez-niziol-nearby-cycles}, where by explicit computations, one shows that $p^{-ks}\varphi^k(R_{\varpi,N}^{[u]}) \subset p^{n(k)} R_{\varpi,N}^{[u]}$, where $n(k)$ depends on $k$ and goes to $+\infty$ as $k \rightarrow +\infty$.
	So it follows that the series of operators $\sum_{k \in \NN} p^{-ks}\varphi^k$ converge as an inverse to $1-p^{-s}\varphi$ on $R_{\varpi,N}^{[u]}$.
\end{proof}

\subsubsection{The operator \texorpdfstring{$\psi$}{-}}

Set $u_{\alpha} := (1 + X_0)^{\alpha_0} X_1^{\alpha_1} \cdots X_d^{\alpha_d}$, where $\alpha = (\alpha_0, \ldots, \alpha_d)$ is a $(d+1)\textrm{-tuple}$ with $\alpha_i \in \{0, \ldots, p-1\}$ for each $0 \leqslant i \leqslant d$.
Over the ring $\Rpi$, we have $O_F\textrm{-linear}$ differential operators $\partial_0 = (1+X_0) \frac{d}{d X_0}$ and $\partial_i = X_i \frac{d}{d X_i}$, for $1 \leqslant i \leqslant d$.
Therefore, for $0 \leqslant i \leqslant d$, we have that $\partial_i u_{\alpha} = \alpha_i u_{\alpha}$ and $\varphi(u_{\alpha}) = u_{\alpha}^p$.

\begin{lem}[{\cite[Proposition 2.15]{colmez-niziol-nearby-cycles}}]\label{lem:psi_action_structure_modp}
	Any $x$ in $\Rpi/p$ can be uniquely written as $x = \sum_{\alpha} c_{\alpha}(x)$, with $\partial_i \circ c_{\alpha}(x) = \alpha_i c_{\alpha}(x)$, for $0 \leqslant i \leqslant d$.
	Moreover, there exists a unique $x_{\alpha}$ in $\Rpi/p$, such that $c_{\alpha}(x) = x_{\alpha}^p u_{\alpha}$.
	Furthermore, if $x$ is in $\Rpi^+/p$, then $c_{\alpha}(x)$ belongs to $\Rpi^+/p$.
\end{lem}

\begin{prop}\label{prop:psi_action_structure}
	Any $x$ in $\Rpi$ can be uniquely written as $x = \sum_{\alpha} c_{\alpha}(x)$, with $c_{\alpha}(x)$ in $\varphi\big(\Rpi\big) u_{\alpha}$.
	Moreover, if $x$ is in $\Rpi^+$ with $c_{\alpha}(x) = \varphi(x_{\alpha}) u_{\alpha}$, then $c_{\alpha}(x)$ belongs to $\Rpi^+$, for all $\alpha$, and $\partial_i c_{\alpha}(x) - \alpha_i c_{\alpha}(x)$ belongs to $p \Rpi^+$, for $0 \leqslant i \leqslant d$.
	Finally, if $x$ is in $\Rpi^{(0, v]+}$ then $c_{\alpha}(x)$ is in $\Rpi^{(0, v]+}$, for all $\alpha$.
\end{prop}
\begin{proof}
	The first two claims follow from Lemma \ref{lem:psi_action_structure_modp} and the last from \cite[Proposition 2.15]{colmez-niziol-nearby-cycles}.
\end{proof}

\begin{defi}
	Define the left inverse $\psi$ of the Frobenius $\varphi$ on $S = \Rpi^+$ or $S = \Rpi$, by the formula $\psi(x) = \varphi^{-1}\big(c_{0}(x)\big)$.
	Since $\Rpi$ is an extension of degree $p^{d+1}$ of $\varphi(\Rpi)$, with basis the $u_{\alpha}$'s, and since $\varphi(u_{\alpha}) = u_{\alpha}^p$, for all $\alpha$, therefore, we have that $\textrm{Tr}_{\Rpi/\varphi(\Rpi)}(u_{\alpha}\big) = 0$, if $\alpha \neq 0$, and we can define $\psi$ intrinsically as $\psi(x) := \tfrac{1}{p^{d+1}} \varphi^{-1} \circ \textrm{Tr}_{\Rpi/\varphi(\Rpi)}(x)$.
\end{defi}

The operator $\psi$ defined above is closely related to the operator defined in Proposition \ref{prop:psi_oper} (also denoted $\psi$; the relation will become clear in \S \ref{subsec:cyclotomic_embeddings}).
Note that $\psi$ is not a ring morphism; it is a left inverse to $\varphi$ and, more generally, we have that $\psi(\varphi(x)y) = x \psi(y)$.
Also, we have that $\partial_i \circ \varphi = p \varphi \circ \partial_i$ and $\partial_i \circ \psi = p^{-1} \psi \circ \partial_i$, for $i = 0, 1, \ldots, d$.
Indeed, the first equality can be obtained by checking on the basis elements $u_{\alpha}$ and the second equality is obtained by an easy computation using Proposition \ref{prop:psi_action_structure}.

For any $k \in \NN$, we can write $X_0^k = \sum_{j=0}^{p-1} \varphi(a_{j,k}) (1+X_0)^j$, for some $a_{j, k}$ in $\Rpi^+$.
Therefore, by continuity, we obtain the following:
\begin{lem}\phantomsection\label{lem:analytic_rings_psi_action}
	\begin{enumromanup}
	\item The definition of $\psi$ extends to surjective maps $\Rpi^{(0, v]+} \rightarrow \Rpi^{(0, pv]+}$, $\Rpi^{[u]} \rightarrow \Rpi^{[pu]}$ and $\Rpi^{[u, v]} \rightarrow \Rpi^{[pu, pv]}$.

	\item For the same reasons, the maps $x \mapsto c_{\alpha}(x)$ also extend and lead to decompositions $S = \oplus_{\alpha} S_{\alpha}$, where $S_{\alpha} = S \cap \varphi(R_{\varpi})u_{\alpha}$ for $S = \Rpi^{\bmstar}$, with $\smstar \in \{ , +, [u], (0, v]+, [u, v]\}$.
		Since $\psi(x) = \varphi^{-1}\big(c_{0}(x)\big)$, therefore, we have that $S^{\psi=0} = \oplus_{\alpha \neq 0} S_{\alpha}$.
	\end{enumromanup}
\end{lem}

\begin{lem}\label{lem:differential_mod_p}
	Let $S = \Rpi^{\bmstar}$, for $\smstar \in \{\hspace{1mm}, +, [u], (0, v]+, [u, v]\}$.
	Then, for $0 \leqslant i \leqslant d$, the operator $\partial_i$ on $S^{\bmstar}_{\alpha} / p S^{\bmstar}_{\alpha}$ is given by multiplication by $\alpha_i$, where $\alpha_i$ is the $i\textrm{-th}$ entry in $\alpha = (\alpha_0, \ldots, \alpha_d)$.
\end{lem}
\begin{proof}
	If $\smstar \in \{\hspace{1mm}, +\}$, then the claim was already shown in Proposition \ref{prop:psi_action_structure}.
	For $\smstar \in \{[u], (0, v]+, [u, v]\}$, the elements of $S^{\bmstar}_{\alpha}$ are those of the form $\sum_{k \in \ZZ} p^{r_k} X_0^k x_k$, where $x_k \in S^+$ goes to $0$ when $k \rightarrow +\infty$ and $r_k$ is determined by ``$\smstar$''.
	Let $x = \sum_{k \in \ZZ} p^{r_k} X_0^k x_k$.
	Then, note that for $1 \leqslant i \leqslant d$, we have that $\partial_{i}(X_0^k a_k) - \alpha_i X_0^k a_k = X_0^k\big(\partial_{i}(a_k) - \alpha_i a_k\big)$ belongs to $p S^+$ by Proposition \ref{prop:psi_action_structure}.
	Therefore, the claim follows for all $1 \leqslant i \leqslant d$ and $\smstar \in \{\hspace{1mm}, +, [u], (0, v]+, [u, v]\}$.
	Next, we will look at the case of $i = 0$.
	We first assume that $x$ is in $S^{[u]}$ and write $x = \sum_{k \in \NN} p^{r_k} x_k \sum_{j=0}^{p-1}\varphi(a_{j, k})(1+X_0)^j$, for some $a_{j, k}$ in $S^+$.
	Then, $c_{\alpha}(x) = \sum_{j=0}^{p-1}\sum_{k \in \NN} p^{r_k}\varphi(a_{j, k}) c_{(\alpha_0-j, \alpha_1, \cdots, \alpha_d)}(x_k)(1+X_0)^j$, where $\alpha_0-j$ denotes its value modulo $p$.
	Since $\partial_0\big(c_{(\alpha_0-j, \alpha_1, \cdots, \alpha_d)}(x_k)\big) - (\alpha_0-j)c_{(\alpha_0-j, \alpha_1, \cdots, \alpha_d)}(x_k)$ belongs to $pS^+$ and $\partial_0 \circ \varphi = p \varphi \circ \partial_{0}$, therefore, we get the desired conclusion for $i =0$ and $x$ in $S^{[u]}$.
	Next, assume that $x$ is in $S^{(0, v]+}$ and using the result for $S$, we get that $\partial_{0}(x) - \alpha_0 x$ belongs to $pS \cap S^{(0, v]+} = pS^{(0, v]+}$.
	Finally, by combining the results for $S^{[u]}$ and $S^{(0, v]+}$, we get the conclusion for any $x$ in $S^{[u, v]}$.
	This allows us to conclude.
\end{proof}

\begin{prop}\label{prop:ring_psi_eigenspace_decomp}
	Assume that $v < p$.
	\begin{enumromanup}
		\item Let $x$ in $\Rpi^{\psi=0}$, then $X_0^k \psi(x) = \psi(\varphi(X_0)^k x)$, for all $k \in \ZZ$.

		\item $\psi\big(X_0^{-pN} \Rpi^{(0, v/p]+}\big) \subset X_0^{-N} \Rpi^{(0, v]+}$, for all $N \in \NN$.

		\item The natural map $\oplus_{\alpha \neq 0} \varphi\big(\Rpi^{(0, v]+}\big)u_{\alpha} \rightarrow \big(\Rpi^{(0, v/p]+}\big)^{\psi=0}$ is an isomorphism.
	\end{enumromanup}
\end{prop}
\begin{proof}
	The claim in (i) follows from an elementary computation.
	Claims in (ii) and (iii) follow from \cite[Proposition 2.16]{colmez-niziol-nearby-cycles}.
\end{proof}

\subsection{Cyclotomic embedding}\label{subsec:cyclotomic_embeddings}

In this subsection, we will describe the relationships between period rings discussed in \S \ref{subsec:period_rings} and \S \ref{subsec:relative_phi_gamma_mod}, as well as, for the ring $\Rpi^{\bmstar}$, where $\smstar \in \{\hspace{1mm}, +, \textpd\}$.
Define a morphism of rings $\iota_{\cycl} : R_{\varpi, \square}^+ \rightarrow \Ainf(\Rbar)$, by sending $X_0 \mapsto \pi_{m} = \varphi^{-m}(\pi)$ and $X_i \mapsto [X_i^{\flat}]$, for $1 \leqslant i \leqslant d$.
The map $\iota_{\cycl}$ admits a unique extension to an embedding $\Rpi^+ \rightarrow \Ainf(\Rbar)$ such that $\theta \circ \iota_{\cycl}$ is the projection $\Rpi^+ \twoheadrightarrow R[\varpi]$ (see \cite[Lemma 3.12]{abhinandan-crystalline-wach}).
This embedding commutes with the respective Frobenii, i.e.\ $\iota_{\cycl} \circ \varphi = \varphi \circ \iota_{\cycl}$.
By continuity, the morphism $\iota_{\cycl}$ extends to embeddings $\Rpi^{\textpd} \subset \Acrys(\Rbar)$, $\Rpi^{[u]} \subset \ARbar^{[u]}$, $\Rpi^{(0, v]+} \subset \ARbar^{(0, v]+}$, $\Rpi^{[u, v]} \subset \ARbar^{[u, v]}$ and $\Rpi \subset \ARbar$.
Denote by $\ARpi^{\bmstar}$ the image of $\Rpi^{\bmstar}$ under $\iota_{\cycl}$.
These rings are stable under the action of $G_R$ and the action factors through $\Gamma_R$; we equip these rings with the induced action of $\Gamma_R$.
Moreover, for $\smstar \in \{+, \textpd, [u], [u, v], (0, v]+\}$, we equip $\ARpi^{\bmstar}$ with a filtration using Definition \ref{defi:filtration_vanishing_varpi} and $\iota_{\cycl}$.
It is easy to see that for $u \leqslant 1 \leqslant v$, the filtration on $\ARpi^{\star}$ coincides with the filtration induced via the embedding $\ARpi^{\bmstar} \subset \BdR^+(\Rbar)$, where we consider the natural filtration on $\BdR^+(\Rbar)$ (see \S \ref{subsec:period_rings}).
From \cite[\S 2.4.2]{colmez-niziol-nearby-cycles}, note that we have $(\varphi, \Gamma_R)\textrm{-equivariant}$ inclusions $\ARpi^{[u\prm]} \subset \ARpi^{\textpd} \subset \ARpi^{[u]}$, for $u \geqslant \frac{1}{p-1}$ and $u\prm \leqslant \frac{1}{p}$.

Note that the preceding discussion works well for $R[\varpi]$, where $\varpi = \zeta_{p^m}-1$ with $m \geqslant 1$.
For $R$, one can repeat the constructions above to obtain the period ring $\AR^+ \subset \ARpi^+$ (see \cite[\S 3.3.2]{abhinandan-crystalline-wach}), equipped with an induced filtration $\Fil^k \AR^+ = \AR^+ \cap \Fil^k \ARpi^+ = \pi^k \AR^+$ (see \cite[Lemma 3.17]{abhinandan-crystalline-wach}).
Recall that,
\begin{lem}[{\cite[Lemma 3.14]{abhinandan-crystalline-wach}}]\label{lem:t_over_pi_unit}
	The element $t/\pi$ is a unit in $\AFpi^{\textpd} \subset \ARpi^{\textpd} \subset \ARpi^{[u]} \subset \ARpi^{[u, v]}$.
\end{lem}

\begin{lem}\label{lem:arpi_fil_pi_cap}
	For $k \in \ZZ$ and $\smstar \in \{+, \textpd, [u], [u, v]\}$, we have $\Fil^k \ARpi^{\bmstar} \cap \pi \ARpi^{\bmstar} = \pi\Fil^{k-1} \ARpi^{\bmstar}$, as submodules of $\ARpi^{\bmstar}$.
\end{lem}
\begin{proof}
	Let $A = \ARpi^{\bmstar}$ and $B = R[\varpi, 1/p] \llbracket P_{\varpi} \rrbracket = R[\varpi, 1/p] \llbracket X_0-\varpi \rrbracket$ (see Definition \ref{defi:filtration_vanishing_varpi} for the latter ring), where $\varpi = \zeta_{p^m}-1$.
	Using the inverse of the isomorphism $\iota_{\cycl} : \Rpi^{\bmstar} \isomorphic \ARpi^{\bmstar} = A$, we may regard $A$ as a subring of $B$.
	Now, we will prove the claim by induction on $k$.
	Note that the claim is trivial for $k \leqslant 0$ and for $k = 1$ we have that $\Fil^k A \cap \pi A = \pi A$.
	So, let $k \in \NN_{\geqslant 2}$ and assume that the claim is true for $k-1$, i.e.\ $\Fil^{k-1} A \cap \pi A = \pi \Fil^{k-2} A$.
	Now, note that $\Fil^k A \cap \pi A = \Fil^k A \cap \Fil^{k-1} A \cap \pi A =  \Fil^k A \cap \pi \Fil^{k-2} A$.
	In particular, to get the claim, it is enough to show that $\Fil^k A \cap \pi \Fil^{k-2} A = \pi \Fil^{k-1} A$.
	Let $x$ be an element of $\Fil^k A \cap \pi \Fil^{k-2} A$ and write $x = \pi y$, for some $y$ in $\Fil^{k-2} A$.
	From the description of the filtration on $A$ in Definition \ref{defi:filtration_vanishing_varpi}, it follows that we can write $x = \xi^k x'$ and $y = \xi^{k-2} y'$, for some $x'$ and $y'$ in $B$ (note that $\iota_{\cycl}(P_{\varpi}) = \xi$).
	Since $B$ is $\xi\textrm{-torsion free}$ and $\pi = \xi \pi_1$, we get that $\xi x' = \pi_1 y'$ in $B$.
	But we have $\pi_1 = (1+\pi_m)^{p^{m-1}}-1 = (\pi_m - \varpi + \zeta_{p^m})^{p^{m-1}}-1 = (\pi_m - \varpi) z + \zeta_p-1$, for some $z$ in $B$ and $\zeta_p = \zeta_{p^m}^{p^{m-1}}$ (note that $\pi_m = \iota_{\cycl}(X_0)$).
	Moreover, from Definition \ref{defi:filtration_vanishing_varpi}, recall that $\xi$ and $\pi_m - \varpi$ generate the same ideal in $B$.
	Therefore, we obtain that $(\zeta_p-1)y' = \xi x' - (\pi_m-\varpi)zy'$ is an element of $\xi B$.
	As $(\zeta_p-1)$ is a unit in $B$, it follows that we have $y' = \xi y''$, for some $y''$ in $B$.
	So, we can write $y = \xi^{k-2} y' = \xi^{k-1} y''$, and see that it belongs to $\xi^{k-1} B \cap A = \Fil^{k-1} A$.
	Hence, $x = \pi y$ is an element of $\pi \Fil^{k-1} A$, in particular, $\Fil^k A \cap \pi \Fil^{k-2} A \subset \pi \Fil^{k-1} A$.
	The other inclusion, i.e.\ $\pi \Fil^{k-1} A \subset \Fil^k A \cap \pi \Fil^{k-2} A$, is obvious.
\end{proof}

\begin{lem}[{\cite[Lemma 2.35]{colmez-niziol-nearby-cycles}}]\label{lem:pi_m_divides_p}
	If $v < p$, then
	\begin{enumromanup}
	\item The element $\pi_m^{-p^{m-1}} \pi_1$ is a unit in $\ARpi^{(0, v]+}$;

	\item In $\ARpi^{(0, v]+}$, the element $p$ is divisible by $\pi_m^{\lfloor (p-1)p^{m-1}/v \rfloor}$, hence also by $\pi_m^{(p-1)p^{m-2}}$;

	\item Let $v = p-1$, then $\pi_m^{-p^m}\pi$ is a unit in $\ARpi^{(0, v/p]+}$ and $p/\pi \in \ARpi^{(0, v/p]+}$.
	\end{enumromanup}
\end{lem}

Next, we prove some claims for the action of $\Gamma_R$.
\begin{lem}\label{lem:gamma_minus_1_pd}
	Let $k \in \NN$ and note that for $\smstar \in \{+, \textpd, [u]\}$ and $i \in \{0, 1, \ldots, d\}$, we have that $(\gamma_i-1) \big(p^m, \pi_m^{p^m}\big)^k \ARpi^{\bmstar} \subset \big(p^m, \pi_m^{p^m}\big)^{k+1} \ARpi^{\bmstar}$ .
\end{lem}
\begin{proof}
	Let $i = 0$ and note that we have $(\gamma_0-1)\pi_m = \pi x$, for some $x \in \ARpi^+$.
	Since $\pi = (1 + \pi_m)^{p^m}-1$, we get that $(\gamma_0-1)\pi_m$ belongs to $\big(p^m\pi_m, \pi_m^{p^m}\big) \ARpi^+$.
	Moreover, $(\gamma_0-1) \pi_m^{p^m} = (\pi x + \pi_m)^{p^m} - \pi_m^{p^m}$ belongs to $\big(p^m\pi_m, \pi_m^{p^m}\big)^2 \ARpi^+$.
	Proceeding by induction on $k \geqslant 1$ and using the fact that $\gamma_0-1$ acts as a twisted derivation (i.e.\  for all $x, y$ in $\ARpi^+$, we have $(\gamma_0-1)xy = (\gamma_0-1)x \cdot y + \gamma_0(x)(\gamma_0-1)y)$, we conclude that $(\gamma_0 - 1) \big(p^m\pi_m, \pi_m^{p^m}\big)^{k} \ARpi^+ \subset \big(p^m\pi_m, \pi_m^{p^m}\big)^{k+1} \ARpi^+$.
	Furthermore, any $f$ in $\ARpi^{\textpd}$ can be written as $f = \sum_{n \in \NN} f_n \pi_m^n/(\lfloor n/e \rfloor!)$, such that $f_n$ is in $\ARpi^+$ and goes to $0$ $p\textrm{-adically}$ as $n \rightarrow +\infty$.
	For notational convenience, we take $n = je$, for some $j$ in $\NN$, and see that $(\gamma_0-1)\pi_m^{je}/j!$ is in $\big(p^m, \pi_m^{p^m}\big) \ARpi^{\textpd}$.
	Proceeding by induction on $k \geqslant 1$ and using that $\gamma_0-1$ acts as a twisted derivation, we conclude that $(\gamma_0-1) \big(p^m, \pi_m^{p^m}\big)^{k} \ARpi^{\textpd} \subset \big(p^m, \pi_m^{p^m}\big)^{k+1} \ARpi^{\textpd}$.
	Next, for $1 \leqslant i \leqslant d$, note that we have $(\gamma_i-1)[X_i^{\flat}] = \pi [X_i^{\flat}]$ is in $\big(p^m, \pi_m^{p^m}\big) \ARpi^+$ and $(\gamma_i-1)\big([X_i^{\flat}]^{-1}\big) = -\pi(1+\pi)^{-1}[X_i^{\flat}]^{-1}$ belongs to $\big(p^m\pi_m, \pi_m^{p^m}\big) \ARpi^+$.
	Proceeding by induction on $k \geqslant 0$ and using the fact that $\gamma_i-1$ also acts as a twisted derivation, we conclude that $(\gamma_i-1) \big(p^m\pi_m, \pi_m^{p^m}\big)^{k} \ARpi^+ \subset \big(p^m\pi_m, \pi_m^{p^m}\big)^{k+1} \ARpi^+$.
	Again, by the description of elements of $\ARpi^{\textpd}$, using the discussion for $\ARpi^+$ and the fact that $\gamma_i-1$ acts as a twisted derivation, we conclude that $(\gamma_i-1) \big(p^m, \pi_m^{p^m}\big)^k \ARpi^{\textpd} \subset \big(p^m, \pi_m^{p^m}\big)^{k+1} \ARpi^{\textpd}$.
	Finally, the claim for $\ARpi^{[u]}$ follows in a similar manner.
\end{proof}

\begin{lem}\phantomsection\label{lem:gamma_minus_1_oc}
	We have that $(\gamma_0-1) \ARpi^{(0, v]+} \subset \big(p^m \pi_m, \pi_m^{p^m}\big) \ARpi^{(0, v]+}$ and $(\gamma_i-1) \ARpi^{(0, v]+} \subset \pi \ARpi^{(0, v]+}$, for $i \in \{1, \ldots, d\}$.
	Moreover, for $i \in \{0, 1, \ldots, d\}$ and $k \in \NN$, we have that $(\gamma_i-1)\big(p^m, \pi_m^{p^m}\big)^k \ARpi^{[u, v]} \subset \big(p^m, \pi_m^{p^m}\big)^{k+1} \ARpi^{[u, v]}$.
\end{lem}
\begin{proof}
	Let $i = 0$ and from the proof of Lemma \ref{lem:gamma_minus_1_pd}, we have that $(\gamma_0-1)\pi_m$ is in $\big(p^m \pi_m, \pi_m^{p^m}\big) \ARpi^+$.
	So we conclude that $(\gamma_0-1)\ARpi^+$ belongs to $\big(p^m \pi_m, \pi_m^{p^m}\big) \ARpi^+$.
	Observe that $\gamma_0(\pi_m) = \chi(\gamma_0)\pi_m a$, where $\chi(\gamma_0) = \exp(p^m)$ is in $\ZZ_p^{\times}$ and $a$ is a unit in $\ARpi^+$.
	So, we can write $(\gamma_0-1)\pi_m^{-1} = p^mz/(\chi(\gamma_0)a\pi_m)$ and, therefore, $(\gamma_0-1)(p/\pi_m)$ belongs to $\big(p^m \pi_m, \pi_m^{p^m}\big) \ARpi^{(0, v]+}$.
	Proceeding by induction on $k \geqslant 1$ and using the fact that $\gamma_0-1$ acts as a twisted derivation, we conclude that $(\gamma_0-1) \big(p^m \pi_m, \pi_m^{p^m}\big)^k \ARpi^{(0, v]+} \subset \big(p^m \pi_m, \pi_m^{p^m}\big)^{k+1} \ARpi^{(0, v]+}$.	
	Next, for $1 \leqslant i \leqslant d$, from the analysis for $\ARpi^+$ in Lemma \ref{lem:gamma_minus_1_pd}, we already have that $(\gamma_i-1) \ARpi^+ \subset \pi \ARpi^+$.
	Since passing from $\ARpi^+$ to $\AR^{(0, v]+}$ involves only the arithmetic variable $\pi_m$, on which $\gamma_i$ acts trivially.
	Therefore, we conclude that $(\gamma_i-1) \ARpi^{(0, v]+} \subset \pi \ARpi^{(0, v]+}$, and proceeding by induction on $k \geqslant 1$ and using that $\gamma_i-1$ acts as a twisted derivation, we get that $(\gamma_i-1) \big(p^m \pi_m, \pi_m^{p^m}\big)^k \ARpi^{(0, v]+} \subset \big(p^m \pi_m, \pi_m^{p^m}\big)^{k+1} \ARpi^{(0, v]+}$.
	This shows the first claim.
	Finally, the claim for $\ARpi^{[u, v]}$ follows by combining discussion above with Lemma \ref{lem:gamma_minus_1_pd} for $\ARpi^{[u]}$.
\end{proof}

\subsection{Filtered Poincar\'e Lemma}\label{subsec:fat_period_rings}

In this subsection we will state and prove a filtered version of the PD-Poincar\'e Lemma which will be useful for \S \ref{sec:syntomic_complex_finite_height}.

\subsubsection{Fat period rings}\label{subsubsec:obese_rings}

We recall the definition from \cite[\S 2.6]{colmez-niziol-nearby-cycles} and \cite[\S 3.4]{abhinandan-crystalline-wach}.
Let $A$ and $B$ be two $p\textrm{-adically}$ complete filtered $O_F\textrm{-algebras}$.
Let $\iota : B \rightarrow A$ be a continuous injective homomorphism of filtered $O_F\textrm{-algebras}$ and let $f : B \otimes_{O_F} A \rightarrow A$ denote the ring homomorphism sending $x \otimes y \mapsto \iota(x)y$.

\begin{defi}\label{defi:fat_ring_const}
	Define $E$ to be the $\padic$ completion of the divided power envelope of $B \otimes_{O_F} A$, with respect to $\kert f$.
\end{defi}

For consistency in notations, in the following definition, we write $\Acrys(\Rbar)$ as $\ARbar^{\textpd}$.
\begin{defi}\label{defi:obese_rings}
	In the notation of Definition \ref{defi:fat_ring_const}, we record the following:
	\begin{enumromanup}
	\item Let $\smstar \in \{\textpd, [u], [u, v]\}$ and define $\fERpi^{\bmstar} = E$, for $B = \Rpi^{\bmstar}$, $A = \ARpi^{\bmstar}$ and $\iota = \iota_{\cycl}$ (see \S \ref{subsec:cyclotomic_embeddings}).

	\item Let $\smstar \in \{\textpd, [u], [u, v]\}$ and define $\fERbar^{\bmstar} = E$, for $B = \Rpi^{\bmstar}$, $A = \ARbar^{\bmstar}$ and $\iota = \iota_{\cycl}$ (see \S \ref{subsec:cyclotomic_embeddings}).
	\end{enumromanup}
\end{defi}

\begin{rem}\label{rem:elements_of_pd_ring}
	Let us note some properties of the ring $E$ in Definition \ref{defi:obese_rings}:
	\begin{enumromanup}
	\item The ring $E$ is the $\padic$ completion of $B \otimes_{O_F} A$ adjoin $(x \otimes 1 - 1 \otimes \iota(x))^{[k]}$, for all $x$ in $B$ and $n \in \NN$, and $(V_i - 1)^{[k]}$ for $0 \leqslant i \leqslant d$ and $k \in \NN$, where $V_i = \frac{X_i \otimes 1}{1 \otimes \iota(X_i)}$ for $1 \leqslant i \leqslant d$ and $V_0 = \frac{1 + (X_0 \otimes 1)}{1 + (1 \otimes \iota(X_0))}$.
		The morphism $f : B \otimes_{O_F} A \rightarrow A$ extends uniquely to a continuous morphism $f : E \rightarrow A$.

	\item The ring $E$ is equipped with an $\ZZ\textrm{-indexed}$ decreasing filtration, which we define to be $\Fil^r E := E$ for $r \leqslant 0$, and for $r \geqslant 0$, define $\Fil^r E$ to be the topological closure of the ideal generated by elements of the form $x_1 x_2 \prod_{i=0}^d (V_i - 1)^{[k_i]}$, with $x_1$ in $\Fil^{r_1}B$, $x_2$ in $\Fil^{r_2}A$ and $r_1 + r_2 + \sum_{i=0}^d k_i \geqslant r$.

	\item From \cite[Lemma 2.36]{colmez-niziol-nearby-cycles}, we have that any element $x$ in $E$ can be uniquely written as $x = \sum_{\smbfk \in \NN^{d+1}} x_{\smbfk}(1-V_0)^{[k_0]} \cdots (1-V_d)^{[k_d]}$, with $x_{\smbfk}$ in $A$ for all $\smbfk = (k_0, k_1, \ldots, k_d) \in \NN^{d+1}$ and $x_{\smbfk} \rightarrow 0$ as $|\smbfk| = \sum_{i=0}^d k_i \rightarrow +\infty$.
		Moreover, $x$ is in $\Fil^r E$ if and only if $x_{\smbfk}$ is in $\Fil^{r - |\smbfk|} A$, for all $\smbfk \in \NN^{d+1}$.

	\item The ring $E$ is equipped with a natural $A\textrm{-linear}$ continuous de Rham differential operator $d : E \rightarrow \Omega^1_{E/A}$.
		Moreover, by the description of the filtration on $E$ in (iii), it is easy to see that the differential operator satisfies Griffiths transversality with respect to the filtration, i.e.\ we have $d : \Fil^r E \rightarrow \Fil^{r-1} E \otimes_E \Omega^1_{E/A}$.
		In the special case that $\iota : B \isomorphic A$, we see that $E$ is further equipped with a natural $B\textrm{-linear}$ continuous de Rham differential operator $d : E \rightarrow \Omega^1_{E/B}$ satisfying Griffiths transversality with respect to the filtration.
	\end{enumromanup}
\end{rem}

\begin{lem}\label{lem:obese_rings_struct}
	Rings in Definition \ref{defi:obese_rings} have desirable properties:
	\begin{enumromanup}
	\item In Definition \ref{defi:obese_rings} (i), the tensor product Frobenii $\varphi \otimes \varphi$ on $\Rpi^{\bmstar} \otimes_{O_F} \ARpi^{\bmstar}$, for $\smstar \in \{\textpd, [u], [u, v]\}$, extend respectively uniquely to continuous morphisms $\fERpi^{\textpd} \rightarrow \fERpi^{\textpd}$, $\fERpi^{[u]} \rightarrow \fERpi^{[u]}$ and $\fERpi^{[u, v]} \rightarrow \fERpi^{[u, v/p]}$.
		Moreover, the actions of $G_R$ on $\ARpi^{\bmstar}$ extend respectively uniquely to continuous actions of $G_R$ on $\fERpi^{\textpd}$, $\fERpi^{[u]}$ and $\fERpi^{[u, v]}$, which commute with the respective Frobenii.
		Furthermore, we have $(\varphi, G_R)\textrm{-equivariant}$ inclusions $\fERpi^{\textpd} \subset \fERpi^{[u]} \subset \fERpi^{[u, v]}$.

	\item In Definition \ref{defi:obese_rings} (ii), the tensor product Frobenii $\varphi \otimes \varphi$ on $\Rpi^{\bmstar} \otimes_{O_F} \ARbar^{\bmstar}$, for $\smstar \in \{\textpd, [u], [u, v]\}$, extend respectively uniquely to continuous morphisms $\fERbar^{\textpd} \rightarrow \fERbar^{\textpd}$, $\fERbar^{[u]} \rightarrow \fERbar^{[u]}$ and $\fERbar^{[u, v]} \rightarrow \fERbar^{[u, v/p]}$.
		Moreover, the actions of $G_R$ on $\ARbar^{\bmstar}$ extend respectively uniquely to continuous actions of $G_R$ on $\fERbar^{\textpd}$, $\fERbar^{[u]}$ and $\fERbar^{[u, v]}$, which commute with the respective Frobenii.
		Furthermore, we have $(\varphi, G_R)\textrm{-equivariant}$ inclusions $\fERbar^{\textpd} \subset \fERbar^{[u]} \subset \fERbar^{[u, v]}$.

	\item The natural $(\varphi, \Gamma_R)\textrm{-equivariant}$ inclusion of rings $\ARpi^{\bmstar} \subset \ARbar^{\bmstar}$ induces a natural $(\varphi, \Gamma_R)\textrm{-equivariant}$ injective homomorphism of rings $\fERpi^{\bmstar} \subset \fERbar^{\bmstar}$.
		Moreover, the filtration and the $\ARpi^{\bmstar}\textrm{-linear}$ connection on $\fERpi^{\bmstar}$ are respectively induced from the filtration and $\ARbar^{\bmstar}\textrm{-linear}$ connection on $\fERbar^{\bmstar}$, in particular, $\Fil^r \fERpi^{\bmstar} = \fERpi^{\bmstar} \cap \Fil^r \fERbar^{\bmstar} \subset \fERbar^{\bmstar}$, for all $r \in \ZZ$.
	\end{enumromanup}
\end{lem}
\begin{proof}
	The first two claims follow from \cite[Lemma 2.38]{colmez-niziol-nearby-cycles}.
	The last claim follows from the description of $\ERpi^{\bmstar}$ and $\fERbar^{\bmstar}$ in Remark \ref{rem:elements_of_pd_ring} and the fact that $\ARpi^{\bmstar} \cap \Fil^r \ARbar^{\bmstar} = \ARpi^{\bmstar} \cap \Fil^r \BdR(\Rbar) = \Fil^r \ARpi^{\bmstar}$.
\end{proof}

\begin{rem}\label{rem:oarpd_erpd}
	From Definition \ref{defi:obese_rings} and Lemma \ref{lem:obese_rings_struct}, note that we have a natural embedding $\OAcrys(\Rbar) \subset \fERbar^{\textpd}$ compatible with the respective Frobenii, $\Acrys(\Rbar)\textrm{-linear}$ connections and actions of $G_R$, and the natural filtration on the former is induced from the filtration on the latter.
	Furthermore, from \S \ref{subsec:wach_crystalline} recall that we have the ring $\OARpi^{\textpd} \subset \OAcrys(\Rbar)$ and from \cite[Remark 4.20]{abhinandan-crystalline-wach} we have an alternative construction of $\OARpi^{\textpd}$ using the embedding $R \subset \Rpi^{\textpd} \isomorphic \ARpi^{\textpd}$ (the last morphism is $\iota_{\cycl}$ in \S \ref{subsec:cyclotomic_embeddings}).
	This induces an embedding $\OARpi^{\textpd} \subset \fERpi^{\textpd}$ compatible with the respective Frobenii and actions of $\Gamma_R$, and the natural filtration on the former is induced from the filtration on the latter.
	Denote the $O_F\textrm{-linear}$ differential operator over $\ARpi^{\textpd}$ as $\partial_A$ and the $O_F\textrm{-linear}$ differential operator over $\Rpi^{\textpd}$ (as well as over $R$) as $\partial_R$.
	Then, the induced differential operators $\partial_R \otimes 1 + 1 \otimes \partial_A$ over $\OARpi^{\textpd}$ and $\fERpi^{\textpd}$ are compatible.
\end{rem}

\begin{lem}\label{lem:erpi_fil_pi_cap}
	For $r \in \ZZ$ and $\smstar \in \{+, \textpd, [u], [u, v]\}$, we have $\Fil^r \fERpi^{\bmstar} \cap \pi \fERpi^{\bmstar} = \pi\Fil^{r-1} \fERpi^{\bmstar}$, as submodules of $\fERpi^{\bmstar}$.
\end{lem}
\begin{proof}
	Let $E := \fERpi^{\bmstar}$ and $A := \ARpi^{\bmstar}$, for $\smstar \in \{+, \textpd, [u], [u, v]\}$.
	The claim is trivial for $r \leqslant 0$, so assume that $r \geqslant 1$.
	Note that we have $\pi \Fil^{r-1} E \subset \Fil^r E \cap \pi E$, so we need to show the reverse inclusion.
	Let $x$ be any element of $\Fil^r E \cap \pi E$, and write $x = \pi y$, for some $y$ in $E$.
	From the description of the filtration on $E$ in Remark \ref{rem:elements_of_pd_ring} (iii), we have a unique presentation of $x$ as $\sum_{\smbfk \in \NN^{d+1}} x_{\smbfk}(1-V_0)^{[k_0]} \cdots (1-V_d)^{[k_d]}$, with $x_{\smbfk}$ in $\Fil^{r-|\smbfk|} A$ for all $\smbfk \in \NN^{d+1}$.
	Moreover, we have a unique presentation of $y$ as $\sum_{\smbfk \in \NN^{d+1}} y_{\smbfk}(1-V_0)^{[k_0]} \cdots (1-V_d)^{[k_d]}$, with $y_{\smbfk}$ in $A$ for all $\smbfk \in \NN^{d+1}$.
	Then, using the equality $x = \pi y$, we get that $x_{\smbfk} = \pi y_{\smbfk}$, for all $\smbfk \in \NN^{d+1}$.
	Now, from Lemma \ref{lem:arpi_fil_pi_cap} and the fact that $A$ is $\pi\textrm{-torsion free}$, it follows that $x_{\smbfk}$ is an element of $\pi \Fil^{r-|\smbfk|-1} A$, hence, $x$ is an element of $\pi \Fil^{r-1} E$.
\end{proof}

Finally, to work with various filtered modules later, we define a filtered ring (analogous to $\OBdR(\Rbar)$) containing all the rings described so far and inducing the same filtrations as described above.
From \cite[Proposition 5.2.2]{brinon-padicrep-relatif}, recall that the natural inclusion $\BdR^+ \subset \OBdR^+(\Rbar)$ extends to a $\BdR^+\textrm{-linear}$ isomorphism of rings $\BdR^+\llbracket T_1, \ldots, T_d \rrbracket \isomorphic \OBdR(\Rbar)$, by sending the indeterminate $T_i \mapsto X_i - [X_i^{\flat}]$, for each $1 \leqslant i \leqslant d$.
We enlarge $\OBdR(\Rbar)$ by setting,
\begin{equation*}
	\pazb^+ := \BdR^+\llbracket T_0, T_1, \ldots, T_d \rrbracket, \hspace{2mm} \textrm{and} \hspace{2mm} \pazb := \pazb^+[1/t],
\end{equation*}
in particular, we have natural inclusions of rings $\OBdR^+(\Rbar) \subset \pazb^+$ and $\OBdR(\Rbar) \subset \pazb$.
We equip the latter rings with filtrations similar to \cite[p.\ 52]{brinon-padicrep-relatif}.
Set $\Fil^r \pazb^+ := (t, T_0, \ldots, T_d)^r \pazb^+$, for all $r \in \NN$, and $\Fil^r \pazb^+ = \pazb^+$, for $r < 0$.
Moreover, set $\Fil^0 \pazb := \sum_{n \in \NN} t^{-n} \Fil^n \pazb^+$ and $\Fil^r \pazb := t^r \Fil^0 \pazb$, for all $r \in \ZZ$.
Similar rings were studied in \cite[\S 3.2.1]{andreatta-iovita-semistable}, in the more general setting of semistable schemes.
Now, employing arguments similar to \cite[Propositions 5.2.5, 5.2.6 \& 5.2.8]{brinon-padicrep-relatif}, the following is clear:
\begin{lem}\label{lem:fil_B}
	Let $x_i$ denote the image of $T_i$ in $\gr^1 \pazb^+$ and $y_i$ denote the image of $T_i/t$ in $\gr^0 \pazb$, for $0 \leqslant i \leqslant d$.
	Then, we have that $\gr^{\bullet} \pazb^+ \isomorphic \CC(\Rbar)[t, x_0, \ldots, x_d]$, where the grading is given by the degree of the polynomial in $t, x_0, \ldots, x_d$, and $\gr^{\bullet} \pazb \isomorphic \CC(\Rbar)[t, t^{-1}, y_0, \ldots, y_d]$, where the grading is given by the degree of $t$, in particular, we have that $\gr^0 \pazb^+ \isomorphic \CC(\Rbar)[y_0, \ldots, y_d]$.
	Moreover, the filtration on $\pazb^+$ is the same as the induced filtration from $\pazb$, i.e.\ $\Fil^r \pazb^+ = \Fil^r \pazb \cap \pazb^+ \subset \pazb$, for all $r \in \ZZ$.
\end{lem}

\begin{rem}\label{rem:fil_B_BdR}
	From Lemma \ref{lem:fil_B} and the description of the filtration on $\OBdR^+(\Rbar)$ in \cite[p.\ 52]{brinon-padicrep-relatif}, we see that the filtration on $\OBdR^+(\Rbar)$ is induced from the filtration on $\pazb^+$, i.e.\ $\Fil^r \OBdR^+(\Rbar) = \OBdR^+(\Rbar) \cap \Fil^r \pazb^+ \subset \pazb^+$, for $r \in \ZZ$.
	Then it also follows that $\Fil^r \OBdR(\Rbar) = \OBdR(\Rbar) \cap \Fil^r \pazb \subset \pazb$, for $r \in \ZZ$.
\end{rem}

Now, recall that we have an inclusion of rings $\ARbar^{[u, v]} \subset \BdR^+(\Rbar)$ (since $u \leq 1 \leq v$), and the former is equipped with a filtration induced from the latter (see \S \ref{subsubsec:overconvergence}).
Then, upon using the description of $\fERpi^{[u, v]}$ from Remark \ref{rem:elements_of_pd_ring} (i), we see that the preceding embedding naturally extends to an injective ring homomorphism $\fERbar^{[u, v]} \rightarrow \pazb^+$, via $V_i-1 \mapsto T_i/[X_i^{\flat}]$, for $1 \leqslant i \leqslant d$, and $V_0-1 \mapsto T_0/(1+\pi_m)$.
Using the description of the filtration on $\fERpi^{[u, v]}$ from Remark \ref{rem:elements_of_pd_ring} and the filtration on $\pazb^+$ from above, we see that,
\begin{lem}\label{lem:fil_B_er}
	The filtration on $\fERbar^{[u, v]}$ is induced from the filtration on $\pazb^+$, i.e.\ $\Fil^r \fERbar^{[u, v]} = \fERbar^{[u, v]} \cap \Fil^r \pazb^+ \subset \pazb^+$, for all $r \in \ZZ$.
\end{lem}

\begin{rem}\label{rem:fil_induced_B}
	Let $S$ be any ring out of $\Acrys(\Rbar)$, $\OAcrys(\Rbar)$, $\OARpi^{\PD}$, $\Rpi^{\bmstar}$, $\fERpi^{\bmstar}$, $\fERbar^{\bmstar}$, for $\smstar \in \{\textpd, [u], [u, v]\}$.
	Then, by Remark \ref{rem:oarpd_erpd}, Remark \ref{rem:fil_B_BdR} and Lemma \ref{lem:fil_B_er}, it is easy to see that $\Fil^r S = S[1/p] \cap \Fil^r \pazb \subset \pazb$ and $\Fil^r(S[1/p]) := S[1/p] \cap \Fil^r \pazb = (\Fil^r S)[1/p] \subset \pazb$, for all $r \in \ZZ$.
\end{rem}

\subsubsection{Filtered modules}

Let $V$ be a de Rham representation of $G_R$ and set $D := \ODdR(V)$ as a finite projective $R[1/p]\textrm{-module}$.
From \S \ref{subsec:relative_padic_reps}, recall that $D$ is equipped with a decreasing, separated and exhaustive filtration by $R[1/p]\textrm{-submodules}$ $\{\Fil^r D\}_{r \in \ZZ}$, such that $\Fil^a D = D$ and $\Fil^b D = 0$, for some $a, b \in \ZZ$, and for each $r \in \ZZ$, the $R[1/p]\textrm{-modules}$ $\Fil^r D$ and $\gr^r D$ are finite projective.
Recall that by the definition of de Rham representations, we have a natural $\OBdR(\Rbar)\textrm{-linear}$ isomorphism $\alpha_{\dr} : \OBdR(\Rbar) \otimes_{R[1/p]} D \isomorphic \OBdR(\Rbar) \otimes_{\QQ_p} V$.
Extending scalars of the isomorphism $\alpha_{\dr}$ along the natural map $\OBdR(\Rbar) \rightarrow \pazb$ from \S \ref{subsubsec:obese_rings}, we obtain the following $\pazb\textrm{-linear}$ isomorphism,
\begin{equation}\label{eq:alpha_B}
	\alpha_{\pazb} : \pazb \otimes_{R[1/p]} D \isomorphic \pazb \otimes_{\QQ_p} V.
\end{equation}

Next, let $S \subset \pazb$ be an $R\textrm{-subalgebra}$ equipped with a filtration induced from the filtration on $\pazb$ (see after Lemma \ref{lem:fil_B}), such that the natural map $R \rightarrow S$ is injective and $p$ is not invertible in $S$.
Now, consider the $S[1/p]\textrm{-module}$ $D_S := S \otimes_R D$ and we equip $D_S$ with the induced filtration, for each $r \in \ZZ$:
\begin{equation}\label{eq:filr_ds}
	\Fil^r D_S := D_S \cap \alpha_{\pazb}^{-1}(\Fil^r \pazb \otimes_{\QQ_p} V).
\end{equation}

\begin{prop}\label{prop:filr_ds_tensorprod}
	The filtration on $D_S$ in \eqref{eq:filr_ds} coincides with the tensor product filtration, i.e.\ for each $r \in \ZZ$, we have that
	\begin{equation}\label{eq:filr_ds_tensorprod}
		F^r D_S := \textstyle\sum_{i+j=r}^r \Fil^i S \otimes_R \Fil^j D \isomorphic \Fil^r D_S.
	\end{equation}
\end{prop}
\begin{proof}
	From Lemma \ref{lem:fil_alpha_B}, we have that $F^r (\pazb \otimes_{R[1/p]} D) \isomorphic \Fil^r (\pazb \otimes_{R[1/p]} D)$, for each $r \in \ZZ$.
	Then, by using Lemma \ref{lem:fil_gr_ds} and Lemma \ref{lem:filr_induced_rat} below, with $S' = \pazb$, we obtain the isomorphism in \eqref{eq:filr_ds_tensorprod}.
\end{proof}

\begin{rem}\label{prop:filr_ds_tensorprod_rat}
	If $p$ is invertible in $S$, then by using $R[1/p]$ in place of $R$ in Proposition \ref{prop:filr_ds_tensorprod}, we get that $\textstyle\sum_{i+j=r}^r \Fil^i S \otimes_{R[1/p]} \Fil^j D \isomorphic \Fil^r D_S$, for each $r \in \ZZ$.
\end{rem}

Following observations were used above:
\begin{lem}\label{lem:fil_gr_ds}
	For each $i, j, r \in \ZZ$ such that $i+j = r$, the natural map $\Fil^i S \otimes_R \Fil^j D \rightarrow D_S$ is injective.
	In particular, the filtration $\{F^r D_S\}_{r \in \ZZ}$ is a well-defined $\ZZ\textrm{-indexed}$ decreasing filtration on $D_S$ by $S[1/p]\textrm{-submodules}$.
	Moreover, we have that $\gr_F^r D_S = \oplus_{i+j=r} \gr^i S \otimes_R \gr^j D$.
\end{lem}
\begin{proof}
	For each $j \in \ZZ$, let us consider the following exact sequence of finite projective $R[1/p]\textrm{-modules}$, in particular, flat $R\textrm{-modules}$,
	\begin{equation}\label{eq:grj_d}
		0 \longrightarrow \Fil^{j+1} D \longrightarrow \Fil^j D \longrightarrow \gr^j D \longrightarrow 0.
	\end{equation}
	Extending scalars in \eqref{eq:grj_d} along the natural map $R \rightarrow S$ and by decreasing induction on $j \geqslant a$, it is easy to see that the natural map $S \otimes_R \Fil^j D \rightarrow S \otimes_R \Fil^a D = S \otimes_R D$ is injective.
	Therefore, for any $i + j = r$, it follows that the natural map $\Fil^i S \otimes_R \Fil^j D \hookrightarrow S \otimes_R \Fil^j D \rightarrow D_S$ is injective, where the first arrow is obtained by tensoring the $R\textrm{-linear}$ inclusion $\Fil^i S \subset S$ with the flat $R\textrm{-module}$ $\Fil^j D$ and the second arrow is as above.
	Hence, for each $r \in \ZZ$, we get that $F^r D_S := \sum_{i+j=r} \Fil^i S \otimes_R \Fil^j D$ is an $S[1/p]\textrm{-submodule}$ of $D_S$.
	It is clear that the filtration is decreasing.
	Next, let us note that upon tensoring \eqref{eq:grj_d} with $\Fil^i S$ and $\gr^i S$, we obtain the following $R\textrm{-linear}$ commutative diagram:
	\begin{equation}\label{eq:grij_ds}
		\begin{tikzcd}[column sep=small, row sep=small]
			& 0 \arrow[d] & 0 \arrow[d] & 0 \arrow[d]\\
			0 \arrow[r] & \Fil^{i+1} S \otimes_R \Fil^{j+1} D \arrow[r] \arrow[d] & \Fil^{i+1} S \otimes_R \Fil^j D \arrow[r] \arrow[d] & \Fil^{i+1} S \otimes_R \gr^j D \arrow[r] \arrow[d] & 0\\
			0 \arrow[r] & \Fil^i S \otimes_R \Fil^{j+1} D \arrow[r] \arrow[d] & \Fil^i S \otimes_R \Fil^j D \arrow[r] \arrow[d] & \Fil^i S \otimes_R \gr^j D \arrow[r] \arrow[d] & 0\\
			0 \arrow[r] & \gr^i S \otimes_R \Fil^{j+1} D \arrow[r] \arrow[d] & \gr^i S \otimes_R \Fil^j D \arrow[r] \arrow[d] & \gr^i S \otimes_R \gr^j D \arrow[r] \arrow[d] & 0\\
			& 0 & 0 & 0.
		\end{tikzcd}
	\end{equation}
	Since $\Fil^j D$ and $\gr^j D$ are finite projective modules over $R[1/p]$, in particular, flat modules over $R$, we get that all rows and columns of \eqref{eq:grij_ds} are exact.
	From the diagram, it easily follows that we have $\gr_F^r D_S = \oplus_{i+j=r} \gr^i S \otimes_R \gr^j D$, for each $r \in \ZZ$.
\end{proof}

\begin{rem}\label{rem:fil_gr_ds_rat}
	If $p$ is invertible in $S$, then by employing arguments similar to the proof of Lemma \ref{lem:fil_gr_ds} (use $R[1/p]$ in place of $R$), we see that the $S\textrm{-module}$ $D_S := S \otimes_{R[1/p]} D$ is equipped with a well-defined $\ZZ\textrm{-indexed}$ decreasing tensor product filtration by $S\textrm{-submodules}$ given as $F^r D_S := \sum_{i+j=r} \Fil^i S \otimes_{R[1/p]} \Fil^j D$.
	Moreover, for each $r \in \ZZ$, we have that $\gr_F^r D_S = \oplus_{i+j=r} \gr^i S \otimes_{R[1/p]} \gr^j D[1/p]$.
\end{rem}

Next, let $S \subset S' \subset \pazb$ be two $R\textrm{-subalgebras}$ equipped with the respective induced filtrations such that the natural map $R \rightarrow S$ is injective.
Set $D_S := S \otimes_R D$ and $D_{S'} := S' \otimes_R D$, equipped with the tensor product filtration as in Lemma \ref{lem:fil_gr_ds}.
Then, we claim the following:
\begin{lem}\label{lem:filr_induced_rat}
	For each $r \in \ZZ$, we have that $F^r D_{S'} \cap D_S = F^r D_S$, as submodules of $D_{S'}$.
\end{lem}
\begin{proof}
	We will prove the claim by assuming that $p$ is not invertible in $S'$; in the case that $p$ is invertible in either $S$ or $S'$ the same argument works by using Remark \ref{rem:fil_gr_ds_rat} and replacing $R$ with $R[1/p]$.
	Now, let us first note that an easy induction on $r$ shows that proving the equality $F^{r+1} D_{S'} \cap D_S = F^{r+1} D_S$ is equivalent to proving the equality $F^{r+1} D_{S'} \cap F^r D_S = F^{r+1} D_S$.
	Next, consider the following diagram with $R\textrm{-linear}$ exact rows,
	\begin{equation}\label{eq:grr_ds}
		\begin{tikzcd}[row sep=small]
			0 \arrow[r] & F^{r+1} D_S \arrow[r] \arrow[d] & F^r D_S \arrow[r] \arrow[d] & \gr^r_F D_S \arrow[r] \arrow[d] & 0\\
			0 \arrow[r] & F^{r+1} D_{S'} \arrow[r] & F^r D_{S'} \arrow[r] & \gr_F^r D_{S'} \arrow[r] & 0.
		\end{tikzcd}
	\end{equation}
	Note that proving the equality $F^{r+1} D_{S'} \cap F^r D = F^{r+1} D_S$ is equivalent to showing that the right vertical arrow in the diagram \eqref{eq:grr_ds} is injective.
	Now, by Lemma \ref{lem:fil_gr_ds} we have that $\gr_F^r D_S = \oplus_{i+j=r} \gr^i S \otimes_R \gr^j D$, for each $r \in \ZZ$.
	Similarly, we also have that $\gr_F^r D_{S'} = \oplus_{i+j=r} \gr^i S' \otimes_R \gr^j D$, for each $r \in \ZZ$.
	Since $\Fil^i S' \cap S = \Fil^i S$, therefore, by using a diagram similar to \eqref{eq:grr_ds}, it follows that the natural $R\textrm{-linear}$ map $\gr^i S \rightarrow \gr^i S'$ is injective, for all $i \in \ZZ$.
	Furthermore, as $\gr^i D$ is flat over $R$, it follows that the natural map $\gr^i S \otimes_R \gr^j D \rightarrow \gr^i S' \otimes_R \gr^j D$ is also injective.
	Hence, we get that the right vertical arrow in \eqref{eq:grr_ds} is injective, allowing us to conclude.
\end{proof}

Now, by using \cite[Proposition 8.4.3]{brinon-padicrep-relatif}, note that the isomorphism $\alpha_{\dr}$ is compatible with the tensor product filtration of Remark \ref{rem:fil_gr_ds_rat} on the source and the filtration on the target is induced by the natural filtration on $\OBdR(\Rbar)$.
As the natural filtration on $\OBdR(\Rbar)$ coincides with the induced filtration via the inclusion $\OBdR(\Rbar) \subset \pazb$ (see Remark \ref{rem:fil_B_BdR}), it follows that we have $F^r (\OBdR(\Rbar) \otimes_{R[1/p]} D) \isomorphic \Fil^r (\OBdR(\Rbar) \otimes_{R[1/p]} D)$, for all $r \in \ZZ$.
By using Lemma \ref{lem:fil_B} in an argument similar to the proof of \cite[Proposition 8.3.2]{brinon-padicrep-relatif}, we obtain the following:
\begin{lem}\label{lem:fil_alpha_B}
	The isomorphism in \eqref{eq:alpha_B} induces an isomorphism $\alpha_{\pazb}(F^r(\pazb \otimes_{R[1/p]} D)) \isomorphic \Fil^r \pazb \otimes_{\QQ_p} V$, for all $r \in \ZZ$.
	In particular, we get that $F^r(\pazb \otimes_{R[1/p]} D) \isomorphic \Fil^r(\pazb \otimes_{R[1/p]} D)$.
\end{lem}
\begin{proof}
	Note that \eqref{eq:alpha_B} is an isomorphism and the filtation on $D$ is exhaustive, so it is enough to show that the maps on the associated graded pieces, induced by \eqref{eq:alpha_B}, are bijective.
	For each $r \in \ZZ$, consider the following diagram:
	\begin{center}
		\begin{tikzcd}[row sep=small]
			\oplus_{i+j=r} \gr^i \OBdR(\Rbar) \otimes_{R[1/p]} \gr^j M[1/p] \arrow[r, "\sim"] \arrow[d] & \gr^r \OBdR(\Rbar) \otimes_{\QQ_p} V \arrow[d]\\
			\oplus_{i+j=r} \gr^i \pazb \otimes_{R[1/p]} \gr^j M[1/p] \arrow[r] & \gr^r \pazb \otimes_{\QQ_p} V,
		\end{tikzcd}
	\end{center}
	where the top horizontal arrow is the isomorphism induced by the filtration compatible $\OBdR(\Rbar)\textrm{-linear}$ isomorphism $\alpha_{\dr}$, the left vertical arrow is induced by the compatibility of filtrations on the source of $\alpha_{\dr}$ and $\alpha_{\pazb}$ (see Lemma \ref{lem:filr_induced_rat}) and the right vertical arrow is induced by the compatibility of filtrations on the target of $\alpha_{\dr}$ and $\alpha_{\pazb}$ (see Lemma \ref{lem:fil_B} and Remark \ref{rem:fil_B_BdR}).
	Now, recall that from Lemma \ref{lem:fil_B} we have $\gr^i \pazb \isomorphic t^i \CC(\Rbar)[y_0, \ldots, y_d]$ and from \cite[Proposition 5.2.6]{brinon-padicrep-relatif} we have that $\gr^i \OBdR(\Rbar) \isomorphic t^i \CC(\Rbar)[y_1, \ldots, y_d]$.
	In particular, we see that $\gr^i \pazb \isomorphic \ZZ[y_0] \otimes_{\ZZ} \gr^i \OBdR(\Rbar)$.
	Therefore, it follows that the bottom horizontal arrow of the diagram above is given as the extension of scalars along $\ZZ \rightarrow \ZZ[y_0]$ of the top horizontal arrow, hence, it is also an isomorphism.
	This allows us to conclude.
\end{proof}

Next, let us note some applications of Proposition \ref{prop:filr_ds_tensorprod}, which will be used in \S \ref{sec:syntomic_complex_finite_height}.
\begin{lem}\label{lem:ms_fil_pi_cap}
	Let $S = \fERpi^{[u, v]}$ and set $D_S := \fERpi^{[u, v]} \otimes_R D$, equipped with the tensor product filtration as in Lemma \ref{lem:fil_gr_ds}.
	Assume that $\Fil^0 D = D$.
	Then, for any $r \in \NN$, we have that $\Fil^r D_S \cap \pi D_S = \pi \Fil^{r-1} D_S$, as submodules of $D_S$.
\end{lem}
\begin{proof}
	The claim is trivial for $r = 0$, so assume that $r \geqslant 1$.
	We will prove the claim by induction on $r$.
	Note that for $r = 1$, we have that $\Fil^r D_S \cap \pi D_S = \pi D_S$.
	So, let $r \in \NN_{\geqslant 2}$ and assume that the claim is true for $r-1$, i.e.\ $\Fil^{r-1} D_S \cap \pi D_S = \pi \Fil^{r-2} D_S$.
	Then, we see that,
	\begin{equation*}
		\Fil^r D_S \cap \pi D_S = \Fil^r D_S \cap \Fil^{r-1} D_S \cap \pi D_S =  \Fil^r D_S \cap \pi \Fil^{r-2} D_S.
	\end{equation*}
	In particular, to get the claim, it is enough to show that $\Fil^r D_S \cap \pi \Fil^{r-2} D_S = \pi \Fil^{r-1} D_S$.
	Now, consider the following diagram with exact rows,
	\begin{equation}\label{eq:ds_fil_pi_cap}
		\begin{tikzcd}[row sep=small]
			0 \arrow[r] & \Fil^{r-1} D_S \arrow[r] \arrow[d, "\pi"] & \Fil^{r-2} D_S \arrow[r] \arrow[d, "\pi"] & \gr^{r-2} D_S \arrow[r] \arrow[d, "\pi"] & 0\\
			0 \arrow[r] & \Fil^r D_S \arrow[r] & \Fil^{r-1} D_S \arrow[r] & \gr^{r-1} D_S \arrow[r] & 0,
		\end{tikzcd}
	\end{equation}
	where the left and middle vertical arrows are $\textrm{multiplication-by-}\pi$ and the right vertical arrow is the induced map, which we again denote as $\textrm{multiplication-by-}\pi$.
	Note that that all the vertical arrows in \eqref{eq:ds_fil_pi_cap} are $R\textrm{-linear}$.
	Moreover, from the diagram \eqref{eq:ds_fil_pi_cap}, we see that showing the equality $\Fil^r D_S \cap \pi \Fil^{r-2} D_S = \pi \Fil^{r-1} D_S$ is equivalent to showing that the right vertical arrow in \eqref{eq:ds_fil_pi_cap} is injective.
	Note that by using Lemma \ref{lem:fil_gr_ds} and Remark \ref{rem:fil_gr_ds_rat}, we have that $\gr^{r-2} D_S = \oplus_{i+j=r-2} \gr^i S \otimes_R \gr^j D$ and similarly for $\gr^{r-1} D_S$.
	Therefore, the right vertical arrow in \eqref{eq:ds_fil_pi_cap} induces $R\textrm{-linear}$ maps $\gr^i S \otimes_R \gr^j D \xrightarrow{\pi} \gr^{i-1} S \otimes_R \gr^j D$, for $i+j = r-2$.
	As $\gr^j D$ is a flat $R\textrm{-module}$ and the preceding map is $R\textrm{-linear}$, it is enough to show that the map $\gr^i S \xrightarrow{\pi} \gr^{i+1} S$, induced from the $\textrm{multiplication-by-}\pi$ map $\Fil^i S \xrightarrow{\pi} \Fil^{i+1} S$, is injective.
	This follows from Lemma \ref{lem:erpi_fil_pi_cap}.
	Hence, we obtain that the right vertical arrow in \eqref{eq:ds_fil_pi_cap} is injective, in particular, $\Fil^r D_S \cap \pi D_S = \pi \Fil^{r-1} D_S$, for each $r \in \NN$.
\end{proof}

Now, let us assume that $D$ is finite free over $R[1/p]$, we have $\Fil^0 D = D$ and there exists a finite free $R\textrm{-submodule}$ $M \subset D$ such that $M[1/p] = D$.
Let $S$ and $S'$ be as in Lemma \ref{lem:filr_induced_rat} and equip $M_S$ and $M_{S'}$ with induced filtrations, i.e.\ $\Fil^r M_S := M_S \cap \Fil^r D_S \subset D_S$ and $\Fil^r M_{S'} := M_{S'} \cap \Fil^r D_{S'} \subset D_{S'}$.
As $M$ is free over $R$, the natural map $M_S \rightarrow M_{S'}$ is injective and we note the following:
\begin{lem}\label{lem:filr_induced}
	For each $r \in \NN$, we have $\Fil^r M_S = \Fil^r M_{S'} \cap M_S$, as submodules of $M_{S'}$.
	Moreover, if $S = \fERpi^{[u, v]}$, then we have $\Fil^r M_S \cap \pi M_S = \pi \Fil^{r-1} M_S$, as submodules of $M_S$.
\end{lem}
\begin{proof}
	The first claim follows from the definition of filtration on $M_S$ and $M_{S'}$ and Lemma \ref{lem:filr_induced_rat}.
	For the second claim, by Lemma \ref{lem:ms_fil_pi_cap}, we have $\Fil^r M_S \cap \pi M_S = \pi \Fil^{r-1} D_S \cap \pi M_S = \pi \Fil^{r-1} M_S$.
\end{proof}

\subsubsection{Poincar\'e Lemma}\label{subsubsec:crys_fil_poincare_lem}

In the notation of Definition \ref{defi:fat_ring_const}, let us set $A = \ARpi^{\bmstar}$, $B = \Rpi^{\bmstar}$ and $E = \fERpi^{\bmstar}$, for $\smstar \in \{\textpd, [u], [u, v]\}$.
Let $\omega_0 := \frac{d[X_0^{\flat}]}{1+[X_0^{\flat}]}$ and $\omega_i := \frac{d[X_i^{\flat}]}{[X_i^{\flat}]}$, for $1 \leqslant i \leqslant d$.
Set $\Omega^1 := \oplus_{i=1}^d \ZZ \omega_i$ and $\Omega^k := \bmwedge^k \Omega^1$, for all $k \in \NN$.
Then, we have $\Omega^k_{E/B} = E \otimes_{\ZZ} \Omega^k$ and from Remark \ref{rem:elements_of_pd_ring} (iv), note that for $r \in \ZZ$, we have the following filtered de Rham complex of $E$ relative to $B$,
\begin{equation*}
	\Fil^r \Omega^{\bullet}_{E/B} := \Fil^r E \longrightarrow \Fil^{r-1} E \otimes_{\ZZ} \Omega^1 \longrightarrow \Fil^{r-2} E \otimes_{\ZZ} \Omega^2 \longrightarrow \cdots.
\end{equation*}

From the discussion before Lemma \ref{lem:filr_induced}, let $M$ be a finite free $R\textrm{-module}$ such that $M[1/p] = \ODcrys(V)$, where $V$ is a positive crystalline representation of $G_R$.
Moreover, we set $M_B := B \otimes_R M$, equipped with a filtration induced from the tensor product filtration on $M_B[1/p]$, and similarly, we set $M_E := E \otimes_R M$, equipped with a filtration induced from the tensor product filtration on $M_E[1/p]$.
Furthermore, the $B\textrm{-linear}$ differential operator on $E$ induces a quasi-nilpotent integrable connection $\partial : M_E \rightarrow M_E \otimes_E \Omega_{E/B}^1$ satisfying Griffiths transversality with respect to the filtration (since $\partial(\Fil^r E) \subset \Fil^{r-1} E$).
In particular, for each $r \in \ZZ$, we have the following filtered de Rham complex,
\begin{align*}
	\Fil^r M_E \otimes \Omega^{\bullet}_{E/B} &:= \Fil^r M_E \longrightarrow \Fil^{r-1} M_E \otimes_E \Omega_{E/B}^1 \longrightarrow \Fil^{r-2} M_E \otimes_E \Omega_{E/B}^2 \longrightarrow \cdots \\
	&\hspace{2mm}= \Fil^r M_E \longrightarrow \Fil^{r-1} M_E \otimes_{\ZZ} \Omega^1 \longrightarrow \Fil^{r-2} M_E \otimes_{\ZZ} \Omega^2 \longrightarrow \cdots.
\end{align*}
Using the equality $M_B = M_E^{\partial=0}$ and Lemma \ref{lem:filr_induced}, let us note that we have $\Fil^r M_B = \Fil^r M_E \cap M_E^{\partial = 0} = (\Fil^r M_E)^{\partial=0}$ and we obtain the following filtered Poincar\'e Lemma:

\begin{lem}\label{lem:fil_poincare_lem_mb}
	The natural map $\Fil^r M_B \rightarrow \Fil^r M_E \otimes \Omega^{\bullet}_{E/B}$ is a quasi-isomorphism.
\end{lem}
\begin{proof}
	We have a natural injection $\epsilon : \Fil^r M_B \rightarrow \Fil^r M_E$, so we give a contracting ($B\textrm{-linear}$) homotopy.
	Note that $M$ is a finite free $R\textrm{-module}$, so we may choose $\{f_1, \ldots, f_h\}$ as an $R\textrm{-basis}$ of $M$.
	Now define a $B\textrm{-linear}$ map $h^0 : M_E \rightarrow M_B$ as $\sum_{j=1}^h a_j f_j \mapsto \sum_{j=1}^h a_{j,0} f_j$, where $a_j$ is in $E$ and $a_{j,0}$ is the projection to the $0\textrm{-th}$ coordinate (see Remark \ref{rem:elements_of_pd_ring} (iii), where 0 corresponds to the coordinate $(0, \ldots, 0)$).
	Moreover, note that after inverting $p$ and using the tensor product filtration on $M_E[1/p]$, we get that $h^0$ induces a $B[1/p]\textrm{-linear}$ map $h^0 : \Fil^r M_E[1/p] \rightarrow \Fil^r M_B[1/p]$.
	In particular, we obtain an induced $B\textrm{-linear}$ map $h^0 : \Fil^r M_E \rightarrow M_B \cap \Fil^r M_B[1/p] = \Fil^r M_B$.
	It is clear that we have $h^0 \epsilon = id$.

	Next, for $q > 0$, define a $B\textrm{-linear}$ map $h^q : M_E \otimes_{\ZZ} \Omega^q \rightarrow M_E \otimes_{\ZZ} \Omega^{q-1}$, given by the formula $h^q\Big(f_j a_j \prod_{i=0}^d (V_i-1)^{[k_i]} V_{i_1} \omega_{i_1} \wedge \cdots \wedge V_{i_q} \omega_{i_q}\Big) = f_j a_j \prod_{i=0}^d (V_i-1)^{[k_i+\delta_{ji_1}]} V_{i_2} \omega_{i_2} \wedge \cdots \wedge V_{i_q} \omega_{i_q}$, if $k_j=0$ and 0 otherwise (here $\delta$ denotes the Kronecker $\delta\textrm{-symbol}$).
	Moreover, note that after inverting $p$ and using the tensor product filtration on $M_E[1/p]$, we get that $h^q$ induces a $B[1/p]\textrm{-linear}$ map $h^q : \Fil^{r-q} M_E[1/p] \otimes_{\ZZ} \Omega^q \rightarrow \Fil^{r-q+1} M_E[1/p] \otimes_{\ZZ} \Omega^{q-1}$.
	In particular, we obtain an induced $B\textrm{-linear}$ map $h^q : \Fil^{r-q} M_E \otimes_{\ZZ} \Omega^q \rightarrow \Fil^{r-q+1} M_E \otimes_{\ZZ} \Omega^{q-1}$.
	It is easy to see $\epsilon h^0 + h^1d = id$ and $d h^q + h^{q+1}d = id$.
	Hence, we obtain the desired $B\textrm{-linear}$ homotopy, proving the claim.
\end{proof}


\section{Finite height \texorpdfstring{$p\textrm{-adic}$}{-} representations}\label{sec:relative_finite_height}

In this section, we will recall the notion of relative Wach modules from \cite{abhinandan-crystalline-wach} and prove some lemmas that will be used later.
We will use the setup and notations of \S \ref{subsec:setup_nota}, in particular, we fix some $m \in \NN_{\geqslant 1}$.

\begin{nota}
	For an algebra $S$ admitting a Frobenius endomorphism $\varphi$ and an $S\textrm{-module}$ $M$ admitting a Frobenius-semilinear endomorphism $\varphi : M \rightarrow M$, we will denote by $\varphi^{\ast}(M) \subset M$ the $S\textrm{-submodule}$ generated by the image of $\varphi$.
\end{nota}

\subsection{Relative Wach modules}\label{subsec:relative_wach_modules}

Set $q := \varphi(\pi)/\pi$ in $\AR^+$ and let $T$ be a free $\ZZ_p\textrm{-representation}$ of $G_R$.
Then, note that we have an $\AR^+\textrm{-submodule}$ $\mbfd^+(T) := (\mbfa^+ \otimes_{\QQ_p} T)^{H_{R}} \subset \mbfd(T)$, equipped with induced commuting actions of $(\varphi, \Gamma_R)$.

\begin{defi}[{\cite[Definition 4.8]{abhinandan-crystalline-wach}}]\label{defi:wach_reps}
	A $\ZZ_p\textrm{-representation}$ $T$ is said to be \textit{positive} and of \textit{finite $q\textrm{-height}$} if there exists a finite projective $\AR^+\textrm{-submodule}$ $\mbfn(T) \subset \mbfd^+(T)$ of rank $=\textrm{rk}_{\ZZ_p} T$, stable under the action of $\varphi$ and $\Gamma_{R}$ and satisfying the following conditions:
	\begin{enumromanup}
		\item The natural $\AR\textrm{-linear}$ map $\AR \otimes_{\AR^+} \mbfn(T) \isomorphic \mbfd(T)$ is a $(\varphi, \Gamma_R)\textrm{-equivariant}$ isomorphism, where $\mbfn(T)$ is equipped with the induced action of $(\varphi, \Gamma_R)$;
		
		\item The $\AR^+\textrm{-module}$ $\mbfn(T) / \varphi^{\ast}(\mbfn(T))$ is killed by $q^{s}$ for some $s \in \NN$;
		
		\item The induced action of $\Gamma_R$ on $\mbfn(T) / \pi \mbfn(T)$ is trivial;
		
		\item There exists $R\prm \subset \Rbar$ finite \'etale over $R$ such that $\mbfa_{R\prm}^+ \otimes_{\AR^+} \mbfn(T)$ is free over $\mbfa_{R\prm}^+$.
	\end{enumromanup}
	The \textit{height} of $T$ is defined to be the smallest $s \in \NN$ satisfying (ii) above.
	Furthermore, a positive finite $q\textrm{-height}$ $\padic$ representation $V$ of $G_R$ is a representation admitting a positive finite $q\textrm{-height}$ $\ZZ_p\textrm{-lattice}$ $T \subset V$ and we set $\mbfn(V) := \mbfn(T)[1/p]$, satisfying properties analogous to (i)-(iv) above.
	The height of $V$ is defined to be the height of $T$.
	For $k \in \ZZ$, let $T(k) := T \otimes_{\ZZ_p} \ZZ_p(k)$, $V(k) := T(k)[1/p]$, define $\mbfn(T(k)) := \frac{1}{\pi^k}\mbfn(T)(k)$ and $\mbfn(V(k)) := \frac{1}{\pi^k} \mbfn(V)(k)$ and set $\textrm{height of } T(k) = (\textrm{height of } T) - k$.
	We call these twists as representations of \textit{finite $q\textrm{-height}$}.
\end{defi}

For general properties of Wach modules, we refer the reader to \cite[\S 4.2]{abhinandan-crystalline-wach}.
Let us note that there is a natural filtration on Wach modules attached to finite $q\textrm{-height}$ representations.
\begin{defi}\label{defi:wach_mod_fil}
	Let $V$ be a finite $q\textrm{-height}$ represenation of $G_R$.
	For each $r \in \ZZ$, set $\Fil^r \mbfn(V) := \{x \textrm{ in } \mbfn(V), \hspace{1mm} \textrm{such that} \hspace{1mm} \varphi(x) \textrm{ is in } q^r \mbfn(V)\}$ and $\Fil^r \mbfn(T) := \Fil^r \mbfn(V) \cap \mbfn(T) \subset \mbfn(V)$.
\end{defi}

\begin{lem}\label{lem:wach_mod_twist_fil}
	We have $\Fil^r \mbfn(T) = \{x \textrm{ in } \mbfn(T), \hspace{1mm} \textrm{such that} \hspace{1mm} \varphi(x) \textrm{ is in } q^r \mbfn(T)\}$.
	Moreover, we have that $\Fil^r \mbfn(T(k)) = \pi^{-k} \Fil^{r+k} \mbfn(T)(k)$ and $\Fil^r \mbfn(V(k)) = \pi^{-k} \Fil^{r+k} \mbfn(V)(k)$.
\end{lem}
\begin{proof}
	The first claim is true because $q^r \mbfn(V) \cap \mbfn(T) = (q^r \BR^+ \cap \AR^+) \otimes_{\AR^+} \mbfn(T) = q^r \mbfn(T)$.
	To show the second claim, let $\pi^{-k} x \otimes \epsilon^{\otimes k}$ be an element of $\Fil^r \pi^{-k} \mbfn(T)(k)$, with $x \in \mbfn(T)$ and $\epsilon^{\otimes k}$ a $\ZZ_p\textrm{-basis}$ of $\ZZ_p(k)$.
	By assumption, $\varphi(\pi^{-k} x \otimes \epsilon^{\otimes k}) = (q\pi)^{-k} \varphi(x) \otimes \epsilon^{\otimes k}$ belongs to $q^r \pi^{-k} \mbfn(T)(k)$.
	Therefore, we see that $\varphi(x)$ belongs to $q^{r+k} \mbfn(T)$, i.e.\ $x$ is in $\Fil^{r+k} \mbfn(T)$.
	The converse is obvious.
\end{proof}

\begin{rem}\label{rem:wach_mod_fil_ainf}
	Set $\Fil^r \Ainf(\Rbar) := \xi^r \Ainf(\Rbar)$ and $\Fil^r \mbfa := \mbfa \cap \Fil^r \Ainf(\Rbar) \subset \Ainf(\Rbar)$, for each $r \in \NN$.
	If $T$ is a positive finite $q\textrm{-height}$ $\ZZ_p\textrm{-representation}$ of $G_R$, then from \cite[Lemma 4.53]{abhinandan-crystalline-wach} note that, for the filtration on Wach modules as in Definition \ref{defi:wach_mod_fil}, we have $\Fil^r \mbfn(T) = \mbfn(T) \cap \Fil^r \Ainf(\Rbar) \otimes_{\ZZ_p} T = \mbfn(T) \cap \Fil^r \mbfa \otimes_{\ZZ_p} T \subset \Ainf(\Rbar) \otimes_{\ZZ_p} T$, for each $r \in \NN$.
\end{rem}

The operator $\psi$ defined in \S \ref{subsec:relative_phi_gamma_mod} commutes with the action of $G_R$, so by linearity, it extends to a map $\psi : \mbfd(T) \rightarrow \mbfd(T)$ and from Proposition \ref{prop:psi_oper} we get that $\psi(\mbfd^+(T)) \subset \mbfd^+(T)$.

\begin{lem}\label{lem:psi_wach_nonpositive}
	Let $T$ be positive finite $q\textrm{-height}$ $\ZZ_p\textrm{-representation}$ of $G_R$ of height $s$.
	Then for $k \geqslant s$, we have $\psi(\mbfn(T(k))) \subset \mbfn(T(k))$.
\end{lem}
\begin{proof}
	Note that we have $q^s \mbfn(T) \subset \varphi^*(\mbfn(T))$.
	So, for $k \geqslant s$ and $x$ in $\mbfn(T(k))$, we must have that $\varphi(\pi^k) x = (q\pi)^k x$ is in $\varphi^{\ast}(\mbfn(T)(k))$.
	Therefore, $\psi(x)$ belongs to $\frac{1}{\pi^k} \mbfn(T)(k) = \mbfn(T(k))$.
\end{proof}

\subsection{Wach modules and crystalline representations}\label{subsec:wach_crystalline}

From \cite[\S 4.3.1]{abhinandan-crystalline-wach}, we have an $R\textrm{-algebra}$ $\OARpi^{\textpd} \subset \OAcrys(\overline{R})$ equipped with a Frobenius endomorphism $\varphi$, a continuous action of $\Gamma_R$, a $\Gamma_R\textrm{-stable}$ filtration and an $\ARpi^{\textpd}\textrm{-linear}$ integrable connection satisfying Griffiths transversality with respect to the filtration and commuting with the action of $\varphi$ and $\Gamma_R$.

\begin{thm}[{\cite[Theorem 4.24, Proposition 4.27, Corollary 4.26]{abhinandan-crystalline-wach}}]\label{thm:crys_wach_comparison}
	Let $V$ be a finite $q\textrm{-height}$ representation of $G_{R}$, then $V$ is crystalline.
	Moreover, if $V$ is positive then we have an isomorphism of $R[1/p]\textrm{-modules}$ $M[1/p] := \big(\OARpi^{\textpd} \otimes_{\AR^+} \mbfn(V)\big)^{\Gamma_R} \isomorphic \ODcrys(V)$, compatible with respective Frobenii, filtrations and connections.
	Furthermore, we have a natural $\OARpi^{\textpd}\textrm{-linear}$ isomorphisms 
	\begin{equation}\label{eq:crys_wach_comparison}
		\OARpi^{\textpd} \otimes_{\AR^+} \mbfn(V) \lisomorphic \OARpi^{\textpd} \otimes_R M[1/p] \isomorphic \OARpi^{\textpd} \otimes_R \ODcrys(V),
	\end{equation}
	compatible with the respective Frobenii, filtrations, connections and the actions of $\Gamma_R$.
\end{thm}

\begin{rem}\label{rem:crys_wach_comparison_struct}
	In Theorem \ref{thm:crys_wach_comparison}, the $\OARpi^{\textpd}\textrm{-module}$ $\OARpi^{\textpd} \otimes_{\AR^+} \mbfn(V)$ is equipped with the following structures:
	a Frobenius endomorphism, given as $\varphi \otimes \varphi$;
	an $\ARpi^{\textpd}\textrm{-linear}$ connection, given by the natural $\ARpi^{\textpd}\textrm{-linear}$ differential operator $\partial_R \otimes 1$ (see Remark \ref{rem:oarpd_erpd} for notations);
	an action of $\Gamma_R$, where any $g$ in $\Gamma_R$ acts as $g \otimes g$;
	an $\NN\textrm{-indexed}$ decreasing filtration given as the tensor product filtration, i.e.\ $\Fil^r \big(\OARpi^{\textpd} \otimes_{\AR^+} \mbfn(V)\big) = \sum_{i+j=r} \Fil^i \OARpi^{\textpd} \otimes_{\AR^+} \Fil^j \mbfn(V)$, which is well defined because each term of the summation is an $\OARpi^{\textpd}\textrm{-submodule}$ of $\OARpi^{\textpd} \otimes_{\AR^+} \mbfn(V)$ (use that as $\AR^+\textrm{-modules}$ $\mbfn(V)$ is finite projective and $\Fil^i \OARpi^{\PD}$ is flat, see \cite[Remark 3.33]{abhinandan-crystalline-wach-ii}).
	The module $M[1/p]$ is equipped with induced structures, in particular, the filtration on $M[1/p]$ is given as $\Fil^r M[1/p] = \big(\Fil^r \big(\OARpi^{\textpd} \otimes_{\AR^+} \mbfn(V)\big)\big)^{\Gamma_R}$ and its compatibility with the Hodge filtration on $\ODcrys(V)$ follows from \cite[\S 4.5.1]{abhinandan-crystalline-wach}.
	Then, in \eqref{eq:crys_wach_comparison}, the middle and right-hand terms are equipped with the following structures: 
	a Frobenius endomorphism, given as $\varphi \otimes \varphi$;
	an $\ARpi^{\textpd}\textrm{-linear}$ connection, given as $\partial_R \otimes 1 + 1 \otimes \partial_D$, where $\partial_D$ is the connection on $\ODcrys(V)$ (see \S \ref{subsec:relative_padic_reps});
	an action of $\Gamma_R$, where any $g$ in $\Gamma_R$ acts as $g \otimes 1$;
	an $\NN\textrm{-indexed}$ decreasing filtration given as the tensor product filtration (see Lemma \ref{lem:fil_gr_ds}), where we use the filtration on $M[1/p]$ as above and the Hodge filtration on $\ODcrys(V)$.
	As the respective connections on $\OARpi^{\textpd}$ and $\ODcrys(V)$ satisfy Griffiths transversality with respect to their respective filtrations, therefore, it follows that the connection on $\OARpi^{\textpd} \otimes_R \ODcrys(V)$ also satisfies Griffiths transversality with respect to the tensor product filtration.
	Then, by the compatibility of the isomorphisms in \eqref{eq:crys_wach_comparison} with connections and filtrations, we see that the respective connection on each term of \eqref{eq:crys_wach_comparison} satisfies Griffiths transversality with respect to the filtration.
	Finally, note that the left-hand isomorphism in \eqref{eq:crys_wach_comparison} is given as $ab \otimes x \mapsfrom a \otimes b \otimes x$.
\end{rem}

The proof of Theorem \ref{thm:crys_wach_comparison} depends on the following important observation:

\begin{lem}[{\cite[Proposition 4.27]{abhinandan-crystalline-wach}}]\label{lem:crys_from_wach_mod}
	Let $V$ be a positive finite $q\textrm{-height}$ representation of $G_R$ such that the $\AR^+\textrm{-module}$ $\mbfn(T)$ is finite free of rank $= \dim_{\QQ_p} V$.
	Then there exists a finite free $R\textrm{-module}$ $M_0 \subset M := \big(\OARpi^{\textpd} \otimes_{\AR^+} \mbfn(T)\big)^{\Gamma_R}$, stable under the Frobenius and such that $M_0[1/p] = M[1/p] \isomorphic \ODcrys(V)$ are free $R[1/p]\textrm{-modules}$ of rank $=\dim_{\QQ_p} V$.
\end{lem}

\begin{prop}\label{prop:odcris_phi_coker}
	Let $V$ be a positive finite $q\textrm{-height}$ representation of $G_R$ of height $s$ such that $\mbfn(T)$ is free over $\AR^+$.
	Let $M_0 \subset M := \big(\OARpi^{\textpd} \otimes_{\AR^+} \mbfn(T)\big)^{\Gamma_R}$ be the free $R\textrm{-module}$ obtained in Lemma \ref{lem:crys_from_wach_mod}.
	Then the $R\textrm{-module}$ $M_0/\varphi^{\ast}(M_0)$ is killed by $p^{ms}$.
\end{prop}
\begin{proof}
	In order to prove the claim, we will use without recalling constructions and notations from the proof of \cite[Proposition 4.28]{abhinandan-crystalline-wach}.
	Let $\smbff = \{f_1, \ldots, f_h\}$ be an $\AR^+\textrm{-basis}$ of $\mbfn(T)$.
	Then from Lemma \ref{lem:crys_from_wach_mod} and the proof of \cite[Proposition 4.28]{abhinandan-crystalline-wach}, we have that $M_0$ is a free $R\textrm{-module}$ with a basis given as $\smbfg = \{g_1, \ldots, g_h\}$, where $\smbfg = \varphi^m(\smbff)\varphi^m(A)$ for some $A$ in $\GL(h, \pazo \widehat{S}_m^{\textpd})$.
	It is easy to see that $M_0$ is independent of the choice of the $\AR^+\textrm{-basis}$ of $\mbfn(T)$.
	Note that we have $q = \varphi(\pi)/\pi = p\varphi(\pi/t)(t/\pi)$, and since $\pi/t$ is a unit in $\OARpi^{\textpd}$ (see Lemma \ref{lem:t_over_pi_unit}), therefore, we obtain that $q$ and $p$ are associates in $\OARpi^{\textpd}$.
	Furthermore, $\mbfn(T)/\varphi^{\ast}(\mbfn(T))$ is killed by $q^s$, where $s$ is the height of $V$.
	So $\big(\OARpi^{\textpd} \otimes_{\AR^+} \mbfn(T)\big) / \varphi^{m, \ast}\big(\OARpi^{\textpd} \otimes_{\AR^+} \mbfn(T)\big)$ is killed by $p^{ms}$, where we write $\varphi^{m, \ast}\big(\OARpi^{\textpd} \otimes_{\AR^+} \mbfn(T)\big) = \oplus_{i=1}^h \OARpi^{\textpd} \varphi^m(f_i)$.
	Now, recall that $\det A$ is a unit in $\pazo \widehat{S}_m^{\textpd}$ (see \cite[Lemma 4.43]{abhinandan-crystalline-wach}), therefore, $\varphi^m(\det A)$ is a unit in $\OARpi^{\textpd}$ and $\varphi^m(A)$ is invertible over $\OARpi^{\textpd}$, in particular, $\OARpi^{\textpd} \otimes_R M_0 \isomorphic \varphi^{m, \ast}\big(\OARpi^{\textpd} \otimes_{\AR^+} \mbfn(T)\big)$.
	So, we get that the cokernel of the natural inclusion $\OARpi^{\textpd} \otimes_R M_0 \subset \OARpi^{\textpd} \otimes_{\AR^+} \mbfn(T)$ is killed by $p^{ms}$.
	Moreover, the observation above also implies that the cokernel of the composition $\varphi^{m, \ast}\big(\OARpi^{\textpd} \otimes_R M_0 \big) \subset \OARpi^{\textpd} \otimes_R M_0 \isomorphic \varphi^{m, \ast}\big(\OARpi^{\textpd} \otimes_{\AR^+} \mbfn(T)\big)$ is killed by $p^{ms}$.
	In other words, we get that $p^{ms} \big(\OARpi^{\textpd} \otimes_R M_0\big) \subset \varphi^{m, \ast}\big(\OARpi^{\textpd} \otimes_R M_0\big) \subset \varphi^{\ast}\big(\OARpi^{\textpd} \otimes_R M_0\big)$.
	Finally, as the action of the Frobenius is $\Gamma_R\textrm{-equivariant}$, therefore, by taking $\Gamma_R\textrm{-invariants}$ we get that $p^{ms} M_0 \subset \varphi^{\ast}(M_0)$, i.e.\ $M_0/\varphi^{\ast}(M_0)$ is killed by $p^{ms}$.
\end{proof}

\begin{rem}\label{rem:odcris_phi_coker}
	From the proof of Proposition \ref{prop:odcris_phi_coker}, note that we have an inclusion $p^s \big(\OARpi^{\textpd} \otimes_{\AR^+} \mbfn(T)\big) \subset \varphi^{\ast}\big(\OARpi^{\textpd} \otimes_{\AR^+} \mbfn(T)\big)$.
	Since the action of Frobenius is $\Gamma_R\textrm{-equivariant}$, therefore, by taking $\Gamma_R\textrm{-invariants}$ of the preceding inclusion, we get that $p^s M \subset \varphi^*(M)$.
	Moreover, from Lemma \ref{lem:crys_from_wach_mod} and Proposition \ref{prop:odcris_phi_coker}, as $M_0 \subset M$, therefore, it also follows that the cokernel of the composition $\OARpi^{\textpd} \otimes_{R} M \rightarrow \OARpi^{\textpd} \otimes_{\AR^+} \mbfn(T)$ is killed by $p^{ms}$ (in fact, the cokernel is killed by $p^s$, see Remark \ref{rem:lattice_depends_on_m}).
\end{rem}

\begin{rem}\label{rem:connection_filtration_M}
	Using Theorem \ref{thm:crys_wach_comparison}, we equip $M \subset M[1/p]$ with a $p\textrm{-adically}$ quasi-nilpotent integrable connection $\partial : M \rightarrow M \otimes_R \Omega^1_R$ and an induced filtration compatible with the tensor product filtration on $\OARpi^{\textpd} \otimes_{\AR^+} \mbfn(V)$ (see \cite[\S 4.5.1]{abhinandan-crystalline-wach}); the connection satisfies Griffiths transversality with respect to the filtration.
	Furthermore, using the explicit description of $M_0$ in Proposition \ref{prop:odcris_phi_coker}, we obtain an induced filtration on $M_0$ and an induced $p\textrm{-adically}$ quasi-nilpotent integrable connection $\partial : M_0 \rightarrow M_0 \otimes_R \Omega^1_R$, satisfying Griffiths transversality with respect to the filtration.
\end{rem}

\begin{rem}\label{rem:lattice_depends_on_m}
	Note that we fixed $m \in \NN_{\geqslant 1}$ in the beginning and the $R\textrm{-modules}$ obtained above depend on this choice.
	In particular, let $1 \leqslant m \leqslant m\prm$ with $\varpi = \zeta_{p^m}-1$ and $\varpi' = \zeta_{p^{m\prm}}-1$.
	Then, we have an inclusion $\OARpi^{\textpd} \subset \pazo \mbfa_{R, \varpi'}^{\textpd}$ and we obtain that $M = \big(\OARpi^{\textpd} \otimes_{\AR^+} \mbfn(T)\big)^{\Gamma_R} \subset \big(\pazo \mbfa_{R, \varpi'}^{\textpd} \otimes_{\AR^+} \mbfn(T)\big)^{\Gamma_R} = M'$.
	As the cokernel of $\OARpi^{\textpd} \otimes_R M \rightarrow \OARpi^{\textpd} \otimes_{\AR^+} \mbfn(T)$ is killed by $p^{ms}$ (see Remark \ref{rem:odcris_phi_coker}) and $\pazo \mbfa_{R, \varpi'}^{\textpd} \otimes_R M \subset \pazo \mbfa_{R, \varpi'}^{\textpd} \otimes_R M'$, therefore, the cokernel of $\pazo \mbfa_{R, \varpi'}^{\textpd} \otimes_R M' \rightarrow \pazo \mbfa_{R, \varpi'}^{\textpd} \otimes_{\AR^+} \mbfn(T)$ is also killed by $p^{ms}$.
	In particular, taking $m = 1$, we see that the cokernel of $\pazo \mbfa_{R, \varpi'}^{\textpd} \otimes_R M' \rightarrow \pazo \mbfa_{R, \varpi'}^{\textpd} \otimes_{\AR^+} \mbfn(T)$ is always killed by $p^s$.
	Finally, let $M_0$ and $M_0\prm$ be $R\textrm{-modules}$ respectively obtained for $m$ and $m\prm$ in Lemma \ref{lem:crys_from_wach_mod}, then we have that $\varphi^{m\prm-m}(M_0\prm) \subset M_0$.
\end{rem}

\subsection{Filtrations and a Poincar\'e Lemma}\label{subsec:wachmod_poincare_lem}

Let $T$ be a positive finite $q\textrm{-height}$ $\ZZ_p\textrm{-representation}$ of $G_R$ and set $V = T[1/p]$.
Let $\mbfn(T)$ denote the associated Wach module over $\AR^+$ and set $M := (\OARpi^{\textpd} \otimes_{\AR^+} \mbfn(T))^{\Gamma_R}$ as a finitely generated $p\textrm{-torsion free}$ $R\textrm{-module}$.
Now consider the following diagram:
\begin{equation}\label{eq:crys_wach_B}
	\begin{tikzcd}
		\pazb \otimes_{R[1/p]} M[1/p] \arrow[r, "\alpha", "\sim"'] \arrow[d, "\wr"'] & \pazb \otimes_{\BR^+} \mbfn(V) \arrow[d, "\wr"', "\beta"]\\
		\pazb \otimes_{R[1/p]} \ODcrys(V) \arrow[r, "\alpha_{\pazb}", "\sim"'] & \pazb \otimes_{\QQ_p} V, 
	\end{tikzcd}
\end{equation}
where $\BR^+ = \AR^+[1/p]$, and the right vertical arrow is the $\pazb\textrm{-linear}$ extension of the natural inclusion $\mbfn(V) \subset \Ainf(\Rbar) \otimes_{\QQ_p} V \subset \pazb \otimes_{\QQ_p} V$; the top horizontal arrow is the extension along $\OARpi^{\textpd} \rightarrow \pazb$ of the $\OARpi^{\textpd}\textrm{-linear}$ isomorphism $\OARpi^{\textpd} \otimes_R M[1/p] \isomorphic \OARpi^{\textpd} \otimes_{\AR^+} \mbfn(V)$ (see the first isomorphism in \eqref{eq:crys_wach_comparison} of Theorem \ref{thm:crys_wach_comparison}); the left vertical arrow is the extension along $R[1/p] \rightarrow \pazb$ of the $R[1/p]\textrm{-linear}$ isomorphism $M[1/p] \isomorphic \ODcrys(V)$ (see the second isomorphism in \eqref{eq:crys_wach_comparison} of Theorem \ref{thm:crys_wach_comparison}), and it is compatible with the respective filtrations; the bottom horizontal arrow is the filtration compatible $\pazb\textrm{-linear}$ isomorphism from \eqref{eq:alpha_B} (see Lemma \ref{lem:fil_alpha_B}).
The diagram commutes by definition and the right vertical arrow is an isomorphism because the other three arrows are isomorphisms (see \cite[\S 4.5]{abhinandan-crystalline-wach} for a similar diagram over $\OBcrys(\Rbar)$).
Using the right vertical arrow of diagram \eqref{eq:crys_wach_B}, for each $r \in \ZZ$, we set
\begin{equation}\label{eq:fil_nB_beta}
	\Fil^r (\pazb \otimes_{\BR^+} \mbfn(V)) := \beta^{-1}(\Fil^r \pazb \otimes_{\QQ_p} V).
\end{equation}
In \eqref{eq:crys_wach_B}, by the compatibility of the left vertical arrow and the bottom horizontal arrow with the respective filtrations, an easy diagram chase shows that, for each $r \in \ZZ$, the top horizontal arrow induces,
\begin{equation}\label{eq:fil_nB_alpha}
	\alpha : \Fil^r(\pazb \otimes_{R[1/p]} M[1/p]) \isomorphic \Fil^r (\pazb \otimes_{\BR^+} \mbfn(V)).
\end{equation}

\subsubsection{Filtration on scalar extensions of Wach modules}

Let $S$ be a ring such that $\AR^+ \subset S \subset \pazb$ and $p$ is not invertible in $S$.
Set $N_S := S \otimes_{\AR^+} \mbfn(T)$ and note that we have a natural embedding $N_S \rightarrow \pazb \otimes_{\BR^+} \mbfn(V)$.
We equip $N_S$ with the induced filtration, i.e.\ for each $r \in \ZZ$, using \eqref{eq:fil_nB_beta}, set
\begin{equation}\label{eq:fil_ns}
	\Fil^r N_S := N_S \cap \Fil^r (\pazb \otimes_{\BR^+} \mbfn(V)) \subset \pazb \otimes_{\BR^+} \mbfn(V).
\end{equation}
Similarly, we set $\Fil^r N_S[1/p] := N_S[1/p] \cap \Fil^r (\pazb \otimes_{\BR^+} \mbfn(V))$, for each $r \in \ZZ$, and it is clear that $\Fil^r N_S = N_S \cap \Fil^r N_S[1/p]$.

\begin{rem}\label{rem:ns_fil_induced}
	Let $S$ and $S'$ be such that $S \subset S' \subset \pazb$ and $p$ is not invertible in $S'$.
	Then, from the definition of the respective filtrations on $N_S$ and $N_{S'}$ in \eqref{eq:fil_ns}, it is clear that $\Fil^r N_S = N_S \cap \Fil^r N_{S'} \subset N_{S'}$.
\end{rem}

\begin{lem}\label{lem:fil_ns_Gstable}
	Let $S \subset \fERbar^{[u, v]}$ be a $G_R\textrm{-stable}$ $\AR^+\textrm{-subalgebra}$.
	Then, the filtration on $N_S$ in \eqref{eq:fil_ns} is stable under the natural action of $G_R$ on $N_S$.
\end{lem}
\begin{proof}
	Let us consider the following commutative diagram,
	\begin{equation}\label{eq:crys_wach_eruv}
		\begin{tikzcd}[row sep=small]
			\fERbar^{[u, v]} \otimes_R M[1/p] \arrow[r, "\alpha", "\sim"'] \arrow[d] & \fERbar^{[u, v]} \otimes_{\AR^+} \mbfn(V) \arrow[d]\\
			\pazb \otimes_{R[1/p]} M[1/p] \arrow[r, "\alpha", "\sim"'] & \pazb \otimes_{\BR^+} \mbfn(V), 
		\end{tikzcd}
	\end{equation}
	where the bottom horizontal arrow is the top horizontal isomorphism of \eqref{eq:crys_wach_B}; the top horizontal arrow is the extension of the $\OARpi^{\textpd}\textrm{-linear}$ isomorphism $\OARpi^{\textpd} \otimes_R M[1/p] \isomorphic \OARpi^{\textpd} \otimes_{\AR^+} \mbfn(V)$ (see the first isomorphism in \eqref{eq:crys_wach_comparison} of Theorem \ref{thm:crys_wach_comparison}), along the $G_R\textrm{-equivariant}$ map $\OARpi^{\textpd} \rightarrow \fERbar^{[u, v]}$ (see Remark \ref{rem:oarpd_erpd}) and compatible with the respective Frobenii, $\ARbar^{[u, v]}\textrm{-linear}$ connections and the actions of $G_R$; the vertical maps are extensions of scalars along the map $\fERbar^{[u, v]} \rightarrow \pazb$ (see Lemma \ref{lem:fil_B_er}).
	Now, by using the definition of filtrations on each term (see \eqref{eq:filr_ds} and \eqref{eq:fil_ns}) and the isomorphism in \eqref{eq:fil_nB_alpha}, the top horizontal arrow induces the following $\fERbar^{[u, v]}\textrm{-linear}$ isomorphism, for each $r \in \ZZ$, 
	\begin{equation}\label{eq:fil_neruv}
		\alpha : \Fil^r \big(\fERbar^{[u, v]} \otimes_R M[1/p]\big) \isomorphic \Fil^r \big(\fERbar^{[u, v]} \otimes_{\AR^+} \mbfn(V)\big).
	\end{equation}
	As the source of \eqref{eq:fil_neruv} is stable under the natural action of $G_R$ on $\fERbar^{[u, v]} \otimes_R M[1/p]$ and the top horizontal arrow of \eqref{eq:crys_wach_eruv} is $G_R\textrm{-equivariant}$, therefore, it follows that the target of \eqref{eq:fil_neruv} is stable under the natural action of $G_R$ on $\fERbar^{[u, v]} \otimes_{\AR^+} \mbfn(V)$.
	Finally, note that we have the $G_R\textrm{-equivariant}$ inclusion $S \subset \fERbar^{[u, v]}$, so by using Remark \ref{rem:ns_fil_induced}, we obtain that $\Fil^r N_S$ is stable under the natural action of $G_R$ on $N_S$.
\end{proof}

\begin{rem}\label{rem:fil_ms_ns}
	Let $S$ be any ring out of $\OAcrys(\Rbar)$, $\fERpi^{\smstar}$ for $\smstar \in \{\textpd, [u], [u, v]\}$, or $\fERbar^{\smstar}$ for $\smstar \in \{\textpd, [u], [u, v]\}$.
	Then, by Lemma \ref{lem:fil_ns_Gstable} we get that, for each $r \in \ZZ$, the isomorphism in \eqref{eq:fil_neruv} induces a $G_R\textrm{-equivariant}$ $S\textrm{-linear}$ isomorphism,
	\begin{equation}\label{eq:fil_ms_ns}
		\alpha : \Fil^r (S \otimes_R M[1/p]) \isomorphic \Fil^r (S \otimes_{\AR^+} \mbfn(V)).
	\end{equation}
	In particular, as the connection on $S \otimes_R M[1/p]$ satisfies Griffiths transversailty with respect to the filtration, therefore, similar to Remark \ref{rem:crys_wach_comparison_struct}, it follows that the connection on $S \otimes_{\AR^+} \mbfn(V)$ satisfies Griffiths transversality with respect to the filtration in \eqref{eq:fil_ns}.
\end{rem}

\begin{rem}\label{rem:fil_neruv_explicit}
	Let $E = \fERpi^{\bmstar}$ or $\fERbar^{\bmstar}$, for $\smstar \in \{\textpd, [u], [u, v]\}$ and we claim that $\Fil^r (E \otimes_{\AR^+} \mbfn(V)) = \sum_{i+j=r} \Fil^i E \cdot \Fil^j \mbfn(V)$, where $\Fil^i E \cdot \Fil^j \mbfn(V)$ denotes the image of $\Fil^i E \otimes_{\AR^+} \Fil^j \mbfn(V) \rightarrow E \otimes_{\AR^+} \mbfn(V)$.
	Indeed, using Lemma \ref{lem:fil_B_er}, Remark \ref{rem:wach_mod_fil_ainf} and \eqref{eq:fil_nB_alpha}, it easily follows that $\Fil^i E \cdot \Fil^j \mbfn(V) \subset \Fil^r(\pazb \otimes_{\AR^+} \mbfn(V))$, in particular, from \eqref{eq:fil_ns} we deduce that $\sum_{i+j=r} \Fil^i E \cdot \Fil^j \mbfn(V) \subset \Fil^r (E \otimes_{\AR^+} \mbfn(V))$.
	To show the reverse inclusion, recall that $\Fil^r M[1/p] \isomorphic \Fil^r \ODcrys(V)$ is a finite projective $R[1/p]\textrm{-module}$ (see Theorem \ref{thm:crys_wach_comparison} and \cite[Proposition 8.3.2]{brinon-padicrep-relatif}), in particular, flat as an $R\textrm{-module}$ and the natural map $\Fil^i E \otimes_R \Fil^j M[1/p] \rightarrow E \otimes_R M[1/p]$ is injective by Lemma \ref{lem:fil_gr_ds}, for each $i, j \in \NN$; we denote the image as $\Fil^i E \cdot \Fil^j M[1/p]$ and note that $\Fil^r (E \otimes_R M[1/p]) = \sum_{i+j=r} \Fil^i E \otimes_R \Fil^j M[1/p] = \sum_{i+j=k} \Fil^i E \cdot \Fil^j M[1/p]$.
	Now, since the isomorphism $E \otimes_R M[1/p] \isomorphic E \otimes_{\AR^+} \mbfn(V)$ is given by the natural multiplication map and the filtration on $M[1/p]$ is given as the tensor product filtration (see Remark \ref{rem:crys_wach_comparison_struct}), therefore, we obtain that the natural map $\sum_{i+j=k} \Fil^i E \cdot \Fil^j M[1/p] \rightarrow \sum_{i+j=r} \Fil^i E \cdot \Fil^j \mbfn(V)$ is injective.
	But from \eqref{eq:fil_ms_ns}, we have that $\Fil^r (E \otimes_R M[1/p]) \isomorphic \Fil^r (E \otimes_{\AR^+} \mbfn(V))$.
	Hence, it follows that $\Fil^r (E \otimes_{\AR^+} \mbfn(V)) = \sum_{i+j=r} \Fil^i E \cdot \Fil^j \mbfn(V)$.
\end{rem}

Next, let $S = \ARpi^{\bmstar}$ for $\smstar \in \{+, \textpd, [u], [u, v], (0, v]+\}$ or $\fERpi^{\bmstar}$ for $\smstar \in \{\textpd, [u], [u, v]\}$ and set $N_S := S \otimes_{\AR^+} \mbfn(T)$.
Then, we have the following:
\begin{lem}\label{lem:ns_fil_pi_cap}
	For each $r \in \ZZ$, we have that $\Fil^r N_S \cap \pi N_S = \pi \Fil^{r-1}N_S$.
\end{lem}
\begin{proof}
	Note that the claim is clear for $r \leqslant 0$, so let $r \geqslant 1$.
	Let $S' = \fERpi^{[u, v]}$ and using the definition of the filtration on $N_{S'}[1/p]$ in \eqref{eq:fil_ns}, the $S'\textrm{-linear}$ isomorphism in \eqref{eq:fil_neruv} and Lemma \ref{lem:ms_fil_pi_cap}, note that
	\begin{align*}
		\Fil^r N_{S'}[1/p] \cap \pi N_{S'}[1/p] &= \alpha(\Fil^r(S' \otimes_R M[1/p])) \cap \alpha(\pi S' \otimes_R M[1/p])\\
		&= \alpha(\Fil^r(S' \otimes_R M[1/p]) \cap \pi (S' \otimes_R M[1/p]))\\
		&= \alpha(\pi \Fil^{r-1} (S' \otimes_R M[1/p])) = \pi \Fil^{r-1} N_{S'}[1/p].
	\end{align*}
	In particular, we get that $\Fil^r N_{S'} \cap \pi N_{S'} = \pi \Fil^{r-1} N_{S'}[1/p] \cap \pi N_{S'} = \pi \Fil^{r-1} N_{S'}$.
	Now, by using the definition of the filtration on $N_S$ in \eqref{eq:fil_ns}, Remark \ref{rem:ns_fil_induced} and the equality above, we get that $\Fil^r N_S \cap \pi N_S \subset \pi \Fil^{r-1} N_{S'} \cap \pi N_S = \pi \Fil^{r-1} N_S$.
	The other inclusion, i.e.\ $\pi \Fil^{r-1}N_S \subset \Fil^r N_S \cap \pi N_S$, is obvious.
\end{proof}

\begin{lem}\label{lem:fil_as_prod}
	For each $r \in \ZZ$, we have that $\Fil^r N_S[1/p] = \sum_{i+j=r} \Fil^i S \cdot \Fil^j \mbfn(V)$, where $\Fil^i S \cdot \Fil^j \mbfn(V)$ denotes the image of $\Fil^i S \otimes_{\AR^+} \Fil^j \mbfn(V) \rightarrow N_S[1/p]$.
\end{lem}
\begin{proof}
	Note that the claim for $\fERpi^{\bmstar}$ was shown in Remark \ref{rem:fil_neruv_explicit}.
	For $\ARpi^{\bmstar}$, the claim for $\smstar \in \{\textpd, [u], [u, v]\}$ follows from the proof of Lemma \ref{lem:fil_poincare_lem_na} (see Remark \ref{rem:fil_as_prod}) and for $\ARpi^+$, the claim follows from Lemma \ref{lem:fil_nrpi+}.
	So, it remains to show the claim for $\ARpi^{(0, v]+}$.
	Let $S = \ARpi^{(0, v]+}$, $A = \ARpi^+$, $B = \ARpi^{[u]}$, $C = \ARpi^{[u, v]}$, and $N[1/p] = \mbfn(V)$.
	Note that by definition, we have $C = S + B$ and the ideal $\Fil^i C$ is topologically generated by $(\Fil^i S + \Fil^i B)C$, for all $i \in \NN$ (see Remark \ref{rem:fil_pd_u}).
	Moreover, from Remark \ref{rem:fil_as_prod}, we have that $\Fil^r N_B[1/p] = \sum_{i+j=r} \Fil^i B \cdot \Fil^j N[1/p]$ and $\Fil^r N_C[1/p] = \sum_{i+j=r} \Fil^i C \cdot \Fil^j N[1/p]$.
	So, by setting $M := \sum_{i+j=r} \Fil^i S \cdot \Fil^j N[1/p]$, we see that $\Fil^r N_C[1/p] = \sum_{i+j=r} \Fil^i C \cdot \Fil^j N[1/p] = M + \Fil^r N_B[1/p] = \Fil^r N_S[1/p] + \Fil^r N_B[1/p]$.
	Now, consider the following diagram with exact rows:
	\begin{center}
		\begin{tikzcd}[row sep=small]
			0 \arrow[r] & M \arrow[r] \arrow[d] & M + \Fil^r N_B[1/p] \arrow[r] \arrow[d, equal] & (\Fil^r N_B[1/p])/(M \cap \Fil^r N_B[1/p]) \arrow[r] \arrow[d] & 0\\
			0 \arrow[r] & \Fil^r N_S[1/p] \arrow [r] & \Fil^r N_C[1/p] \arrow [r] & (\Fil^r N_C[1/p])/(\Fil^r N_S[1/p]) \arrow[r] & 0,
		\end{tikzcd}
	\end{center}
	where the left vertical arrow is injective (by an argument similar to the first part of Remark \ref{rem:fil_neruv_explicit}).
	To get the claim, it is enough to show that the right vertical arrow is bijective.
	Note that we have $(\Fil^r N_C[1/p])/(\Fil^r N_S[1/p]) = (\Fil^r N_S[1/p] + \Fil^r N_B[1/p])/(\Fil^r N_S[1/p]) = (\Fil^r N_B[1/p])/(\Fil^r N_S[1/p] \cap \Fil^r N_B[1/p])$.
	It is clear that $M \cap \Fil^r N_B[1/p] \subset \Fil^r N_S[1/p] \cap \Fil^r N_B[1/p]$, and we claim that the reverse inclusion also holds.
	Indeed, as $N[1/p]$ is a finite projective $\AR^+[1/p]\textrm{-module}$ and $A = S \cap B \subset C$, therefore, we get that $N_A[1/p] = N_S[1/p] \cap N_B[1/p] \subset N_C[1/p]$.
	Then, it follows that $\Fil^r N_S[1/p] \cap \Fil^r N_B[1/p] \subset N_S[1/p] \cap N_B[1/p] = N_A[1/p]$, in particular, we see that $\Fil^r N_S[1/p] \cap \Fil^r N_B[1/p] = \Fil^r N_A[1/p] \cap \Fil^r N_B[1/p] \subset M \cap \Fil^r N_B[1/p]$, where the equality follows from Remark \ref{rem:ns_fil_induced} and the inclusion follows by using the description of $\Fil^r N_A[1/p]$ from Lemma \ref{lem:fil_nrpi+}.
	So, we obtain that the right vertical arrow in the diagram is bijective, and hence, the left vertical arrow is bijective as well, i.e.\ $\Fil^r N_S[1/p] = \sum_{i+j=r} \Fil^i S \cdot \Fil^j \mbfn(V)$.
\end{proof}

Set $\Fil^i \Ainf(\Rbar) := \Ainf(\Rbar) \cap \Fil^i \Acrys(\Rbar) = \xi^i \Ainf(\Rbar) \subset \Acrys(\Rbar)$, for $i \in \ZZ$.
\begin{lem}\label{lem:fil_nrpi+}
	For $S = \ARpi^+$ and any $r \in \ZZ$, we have that $\Fil^r N_S[1/p] = (\Fil^r\Ainf(\Rbar) \otimes_{\ZZ_p} V) \cap N_S[1/p] = \sum_{i+j=r} \Fil^i \ARpi^+ \cdot \Fil^j \mbfn(V)$.
\end{lem}
\begin{proof}
	The first equality is obvious from the definition of the filtration on $N_S[1/p]$ in \eqref{eq:fil_ns} and Remark \ref{rem:ns_fil_induced}.
	For the second equality, we will show a stronger claim: $\Fil^r N_S = \sum_{i+j=r} \Fil^i \ARpi^+ \cdot \Fil^j \mbfn(T)$.
	From the first equality, note that we have $\Fil^r N_S = (\Fil^r\Ainf(\Rbar) \otimes_{\ZZ_p} V) \cap N_S = (\Fil^r\Ainf(\Rbar) \otimes_{\ZZ_p} T) \cap N_S$.
	Let us set $F^r N_S := \sum_{i+j=r} \Fil^i \ARpi^+ \cdot \Fil^j \mbfn(T)$, for each $r \in \NN$, and note that the inclusion $F^r N_S \subset \Fil^r N_S$ is obvious.
	To prove the reverse inclusion, we will simplify the claim a bit.
	Note that the natural map $\ARpi^+ \otimes_{\AR^+} \Fil^r \mbfn(T) \rightarrow N_S$ is injective because the morphism $\AR^+ \rightarrow \ARpi^+$ is flat.
	So it follows that we have $F^r N_S = \sum_{i+j=r} \Fil^i \ARpi^+ \otimes_{\AR^+} \Fil^j \mbfn(T) = \xi F^{r-1} N_S + \ARpi^+ \otimes_{\AR^+} \Fil^r \mbfn(T)$.
	Now, to show the inclusion $\Fil^r N_S \subset F^r N_S$, we will proceed by induction on $r \in \NN$.
	The case $r=0$ is trivial, so assume that $r \geqslant 1$ and the claim holds for all $k \leqslant r-1$.
	Let us note that inside $\Ainf(\Rbar) \otimes_{\ZZ_p} T$, we have $\Fil^r N_S \cap \xi \Fil^{r-2} N_S = (\xi^r \Ainf(\Rbar) \otimes_{\ZZ_p} T) \cap N_S \cap (\xi^{r-1} \Ainf(\Rbar) \otimes_{\ZZ_p} T) \cap \xi N_S = \xi \Fil^{r-1} N_S$.
	Therefore, it follows that the natural inclusion $ \Fil^r N_S \subset \Fil^{r-1} N_S$ induces an injective map $(\Fil^r N_S)/(\xi \Fil^{r-1} N_S) \rightarrow (\Fil^{r-1} N_S)/(\xi \Fil^{r-2} N_S)$, where we have,
	\begin{equation*}
		(\Fil^{r-1} N_S)/(\xi \Fil^{r-2} N_S) = (\ARpi^+ \otimes_{\AR^+} \Fil^{r-1} \mbfn(T))/((\ARpi^+ \otimes_{\AR^+} \Fil^{r-1} \mbfn(T)) \cap (\xi \Fil^{r-2} N_S)).
	\end{equation*}
	In particular, given any element $x$ in $\Fil^r N_S$, we can write $x = \xi y + z$, for some $y \in \Fil^{r-1} N_S = F^{r-1} N_S$ and $z \in \ARpi^+ \otimes_{\AR^+} \Fil^{r-1} \mbfn(T)$.
	To obtain the claim, it is enough to show that $z$ is an element of $F^r N_S$.

	Note that we have $\Fil^r N_S = (\xi^r\Ainf(\Rbar) \otimes_{\ZZ_p} T) \cap N_S$, so we see that $z = x - \xi y = \xi^r z'$, for some $z' \in \Ainf(\Rbar) \otimes_{\ZZ_p} T$.
	Recall that we have $\ARpi^+ = \AR^+[\pi_m]$, where $\pi_m = \varphi^{-m}(\pi)$, and it follows that any element $a \in \ARpi^+$ has a unique presentation as $a = \sum_{i=0}^e a_i (1+\pi_m)^{i/p}$, with $a_i \in \AR^+$ and $e = p^{m-1}(p-1)$.
	Now, let us write $z = \sum_j f_j n_j$, for some $f_j \in \ARpi^+$ and $n_j \in \Fil^{r-1} \mbfn(T)$.
	Then, expressing each $f_j$ as above, i.e.\ in terms of the powers of $1+\pi_m$, and rearranging the sum for $z$ in terms of the powers of $1+\pi_m$, we get that $z = \sum_{i=0}^e z_i (1+\pi_m)^{i/p}$, for some $z_i \in \Fil^{r-1} \mbfn(T)$ (obtained from elements $n_j$ above).
	Now, by using Remark \ref{rem:wach_mod_fil_ainf}, we can write each $z_i$ as $\xi^{r-1} w_i$, for some $w_i \in \Ainf(\Rbar) \otimes_{\ZZ_p} T$.
	Plugging the values of $z$ and $z_i$ into the equality $z = \sum_{i=0}^e z_i (1+\pi_m)^{i/p}$ and noting that $\Ainf(\Rbar) \otimes_{\ZZ_p} T$ is $\xi\textrm{-torsion}$ free, we get that $\xi z' = \sum_{i=0}^e w_i (1+\pi_m)^{i/p}$.
	Reducing the latter equality modulo $\xi \Ainf(\Rbar) \otimes_{\ZZ_p} T$, we obtain the equality $\sum_{i=0}^e w_i \zeta_{p^m}^{i/p} = 0 \textrm{ mod } \xi$ in $\CC^+(\Rbar) \otimes_{\ZZ_p} T$, which is possible only if $w_0 = w_1 = \cdots = w_e \textrm{ mod } \xi \Ainf(\Rbar) \otimes_{\ZZ_p} T$.
	So we write $\xi z' = \xi w_0 + \sum_{i=1}^e (w_i-w_0) (1+\pi_m)^{i/p}$, with $w_i-w_0 \in \xi\Ainf(\Rbar) \otimes_{\ZZ_p} T$, for each $1 \leqslant i \leqslant e$.
	In particular, we get that $z = \xi^r z' = \xi^r w_0 + \sum_{i=1}^e \xi^{r-1}(w_i-w_0) (1+\pi_m)^{i/p} = \xi z_0 + \sum_{i=1}^e (z_i-z_0) (1+\pi_m)^{i/p}$.
	Note that $z_0$ is in $\Fil^{r-1} \mbfn(T)$ and $z_i - z_0 = \xi^{r-1}(w_i-w_0)$ is in $(\xi^r\Ainf(\Rbar) \otimes_{\ZZ_p} T) \cap \Fil^{r-1} \mbfn(T) = \Fil^r \mbfn(T)$ (see Remark \ref{rem:wach_mod_fil_ainf}), for each $1 \leqslant i \leqslant e$.
	Therefore, it follows that $z$ belongs to $\xi \Fil^{r-1} N_S + \ARpi^+ \otimes_{\AR^+} \Fil^r \mbfn(T) = F^r N_S$.
	This allows us to conclude.
\end{proof}

Next, let $k \in \ZZ$ and consider the $p\textrm{-adic}$ representation $V(k)$ of $G_R$.
Using \eqref{eq:fil_ns} and Lemma \ref{lem:fil_ns_Gstable}, we define a $\Gamma_R\textrm{-stable}$ filtration on $\fERpi^{[u, v]} \otimes_{\AR^+} \mbfn(V(k))$ as follows:
\begin{equation}\label{eq:fil_neruv_k}
	\Fil^r \big(\fERpi^{[u, v]} \otimes_{\AR^+} \mbfn(V(k))\big) := \pi^{-k} \Fil^{r+k} \big(\fERpi^{[u, v]} \otimes_{\AR^+} \mbfn(V)\big)(k).
\end{equation}
From the explicit description of the filtration in Remark \ref{rem:fil_neruv_explicit} and by using Lemma \ref{lem:wach_mod_twist_fil}, it follows that we have $\Fil^r \big(\fERpi^{[u, v]} \otimes_{\AR^+} \mbfn(V(k))\big) = \sum_{i+j=r} \Fil^i \fERpi^{[u, v]} \cdot \Fil^j \mbfn(V(k))$.
Furthermore, let $S \subset \fERpi^{[u, v]}$ be as above (see before Lemma \ref{lem:ns_fil_pi_cap}).
Then, we note that we have a natural embedding $S \otimes_{\AR^+} \mbfn(T(k)) \rightarrow \fERpi^{[u, v]} \otimes_{\AR^+} \mbfn(V(k))$, and we equip the former with an induced $\Gamma_R\textrm{-stable}$ filtration, i.e.\ for each $r \in \ZZ$, set
\begin{equation}\label{eq:fil_ns_k}
	\Fil^r \big(S \otimes_{\AR^+} \mbfn(T(k))\big) := \big(S \otimes_{\AR^+} \mbfn(T(k))\big) \cap \Fil^r \big(\fERpi^{[u, v]} \otimes_{\AR^+} \mbfn(V(k))\big) \subset \fERpi^{[u, v]} \otimes_{\AR^+} \mbfn(V(k)).
\end{equation}
Using \eqref{eq:fil_neruv_k} and Remark \ref{rem:ns_fil_induced}, it easily follows that,
\begin{lem}\label{lem:fil_ns_k_twist}
	For each $r \in \ZZ$, we have $\Fil^r \big(S \otimes_{\AR^+} \mbfn(T(k))\big) = \pi^{-k} \Fil^{r+k} \big(S \otimes_{\AR^+} \mbfn(T)\big)(k)$.
\end{lem}

\subsubsection{Filtered Poincar\'e Lemma}\label{subsubsec:fil_poincare_lem_wach}

In the notation of \S \ref{subsubsec:crys_fil_poincare_lem}, let us set $A = \ARpi^{\bmstar}$ (resp.\ $\ARbar^{\bmstar}$), $B = \Rpi^{\bmstar}$ and $E = \fERpi^{\bmstar}$ (resp.\ $\fERbar^{\bmstar}$), for $\smstar \in \{\textpd, [u], [u, v]\}$.
Let $\omega_0 := \frac{d X_0}{1+X_0}$ and $\omega_i := \frac{d X_i}{X_i}$, for $1 \leqslant i \leqslant d$.
Set $\Omega^1 := \oplus_{i=1}^d \ZZ \omega_i$ and $\Omega^k := \bmwedge^k \Omega^1$.
Then, we have $\Omega^k_{E/A} = E \otimes_{\ZZ} \Omega^k$ and from Remark \ref{rem:elements_of_pd_ring} (iv), note that for $r \in \ZZ$, we have the following filtered de Rham complex of $E$ relative to $A$,
\begin{equation*}
	\Fil^r \Omega^{\bullet}_{E/A} := \Fil^r E \longrightarrow \Fil^{r-1} E \otimes_{\ZZ} \Omega^1 \longrightarrow \Fil^{r-2} E \otimes_{\ZZ} \Omega^2 \longrightarrow \cdots.
\end{equation*}

Let $T$ be a positive finite $q\textrm{-height}$ $\ZZ_p\textrm{-representation}$ of $G_R$ as above and assume that $\mbfn(T)$ is finite free over $\AR^+$.
Let us set $N_A := A \otimes_{\AR^+} \mbfn(T)$, equipped with a filtration as in \eqref{eq:fil_ns}, and similarly, we set $N_E := E \otimes_{\AR^+} \mbfn(T)$, equipped with a filtration as in \eqref{eq:fil_ns}.
Note that the $A\textrm{-linear}$ differential operator on $E$ induces a quasi-nilpotent integrable connection $\partial : N_E \rightarrow N_E \otimes_E \Omega_{E/A}^1$ satisfying Griffiths transversality with respect to the filtration (since the same is true after inverting $p$, see Remark \ref{rem:fil_ms_ns}).
In particular, for each $r \in \ZZ$, we have the following filtered de Rham complex,
\begin{align*}
	\Fil^r N_E \otimes \Omega^{\bullet}_{E/A} &:= \Fil^r N_E \longrightarrow \Fil^{r-1} N_E \otimes_E \Omega_{E/A}^1 \longrightarrow \Fil^{r-2} N_E \otimes_E \Omega_{E/A}^2 \longrightarrow \cdots \\
	&\hspace{2mm}= \Fil^r N_E \longrightarrow \Fil^{r-1} N_E \otimes_{\ZZ} \Omega^1 \longrightarrow \Fil^{r-2} N_E \otimes_{\ZZ} \Omega^2 \longrightarrow \cdots.
\end{align*}
Using the equality $N_A = N_E^{\partial=0}$ and \eqref{eq:fil_ns}, we note that $\Fil^r N_A = \Fil^r N_E \cap N_E^{\partial = 0} = (\Fil^r N_E)^{\partial=0}$.
Then, we have the following filtered Poincar\'e Lemma:

\begin{lem}\label{lem:fil_poincare_lem_na}
	The natural map $\Fil^r N_A \rightarrow \Fil^r N_E \otimes \Omega^{\bullet}_{E/A}$ is a quasi-isomorphism.
\end{lem}
\begin{proof}
	The claim follows by employing an argument similar to the proof of Lemma \ref{lem:fil_poincare_lem_mb}, where we use the description of filtration on $N_E[1/p]$ from Remark \ref{rem:fil_neruv_explicit}.
	We omit the details.
\end{proof}

\begin{rem}\label{rem:fil_as_prod}
	From the proof of Lemma \ref{lem:fil_poincare_lem_na}, using the map $h^0 : \Fil^r N_E[1/p] \rightarrow \Fil^r N_A[1/p]$, it follows that for any $r \in \ZZ$, we have $\Fil^r N_A[1/p] = \sum_{i+j=r} \Fil^i A \cdot \Fil^j \mbfn(V)$, where $\Fil^i A \cdot \Fil^j \mbfn(V)$ denotes the image of $\Fil^i A \otimes_{\AR^+} \Fil^j \mbfn(V) \rightarrow A \otimes_{\AR^+} \mbfn(V)$.
\end{rem}

\subsection{Relative Fontaine--Laffaille modules}\label{subsec:fontaine_laffaile_to_wach}

In this subsection we will consider the category of relative Fontaine--Laffaille modules $\MF_{[0, s], \free}(R, \Phi, \partial)$ defined in \cite[\S 4]{tsuji-ainf-genrep} as a full subcategory of the abelian category $\mathfrak{MF}_{[0, s]}^{\nabla}(R)$ introduced in \cite[\S II]{faltings-crystalline}.
Let $s \in \NN$ such that $s \leqslant p-2$.
\begin{defi}\label{defi:rel_fontaine_laffaille}
	Define the category of \textit{free relative Fontaine--Laffaille} modules of level $[0, s]$, denoted by $\MF_{[0, s], \free}(R, \Phi, \partial)$, as follows:\\
	An object with weights/level in the interval $[0, s]$ is a quadruple $(M, \Fil^{\bullet} M, \partial, \Phi)$ such that,
	\begin{enumromanup}
	\item $M$ is a free $R\textrm{-module}$ of finite rank.
		It is equipped with a decreasing filtration $\{\Fil^k M\}_{k \in \ZZ}$ by finite $R\textrm{-submodules}$, with $\Fil^0 M = M$ and $\Fil^{s+1} M = 0$, and such that $\gr^k_{\Fil} M$ is a finite free $R\textrm{-module}$ for all $k \in \ZZ$.

	\item The connection $\partial : M \rightarrow M \otimes_{R} \Omega^1_{R}$ is quasi-nilpotent and integrable and satisfies Griffiths transversality with respect to the filtration, i.e.\ $\partial(\Fil^k M) \subset \Fil^{k-1} M \otimes_{R} \Omega^1_{R}$ for all $k \in \ZZ$.	
	
	\item Let $(\varphi^{\ast}(M), \varphi^{\ast}(\partial))$ denote the pullback of $(M, \partial)$ by $\varphi : R \rightarrow R$ and equip it with a decreasing filtration $\Fil^k_p(\varphi^{\ast}(M)) = \sum_{i \in \NN} (p^i/i!) \varphi^{\ast}(\Fil^{k-i} M)$, for $k \in \ZZ$.
		Suppose that there is an $R\textrm{-linear}$ morphism $\Phi : \varphi^{\ast}(M) \rightarrow M$ such that $\Phi$ is compatible with connections, $\Phi\big(\Fil^k_p(\varphi^{\ast}(M))\big) \subset p^k M$ for $0 \leqslant k \leqslant s$, and $\sum_{k=0}^{s} p^{-k}\Phi\big(\Fil^k_p(\varphi^{\ast}(M))\big) = M$.
		Denote the composition $M \rightarrow \varphi^{\ast}(M) \xrightarrow{\Phi} M$ by $\varphi$.
	\end{enumromanup}
	A morphism between two objects of the category $\MF_{[0, s], \free}(R, \Phi, \partial)$ is a continuous $R\textrm{-linear}$ map compatible with the homomorphism $\Phi$ and the connection $\partial$ on each side.
\end{defi}

\begin{rem}
	In Definition \ref{defi:rel_fontaine_laffaille} (iii), note that $\varphi^*(M)$ denotes the $R\textrm{-module}$ $R \otimes_{\varphi, R} M$ on which the $O_F\textrm{-linear}$ connection is given by the formula $\varphi^*(\partial)(a \otimes x) = da \otimes x + a \otimes \partial(x)$, for any $a$ in $R$ and $x$ in $M$.
	Furthermore, compatibility of the $R\textrm{-linear}$ morphism $\Phi : \varphi^*(M) \rightarrow M$ with connections means that for any $a$ in $R$ and $x$ in $M$, we must have $\partial \circ \Phi(a \otimes x) = \Phi \circ \varphi^*(\partial)(a \otimes x)$.
\end{rem}

To an object $M$ in $\MF_{[0, s], \free}(R, \varphi, \Fil)$, we can functorially associate a $\ZZ_p\textrm{-module}$ as $T_{\crys}^{\ast}(M) := \Hom_{R, \hspace{0.3mm}\Fil, \hspace{0.3mm}\varphi, \hspace{0.3mm}\partial}(M, \pazo \mbfa_{\crys}(\overline{R}))$, i.e.\ $R\textrm{-linear}$ maps from $M$ to $\pazo \mbfa_{\crys}(\overline{R})$, compatible with the respective Frobenii, filtrations and connections.
Set $T_{\crys}(M) := \Hom_{\ZZ_p}(T_{\crys}^{\ast}(M), \ZZ_p)$ and note that it is a finite free $\ZZ_p\textrm{-module}$ of rank $=\textup{rk}_R M$, admitting a continuous action of $G_R$.
By \cite{faltings-crystalline} and \cite{tsuji-ainf-genrep}, the $\padic$ representation $V_{\crys}(M) := \QQ_p \otimes_{\ZZ_p} T_{\crys}(M)$ is crystalline with Hodge--Tate weights in the interval $[-s, 0]$.

\begin{thm}[{\cite[Theorem 5.4]{abhinandan-crystalline-wach}}]\label{thm:fl_to_wach}
	For a free relative Fontaine--Laffaille module $M$ over $R$ of level $[0, s]$, the associated $\padic$ representation $V_{\crys}(M) := \QQ_p \otimes_{\ZZ_p} T_{\crys}(M)$ of $G_R$ is a positive finite $q\textrm{-height}$ representation (in the sense of Definition \ref{defi:wach_reps}).
\end{thm}

\begin{rem}\phantomsection\label{rem:fl_wach_comparison}
	\begin{enumromanup}
	\item The results of \cite{abhinandan-crystalline-wach} are shown for $s = p-2$.
		However, all the arguments can be adapted almost verbatim (by replacing $p-2$ everywhere by any $0 \leqslant s \leqslant p-2$).

	\item Let $M$ be a free relative Fontaine--Laffaille module over $R$ of level $[0, s]$ and let $T = T_{\crys}(M)$ be its associated $\ZZ_p\textrm{-representation}$ of $G_R$.
		Then, from Theorem \ref{thm:fl_to_wach} we have a free relative Wach module $\mbfn(T)$ over $\AR^+$, associated to $T$.
		Moreover, by combining \cite[Propositions 5.23 \& 5.27]{abhinandan-crystalline-wach} and the proof of \cite[Theorem 5.4]{abhinandan-crystalline-wach}, we have a natural isomorphism $\OARpi^{\textpd} \otimes_R M \isomorphic \OARpi^{\textpd} \otimes_{\AR^+} \mbfn(T)$, compatible with the respective Frobenii, filtrations, connections and the actions of $\Gamma_R$.

	\item From the proof of \cite[Theorem 5.4]{abhinandan-crystalline-wach}, one can observe that $M/\Phi(\varphi^{\ast}(M))$ is $p^s\textrm{-torsion}$ and $s$ equals the maximum among the absolute value of Hodge--Tate weights of $V_{\crys}(M)$.
	\end{enumromanup}
\end{rem}

\begin{rem}\label{rem:modpn_flmodules}
	In Defintion \ref{defi:rel_fontaine_laffaille}, we considered finite free $R\textrm{-modules}$.
	For $R/p^n\textrm{-module}$ $M/p^n$, the associated $\ZZ/p^n\textrm{-representation}$ of $G_R$ is given as $T_{\crys}(M/p^n) = T_{\crys}(M)/p^n$.
	Moreover, we associate a Wach module to $T/p^n = T_{\crys}(M)/p^n$ as $\mbfn(T/p^n) := \mbfn(T)/p^n$ and we have a natural isomorphism $\OARpi^{\textpd}/p^n \otimes_{\AR^+/p^n} \mbfn(T/p^n) \isomorphic \OARpi^{\textpd}/p^n \otimes_{R/p^n} M/p^n$ compatible with the respective Frobenii, filtrations, connections and the actions of $\Gamma_R$ (see \cite[\S 5.3]{abhinandan-crystalline-wach}).
\end{rem}


\section{Galois cohomology complexes}\label{sec:galois_cohomology}

In this section, we will describe Koszul complexes computing the cohomology for the action of $\Gamma_R$ and $\Lie\Gamma_R$ on certain modules.

\subsection{Relative Fontaine--Herr complex}\label{subsec:fontaine_herr_complex}

From \S \ref{subsec:relative_phi_gamma_mod}, recall that we have an equivalence between $\ZZ_p\textrm{-representations}$ of $G_R$ and \'etale $(\varphi, \Gamma_R)\textrm{-modules}$ over $\mbfa_R$, so it is natural to expect that the continuous $G_R\textrm{-cohomology}$ groups of a $\ZZ_p\textrm{-representation}$ $T$ could be computed using its associated \'etale $(\varphi, \Gamma_R)\textrm{-module}$ $\mbfd(T)$.
Below, we will consider the continuous cohomology (for the weak topology) of \'etale $(\varphi, \Gamma_R)\textrm{-modules}$ over $\AR$ and $\AR^{\dagger}$ (see \S \ref{subsec:relative_phi_gamma_mod}).
\begin{defi}
	Let $D$ be an \'etale $(\varphi, \Gamma_R)\textrm{-module}$ over $\AR$ or $\AR^{\dagger}$.
	In the derived category of abelian groups, let $\RGamma_{\cont}(\Gamma_R, D)$ denote the complex of continuous cochains with values in $D$.
\end{defi}

\begin{thm}[{\cite{herr-galois-cohomology}, \cite[Theorem 3.3, Theorem 7.10.6]{andreatta-iovita-relative-phiGamma}}]\label{thm:galois_cohomology_herr_complex}
	Let $T$ in $\Rep_{\ZZ_p}(G_R)$ and let $\mbfd(T)$ and $\mbfd^{\dagger}(T)$ be the associated \'etale $(\varphi, \Gamma_R)\textrm{-module}$ over $\AR$ and $\AR^{\dagger}$, respectively.
	Then we have natural quasi-isomorphisms 
	\begin{align*}
		\big[\RGamma_{\cont}(\Gamma_R, \mbfd(T)) \xrightarrow{1-\varphi} \RGamma_{\cont}(\Gamma_R, \mbfd(T))\big] &\simeq \RGamma_{\cont}(G_R, T),\\
		\big[\RGamma_{\cont}(\Gamma_R, \mbfd^{\dagger}(T)) \xrightarrow{1-\varphi} \RGamma_{\cont}(\Gamma_R, \mbfd^{\dagger}(T))\big] &\simeq \RGamma_{\cont}(G_R, T).
	\end{align*}
\end{thm}

\begin{rem}\label{rem:galcoh_rpi}
	Theorem \ref{thm:galois_cohomology_herr_complex} is also valid for $S = R[\varpi]$, where $\varpi = \zeta_{p^m}-1$, and we replace $G_R$ by $G_S \triangleleft G_R$, $\Gamma_R$ by $\Gamma_S = \Gamma_R' \rtimes \Gamma_K \triangleleft \Gamma_R$ and consider complexes in terms of \'etale $(\varphi, \Gamma_S)\textrm{-modules}$ over respective period rings $\ARpi$ and $\ARpi^{\dagger}$ (defined in an obvious way).
\end{rem}

\subsection{Koszul complexes}\label{subsec:gal_coho_kos_complex}

Recall that $K = F(\zeta_{p^m})$ for $m \in \NN_{\geqslant 1}$.
Let $S = R[\varpi]$ for $\varpi = \zeta_{p^m}-1$.
From \S \ref{subsec:relative_phi_gamma_mod}, recall that $S_{\infty}[1/p] = R_{\infty}[1/p]$ is a Galois extension of $S[1/p]$, with Galois group $\Gamma_S = \Gamma_R' \rtimes \Gamma_K \triangleleft \Gamma_R$.
Also recall that we fixed topological generators $\{\gamma_0, \gamma_1, \ldots, \gamma_d\}$ of $\Gamma_S$ such that $\{\gamma_1, \ldots, \gamma_d\}$ are topological generators of $\Gamma_S' := \Gamma_R'$ and $\gamma_0$ is a lift (to $\Gamma_S$) of a topological generator of $\Gamma_K$.
Furthermore, $\chi$ denotes the $p\textrm{-adic}$ cyclotomic character and recall that $c = \chi(\gamma_0) = \exp(p^m)$.

In this subsection, we will recall the definition of Koszul complexes from \cite[\S 4.2]{colmez-niziol-nearby-cycles} computing continuous $\Gamma_S\textrm{-cohomology}$ of topological modules admitting a continuous action of $\Gamma_S$, in particular, \'etale $(\varphi, \Gamma_S)\textrm{-modules}$ (see Remark \ref{rem:galcoh_rpi}).
Let $\tau_i = \gamma_i - 1$, for $1 \leqslant i \leqslant d$, and set $K(\tau_i) : 0 \longrightarrow \ZZ_p \llbracket \tau_i \rrbracket  \xrightarrow{\hspace{2mm}\tau_i\hspace{2mm}} \ZZ_p \llbracket \tau_i \rrbracket  \longrightarrow 0$, where the middle map is multiplication by $\tau_i$ and the right-hand term is placed in degree 0.
\begin{defi}\label{defi:complex_ktau}
	Define $K(\tau_1, \ldots, \tau_d) := K(\tau_1) \widehat{\otimes}_{\ZZ_p} K(\tau_2) \widehat{\otimes}_{\ZZ_p} \cdots \widehat{\otimes}_{\ZZ_p} K(\tau_d)$, to be the \textit{Koszul complex} associated to $(\tau_1, \ldots, \tau_d)$.
\end{defi}

\begin{rem}\label{rem:koszul_complex_d}
	The degree $q$ term in the complex $K(\tau_1, \ldots, \tau_d)$ (Definition \ref{defi:complex_ktau}) equals the exterior power $\wedge_A^q A^d$, where $A = \ZZ_p \llbracket \tau_1, \ldots, \tau_d \rrbracket \isomorphic \ZZ_p\llbracket \Gamma_S' \rrbracket$, the last term denotes the Iwasawa algebra of $\Gamma_S'$.
	The differential $d^1_{q-1} : \wedge_A^q A^d \rightarrow \wedge_A^{q-1} A^d$ is given as $d^1_{q-1}\big(e_{i_1 \cdots i_q}\big) = \sum_{k=1}^q (-1)^{k+1} e_{i_1 \cdots \widehat{i_k} \cdots i_q} \tau_{i_k}$, in the standard basis $\{e_{i_1 \cdots i_q}, 1 \leqslant i_1 < \cdots < i_q \leqslant d\}$ of $\wedge_A^q A^d$.
	In the category of topological $A\textrm{-modules}$, the augmentation map $A \rightarrow \ZZ_p$ makes $K(\tau_1, \ldots, \tau_d)$ into a resolution of $\ZZ_p$.
	Explicitly, we have that,
	\begin{center}
		\begin{tikzcd}[row sep=large, column sep=large]
			K(\tau_1, \ldots, \tau_d) = 0 \arrow[r] & A^{I_d'} \arrow[r, "d^1_{d-1}"] & \cdots \arrow[r, "d^1_1"] & A^{I_1'} \arrow[r, "d_0^1"] & A \arrow[r] & 0,
		\end{tikzcd}
	\end{center}
	where $A^{I_q'} = \oplus_{I_q'} A$, for $I_q\prm = \{(i_1, \ldots, i_q), \hspace{1mm} 1 \leqslant i_1 < \cdots < i_q \leqslant d\}$, and the differentials are as described above. 
	Similarly, for $c = \chi(\gamma_0)$, we can define the Koszul complex $K(\tau_1^c, \ldots, \tau_d^c)$, where $\tau_i^c := \gamma_i^c - 1$.
\end{rem}

\begin{defi}\label{defi:complex_K_lambda}
	Let $\Lambda := \ZZ_p \llbracket \Gamma_S \rrbracket $ and define the complex
	\begin{center}
		\begin{tikzcd}[row sep=large, column sep=large]
			K(\Lambda) := 0 \arrow[r] & \Lambda^{I_d'} \arrow[r, "d^1_{d-1}"] & \cdots \arrow[r, "d^1_1"] & \Lambda^{I_d'} \arrow[r, "d^1_0"] & \Lambda \arrow[r] & 0,
		\end{tikzcd}
	\end{center}
	where we have $\Lambda^{I_q'} = \oplus_{I_q'} \Lambda$ and the indexing sets $I_q'$ were described in Remark \ref{rem:koszul_complex_d}.
	From \cite[Lemma 4.3]{morita-galois-cohomology}, we have an isomorphism of complexes $\lim_m \ZZ_p\big[\Gamma_K / \hspace{0.5mm} (\Gamma_K)^{p^m}\big] \otimes_{\ZZ_p} K(\tau_1, \ldots, \tau_d) \isomorphic K(\Lambda)$.
	Similarly, one can obtain $K^c(\Lambda)$ from $K(\tau_1^c, \ldots, \tau_d^c)$.
	Both $K(\Lambda)$ and $K^c(\Lambda)$ are resolutions of $\ZZ_p \llbracket \Gamma_K \rrbracket$ in the category of topological left $\Lambda\textrm{-modules}$.
\end{defi}

\begin{exm}
	For $d = 2$, the complex $K(\Lambda)$ in Definition \ref{defi:complex_K_lambda} is given as follows:
	\begin{center}
		\begin{tikzcd}[row sep=large, column sep=large]
			0 \arrow[r] & \Lambda \arrow[r, "d^1_1"] & \Lambda \oplus \Lambda \arrow[r, "d^1_0"] & \Lambda \arrow[r] & 0,
		\end{tikzcd}
	\end{center}
	where $d^1_1(x) = (-x\tau_2, x\tau_1)$ and $d^1_0(y, z) = y\tau_1 + z\tau_2$.
\end{exm}

\begin{defi}\label{defi:tau_0}
	Define a map $\tau_0 : K^c(\Lambda) \rightarrow K(\Lambda)$ by setting in each degree $\tau_0^0 = \gamma_0-1$ and $\tau_0^q : \big(a_{i_1 \cdots i_q}\big) \mapsto \big(a_{i_1 \cdots i_q}\big(\gamma_0 - \delta_{i_1 \cdots i_q}\big)\big)$, for $1 \leqslant q \leqslant d$, $1 \leqslant i_1 < \cdots < i_q \leqslant d$ and $\delta_{i_1 \cdots i_q} = \delta_{i_q} \cdots \delta_{i_1}$, with $\delta_{i_j} = \big(\gamma_{i_j}^c - 1\big)\big(\gamma_{i_j} - 1\big)^{-1}$.
\end{defi}

Let $M$ be a topological $\ZZ_p\textrm{-module}$ admitting a continuous action of $\Gamma_S$. 

\begin{defi}\label{defi:koszul_complex_gammaR_prime}\label{defi:koszul_complex_gamma}
	Define the \textit{$\Gamma_S'\textrm{-Koszul complexes}$} of $M$ by setting $\Kos(\Gamma_S\prm, M) := \Hom_{\Lambda, \cont}(K(\Lambda), M)$ and $\Kos^c(\Gamma_S\prm, M) := \Hom_{\Lambda, \cont}(K^c(\Lambda), M)$.
	Moreover, define the \textit{$\Gamma_S\textrm{-Koszul complex}$} of $M$ as $\Kos(\Gamma_S, M) := \big[\Kos(\Gamma_S\prm, M) \xrightarrow{\hspace{2mm} \tau_0 \hspace{2mm}} \Kos^c(\Gamma_S\prm, M)\big]$.
\end{defi}

\begin{prop}[{\cite[Lazard]{lazard-groupes-analytiques-padiques}, \cite[\S 4.2]{colmez-niziol-nearby-cycles}}]\label{prop:koszul_complex_cont_cohomology}
	There exists a natural quasi-isomorphism of complexes $\Kos(\Gamma_S, M) \simeq \RGamma_{\cont}(\Gamma_S, M)$.
\end{prop}

\begin{defi}\label{defi:koszul_phigammaD}
	Let $D$ be an \'etale $(\varphi, \Gamma_S)\textrm{-module}$ over $\ARpi$ and set
	\begin{displaymath}
		\Kos(\varphi, \Gamma_S, D) :=
		\left[
			\vcenter
			{
				\xymatrix
				{
					\Kos(\Gamma_S\prm, D) \ar[r]^{1-\varphi} \ar[d]^{\tau_0} & \Kos(\Gamma_S\prm, D) \ar[d]^{\tau_0} \\
					\Kos^c(\Gamma_S\prm, D) \ar[r]^{1-\varphi} & \Kos^c(\Gamma_S\prm, D)
				}
			}
		\right].
	\end{displaymath}
\end{defi}

Note that from Proposition \ref{prop:koszul_complex_cont_cohomology} and Definition \ref{defi:koszul_phigammaD} we have a natural quasi-isomorphism of complexes $\Kos(\varphi, \Gamma_S, D) \simeq \big[\RGamma_{\cont}(\Gamma_S, D) \xrightarrow{\hspace{1mm} 1-\varphi \hspace{1mm}} \RGamma_{\cont}(\Gamma_S, D)\big]$.
So we conclude the following:
\begin{prop}\label{prop:phigamma_complex_galcoh}
	Let $T$ be in $\Rep_{\ZZ_p}(G_S)$ and $\Dpi(T)$ the associated \'etale $(\varphi, \Gamma_S)\textrm{-module}$ over $\ARpi$.
	Then we have a natural quasi-isomorphism of complexes $\Kos(\varphi, \Gamma_S, \Dpi(T)) \simeq \RGamma_{\cont}(G_S, T)$.
\end{prop}

\subsection{Lie algebra cohomology}\label{subsec:lie_algebra_coh}

In this subsection we will fix constants $u, v \in \RR$ such that $(p-1)/p \leqslant u \leqslant v/p < 1 < v$, for example, one can take $u = (p-1)/p$ and $v = p-1$.

\subsubsection{Convergence of operators}

From \S \ref{subsec:cyclotomic_embeddings}, recall that we have rings $\ARpi^{\textpd}$, $\ARpi^{[u]}$ and $\ARpi^{[u, v]}$ equipped with a continuous action of $\Gamma_S \triangleleft \Gamma_R$.

\begin{lem}\label{lem:log_gamma_converges}
	For $i \in \{0, 1, \ldots, d\}$ the operators $\nabla_i := \log \gamma_i = \sum_{k \in \NN}(-1)^k ((\gamma_i-1)^{k+1})/(k+1)$ converge as a series of operators on $\ARpi^{\textpd}$, $\ARpi^{[u]}$ and $\ARpi^{[u, v]}$.
\end{lem}
\begin{proof}
	From Lemma \ref{lem:gamma_minus_1_pd}, note that we have $(\gamma_0 - 1) \big(p^m, \pi_m^{p^m}\big)^k \ARpi^{\textpd} \subset \big(p^m, \pi_m^{p^m}\big)^{k+1} \ARpi^{\textpd}$, for all $k \geqslant 0$.
	Using the fact that $\gamma_0-1$ acts as a twisted derivation, we see that, for any $x$ in $\ARpi^{\textpd}$, the expression $(\gamma_0 - 1)^k x$ belongs to $\big(p^m, \pi_m^{p^m}\big)^k \ARpi^{\textpd}$.
	Therefore, to check that the series $\nabla_0(x) = \sum_{k \in \NN}(-1)^k ((\gamma_0-1)^{k+1}(x))/(k+1)$ converges in $\ARpi^{\textpd}$, it is enough to show that for a fixed $0 \leqslant j \leqslant k$, the $\padic$ valuation of $(\lfloor p^m j/e \rfloor!)(p^{m(k-j)}/k)$ goes to $+\infty$ as $k \rightarrow +\infty$, which follows from an elementary computation.
	In particular, we have that $\nabla_0(x)$ converges in $\ARpi^{\textpd}$.

	Now, let us consider $\gamma_i $ for $i \in \{1, \ldots, d\}$.
	Again, from Lemma \ref{lem:gamma_minus_1_pd}, note that we have $(\gamma_i - 1) \big(p^m, \pi_m^{p^m}\big)^k \ARpi^{\textpd} \subset \big(p^m, \pi_m^{p^m}\big)^{k+1} \ARpi^{\textpd}$, for all $k \geqslant 0$.
	Using the fact that $\gamma_i-1$ acts as a twisted derivation, we conclude that for any $x$ in $\ARpi^{\textpd}$, the expression $(\gamma_i-1)^k x$ belongs to $\big(p^m, \pi_m^{p^m}\big)^k \ARpi^{\textpd}$.
	Therefore, using an estimate similar the case of $\gamma_0$, we conclude that the series $\nabla_i(x) = \sum_{k \in \NN}(-1)^k ((\gamma_i-1)^{k+1}(x))/(k+1)$ converges in $\ARpi^{\textpd}$.
	The case of $\ARpi^{[u]}$ and $\ARpi^{[u, v]}$ follow from similar arguments (use Lemma \ref{lem:gamma_minus_1_oc} for $\ARpi^{[u, v]}$).
	This allows us to conclude.
\end{proof}

Next, note that formally we can write,
\begin{align*}
	\tfrac{\log (1 + X)}{X} = 1 + a_1 X + a_2 X^2 + a_3 X^3 + \cdots,\\
	\tfrac{X}{\log (1 + X)} = 1 + b_1 X + b_2 X^2 + b_3 X^3 + \cdots,
\end{align*}
where $\upsilon_p(a_k) \geqslant -k/(p-1)$, for all $k \geqslant 1$, and therefore, $\upsilon_p(b_k) \geqslant -k/(p-1)$, for all $k \geqslant 1$.
Setting $X = \gamma_i - 1$, for $i \in \{0, 1, \ldots, d\}$, we make the following claim:
\begin{lem}\label{lem:nabla_quotient_tau_invertible}
	For $i \in \{0, 1, \ldots, d\}$, the operators $\nabla_i/(\gamma_i-1) = (\log \gamma_i)/(\gamma_i-1)$ and $(\gamma_i-1)/\nabla_i = (\gamma_i-1)/(\log \gamma_i)$ converge as series of operators on $\ARpi^{\textpd}$, $\ARpi^{[u]}$ and $\ARpi^{[u, v]}$.
\end{lem}
\begin{proof}
	We will only show that these series converge on $\ARpi^{\textpd}$; the case of $\ARpi^{[u]}$ and $\ARpi^{[u, v]}$ follow similarly (using Lemma \ref{lem:gamma_minus_1_oc} for $\ARpi^{[u, v]}$).
	Note that we have $\upsilon_p(a_k) \geqslant -k/(p-1)$ and $\upsilon_p(b_k) \geqslant -k/(p-1)$, for all $k \geqslant 1$, so it is enough to show the convergence of $(\gamma_i-1)/(\log \gamma_i)$.
	Now from Lemma \ref{lem:gamma_minus_1_pd}, we have that for $k \geqslant 1$, $(\gamma_i - 1) \big(p^m, \pi_m^{p^m}\big)^k \ARpi^{\textpd} \subset \big(p^m, \pi_m^{p^m}\big)^{k+1} \ARpi^{\textpd}$.
	Since $\gamma_i-1$ acts as a twisted derivation, therefore for any $x$ in $\ARpi^{\textpd}$, from the proof of Lemma \ref{lem:log_gamma_converges}, we have that $(\gamma_i - 1)^k x$ belongs to $\big(p^m, \pi_m^{p^m}\big)^k \ARpi^{\textpd}$.
	Therefore, to check that the series $\sum_{k \in \NN}(-1)^k b_k(\gamma_i-1)^k x$ converges in $\ARpi^{\textpd}$, it is enough to show that for a fixed $0 \leqslant j \leqslant k$, the $\padic$ valuation of $b_k p^{m(k-j)} (\lfloor p^m j/e \rfloor!)$ goes to $+\infty$ as $k \rightarrow +\infty$, which follows from an elementary computation.
	So, we get that the series $(\gamma_i-1)/(\log \gamma_i)$ converges on $\ARpi^{\textpd}$.
	This concludes our proof.
\end{proof}

\subsubsection{Koszul Complexes for \texorpdfstring{$\Lie \Gamma_S$}{-}}\label{subsubsec:lie_gamma_koszul_complex}

For $0 \leqslant i \leqslant d$, let $\nabla_i$ denote the operators defined as above.
The Lie algebra $\Lie \Gamma_S\prm$ of the $p\textrm{-adic}$ Lie group $\Gamma_S\prm$ is a finite free $\ZZ_p\textrm{-module}$ of rank $d$, i.e.\ $\Lie \Gamma_S\prm = \ZZ_p[\nabla_i]_{1 \leqslant i \leqslant d}$ and the Lie algebra $\Lie \Gamma_S$ of the $p\textrm{-adic}$ Lie group $\Gamma_S$ is a finite free $\ZZ_p\textrm{-module}$ of rank $d+1$, i.e.\ $\Lie \Gamma_S = \ZZ_p[\nabla_i]_{0 \leqslant i \leqslant d}$.
Moreover, we have $[\nabla_i, \nabla_j] = \nabla_i \circ \nabla_j - \nabla_j \circ \nabla_i = 0$, for $1 \leqslant i, j \leqslant d$, and $[\nabla_0, \nabla_i] = \nabla_0 \circ \nabla_i - \nabla_i \circ \nabla_0 = p^m \nabla_i$, for $1 \leqslant i \leqslant d$.
In particular, $\Lie \Gamma_S'$ is commutative as a $\ZZ_p\textrm{-algebra}$, however, $\Lie \Gamma_S$ is noncommutative.
Let $M$ be a topological $\ZZ_p\textrm{-module}$ admitting a continuous action of $\Lie \Gamma_S$.

\begin{defi}\label{defi:koszul_complex_lie_gammaR_prime}
	Define the complex $\Kos(\Lie \Gamma_S\prm, M) := M \longrightarrow M^{I_1\prm} \longrightarrow \cdots \longrightarrow M^{I_d\prm}$, with differentials dual to those in Remark \ref{rem:koszul_complex_d} (with $\tau_i$ replaced by $\nabla_i$).
\end{defi}

Consider a morphism of complexes $\nabla_0 : \Kos(\Lie \Gamma_S\prm, M) \rightarrow \Kos(\Lie \Gamma_S\prm, M)$ defined on the $q\textrm{-th}$ term as $\nabla_0 - qp^m : M^{I_q\prm} \rightarrow M^{I_q\prm}$.
\begin{defi}\label{defi:koszul_complex_lie_gammaR}
	Define the \textit{$\Lie \Gamma_S\textrm{-}$Koszul complex} with values in $M$ as
	\begin{equation*}
		\Kos(\Lie \Gamma_S, M) := \big[\Kos(\Lie \Gamma_S\prm, M) \xrightarrow{\hspace{2mm} \nabla_0 \hspace{2mm}} \Kos(\Lie \Gamma_S\prm, M)\big].
	\end{equation*}
\end{defi}

\begin{prop}[{\cite[Lazard]{lazard-groupes-analytiques-padiques}, \cite[\S 4.3]{colmez-niziol-nearby-cycles}}]
	There exist natural quasi-isomorphisms of complexes $\RGamma_{\cont}(\Lie \Gamma_S\prm, M) \simeq \Kos(\Lie \Gamma_S\prm, M)$ and $\RGamma_{\cont}(\Lie \Gamma_S, M) \simeq \Kos(\Lie \Gamma_S, M)$.
\end{prop}


\section{Syntomic complexes and finite height representations}\label{sec:syntomic_complex_finite_height}

We will assume the setup of \S \ref{sec:relative_padic_Hodge_theory}.
Recall that we fixed some $m \in \NN_{\geqslant 1}$ and from \S \ref{subsec:pd_envelope}, we have rings $\Rpi^{\bmstar}$ for $\smstar \in \{\hspace{1mm}, +, \textpd, [u], (0, v]+, [u, v]\}$.
Unless otherwise stated, we will assume $u = (p-1)/p$ and $v = p-1$.
Note that the $p\textrm{-adic}$ completion of the module of differentials of $R$ relative to $\ZZ$ is given as $\Omega^1_{R} = \oplus_{i=1}^d R \dlog X_i$.
Also, for $\smstar \in \{+, \textpd, [u], [u, v]\}$, we have $\Omega^1_{\Rpi^{\bmstar}} = \Rpi^{\bmstar} \tfrac{dX_0}{1+X_0} \oplus \big(\oplus_{i=1}^d \Rpi^{\bmstar} \dlog X_i\big)$.

\subsection{Formulation of the main result}\label{subsec:main_result}

In \S \ref{sec:syntomic_complex_finite_height} and \S \ref{sec:syntomic_galcoh} we will work with the following class of representations:
\begin{assum}\label{assum:relative_crystalline_wach_free}
	Let $T$ be a positive finite $q\textrm{-height}$ $\ZZ_p\textrm{-representation}$ of $G_R$ of height $s$, and we set $V = T[1/p]$ (see Definition \ref{defi:wach_reps}).
	Assume that the Wach module $\mbfn(T)$ is free of rank $=\textrm{rk}_{\ZZ_p} T$ over $\AR^+$ and $M \subset \ODcrys(V)$ is a free $R\textrm{-submodule}$ of rank $=\textrm{rk}_{\ZZ_p} T$ such that $M$ is stable under the induced Frobenius, $M[1/p] = \ODcrys(V)$ and the induced connection over $M$ is $p\textrm{-adically}$ quasi-nilpotent, integrable and satisfies Griffiths transversality with respect to the induced filtration.
	Furthermore, assume that $p^s M \subset \varphi^{\ast}(M)$ and there is a natural map $\OARpi^{\textpd} \otimes_{R} M \rightarrow \OARpi^{\textpd} \otimes_{\AR^+} \mbfn(T)$ compatible with the respective Frobenii, filtrations, connections and actions of $\Gamma_R$, and such that it is a $p^N\textrm{-isomorphism}$ with $N = n(T, e) \in \NN$, for $e = [K:F] = p^{m-1}(p-1)$.
\end{assum}

\begin{exm}\label{exm:choice_odcrist}
	Following are some cases in which Assumption \ref{assum:relative_crystalline_wach_free} is satisfied:
	\begin{enumromanup}
	\item Assuming that $\mbfn(T)$ is a free $\mbfa_R^+\textrm{-module}$, from Proposition \ref{prop:odcris_phi_coker} and Remark \ref{rem:connection_filtration_M} we have that the $R\textrm{-module}$ $M := M_0$ (in the notation of the proposition) satisfies Assumption \ref{assum:relative_crystalline_wach_free} with $m=1$ and $n(T, e) = s$.

	\item Let $M = \big(\OARpi^{\textpd} \otimes_{\AR^+} \mbfn(T)\big)^{\Gamma_R}$ with an additional assumption that it is free over $R$ of rank $=\textrm{rk}_{\ZZ_p} T$.
		Then, the module $M$ depends on $T$ and $m \in \NN_{\geqslant 1}$ (see Remark \ref{rem:lattice_depends_on_m}), and it satisfies Assumption \ref{assum:relative_crystalline_wach_free} with $n(T, e) = s$ (see Remark \ref{rem:odcris_phi_coker}, Remark \ref{rem:connection_filtration_M} and Remark \ref{rem:lattice_depends_on_m}).

	\item For our intended global applications to relative Fontaine--Laffaille modules, we note that for representations arising from finite free relative Fontaine--Laffaille modules of level $[0, s]$ with $s \leqslant p-2$ as in \S \ref{subsec:fontaine_laffaile_to_wach}, the conditions of Assumption \ref{assum:relative_crystalline_wach_free} are automatically satisfied, with $M$ being the relative Fontaine--Laffaille module and $n(T, e) = 0$ (see Remark \ref{rem:fl_wach_comparison}).
	\end{enumromanup}
\end{exm}

Let us first consider the case of $S = R[\varpi]$.
From \S \ref{subsec:pd_envelope} we have the divided power ring $\Rpi^{\textpd} \twoheadrightarrow S$ and we have a finite free $\Rpi^{\textpd}\textrm{-module}$ $\Mpi^{\textpd} := \Rpi^{\textpd} \otimes_{R} M$ equipped with a Frobenius-semilinear endomorphism $\varphi$ given by the diagonal action on each component of the tensor product, and a filtration $\{\Fil^k \Mpi^{\textpd}\}_{k \in \NN}$ induced from the tensor product filtration on $\Mpi^{\textpd}[1/p]$ (see the discussion before Lemma \ref{lem:filr_induced}).
Moreover, the $O_F\textrm{-linear}$ integrable connection on $M$ and the continuous $O_F\textrm{-linear}$ de Rham differential operator on $\Rpi^{\textpd}$ induce an $O_F\textrm{-linear}$ integrable connection $\partial : \Mpi^{\textpd} \rightarrow \Mpi^{\textpd} \otimes_{\Rpi^{\textpd}} \Omega^1_{\Rpi^{\textpd}}$ defined by sending $a \otimes x \mapsto a \otimes \partial_M(x) + x da$.
It is easy to see that the connection $\partial$ on $\Mpi^{\PD}$ satisfies Griffiths transversality with respect to the filtration since the same is true for the connection on $M$ and the differential operator on $\Rpi^{\textpd}$.
In particular, we have the following filtered de Rham complex:
\begin{equation}\label{eq:fildeRham_S}
	\Fil^r \cald_{S, M}^{\bullet} := \Fil^r \Mpi^{\textpd} \longrightarrow \Fil^{r-1} \Mpi^{\textpd} \otimes_{\Rpi^{\textpd}} \Omega^1_{\Rpi^{\textpd}} \longrightarrow \cdots.
\end{equation}

Fix a basis of $\Omega^1_{\Rpi^{\textpd}}$ as $\big\{\frac{dX_0}{1+X_0}, \frac{dX_1}{X_1}, \ldots, \frac{dX_d}{X_d}\big\}$ and we will equip $\Omega^1_{\Rpi^{\textpd}}$ with an action of Frobenius next.
Let $j \in \NN$ and $I_j = \{0 \leqslant i_1 < \cdots < i_j \leqslant d\}$.
For $\smbfi = (i_1, \ldots, i_j) \in I_j$, set $\omega_{\smbfi} := \tfrac{dX_0}{1+X_0} \wedge \tfrac{dX_{i_2}}{X_{i_2}} \wedge \cdots \wedge \tfrac{dX_{i_j}}{X_{i_j}}$, if $i_1 = 0$, and $\omega_{\smbfi} := \tfrac{dX_{i_1}}{X_{i_1}} \wedge \cdots \wedge \tfrac{dX_{i_j}}{X_{i_j}}$, otherwise.
Define the operators $\varphi$ and $\psi$ on $\Omega^j_{\Rpi^{\textpd}}$ by the following formulas:
\begin{equation}\label{eq:phi_psi_diff}
	\varphi\big(\textstyle\sum_{\smbfi \in I_j} x_{\smbfi} \omega_{\smbfi}\big) = \textstyle\sum_{\smbfi \in I_j} \varphi(x_{\smbfi}) \omega_{\smbfi} \hspace{2mm} \textrm{and} \hspace{2mm} \psi\big(\textstyle\sum_{\smbfi \in I_j} x_{\smbfi} \omega_{\smbfi}\big) = \textstyle\sum_{\smbfi \in I_j} \psi(x_{\smbfi}) \omega_{\smbfi}.
\end{equation}

\begin{rem}\label{rem:unusual_frobenius}
	Note that \eqref{eq:phi_psi_diff} is not the natural definition of Frobenius, since we have $d(\varphi(x)) = p\varphi(dx)$ in \eqref{eq:phi_psi_diff}.
	But in order to define $\psi$ integrally, we need to divide the usual Frobenius on $\Omega^1_{\Rpi^{\bmstar}}$ by powers of $p$.
	Recall that with the usual definition of Frobenius we have $\varphi \partial = \partial \varphi$ over $M \subset \ODcrys(V)$ (see \S \ref{subsec:relative_padic_reps}).
	However, using \eqref{eq:phi_psi_diff} for $\Omega^1_R$ as well, we see that for any $f \in M$, we now have $\partial_M(\varphi(f))$ = $\sum_{i=1}^d \partial_i(\varphi(f)) \omega_i$ = $\sum p \varphi(\partial_i(f)) \omega_i$ = $p \varphi(\partial_M(f))$.
\end{rem}

\begin{defi}\label{defi:syntomic_complex_coeff}
	Let $r \in \NN$ and consider the complex $\Fil^r \cald_{S, M}^{\bullet}$ as above.
	For $n \in \NN$, let $S_n = S \otimes \ZZ/p^n$ and $M_n = M \otimes \ZZ/p^n$.
	Define the \textit{syntomic complex} and the \textit{syntomic cohomology} of $S$ with coefficients in $M$ as
	\begin{align*}
		&\Syn(S, M, r) := \big[\hspace{1mm} \Fil^r \cald_{S, M}^{\bullet} \xrightarrow{\hspace{1mm}p^r - p^{\bullet} \varphi\hspace{1mm}} \cald_{S, M}^{\bullet}\hspace{1mm}\big], \hspace{2mm} H^{\ast}_{\syn}(S, M, r) := H^{\ast}(\Syn(S, M, r));\\
		&\Syn(S, M, r)_n := \Syn(S, M, r) \otimes \ZZ/p^n, \hspace{12mm} H^{\ast}_{\syn}(S_n, M_n, r) := H^{\ast}(\Syn(S, M, r)_n).
	\end{align*}
\end{defi}

Our main local result is as follows:

\begin{thm}\label{thm:syntomic_complex_galois_cohomology}
	Consider the setting of Assumption \ref{assum:relative_crystalline_wach_free} and let $r \in \ZZ$ such that $r \geqslant s + 1$.
	Then there exists $p^N\textrm{-quasi-isomorphisms}$ 
	\begin{align*}
		\alpha_r^{\Laz} : \tau_{\leqslant r-s-1} \Syn(S, M, r) &\simeq \tau_{\leqslant r-s-1} \RGamma_{\cont}(G_S, T(r)),\\
		\alpha_{r, n}^{\Laz} : \tau_{\leqslant r-s-1} \Syn(S, M, r)_n &\simeq \tau_{\leqslant r-s-1} \RGamma_{\cont}(G_S, T/p^n(r)),
	\end{align*}
	where $N = N(T, e, r) \in \NN$ depends on the representation $T$, $e = [K:F]$ and the twist $r$.
\end{thm}

\begin{rem}
	For $M$ as in Example \ref{exm:choice_odcrist} (ii), note that in Theorem \ref{thm:syntomic_complex_galois_cohomology}, the constant $N$ can precisely be given as $N = 14r + 9s + 2$ (see \S \ref{subsec:proof_lazard_comp}).
\end{rem}

\begin{rem}
	Almost all statements and proofs in \S \ref{sec:syntomic_complex_finite_height} and \S \ref{sec:syntomic_galcoh} are true for $m \geqslant 1$.
	However, for some lemmas in \S \ref{subsec:change_annulus_2} and \S \ref{subsec:change_disk} we need to assume that $m \geqslant 2$.
	So from now on, the reader may safely assume that $m \geqslant 2$ in \S \ref{sec:syntomic_complex_finite_height} and \S \ref{sec:syntomic_galcoh} and obtain Theorem \ref{thm:syntomic_complex_galois_cohomology} for $m=1$, using the Galois descent of Lemma \ref{lem:syn_galois_descent}.
\end{rem}

Using Theorem \ref{thm:syntomic_complex_galois_cohomology}, we can obtain a similar statement over $R$.
Recall that $R$ is smooth over $O_F$ and for $r \in \ZZ$, we have the following filtered de Rham complex:
\begin{equation}\label{eq:fildeRham_R}
	\Fil^r \cald_{R, M}^{\bullet} := \Fil^r M \longrightarrow \Fil^{r-1} M \otimes_{R} \Omega^1_{R} \longrightarrow \Fil^{r-2} M \otimes_{R} \Omega^2_{R} \longrightarrow \cdots.
\end{equation}
\begin{rem}\label{rem:fildeRham_R}
	One can also consider the formulation of filtered de Rham complex for $R$ as in \eqref{eq:fildeRham_S}.
	In that case one considers a surjection $R_{\varpi}^+ \twoheadrightarrow R$ via the map $X_0 \mapsto 0$.
	By writing down the corresponding de Rham complex one readily sees that it is quasi-isomorphic to $\cald_{R, M}^{\bullet}$.
\end{rem}

Using \eqref{eq:fildeRham_R}, similar to Definition \ref{defi:syntomic_complex_coeff}, one can define the syntomic complex of $R$ with coefficients in $M$.
Then using Theorem \ref{thm:syntomic_complex_galois_cohomology} for $\varpi = \zeta_{p^2}-1$ (in particular, Example \ref{exm:choice_odcrist} (ii) for $m=2$), Corollary \ref{cor:lazard_fmlocal_comparison_R} and Galois descent in Lemma \ref{lem:syn_galois_descent} for $e = p(p-1)$), we obtain the following:

\begin{cor}\label{cor:syntomic_complex_galois_cohomology}
	Consider the setting of Assumption \ref{assum:relative_crystalline_wach_free} and let $r \in \ZZ$ such that $r \geqslant s + 1$.
	Then there exists $p^N\textrm{-quasi-isomorphisms}$
	\begin{align*}
		\tau_{\leqslant r-s-1} \Syn(R, M, r) &\simeq \tau_{\leqslant r-s-1} \RGamma_{\cont}(G_R, T(r)),\\
		\tau_{\leqslant r-s-1} \Syn(R, M, r)_n &\simeq \tau_{\leqslant r-s-1} \RGamma_{\cont}(G_R, T/p^n(r)),
	\end{align*}
	where $N = N(p, r, s) \in \NN$ depending on the prime $p$, twist $r$ and height $s$ of $T$.
\end{cor}

\begin{rem}
	For $M$ as in Example \ref{exm:choice_odcrist} (ii), note that in Corollary \ref{cor:syntomic_complex_galois_cohomology}, the constant $N$ can precisely be given as $N = 18r + 9s + 3p(p-1) + 2$ (see \S \ref{subsec:proof_lazard_comp}).
\end{rem}

In Theorem \ref{thm:syntomic_complex_galois_cohomology} we only prove the $\padic$ case.
The modulo $p^n$ case follows in a similar manner.
The complete proof is divided in two main steps: 
first, we will modify the syntomic complexes with coefficients in $M$ to relate it to a ``differential'' Koszul complex with coefficients in $\mbfn(T)$ (see Proposition \ref{prop:syntomic_to_phi_gamma}).
Next, we will modify the Koszul complex from the first step to obtain a Koszul complex computing the continuous $G_S\textrm{-cohomology}$ of $T(r)$ (see Definition \ref{thm:syntomic_complex_galois_cohomology} and Proposition \ref{prop:differential_koszul_complex_galois_cohomology}).
The key to the connection between these two steps will be provided by the comparison isomorphism in Theorem \ref{thm:crys_wach_comparison} and a filtered Poincar\'e Lemma.
In the rest of \S \ref{sec:syntomic_complex_finite_height} we will show the first step.
The second step will be worked out in \S \ref{sec:syntomic_galcoh}.

\subsection{Syntomic complexes with coefficients}\label{subsec:syntomic_coefficents}

For $\bmstar \in \{[u], [u, v], [u, v/p]\}$, define a finite free $R_{\varpi}^{\bmstar}\textrm{-module}$ $\Mpi^{\bmstar} := \Rpi^{\bmstar} \otimes_{R} M$.
Via the diagonal action of Frobenius on each component, define Frobenius-semilinear operators $\varphi : \Mpi^{[u]} \rightarrow \Mpi^{[u]}$ and $\varphi : \Mpi^{[u, v]} \rightarrow \Mpi^{[u, v/p]}$.
Equip $\Mpi^{\bmstar}$ with a filtration $\{\Fil^k \Mpi^{\bmstar}\}_{k \in \NN}$ induced from the tensor product filtration on $\Mpi^{\bmstar}[1/p]$ (see the discussion before Lemma \ref{lem:filr_induced}).
Furthermore, the $O_F\textrm{-linear}$ integrable connection on $M$ and the continuous $O_F\textrm{-linear}$ de Rham differential operator on $\Rpi^{\bmstar}$ induce an $O_F\textrm{-linear}$ integrable connection on $\Mpi$, which satisfies Griffiths transversality with respect to the filtration since the same is true for the connection on $M$ and the differential operator on $\Rpi^{\bmstar}$.
In particular, we have the following filtered de Rham complex:
\begin{equation}\label{eq:deRham_complex_coeff}
	\Fil^r \cald_{R_{\varpi}^{\bmstar}, M}^{\bullet} := \Fil^r \Mpi^{\bmstar} \longrightarrow \Fil^{r-1} \Mpi^{\bmstar} \otimes \Omega^1_{\Rpi^{\bmstar}} \longrightarrow \Fil^{r-2} \Mpi^{\bmstar} \otimes \Omega^2_{\Rpi^{\bmstar}} \longrightarrow \cdots.
\end{equation}
Moreover, for $\bmstar \in \{[u], [u, v], [u, v/p]\}$, we define operators $\varphi$ and $\psi$ on $\Omega^j_{\Rpi^{\bmstar}}$ as in \eqref{eq:phi_psi_diff}.
From \eqref{eq:deRham_complex_coeff}, for $\bmstar \in \{[u], [u, v]\}$, denote by $\cald_{R_{\varpi}^{\bmstar}, M}^{\bullet}$ the source de Rham complex and for $\bmstar \in \{[u], [u,v/p]\}$, denote by $\cale_{R_{\varpi}^{\bmstar}, M}^{\bullet}$ the target de Rham complex.
\begin{defi}
	Define $\Syn(\Mpi^{\bmstar}, r) := \big[\hspace{1mm} \Fil^r \cald_{\Rpi^{\bmstar}, M}^{\bullet} \xrightarrow{\hspace{1mm}p^r - p^{\bullet} \varphi\hspace{1mm}} \cale_{\Rpi^{\bmstar}, M}^{\bullet}\hspace{1mm}\big]$.
\end{defi}

\subsection{Change of the disk of convergence}\label{subsec:change_disk_convergence}

In this section, we will denote the syntomic complex $\Syn(S, M, r)$ in Definition \ref{defi:syntomic_complex_coeff} as $\Syn(\Mpi^{\textpd}, r)$.

\begin{prop}\label{prop:syntomic_pd_to_u}
	For $\tfrac{1}{p-1} \leqslant u \leqslant 1$, the natural morphism between syntomic complexes $\Syn\big(\Mpi^{\textpd}, r\big) \rightarrow \Syn\big(\Mpi^{[u]}, r\big)$, induced by the inclusion $\Mpi^{\textpd} \subset \Mpi^{[u]}$, is a $p^{2r}\textrm{-isomorphism}$. 
\end{prop}

The proposition follows from the following lemma by setting $k=r$.
\begin{lem}
	Let $j, k \in \NN$.
	If $\frac{1}{p-1} \leqslant u \leqslant 1$, the following map is a $p^{k+r}\textrm{-isomorphism}$
	\begin{equation*}
		p^k - p^j \varphi : \Fil^r \Mpi^{[u]} \otimes \Omega^j_{\Rpi^{[u]}} / \Fil^r \Mpi^{\textpd} \otimes \Omega^j_{\Rpi^{\textpd}} \longrightarrow \Mpi^{[u]} \otimes \Omega^j_{\Rpi^{[u]}} / \Mpi^{\textpd} \otimes \Omega^j_{\Rpi^{\textpd}}.
	\end{equation*}
\end{lem}
\begin{proof}
	The proof is motivated by \cite[Lemma 3.2]{colmez-niziol-nearby-cycles}.
	Note that we can decompose everything in the basis of the $\omega_{\mathbf{i}}$'s, where $\smbfi \in I_j = \{0 \leqslant i_1 < \cdots < i_j \leqslant d\}$.
	Then by the definition of Frobenius on $\omega_{\smbfi}$ we are reduced to showing that $p^k - p^j \varphi : \Fil^r \Mpi^{[u]} / \Fil^r \Mpi^{\textpd} \rightarrow \Mpi^{[u]} / \Mpi^{\textpd}$ is a $p^{k+r}\textrm{-isomorphism}$.
	Since $\varphi\big(\Rpi^{[u]}\big) \subset \Rpi^{[u/p]} \subset \Rpi^{\textpd}$, for $\frac{1}{p-1} \leqslant u \leqslant 1$, therefore, we have $\Mpi^{\textpd} \subset \Mpi^{[u]}$ and $\varphi\big(\Mpi^{[u]}\big) \subset \Mpi^{\textpd}$ .

	For $p^k\textrm{-injectivity}$, recall that we have $\Fil^r \Mpi^{[u]} = \Mpi^{[u]} \cap \Fil^r \Mpi^{\textpd}$ (see Lemma \ref{lem:filr_induced}), so for any $x$ in $\Fil^r \Mpi^{[u]}$ it suffices to show that if $(p^k - p^j\varphi)x \in \Mpi^{\textpd}$ then $p^k x \in \Mpi^{\textpd}$.
	As we can write $p^k x$ = $(p^k - p^j\varphi)x + p^j\varphi(x)$ and $\varphi\big(\Mpi^{[u]}\big) \subset \Mpi^{\textpd}$, therefore, we get that $p^k x \in \Mpi^{\textpd}$.
	Next, let us show the $p^{k+r}\textrm{-surjectivity}$.
	Let $\{f_1, \ldots, f_h\}$ be an $R\textrm{-basis}$ of $M$ and take $x = \sum_{i=1}^h a_i \otimes f_i \in \Mpi^{[u]}$.
	Let $N = \frac{ke}{u(p-1)}$, then from the definition of $\Rpi^{[u]}$ we can write $a_i = a_{i1} + a_{i2}$, with $a_{i2} \in R_{\varpi, N}^{[u]}$ and $a_{i1} \in p^{-\lfloor Nu/e \rfloor}\Rpi^+ \subset p^{-k} \Rpi^{\textpd}$, where we write $R_{\varpi,N}^{[u]}$ as in the notation of Lemma \ref{lem:id_minus_frob_bijective} (it consists of power series in $X_0$ involving terms $X_0^s$ for $s \geqslant N$).
	Now let $x_1 = \sum_{i=1}^h a_{i1} \otimes f_i$ and $x_2 = \sum_{i=1}^h a_{i2} \otimes f_i$, so that $x = x_1 + x_2$.
	By Lemma \ref{lem:id_minus_frob_bijective} and the fact that $M$ is stable under $\varphi$, it follows that $(1-p^{j-k}\varphi)$ is bijective on $R_{\varpi, N}^{[u]} \otimes_R M$ (note that the series of operators $\sum_{i \in \NN} p^{(j-k)i}\varphi^i$ converge as an inverse to $1-p^{j-k}\varphi$ on $R_{\varpi, N}^{[u]} \otimes_R M$).
	In particular, we can write $x_2 = (1-p^{j-k}\varphi)z$, for some $z = \sum_{i=1}^h b_i \otimes f_i \in \Mpi^{[u]}$.
	Also, by Lemma \ref{lem:split_pd_elements} we can write $b_i = b_{i1} + b_{i2}$, with $b_{i1} \in \Fil^r \Rpi^{[u]}$ and $b_{i2} \in p^{-\lfloor ru \rfloor}\Rpi^+$.
	By setting $z_1 = \sum_{i=1}^h b_{i1} \otimes f_i \in \Fil^r \Mpi^{[u]}$ and $z_2 = \sum_{i=1}^h b_{i2} \otimes f_i \in p^{-r} \Mpi^{\textpd}$, we obtain that $(1-p^{j-k}\varphi)z_2 = p^{-k}(p^k-p^j\varphi)z_2 \in p^{-k-r} \Mpi^{\textpd}$.
	Using the preceding observation in the expression for $x$, we get that $x - (1-p^{j-k}\varphi)z_1 = x_1 + (1-p^{j-k}\varphi)z_2 \in p^{-k} \Mpi^{\textpd} + p^{-k-r} \Mpi^{\textpd} \subset p^{-k-r} \Mpi^{\textpd}$.
	Therefore, we obtain that $x \in p^{-k-r} \Mpi^{\textpd} + p^{-k}(p^k - p^j \varphi) \Fil^r \Mpi^{[u]}$, allowing us to conclude.
\end{proof}

\subsection{Change of the annulus of convergence}\label{subsec:change_annulus_converge}

We will consider the base change of the syntomic complex from $\Rpi^{\textpd}$ to $\Rpi^{[u, v]}$.

\begin{prop}\label{prop:syntomic_u_to_uv_comp}
	For $pu \leqslant v$, there exists a $p^{2r+4s}\textrm{-quasi-isomorphism}$
	\begin{equation*}
		\tau_{\leqslant r-s-1} \Syn\big(\Mpi^{[u]}, r\big) \simeq \tau_{\leqslant r-s-1} \Syn\big(\Mpi^{[u, v]}, r\big),
	\end{equation*}
	i.e.\ we have $p^{2r+4s}\textrm{-isomorphisms}$ $H^k_{\syn}\big(\Mpi^{[u]}, r\big) \simeq H^k_{\syn}\big(\Mpi^{[u, v]}, r\big)$ for $0 \leqslant k \leqslant r - s - 1$.
\end{prop}
\begin{proof}
	The claim follows by combining the results from Lemmas \ref{lem:syntomic_u_psi_eigenspace}, \ref{lem:syntomic_u_to_uv} \& \ref{lem:syntomic_uv_psi_eigenspace}.
\end{proof}

To prove the claim in Proposition \ref{prop:syntomic_u_to_uv_comp}, we will pass to the corresponding (quasi-isomorphic) $\psi\textrm{-complex}$.
Recall that we have $\varphi^{\ast}\big(\ODcrys(V)\big) \isomorphic \ODcrys(V)$.
Let $\smbff = \{f_1, \ldots, f_h\}$ denote an $R\textrm{-basis}$ of $M$.
Then $\smbff$ and $\varphi(\smbff)$ form two different basis of $\ODcrys(V)$ over $R[1/p]$.
So, we can write $\smbff = \varphi(\smbff)X$, where $X = (x_{ij}) \in \Mat\big(h, R[1/p])$.
For our choice of $M$ (see Assumption \ref{assum:relative_crystalline_wach_free}) and using Theorem \ref{thm:crys_wach_comparison} and Proposition \ref{prop:odcris_phi_coker}, we have $x_{ij} \in p^{-s} R$, where $1 \leqslant i, j \leqslant h$ and $s$ is the height of $V$.
Define $\psi : M^{[u]} = \Rpi^{[u]} \otimes_{R} M \rightarrow p^{-s} \Rpi^{[pu]} \otimes_R M$ by sending $\smbff \smbfy^\intercal \mapsto \smbff \psi(X\smbfy^\intercal)$, where we consider the operator $\psi$ on $\Rpi^{[u]}$ defined in \S \ref{subsec:cyclotomic_frob}.
It is easy to show that this map is well defined, i.e.\ independent of choice of a basis for $M$.
Using the operator $\psi$ on $\Mpi^{[u]}$ as above and on $\Omega^{\bullet}_{\Rpi^{[u]}}$ as in \eqref{eq:phi_psi_diff}, define the complex
\begin{equation*}
	\Syn^{\psi}\big(\Mpi^{[u]}, r\big) := \Big[ \Fil^r \Mpi^{[u]} \otimes \Omega_{\Rpi^{[u]}}^{\bullet} \xrightarrow{p^{r+s}\psi - p^{\bullet + s}} \Mpi^{[pu]} \otimes \Omega_{\Rpi^{[pu]}}^{\bullet} \Big].
\end{equation*}

\begin{lem}\label{lem:syntomic_u_psi_eigenspace}
	The commutative diagram
	\begin{center}
		\begin{tikzcd}[column sep=large]
			\Fil^r \Mpi^{[u]} \otimes \Omega_{\Rpi^{[u]}}^{\bullet} \arrow[d, "id"] \arrow[rr, "p^r-p^{\bullet}\varphi"] & & \Mpi^{[u]} \otimes \Omega^{\bullet}_{\Rpi^{[u]}} \arrow[d, "p^{s}\psi"]\\
			\Fil^r \Mpi^{[u]} \otimes \Omega_{\Rpi^{[u]}}^{\bullet} \arrow[rr, "p^{r+s}\psi-p^{\bullet+s}"] & & \Mpi^{[pu]} \otimes \Omega^{\bullet}_{\Rpi^{[pu]}},
		\end{tikzcd}
	\end{center}
	defines a $p^{2s}\textrm{-quasi-isomorphism}$ from $\Syn\big(\Mpi^{[u]}, r\big)$ to $\Syn^{\psi}\big(\Mpi^{[u]}, r\big)$.
\end{lem}
\begin{proof}
	First, we will look at the cokernel complex which is the cokernel of the right vertical arrow.
	By definition, we have that $\psi(\Mpi^{[u]}) \subset p^{-s} \Mpi^{[pu]}$, in particular, $p^{s} \psi(\Mpi^{[u]}) \subset \Mpi^{[pu]}$.
	Moreover, note that the operator $\psi : \Rpi^{[u]} \rightarrow \Rpi^{[pu]}$ is surjective and $p^s M \subset \varphi^{\ast}(M)$ (see Assumption \ref{assum:relative_crystalline_wach_free}).
	Therefore, $\Mpi^{[pu]} = \Rpi^{[pu]} \otimes_{R} M \subset \psi(\Rpi^{[u]} \otimes_{R} \varphi^{\ast}(M)) \subset \psi(\Mpi^{[u]})$.
	Hence, $p^s\psi(\Mpi^{[u]})$ is $p^{s}\textrm{-isomorphic}$ to $\Mpi^{[pu]}$ and the cokernel complex is killed by $p^{s}$.

	Next, for the kernel complex, we proceed as follows:
	let $M = \oplus_{j=1}^h R f_j$, therefore $\Mpi^{[u]} = \oplus_{j=1}^h \Rpi^{[u]} f_j$.
	Recall that $M / \varphi^{\ast}(M)$ is killed by $p^s$, so we have a $p^{s}\textrm{-isomorphism}$ $\oplus_{j=1}^h \Rpi^{[u]} \varphi(f_j) \isomorphic \Mpi^{[u]}$.
	Note that an element $y = \sum_{j=1}^h y_j \varphi(f_j)$ is in $\big(\oplus_{j=1}^h \Rpi^{[u]} \varphi(f_j)\big)^{\psi=0}$ if and only if $y_j$ is in $(\Rpi^{[u]})^{\psi=0}$.
	Indeed, $\psi(y) = 0$ if and only if $\sum_{j=1}^h \psi(y_j) f_j = 0$, and since $f_j$ are linearly independent over $R[1/p]$, therefore, we see that $\psi(y) = 0$ if and only if $\psi(y_j) = 0$, for all $1 \leqslant j \leqslant h$.
	In particular, we obtain a $p^{s}\textrm{-isomorphism}$ $\big(\Mpi^{[u]}\big)^{\psi=0} \lisomorphic \big(\oplus_{j=1}^h \Rpi^{[u]} \varphi(f_j)\big)^{\psi=0} = \oplus_{j=1}^h (\Rpi^{[u]})^{\psi=0} \varphi(f_j)$.

	Using the definiton of $\psi$ on $\Omega^k_{\Rpi^{[u]}}$ in the chosen basis of \eqref{eq:phi_psi_diff}, it easily follows that $\big(M \otimes_{R} \Omega^k_{\Rpi^{[u]}}\big)^{\psi=0} = \big(\Mpi^{[u]}\big)^{\psi=0} \otimes_{\ZZ} \Omega^k$.
	Recall that from Lemma \ref{lem:analytic_rings_psi_action} (ii), we have a decomposition $(\Rpi^{[u]})^{\psi=0} = \oplus_{\alpha \neq 0} R_{\varpi, \alpha}^{[u]} = \oplus_{\alpha \neq 0} \Rpi^{[u]} u_{\alpha}$, where $u_{\alpha} = (1+X_0)^{\alpha_0} X_1^{\alpha_1} \cdots X_d^{\alpha_d}$, where $\alpha = (\alpha_0, \ldots, \alpha_d)$ is a $(d+1)\textrm{-tuple}$ with $\alpha_i \in \{0, \ldots, p-1\}$ for each $0 \leqslant i \leqslant d$.
	Moreover, we have $\partial_i(u_{\alpha}) = \alpha_i u_{\alpha}$, for each $0 \leqslant i \leqslant d$.
	In particular, $\partial_i(R_{\varpi, \alpha}^{[u]}) \subset R_{\varpi, \alpha}^{[u]}$.
	Now, using the decomposition of $(\Rpi^{[u]})^{\psi=0}$, we set $M_{\alpha} = \oplus_{j=1}^h R_{\varpi, \alpha}^{[u]} \varphi(f_j)$ and obtain that $\big(\Mpi^{[u]}\big)^{\psi=0}$ is $p^{s}\textrm{-isomorphic}$ to $\oplus_{\alpha \neq 0} M_{\alpha}$.
	From the $O_F\textrm{-linear}$ continuous de Rham differential operator on $R_{\varpi, \alpha}^{[u]}$ and the $O_F\textrm{-linear}$ integrable connection on $\Mpi^{[u]}$, we obtain an induced $O_F\textrm{-linear}$ integrable connection $\partial : M_{\alpha} \rightarrow M_{\alpha} \otimes \Omega^1_{R_{\varpi, \alpha}^{[u]}} = M_{\alpha} \otimes_{\ZZ} \Omega^1$.
	Then the decomposition of $\big(\Mpi^{[u]})^{\psi=0}$ shows that the kernel complex in the claim is $p^{s}\textrm{-isomorphic}$ to direct sum of the following complexes:
	\begin{equation}\label{eq:deRham_complex_alpha}
		0 \longrightarrow M_{\alpha} \longrightarrow M_{\alpha} \otimes \Omega^1 \longrightarrow M_{\alpha} \otimes \Omega^2 \longrightarrow \cdots,
	\end{equation}
	where $\alpha \ne 0$.
	We will show that \eqref{eq:deRham_complex_alpha} is exact for each $\alpha$; the idea of the proof is based on \cite[Lemma 3.4]{colmez-niziol-nearby-cycles}.
	Since everything is $p\textrm{-adically}$ complete and $p\textrm{-torsion free}$, we only need to show the exactness of \eqref{eq:deRham_complex_alpha} modulo $p$.
	Note that for $y = \sum_{j=1}^h y_j \varphi(f_j) \in M_{\alpha}$, we have $\partial\big(\sum_{j=1}^h y_j \varphi(f_j)\big) = \sum_{j=1}^h y_j \partial_M(\varphi(f_j)) + \varphi(f_j) \partial(y_j)$, where $\partial_M$ denotes the connection on $M$.
	Recall that from Remark \ref{rem:unusual_frobenius} we have $\varphi \partial_M = p \partial_M \varphi$.
	So $\partial(y) - \sum_{i=1}^h \varphi(f_j) \partial(y_j) \in pM_{\alpha}$.
	Moreover, by using Lemma \ref{lem:differential_mod_p} we have $\partial_i(y_j) - \alpha_i y_j \in pR_{\varpi, \alpha}^{[u]}$.
	So we get that the complex \eqref{eq:deRham_complex_alpha} has a very simple shape modulo $p$: if $d = 0$ it is just $M_{\alpha} \xrightarrow{\hspace{1mm} \alpha_0 \hspace{1mm}} M_{\alpha}$; if $d=1$ it is the complex $M_{\alpha} \xrightarrow{(\alpha_0, \alpha_1)} M_{\alpha} \oplus M_{\alpha} \xrightarrow{-\alpha_1+\alpha_0} M_{\alpha}$; for general $d$ it is the total complex attached to a $(d+1)\textrm{-dimensional}$ cube with all vertices equal to $M_{\alpha}$ and arrows in the $i\textrm{-th}$ direction equal to $\alpha_i$.
	As one of the $\alpha_i$ is invertible by assumption, this implies that the cohomology of the total complex is 0 and \eqref{eq:deRham_complex_alpha} is exact for each $\alpha$.
	This allows us to conclude.
\end{proof}

Following the defintion of $\psi$ over $M^{[u]}$ (see the discussion before Lemma \ref{lem:syntomic_u_psi_eigenspace}), one can define an operator $\psi : \Rpi^{[u, v]} \otimes_{R} M \rightarrow p^{-s} \Rpi^{[pu, pv]} \otimes_{R} M$ as a left inverse to $\varphi$ and set
\begin{equation*}
	\Syn^{\psi}\big(\Mpi^{[u, v]}, r\big) := \Big[ \Fil^r \Mpi^{[u, v]} \otimes \Omega_{\Rpi^{[u, v]}}^{\bullet} \xrightarrow{p^{r+s}\psi - p^{\bullet + s}} \Mpi^{[pu, v]} \otimes \Omega_{\Rpi^{[pu, v]}}^{\bullet} \Big].
\end{equation*}

\begin{lem}\label{lem:syntomic_u_to_uv}
	 For $u \leqslant 1 \leqslant v$ the natural morphism of complexes $\Syn^{\psi}\big(\Mpi^{[u]}, r\big) \rightarrow \Syn^{\psi}\big(\Mpi^{[u, v]}, r\big)$ is a $p^{2r}\textrm{-quasi-isomorphism}$ in degrees $k \leqslant r - s - 1$.
\end{lem}
\begin{proof}
	The map between the complexes is induced by the following diagram:
	\begin{center}
		\begin{tikzcd}[row sep=small, column sep=large]
			\Fil^r \Mpi^{[u]} \otimes \Omega_{\Rpi^{[u]}}^{\bullet} \arrow[d] \arrow[rr, "p^{r+s}\psi-p^{\bullet+s}"] & & \Mpi^{[pu]} \otimes \Omega^{\bullet}_{\Rpi^{[pu]}} \arrow[d]\\
			\Fil^r \Mpi^{[u, v]} \otimes \Omega_{\Rpi^{[u, v]}}^{\bullet} \arrow[rr, "p^{r+s}\psi-p^{\bullet+s}"] & & \Mpi^{[pu, v]} \otimes \Omega^{\bullet}_{\Rpi^{[pu, v]}},
		\end{tikzcd}
	\end{center}
	where the vertical arrows are natural maps induced by the inclusion $\Rpi^{[u]} \subset \Rpi^{[u, v]}$.
	Therefore, it suffices to show that the mapping fibre
	\begin{equation*}
		\Big[ \Fil^r \Mpi^{[u, v]} \otimes \Omega_{\Rpi^{[u, v]}}^{\bullet} / \Fil^r \Mpi^{[u]} \otimes \Omega_{\Rpi^{[u]}}^{\bullet} \xrightarrow{\hspace{1mm} p^{r+s} \psi - p^{\bullet+s} \hspace{1mm}} \Mpi^{[pu, v]} \otimes \Omega_{\Rpi^{[pu, v]}}^{\bullet} / \Mpi^{[pu]} \otimes \Omega_{\Rpi^{[pu]}}^{\bullet} \Big],
	\end{equation*}
	is $p^{2r}\textrm{-acyclic}$.
	By Lemma \ref{lem:fil_quotient_robba_u_module}, we can ignore the filtration, and by working in the basis $\{\omega_{\mathbf{i}}, \hspace{1mm} \smbfi \in I_k\}$ of $\Omega^k$, it is enough to show that $p^{r+s}\psi - p^{k+s} : \Mpi^{[u, v]} / \Mpi^{[u]} \longrightarrow \Mpi^{[pu, v]} / \Mpi^{[pu]}$ is a $p^r\textrm{-isomorphism}$ for $k \leqslant r - s - 1$.
	But note that $\Mpi^{[u, v]} / \Mpi^{[u]} \isomorphic \Mpi^{[pu, v]} / \Mpi^{[pu]}$, therefore, we see that $1 - p^i \psi$ is an endomorphism of this quotient, for $i = r - k$.
	Moreover, for $i \geqslant s + 1$, we get that $1 - p^i \psi$ is invertible on $\Mpi^{[u, v]} / \Mpi^{[u]}$ with the inverse given as $1 + p^{i-s} (p^{s} \psi) + p^{2(i-s)} (p^{s} \psi)^2 + \cdots$.
	Therefore, it follows that $p^{r+s}\psi - p^{k+s} = p^{k+s}(p^{r-k}\psi-1)$ is a $p^{k+s}\textrm{-isomorphism}$.
	Since $k+s \leqslant r-1$, we obtain that the complex in the claim is $p^{2r}\textrm{-acyclic}$.
\end{proof}

\begin{lem}\label{lem:fil_quotient_robba_u_module}
	The natural map $\Fil^r \Mpi^{[u, v]} /\hspace{0.3mm} \Fil^r \Mpi^{[u]} \rightarrow \Mpi^{[u, v]} / \Mpi^{[u]}$ is a $p^r\textrm{-isomorphism}$ for $u \leqslant 1 \leqslant v$.
\end{lem}
\begin{proof}
	The map in the claim is injective by Lemma \ref{lem:filr_induced}.
	For $p^r\textrm{-surjectivity}$, let $\{f_1, \ldots, f_h\}$ be an $R\textrm{-basis}$ of $M$ and let $x = \sum_{i=1}^h b_i \otimes f_i \in \Rpi^{[u, v]} \otimes_R M$.
	By \cite[Lemma 3.5]{colmez-niziol-nearby-cycles}, we have a $p^r\textrm{-isomorphism}$ $\Fil^r \Rpi^{[u, v]} /\hspace{0.3mm} \Fil^r \Rpi^{[u]} \isomorphic \Rpi^{[u, v]} / \Rpi^{[u]}$, so we can write $p^r b_i = b_{i1} + b_{i2}$, with $b_{i1} \in \Fil^r \Rpi^{[u, v]}$ and $b_{i2} \in \Rpi^{[u]}$.
	Since $\sum_{i=1}^h b_{i1} \otimes f_i \in \Fil^r \Mpi^{[u, v]}$, we get the desired conclusion.
\end{proof}

\begin{lem}\label{lem:syntomic_uv_psi_eigenspace}
	The commutative diagram
	\begin{center}
		\begin{tikzcd}[column sep=large]
			\Fil^r \Mpi^{[u, v]} \otimes \Omega_{\Rpi^{[u, v]}}^{\bullet} \arrow[d, "id"] \arrow[rr, "p^r-p^{\bullet}\varphi"] & & \Mpi^{[u, v/p]} \otimes \Omega^{\bullet}_{\Rpi^{[u, v/p]}} \arrow[d, "p^{s}\psi"]\\
			\Fil^r \Mpi^{[u, v]} \otimes \Omega_{\Rpi^{[u, v]}}^{\bullet} \arrow[rr, "p^{r+s}\psi-p^{\bullet+s}"] & & \Mpi^{[pu, v]} \otimes \Omega^{\bullet}_{\Rpi^{[pu, v]}},
		\end{tikzcd}
	\end{center}
	defines a $p^{2s}\textrm{-quasi-isomorphism}$ from $\Syn\big(\Mpi^{[u, v]}, r\big)$ to $\Syn^{\psi}\big(\Mpi^{[u, v]}, r\big)$.
\end{lem}
\begin{proof}
	Proof of the claim follows in manner similar to the proof of Lemma \ref{lem:syntomic_u_psi_eigenspace} by replacing $\Mpi^{[u]}$ with $\Mpi^{[u, v]}$ and $\Rpi^{[u]}$ with $\Rpi^{[u, v]}$.
	One only needs to note that Lemma \ref{lem:analytic_rings_psi_action} (ii) and Lemma \ref{lem:differential_mod_p} are true for the ring $\Rpi^{[u, v]}$ as well.
	We omit the proof.
\end{proof}

\subsection{Differential Koszul Complex}\label{subsec:differential_koszul}

Our next goal is to relate the syntomic complex $\Syn\big(\Mpi^{[u, v]}, r\big)$ in \S \ref{subsec:change_annulus_converge} to a complex with coefficients in the Wach module $\mbfn(T)$ from Assumption \ref{assum:relative_crystalline_wach_free} (see Proposition \ref{prop:syntomic_to_phi_gamma}).
Before stating the result we will verify some results in order to define the latter complex.

Let $\Omega^1_{\ARpi^{[u, v]}}$ denote the $p\textrm{-adic}$ completion of the module of differentials of $\ARpi^{[u, v]}$ relative to $\ZZ$.
Via the isomorphism $\iota_{\cycl} : \Rpi^{[u, v]} \isomorphic \ARpi^{[u, v]}$, we choose a basis $\{\omega_0, \omega_1, \ldots, \omega_d\}$ of $\Omega^1_{\ARpi^{[u, v]}}$ obtained as the image of $\big\{\frac{d X_0}{1+X_0}, \frac{dX_1}{X_1}, \ldots, \frac{dX_d}{X_d}\big\}$ under $\iota_{\cycl}$ (see \S \ref{subsec:pd_envelope}), in particular, we have the differential operators $\partial_i$ over $\ARpi^{[u, v]}$, for $0 \leqslant i \leqslant d$.
Moreover, from Definition \ref{defi:filtration_vanishing_varpi}, $\ARpi^{[u, v]}$ is endowed with a filtration and we have the filtered de Rham complex $\Fil^r \Omega_{\ARpi^{[u, v]}}^{\bullet}$.
The differential operators $\partial_i$ are related to the infinitesimal action of $\Gamma_R$ by the relation $\nabla_i := \log \gamma_i = t \partial_i$, for $0 \leqslant i \leqslant d$ and where $\log \gamma_i = \sum_{k \in \NN} (-1)^k (\gamma_i-1)^{k+1}/(k+1)$.

Let us set $\Npi^{[u, v]}(T) := \ARpi^{[u, v]} \otimes_{\AR^+} \mbfn(T)$ equipped with a $\Gamma_R\textrm{-stable}$ filtration as in \eqref{eq:fil_ns}.
Recall that for an indeterminate $X$ we have formal expressions $\frac{\log (1 + X)}{X}$ and $\frac{X}{\log (1 + X)}$ (see before Lemma \ref{lem:nabla_quotient_tau_invertible}).
\begin{lem}\label{lem:nabla_converges_relative}
	For $i \in \{0, 1, \ldots, d\}$ the operators $\nabla_i = \log \gamma_i$, $\nabla_i/(\gamma_i-1) = (\log \gamma_i)/(\gamma_i-1)$ and $(\gamma_i-1)/\nabla_i = (\gamma_i-1)/(\log \gamma_i)$ converge as series of operators on $\Npi^{[u, v]}(T)$.
	The same is true for $\ARpi^{[u, v]} \otimes_{\AR^+} \mbfn(T(r))$, for any $r \in \ZZ$, and $\Fil^k \Npi^{[u, v]}(T(r))$, for any $k \in \ZZ$.
\end{lem}
\begin{proof}
	We will only show the claim for the operator $\nabla_i$, the claim for the convergence of operators $\nabla_i/(\gamma_i-1)$ and $(\gamma_i-1)/\nabla_i$ follows in a manner similar to Lemma \ref{lem:nabla_quotient_tau_invertible}.
	For $0 \leqslant i \leqslant d$, we have that $\gamma_i-1$ acts as a twisted derivation, i.e.\ for any $a \in \ARpi^{[u, v]}$ and $x \in \mbfn(T)$, we have $(\gamma_i-1)(ax) = (\gamma_i-1)a \cdot x + \gamma_i(a) (\gamma_i-1)x$.
	Note that the action of $\Gamma_R$ is trivial on $\mbfn(T)/\pi\mbfn(T)$.
	So using Lemma \ref{lem:gamma_minus_1_oc} and the preceding discussion we have $(\gamma_i-1)\big(p^m, \pi_m^{p^m}\big)^k \Npi^{[u, v]}(T) \subset \big(p^m, \pi_m^{p^m}\big)^{k+1} \Npi^{[u, v]}(T)$.
	Now, similar to the proof of Lemma \ref{lem:log_gamma_converges}, for $k \geqslant 0$, it follows that we have $(\gamma_i-1)^k \Npi^{[u, v]}(T) \subset \big(p^m, \pi_m^{p^m}\big)^k \Npi^{[u, v]}(T)$.
	The same estimation of the $p\textrm{-adic}$ valuation of the coefficients as in the proof Lemma \ref{lem:log_gamma_converges} helps us in concluding that $\log \gamma_i$ converges as a series of operators on $\Npi^{[u, v]}(T)$.

	Next, from Lemma \ref{lem:fil_ns_k_twist} recall that $\Fil^k \Npi^{[u, v]}(T(r)) = \pi^{-r} \Fil^{r+k} \Npi^{[u, v]}(T)(r)$.
	As $t/\pi$ is a unit in $\ARpi^{[u, v]}$ (see Lemma \ref{lem:t_over_pi_unit}) and the action of $\Gamma_S$ is trivial on $t^{-r} \otimes \epsilon^{\otimes r}$, where $\epsilon^{\otimes r}$ denotes a $\ZZ_p\textrm{-basis}$ of $\ZZ_p(r)$, therefore, it is enough to show that $\nabla_i$ converges on $\Fil^k \Npi^{[u, v]}(T)$, for all $k \in \NN$.
	Now, recall that from Remark \ref{rem:fil_ms_ns} we have a $\Gamma_R\textrm{-equivariant}$ isomorphism of $\fERpi^{[u, v]}\textrm{-modules}$ $\alpha : \Fil^r (\fERpi^{[u, v]} \otimes_R M[1/p]) \isomorphic \Fil^k(\fERpi^{[u, v]} \otimes_{\AR^+} \mbfn(V))$ (see \eqref{eq:fil_ms_ns}).
	Moreover, note that $\nabla_i$ converges on $\fERpi^{[u, v]}$, since it converges on $\ARpi^{[u, v]}$ (see Lemma \ref{lem:log_gamma_converges}) and $\Gamma_R$ acts trivially on $\Rpi^{[u, v]}$.
	So, by using that the filtration on $\fERpi^{[u, v]} \otimes_R M[1/p]$ is given as the tensor product filtration (see Lemma \ref{lem:fil_gr_ds}), the action of $\Gamma_S$ is trivial on $M[1/p]$ and the ideal $\Fil^j \fERpi^{[u, v]}$ is closed in $\fERpi^{[u, v]}$ for all $j \in \NN$ (see Remark \ref{rem:elements_of_pd_ring} (ii)), it follows that $\nabla_i$ converges on $\Fil^r (\fERpi^{[u, v]} \otimes_R M[1/p])$, and since $\alpha$ is $\Gamma_R\textrm{-equivariant}$, therefore, $\nabla_i$ also converges on $\Fil^k(\fERpi^{[u, v]} \otimes_{\AR^+} \mbfn(V))$.
	Combining the two discussions above, it follows that $\nabla_i(\Fil^k\Npi^{[u, v]}(T)) \subset \Fil^k(\fERpi^{[u, v]} \otimes_{\AR^+} \mbfn(V)) \cap \Npi^{[u, v]}(T) = \Fil^k \Npi^{[u, v]}(T)$ (see Remark \ref{rem:ns_fil_induced}).
	A similar argument shows that the operators $\nabla_i/(\gamma_i-1)$ and $(\gamma_i-1)/\nabla_i$ also converge on $\Fil^k \Npi^{[u, v]}(T)$.
	This allows us to conclude.
\end{proof}

\begin{lem}\label{lem:nablai_griffiths_transversal}
	For the filtered modules and operators $\nabla_i$ defined above and $0 \leqslant i \leqslant d$, we have that $\nabla_i\big(\Fil^k \Npi^{[u, v]}(T)\big) \subset \pi \Fil^{k-1} \Npi^{[u, v]}(T) = t \Fil^{k-1} \Npi^{[u, v]}(T)$.
\end{lem}
\begin{proof}
	Note that the action of $\Gamma_R$ is trivial on $\Npi^{[u, v]}(T) / \pi \Npi^{[u, v]}(T)$.
	So using Lemma \ref{lem:nabla_converges_relative}, we infer that $\nabla_i\big(\Fil^k \Npi^{[u, v]}(T)\big) \subset \Fil^k \Npi^{[u, v]}(T) \cap \pi \Npi^{[u, v]}(T) = \pi \Fil^{k-1} \Npi^{[u, v]}(T)$, where the last equality follows from Lemma \ref{lem:ns_fil_pi_cap}.
	As $t/\pi$ is a unit in $\fERpi^{[u, v]}$ (see Lemma \ref{lem:t_over_pi_unit}), we can also write $\nabla_i\big(\Fil^k \Npi^{[u, v]}(T)\big) \subset t\Fil^{k-1} \Npi^{[u, v]}(T)$.
\end{proof}

For $0 \leqslant i \leqslant d$, it is easy to see that we have $\nabla_i = \log \gamma_i = \lim_{n \rightarrow +\infty} (\gamma_i^{p^n}-1)/p^n$, from which one can easily show that $\nabla_i$ satisfies a Leibniz rule (see the proof \cite[Theorem 4.2]{morrow-tsuji-coeff} for a similar argument).
Now using Lemma \ref{lem:nabla_converges_relative} we define differential operators $\partial_i$ over $\Npi^{[u, v]}(T)$ as $\partial_i := \nabla_i/t = (\log \gamma_i)/t$.
In the basis $\big\{\omega_0, \ldots, \omega_d\big\}$ of $\Omega^1_{\ARpi^{[u, v]}}$, we set $\partial = (\partial_0, \ldots, \partial_d)$ and obtain a connection $\partial : \Npi^{[u, v]}(T) \rightarrow  \Npi^{[u, v]}(T) \otimes \Omega^1_{\ARpi^{[u, v]}}$ by sending $ax \mapsto a \partial(x) + x \otimes da$.

\begin{lem}\label{lem:connection_npiuv}
	The connection $\partial$ on $\Npi^{[u, v]}(T)$ is integrable, satisfies a Leibniz rule and Griffiths transversality with respect to the filtration, i.e. $\partial_i(\Fil^k \Npi^{[u, v]}(T)) \subset \Fil^{k-1} \Npi^{[u, v]}(T)$, for $0 \leqslant i \leqslant d$.
\end{lem}
\begin{proof}
	From \S \ref{subsubsec:lie_gamma_koszul_complex} recall that $[\nabla_i, \nabla_j] = 0$ for $1 \leqslant i, j \leqslant d$ and $[\nabla_0, \nabla_i] = p^m \nabla_i$, for $1 \leqslant i \leqslant d$.
	So it follows that over $\Npi^{[u, v]}(T)$ we have a composition of operators $t^2(\partial_i \circ \partial_j - \partial_j \circ \partial_i) = t \partial_i (t \partial_j) - t \partial_j (t \partial_i) = \nabla_i \circ \nabla_j - \nabla_j \circ \nabla_i = 0$, for $1 \leqslant i, j \leqslant d$.
	Next, for $1 \leqslant i \leqslant d$, we have $\nabla_0 \circ \nabla_i - \nabla_i \circ \nabla_0 = t\partial_0 \circ (t \partial_i) - t\partial_i \circ (t\partial_0) = t p^m \partial_i + t^2 \partial_0 \circ \partial_i - t^2\partial_i \circ \partial_0 = p^m \nabla_i + t^2 (\partial_0 \circ \partial_i - \partial_i \circ \partial_0)$.
	In particular, $\partial_0 \circ \partial_i - \partial_i \circ \partial_0 = 0$.
	Since $\partial \circ \partial = (\partial_i \circ \partial_j)_{i, j}$, for $0 \leqslant i \leqslant j \leqslant d$, and $\Npi^{[u, v]}(T)$ is $t\textrm{-torsion}$ free, we conclude that the connection $\partial$ is integrable.
	Moreover, it is clear that $\partial$ satisfies a Leibniz rule and it satisfies Griffiths transversailty because we have $\partial_i\big(\Fil^k \Npi^{[u, v]}(T)\big) = t^{-1} \nabla_i\big(\Fil^k \Npi^{[u, v]}(T)\big) \subset \Fil^{k-1} \Npi^{[u, v]}(T)$, using Lemma \ref{lem:nablai_griffiths_transversal}.
\end{proof}

Let $S = \ARpi^{[u, v]}$, then from Lemma \ref{lem:connection_npiuv}, we have the filtered de Rham complex $\Fil^r \Npi^{[u, v]}(T) \otimes \Omega^{\bullet}_S$.
In the chosen basis $\{\omega_1, \ldots, \omega_d\}$ of $\Omega^1_{S}$, an element of $\Omega^q_{S} = \wedge^q \Omega^1_{S}$ can be expressed as $\sum_{\mathbf{i}} x_{\mathbf{i}} \omega_{\mathbf{i}}$ in a unique manner, where $x_{\mathbf{i}} \in S$ and $\omega_{\mathbf{i}} = \omega_{i_1} \wedge \cdots \wedge \omega_{i_q}$, for $\mathbf{i} = (i_1, \ldots, i_q) \in I_q = \{0 \leqslant i_1 < \cdots < i_q \leqslant d\}$.
In this case, the map involving differential operators becomes $(\partial_i) : \big(\Fil^{k-q} \Npi^{[u, v]}(T)\big)^{I_q} \rightarrow \big(\Fil^{k-q-1} \Npi^{[u, v]}(T)\big)^{I_{q+1}}$ for $0 \leqslant i \leqslant d$.

\begin{defi}\label{def:koszul_complex_differential}
	Define the \textit{$\partial\textrm{-Koszul complex}$} for $\Fil^k \Npi^{[u, v]}(T)$ as
	\begin{equation*}
		\Kos\big(\partial_A, \Fil^k \Npi^{[u, v]}(T)\big) : \Fil^k \Npi^{[u, v]}(T) \xrightarrow{(\partial_i)} \big(\Fil^{k-1} \Npi^{[u, v]}(T)\big)^{I_1} \longrightarrow \cdots.
	\end{equation*}
\end{defi}

\begin{rem}\phantomsection\label{rem:koszul_complex_differential}
	\begin{enumromanup}
	\item By definition, it follows that we have a natural isomorphism between complexes $\Fil^k \Npi^{[u, v]}(T) \otimes \Omega^{\bullet}_{\ARpi^{[u, v]}} \isomorphic \Kos\big(\partial_A, \Fil^k \Npi^{[u, v]}(T)\big)$.

	\item Let $I_q\prm = \{(i_1, \ldots, i_q), \hspace{1mm} \textrm{such that} \hspace{1mm} 1 \leqslant i_1 < \cdots < i_q \leqslant d\}$ and $\partial\prm = (\partial_1, \ldots, \partial_d)$.
		Set
		\begin{equation*}
			\Kos\big(\partial_A\prm, \Fil^k \Npi^{[u, v]}(T)\big) : \Fil^k \Npi^{[u, v]}(T) \xrightarrow{(\partial_i)} \big(\Fil^{k-1} \Npi^{[u, v]}(T)\big)^{I_1\prm} \longrightarrow \cdots,
		\end{equation*}
		and note that $\Kos\big(\partial_A, \Fil^k \Npi^{[u, v]}(T)\big) = \big[ \Kos\big(\partial_A\prm, \Fil^k \Npi^{[u, v]}(T)\big) \xrightarrow{\hspace{2mm}\partial_0\hspace{2mm}} \Kos\big(\partial_A\prm, \Fil^{k-1} \Npi^{[u, v]}(T)\big) \big]$.

	\item Computations carried out in this section are true over the ring $\ARpi^{[u, v/p]}$ as well.
	\end{enumromanup}
\end{rem}

\subsection{Poincar\'e Lemma}\label{subsec:poincare_lemma_app}

From Definition \ref{defi:obese_rings}, Remark \ref{rem:elements_of_pd_ring} and Lemma \ref{lem:obese_rings_struct}, recall that, for $\smstar \in \{\textpd, [u], [u, v]\}$, we have rings $\fERpi^{\bmstar}$ equipped with a filtration, Frobenius $\varphi$ sending $\fERpi^{\textpd} \rightarrow \fERpi^{\textpd}$, $\fERpi^{[u]} \rightarrow \fERpi^{[u]}$ and $\fERpi^{[u, v]} \rightarrow \fERpi^{[u, v/p]}$ and an action of $G_R$ which commutes with the Frobenius.
Moreover, from Remark \ref{rem:oarpd_erpd}, we have a subring $\OARpi^{\textpd} \subset \OAcrys(\Rbar)$ equipped with induced structures and we have a natural embedding $\OARpi^{\textpd} \subset \ERpi^{\textpd}$ compatible with the respective Frobenii, filtrations, $\ARpi^{\textpd}\textrm{-linear}$ connections and actions of $\Gamma_R$.

From Assumption \ref{assum:relative_crystalline_wach_free}, we have a natural map $\OARpi^{\textpd} \otimes_{R} M \rightarrow \OARpi^{\textpd} \otimes_{R} \mbfn(T)$, which is a $p^{n(T, e)}\textrm{-isomorphism}$ compatible with the respective Frobenii, filtrations, connections and the actions of $\Gamma_R$.
Recall that $\Mpi^{[u, v]} = \Rpi^{[u, v]} \otimes_{R} M$ and $\Npi^{[u, v]}(T) = \ARpi^{[u, v]} \otimes_{\AR^+} \mbfn(T)$ and after extension of scalars we have a map $\fERpi^{[u, v]} \otimes_{\Rpi^{[u, v]}} \Mpi^{[u, v]} \rightarrow \fERpi^{[u, v]} \otimes_{\ARpi^{[u, v]}} \Npi^{[u, v]}(T)$, which is a $p^{n(T, e)}\textrm{-isomorphism}$ compatible with the respective Frobenii, connections and the actions of $\Gamma_R$.
Moreover, in the $p^{n(T, e)}\textrm{-isomorphism}$ above, the left hand term is equipped with a filtration as described in the discussion before Lemma \ref{lem:filr_induced} and the right hand term is equipped with a filtration as in \eqref{eq:fil_ns}, which is compatible with the filtration on the left hand term by definition.

Let $R_1 := \ARpi^{[u, v]}$, $R_2 := \Rpi^{[u, v]}$ and $R_3 := \fERpi^{[u, v]}$. 
Set $X_{0, 1} := \pi_m$, $X_{0, 2} := X_0$ and set $X_{i, 1} := [X_i^{\flat}]$ and $X_{i, 2} := X_i$, for $1 \leqslant i \leqslant d$.
For $j = 1, 2$, set $\Omega^1_j := \ZZ\tfrac{dX_{0, j}}{1+X_{0, j}} \oplus_{i=1}^d \ZZ \tfrac{dX_{i, j}}{X_{i, j}}$ and $\Omega^1_3 := \Omega^1_1 \oplus \Omega^1_2$.
For $j = 1, 2, 3$, let $\Omega^k_j = \wedge^k \Omega_j$.
Therefore, we see that $\Omega^k_{R_j} = R_j \otimes \Omega^k_j$.
Recall that from \eqref{eq:deRham_complex_coeff} we have the filtered de Rham complex $\Fil^r \Mpi^{[u, v]} \otimes \Omega_1^{\bullet}$.
Set $\Delta_2 := \fERpi^{[u, v]} \otimes_{\Rpi^{[u, v]}} \Mpi^{[u, v]}$ equipped with a filtration as described in the discussion before Lemma \ref{lem:filr_induced}.
Using the $O_F\textrm{-linear}$ de Rham differential operator $\partial_{R_3} : \Fil^r \fERpi^{[u, v]} \rightarrow \Fil^{r-1} \fERpi^{[u, v]} \otimes_{\ZZ} \Omega_3^1$ and the $O_F\textrm{-linear}$ integrable connection $\partial_{R_2} : \Fil^r \Mpi^{[u, v]} \rightarrow \Fil^{r-1} \Mpi^{[u, v]} \otimes_{\ZZ} \Omega_2^1$, we obtain an $O_F\textrm{-linear}$ integrable connection on $\Delta_2$ as $\partial_{R_3} : \Delta_2 \rightarrow \Delta_2 \otimes_{\ZZ} \Omega_3^1$ by sending $ax \mapsto a \partial_{R_2}(x) + \partial_{R_3}(a) x$.
Moreover, the connection $\partial_{R_3}$ on $\Delta_2$ satisfies Griffiths transversality with respect to the filtration, i.e.\ $\partial_{R_3} : \Fil^r \Delta_2 \rightarrow \Fil^{r-1} \Delta_2 \otimes_{\ZZ} \Omega_3^1$, since the same is true for the differential operator on $\fERpi^{[u, v]}$ and the connection on $\Mpi^{[u, v]}$.
In particular, we have the filtered de Rham complex $\Fil^r \Delta_2 \otimes \Omega_3^{\bullet}$.
\begin{lem}\label{lem:poincare_lemma_2}
	The natural map $\Fil^r \Mpi^{[u, v]} \otimes \Omega_2^{\bullet} \rightarrow \Fil^r \Delta_2 \otimes \Omega_3^{\bullet}$ is a quasi-isomorphism.
\end{lem}
\begin{proof}
	In the notation of \S \ref{subsubsec:crys_fil_poincare_lem}, note that we have $A = R_1$, $B = R_2$ and $E = R_3$.
	Moreover, by definition, it is clear that $\Fil^r \Mpi^{[u, v]} = (\Fil^r \Delta_2)^{\partial_{R_1} = 0}$.
	Therefore, by using Lemma \ref{lem:fil_poincare_lem_mb}, we obtain the claim.
\end{proof}

Similar to above and using the discussion of \S \ref{subsec:differential_koszul}, it is easy to see that for $R_1 = \ARpi^{[u, v]}$ we have a filtered de Rham complex $\Fil^r \Npi^{[u, v]}(T) \otimes \Omega_1^{\bullet}$.
Let $\Delta_1 := \fERpi^{[u, v]} \otimes_{\ARpi^{[u, v]}} \Npi^{[u, v]}(T)$ equipped with the filtration described in \eqref{eq:fil_ns}.
Then similar to the case of $\Delta_2$, we have a filtered de Rham complex $\Fil^r \Delta_1 \otimes \Omega_3^{\bullet}$ and similar to Lemma \ref{lem:poincare_lemma_2} we obtain the following:
\begin{lem}\label{lem:poincare_lemma_1}
	The natural map $\Fil^r \Npi^{[u, v]}(T) \otimes \Omega_1^{\bullet} \rightarrow \Fil^r \Delta_1 \otimes \Omega_3^{\bullet}$ is a quasi-isomorphism.
\end{lem}
\begin{proof}
	In the notation of \S \ref{subsubsec:fil_poincare_lem_wach}, note that we have $A = R_1$, $B = R_2$ and $E = R_3$.
	Using the equality $\Npi^{[u, v]}(T) = \Delta_1^{\partial=0}$ and \eqref{eq:fil_ns}, we note that $\Fil^r \Npi^{[u, v]}(T) = \Fil^r \Delta_1 \cap \Delta_1^{\partial = 0} = (\Fil^r \Npi^{[u, v]}(T))^{\partial=0}$.
	Therefore, by using Lemma \ref{lem:fil_poincare_lem_na}, we obtain the claim.
\end{proof}

\begin{rem}
	Statements analogous to Lemma \ref{lem:poincare_lemma_2} and Lemma \ref{lem:poincare_lemma_1} for $\Rpi^{[u, v/p]}$ and $\ARpi^{[u, v/p]}$ (instead of $\Rpi^{[u, v]}$ and $\ARpi^{[u, v]}$) respectively, are also true.
\end{rem}

\begin{defi}\label{def:koszul_complex_phi_differential}
	Let $\Npi^{[u, v]}(T)$ as above equipped with a Frobenius-semilinear morphism $\varphi : \Npi^{[u, v]}(T) \rightarrow \Npi^{[u, v/p]}(T)$.
	Using Definition \ref{def:koszul_complex_differential} and Remark \ref{rem:koszul_complex_differential} set
	\begin{displaymath}
		\Kos\big(\varphi, \partial_A, \Fil^r \Npi^{[u, v]}(T)\big) :=
		\left[
			\vcenter
			{
				\xymatrix @C+2pc
				{
					\Kos\big(\partial_A\prm, \Fil^r \Npi^{[u, v]}(T)\big) \ar[r]^{p^r-p^{\bullet}\varphi} \ar[d]^{\partial_0} & \Kos\big(\partial_A\prm, \Npi^{[u, v/p]}(T)\big) \ar[d]^{\partial_0} \\
					\Kos\big(\partial_A\prm, \Fil^{r-1} \Npi^{[u, v]}(T)\big) \ar[r]^{p^r-p^{\bullet+1}\varphi} & \Kos\big(\partial_A\prm, \Npi^{[u, v/p](T)}\big)
				}
			}
		\right].
	\end{displaymath}
\end{defi}

\begin{prop}\label{prop:syntomic_to_phi_gamma}
	There exists a natural $p^{2n(T, e)}\textrm{-quasi-isomorphism}$ between complexes $\Syn\big(\Mpi^{[u, v]}, r\big)$ and $\Kos\big(\varphi, \partial_A, \Fil^r \Npi^{[u, v]}(T)\big)$, where $n(T, e) \in \NN$ as in Assumption \ref{assum:relative_crystalline_wach_free}.
\end{prop}
\begin{proof}
	Note that using Lemma \ref{lem:poincare_lemma_2} with $R_1 = \Rpi^{[u, v]}$, $R_3 = \fERpi^{[u, v]}$, $\Delta_1 = \fERpi^{[u, v]} \otimes_{\Rpi^{[u, v]}} \Mpi^{[u, v]}$ and $\Delta_1\prm = \fERpi^{[u, v/p]} \otimes_{\Rpi^{[u, v/p]}} \Mpi^{[u, v/p]}$, we have natural quasi-isomorphisms of complexes $\Syn(\Mpi^{[u, v]}, r) \simeq \Big[ \Fil^r \Mpi^{[u, v]} \otimes \Omega_1^{\bullet} \xrightarrow{p^r - p^{\bullet}\varphi} \Mpi^{[u, v/p]} \otimes \Omega_1^{\bullet} \Big] \simeq \Big[ \Fil^r \Delta_1 \otimes \Omega_3^{\bullet} \xrightarrow{p^r - p^{\bullet}\varphi} \Delta_1\prm \otimes \Omega_3^{\bullet} \Big]$.
	Next, using Lemma \ref{lem:poincare_lemma_2} with $R_2 = \ARpi^{[u, v]}$, $R_3 = \fERpi^{[u, v]}$, $\Delta_2 = \fERpi^{[u, v]} \otimes_{\ARpi^{[u, v]}} \Npi^{[u, v]}(T)$ and $\Delta_2\prm = \fERpi^{[u, v/p]} \otimes_{\ARpi^{[u, v/p]}} \Npi^{[u, v/p]}$, together with Remark \ref{rem:koszul_complex_differential}, note that we have natural quasi-isomorphisms of complexes $\Kos(\varphi, \partial_A, \Fil^r \Npi^{[u, v]}(T)) \simeq \Big[ \Fil^r \Npi^{[u, v]}(T) \otimes \Omega_2^{\bullet} \xrightarrow{p^r - p^{\bullet}\varphi} \Fil^r \Npi^{[u, v/p]} \otimes \Omega_2^{\bullet} \Big] \simeq \Big[ \Fil^r \Delta_2 \otimes \Omega_3^{\bullet} \xrightarrow{p^r - p^{\bullet}\varphi} \Delta_2\prm \otimes \Omega_3^{\bullet} \Big]$.
	Finally, using the $p^{n(T, e)}\textrm{-isomorphism}$ $\fERpi^{[u, v]} \otimes_{\Rpi^{[u, v]}} \Mpi^{[u, v]} \isomorphic \fERpi^{[u, v]} \otimes_{\ARpi^{[u, v]}} \Npi^{[u, v]}(T)$ from Assumption \ref{assum:relative_crystalline_wach_free}, we have $p^{n(T, e)}\textrm{-isomorphisms}$ $\Fil^r \Delta_1 \simeq \Fil^r \Delta_2$ and $\Delta_1\prm \simeq \Delta_2\prm$.
	Hence, from the discussion above, we obtain a natural $p^{2n(T, e)}\textrm{-quasi-isomorphism}$ of complexes $\Syn\big(\Mpi^{[u, v]}, r\big) \simeq \Kos\big(\varphi, \partial_A, \Fil^r \Npi^{[u, v]}(T)\big)$.
\end{proof}


\section{Syntomic complexes and \texorpdfstring{$(\varphi, \Gamma)\textrm{-modules}$}{-}}\label{sec:syntomic_galcoh}

In this section, we will work under the setup of Assumption \ref{assum:relative_crystalline_wach_free} and carry out the second step of the proof of Theorem \ref{thm:syntomic_complex_galois_cohomology}.
Recall that we have a finite free $\ARpi^{[u, v]}\textrm{-module}$ $\Npi^{[u, v]}(T) = \ARpi^{[u, v]} \otimes_{\AR^+} \mbfn(T)$ equipped with a $\Gamma_R\textrm{-stable}$ filtration as in \eqref{eq:fil_ns} and from Definition \ref{def:koszul_complex_phi_differential}, we have the complex $\Kos\big(\varphi, \partial_A, \Fil^r \Npi^{[u, v]}(T)\big)$.
Let $S = R[\varpi]$ and from the theory of \'etale $(\varphi, \Gamma_S)\textrm{-modules}$ in \S \ref{subsec:relative_phi_gamma_mod}, we have $\Dpi(T(r)) = \ARpi \otimes_{\AR} \mbfd(T(r))$, and from Defintion \ref{defi:koszul_phigammaD} we have the complex $\Kos\big(\varphi, \Gamma_S, \Dpi(T(r))\big)$.
In this section, our goal is to show the following:
\begin{prop}\label{prop:differential_koszul_complex_galois_cohomology}
	There exist natural $p^N\textrm{-quasi-isomorphisms}$ of complexes
	\begin{equation*}
		\tau_{\leqslant r} \Kos\big(\varphi, \partial_A, \Fil^r \Npi^{[u, v]}(T)\big) \simeq \tau_{\leqslant r} \Kos\big(\varphi, \Gamma_S, \Dpi(T(r))\big),
	\end{equation*}
	where $N = N(r, s) \in \NN$ depending only on the height $s$ of the representation $T$ and twist $r$.
\end{prop}

\subsection{Proof of Theorem \ref{thm:syntomic_complex_galois_cohomology}}\label{subsec:proof_lazard_comp}

Note that by combining Proposition \ref{prop:syntomic_pd_to_u} and Proposition \ref{prop:syntomic_u_to_uv_comp}, we have a natural $p^{4r+4s}\textrm{-quasi-isomorphism}$ of complexes $\tau_{\leqslant r-s-1} \Syn\big(\Mpi^{\textpd}, r\big) \simeq \tau_{\leqslant r-s-1} \Syn\big(\Mpi^{[u, v]}, r\big)$.
Next, from Proposition \ref{prop:syntomic_to_phi_gamma}, we have a natural $p^{2n(T, e)}\textrm{-quasi-isomorphism}$ of complexes $\Syn\big(\Mpi^{[u, v]}, r\big) \simeq \Kos\big(\varphi, \partial_A, \Fil^r \Npi^{[u, v]}(T)\big)$.
Furthermore, by Proposition \ref{prop:differential_koszul_complex_galois_cohomology}, we have a natural $p^{10r+3s+2}\textrm{-quasi-isomorphism}$ of complexes $\tau_{\leqslant r} \Kos\big(\varphi, \partial_A, \Fil^r \Npi^{[u, v]}(T)\big) \simeq \tau_{\leqslant r} \Kos\big(\varphi, \Gamma_S, \Dpi(T(r))\big)$, where $\tau_{\leqslant}$ denotes the canonical truncation (for the explicit constant, see the proof of Proposition \ref{prop:differential_koszul_complex_galois_cohomology} at the end of \S \ref{subsec:change_disk}).
Finally, by Proposition \ref{prop:koszul_complex_cont_cohomology} and Theorem \ref{thm:galois_cohomology_herr_complex}, we have a natural quasi-isomorphism of complexes $\Kos\big(\varphi, \Gamma_S, \Dpi(T(r))\big) \simeq \RGamma_{\cont}(G_S, T(r))$.
Combining all these statements gives us the desired conclusion with $N = 2n(T, e) + 14r + 7s + 2$. 
\null\nobreak\hfill\ensuremath{\blacksquare}\\

In the rest of this section, we will prove Proposition \ref{prop:differential_koszul_complex_galois_cohomology}.

\subsection{From differential forms to the infinitesimal action of \texorpdfstring{$\Gamma_S$}{-}}\label{subsec:diff_to_lie}

Note that Lemma \ref{lem:nabla_converges_relative} describes the action of $\Lie \Gamma_S$ on $\Fil^r \Npi^{[u, v]}(T)$.
Then for the Lie subgroup $\Gamma_S' \subset \Gamma_S$ (see \S \ref{subsec:relative_phi_gamma_mod} for notations), using Definition \ref{defi:koszul_complex_lie_gammaR_prime} we have the complex $\Kos\big(\Lie \Gamma_S\prm, \Fil^r \Npi^{[u, v]}(T)\big)$ and we consider its subcomplex, i.e.\ a complex made of submodules in each degree stable under the differentials of the complex, as follows:
\begin{align*}
	\pazk\big(\Lie \Gamma_S\prm, \Fil^r \Npi^{[u, v]}(T)\big) := \hspace{1mm} \Fil^r \Npi^{[u, v]}(T) \xrightarrow{\hspace{1mm}(\nabla_i)\hspace{1mm}} &\big(t\Fil^{r-1} \Npi^{[u, v]}(T)\big)^{I_1\prm} \longrightarrow \cdots\\
	&\cdots \longrightarrow \big(t^k \Fil^{r-k} \Npi^{[u, v]}(T)\big)^{I_k\prm} \longrightarrow \cdots.
\end{align*}
Using the same differentials, we can define a complex $\pazk\big(\Lie \Gamma_S\prm, t\Fil^{r-1} \Npi^{[u, v]}(T)\big)$ as a subcomplex of $\Kos\big(\Lie \Gamma_S\prm, \Fil^r \Npi^{[u, v]}(T)\big)$.
Now consider a morphism of complexes $\nabla_0 : \pazk\big(\Lie \Gamma_S\prm, \Fil^r \Npi^{[u, v]}(T)\big) \rightarrow \pazk\big(\Lie \Gamma_S\prm, t\Fil^{r-1} \Npi^{[u, v]}(T)\big)$, given as $\nabla_0 = \log \gamma_0$ in degree 0 and as $\nabla_0 - kp^m : \big(t^k \Fil^{r-k} \Npi^{[u, v]}(T(r))\big)^{I_k\prm} \rightarrow \big(t^{k+1} \Fil^{r-k-1} \Npi^{[u, v]}(T(r))\big)^{I_k\prm}$ on the $k\textrm{-th}$ term of the definition above, for $1 \leqslant k \leqslant d$.
The morphism of complexes is well defined because we have $\nabla_0 \nabla_i -\nabla_i \nabla_0 = p^m \nabla_i$, for $1 \leqslant i\leqslant d$ (see \S \ref{subsubsec:lie_gamma_koszul_complex} and the discussion after Definition \ref{defi:koszul_complex_lie_gammaR_prime}).
Write the total complex of the diagram thus obtained as $\pazk\big(\Lie \Gamma_S, \Fil^r \Npi^{[u, v]}(T)\big)$, which is a subcomplex of $\Kos\big(\Lie \Gamma_S, \Fil^r \Npi^{[u, v]}(T)\big)$ by definition.
Similarly, we can define complexes $\pazk\big(\Lie \Gamma_S\prm, \Npi^{[u, v/p]}(T)\big)$ and $\pazk\big(\Lie \Gamma_S\prm, t\Npi^{[u, v/p]}(T)\big)$ and a map $\nabla_0$ from the former to the latter complex.

Recall that from Definition \ref{def:koszul_complex_phi_differential} we have the Koszul complex $\Kos\big(\varphi, \partial_A, \Fil^r \Npi^{[u, v]}(T)\big)$.
Note that we have $\nabla_i = t \partial_i$, for all $0 \leqslant i \leqslant d$ (see \S \ref{subsec:differential_koszul}).
So we consider a morphism of complexes $\Kos\big(\partial_A\prm, \Fil^r \Npi^{[u, v]}(T)\big) \rightarrow \pazk\big(\Lie \Gamma_S\prm, \Fil^r \Npi^{[u, v]}(T)\big)$, given by the identity map in degree 0 and multiplication by $t^k$ on the $k\textrm{-th}$ term of the definition above, i.e.\ $\big(\Fil^{r-k} \Npi^{[u, v]}(T(r))\big)^{I_k\prm} \xrightarrow{\times t^k} \big(t^k \Fil^{r-k} \Npi^{[u, v]}(T(r))\big)^{I_k\prm}$, for $1 \leqslant k \leqslant d$.
It is clear that the map thus defined is bijective, i.e.\ we obtain an isomorphism of complexes.
Similarly, multiplying by powers of $t$ as above, we obtain an isomorphism of complexes $\Kos\big(\partial_A\prm, \Fil^{r-1} \Npi^{[u, v]}(T)\big) \isomorphic \pazk\big(\Lie \Gamma_S\prm, t \Fil^{r-1} \Npi^{[u, v]}(T)\big)$.
Furthermore, one can do a similar construction for $\Npi^{[u, v/p]}(T)$ to obtain isomorphism of complexes $\Kos\big(\partial_A\prm, \Npi^{[u, v/p]}(T)\big) \isomorphic \pazk\big(\Lie \Gamma_S\prm, \Npi^{[u, v/p]}(T)\big)$ and $\Kos\big(\partial_A\prm, \Npi^{[u, v/p]}(T)\big) \isomorphic \pazk\big(\Lie \Gamma_S\prm, t\Npi^{[u, v/p]}(T)\big)$.
As each term of these complexes admit a Frobenius-semilinear morphism $\varphi : t^j \Fil^{r-j} \Npi^{[u, v]}(T) \rightarrow t^j \Npi^{[u, v/p]}(T)$, we obtain the following morphism of complexes (see Definition \ref{def:koszul_complex_phi_differential} for the source complex):
\begin{align*}
	\begin{split}
		\Kos\big(\varphi, \partial_A, \Fil^r \Npi^{[u, v]}(T)\big)
		\longrightarrow
		\left[
			\vcenter
			{
				\xymatrix@C+=1.3cm
				{
					\pazk\big(\Lie \Gamma_S\prm, \Fil^r \Npi^{[u, v]}(T)\big) \ar[r]^{p^r-\varphi} \ar[d]^{\nabla_0} & \pazk\big(\Lie \Gamma_S\prm, \Npi^{[u, v/p]}(T)\big) \ar[d]^{\nabla_0} \\
					\pazk\big(\Lie \Gamma_S\prm, t\Fil^{r-1} \Npi^{[u, v]}(T)\big) \ar[r]^{p^r-\varphi} & \pazk\big(\Lie \Gamma_S\prm, t\Npi^{[u, v/p]}(T)\big)
				}
			}
		\right].
	\end{split}
\end{align*}
From the discussion above, it follows that,
\begin{lem}\label{lem:diff_to_lie}
	The morphism of complexes described above is an isomorphism.
\end{lem}

Recall that $s$ is the height of $T$ and we fixed some $r \geqslant s + 1$.
Set $\Npi^{[u, v]}(T(r)) := \ARpi^{[u, v]} \otimes_{\AR^+} \mbfn(T(r))$, equipped with the natural action of $\Gamma_R$ and a $\Gamma_R\textrm{-stable}$ filtration as in \eqref{eq:fil_ns_k}.
Then, from Lemma \ref{lem:nabla_converges_relative}, recall that the operators $\nabla_i$ are well defined over $\Fil^k \Npi^{[u, v]}(T(r))$, for $0 \leqslant i \leqslant d$.
Using these operators, we consider a subcomplex of the Koszul complex $\Kos\big(\Lie \Gamma_S\prm, \Fil^0 \Npi^{[u, v]}(T(r))\big)$ (Definition \ref{defi:koszul_complex_lie_gammaR_prime}), as follows:
\begin{align*}
	\pazk\big(\Lie \Gamma_S\prm, \Fil^0 \Npi^{[u, v]}(T(r))\big) := \Fil^0 \Npi^{[u, v]}(T(r)) &\xrightarrow{\hspace{1mm}(\nabla_i)\hspace{1mm}} \big(t\Fil^{-1} \Npi^{[u, v]}(T(r))\big)^{I_1\prm} \longrightarrow \cdots\\
	&\cdots \longrightarrow \big(t^k \Fil^{-k} \Npi^{[u, v]}(T(r))\big)^{I_k\prm} \longrightarrow \cdots.
\end{align*}
Similarly, we can define a complex $\pazk\big(\Lie \Gamma_S\prm, t\Fil^{-1} \Npi^{[u, v]}(T(r))\big)$ as a subcomplex of the Koszul complex $\Kos\big(\Lie \Gamma_S\prm, \Fil^0 \Npi^{[u, v]}(T(r))\big)$.
Moreover, similar to the discussion before Lemma \ref{lem:diff_to_lie}, we can define a morphism of complexes $\nabla_0 : \pazk\big(\Lie \Gamma_S\prm, \Fil^0 \Npi^{[u, v]}(T(r))\big) \longrightarrow \pazk\big(\Lie \Gamma_S\prm, t\Fil^{-1} \Npi^{[u, v]}(T(r))\big)$.
The associated total complex, written as $\pazk\big(\Lie \Gamma_S, \Fil^r \Npi^{[u, v]}(T)\big)$, is a subcomplex of the Koszul complex $\Kos\big(\Lie \Gamma_S, \Fil^0 \Npi^{[u, v]}(T(r))\big)$.
Furthermore, by a similar construction, we can define the complexes $\pazk\big(\Lie \Gamma_S\prm, \Npi^{[u, v/p]}(T(r))\big)$ and $\pazk\big(\Lie \Gamma_S\prm, t\Npi^{[u, v/p]}(T(r))\big)$ and a morphism $\nabla_0$ from the former to the latter.

Next, from Lemma \ref{lem:fil_ns_k_twist}, recall that $\Fil^k \Npi^{[u, v]}(T(r)) = \pi^{-r} \Fil^{k+r} \Npi^{[u, v]}(T)(r)$, for each $k \in \ZZ$.
Let $\epsilon^{-r}$ denote a $\ZZ_p\textrm{-basis}$ of $\ZZ_p(-r)$, then we see that $(t^r \otimes \epsilon^{-r}) \Fil^k \Npi^{[u, v]}(T(r)) = (t/\pi)^r \Fil^{r+k} \Npi^{[u, v]}(T) = \Fil^{r+k} \Npi^{[u, v]}(T)$, where the last equality follows since $t/\pi$ is a unit in $\ARpi^{[u, v]}$ (see Lemma \ref{lem:t_over_pi_unit}).
Now, consider a morphism of complexes $\pazk\big(\Lie \Gamma_S\prm, \Fil^0 \Npi^{[u,v]}(T(r))\big) \rightarrow \pazk\big(\Lie \Gamma_S\prm, \Fil^r \Npi^{[u,v]}(T)\big)$ given as multiplication by $t^r \otimes \epsilon^{-r}$ in each degree, in particular, it is given as $\big(t^k\Fil^{-k} \Npi^{[u, v]}(T(r))\big)^{I_k\prm} \xrightarrow{\times (t^r \otimes \epsilon^{-r})} \big(t^k \Fil^{r-k} \Npi^{[u, v]}(T)\big)^{I_k\prm}$ on the $k\textrm{-th}$ term of the definition above, for $1 \leqslant k \leqslant d$.
Note that the map thus defined is bijective on each term by the preceding discussion.
Similarly, we have $(t^r \otimes \epsilon^{-r}) \Npi^{[u, v/p]}(T(r)) = (t/\pi)^r \Npi^{[u, v/p]}(T) = \Npi^{[u, v/p]}(T)$, which yields an isomorphism of complexes $\pazk\big(\Lie \Gamma_S\prm, \Npi^{[u, v/p]}(T(r))\big) \isomorphic \pazk\big(\Lie \Gamma_S\prm, t\Npi^{[u, v/p]}(T(r))\big)$.
Putting these together, we obtain that,
\begin{lem}\label{lem:twist_lie_koszul_complex}
	The morphism of complexes below, given as multiplication by $t^r \otimes \epsilon^{-r}$ on each term, is an isomorphism:
\begin{align*}
	\begin{split}
		\left[
			\vcenter
			{
				\xymatrix@C+=1.3cm
				{
					\pazk\big(\Lie \Gamma_S\prm, \Fil^0 \Npi^{[u, v]}(T(r))\big) \ar[r]^{p^r(1-\varphi)} \ar[d]^{\nabla_0} & \pazk\big(\Lie \Gamma_S\prm, \Npi^{[u, v/p]}(T(r))\big) \ar[d]^{\nabla_0} \\
					\pazk\big(\Lie \Gamma_S\prm, t\Fil^{-1} \Npi^{[u, v]}(T(r))\big) \ar[r]^{p^r(1-\varphi)} & \pazk\big(\Lie \Gamma_S\prm, t\Npi^{[u, v/p]}(T(r))\big)
				}
			}
		\right]
		&\isomorphic\\
		&\hspace{-75mm}\left[
			\vcenter
			{
				\xymatrix@C+=1.3cm
				{
					\pazk\big(\Lie \Gamma_S\prm, \Fil^r \Npi^{[u, v]}(T)\big) \ar[r]^{p^r-\varphi} \ar[d]^{\nabla_0} & \pazk\big(\Lie \Gamma_S\prm, \Npi^{[u, v/p]}(T)\big) \ar[d]^{\nabla_0} \\
					\pazk\big(\Lie \Gamma_S\prm, t\Fil^{r-1} \Npi^{[u, v]}(T)\big) \ar[r]^{p^r-\varphi} & \pazk\big(\Lie \Gamma_S\prm, t\Npi^{[u, v/p]}(T)\big)
				}
			}
		\right].
	\end{split}
\end{align*}
\end{lem}

In order to change from ``$\Lie \Gamma_S\textrm{-Koszul complexes}$'' to ``$\Gamma_S\textrm{-Koszul complexes}$'', we modify the source complex in Lemma \ref{lem:twist_lie_koszul_complex} to define $\pazk\big(\varphi, \Lie \Gamma_S, \Npi^{[u, v]}(T(r))\big)$, as follows:
\begin{displaymath}
	\left[
		\vcenter
		{
			\xymatrix
			{
				\pazk\big(\Lie \Gamma_S\prm, \Fil^0 \Npi^{[u, v]}(T(r))\big) \ar[r]^{1-\varphi} \ar[d]_{\nabla_0} & \pazk\big(\Lie \Gamma_S\prm, \Npi^{[u, v/p]}(T(r))\big) \ar[d]^{\nabla_0} \\
				\pazk\big(\Lie \Gamma_S\prm, t\Fil^{-1} \Npi^{[u, v]}(T(r))\big) \ar[r]^{1-\varphi} & \pazk\big(\Lie \Gamma_S\prm, t\Npi^{[u, v/p]}(T(r))\big)
			}
		}
	\right].
\end{displaymath}
By definition, the complex $\pazk\big(\varphi, \Lie \Gamma_S, \Npi^{[u, v]}(T(r))\big)$ is $p^{4r}\textrm{-isomorphic}$ to the source complex in Lemma \ref{lem:twist_lie_koszul_complex}.
Combining this with Lemma \ref{lem:diff_to_lie} and Lemma \ref{lem:twist_lie_koszul_complex}, we get that,
\begin{prop}\label{prop:diff_to_lie}
	There exists a natural $p^{4r}\textrm{-quasi-isomorphism}$ of complexes
	\begin{equation*}
		\Kos\big(\varphi, \partial_A, \Fil^r \Npi^{[u, v]}(T)\big) \simeq \pazk\big(\varphi, \Lie \Gamma_S, \Npi^{[u, v]}(T(r))\big).
	\end{equation*}
\end{prop}

\subsection{From the infinitesimal action of \texorpdfstring{$\Gamma_S$}{-} to the continuous action of \texorpdfstring{$\Gamma_S$}{-}}\label{subsec:lie_to_gamma}

In this subsection, we will study Koszul complexes involving operators $\gamma_i-1$ over $\Npi^{[u, v]}(T(r))$.
Note that we have $(\gamma_i-1) \Fil^k \Npi^{[u, v]}(T(r)) \subset \Fil^k \Npi^{[u, v]}(T(r)) \cap \pi \Npi^{[u, v]}(T(r)) = \pi \Fil^{k-1} \Npi^{[u, v]}(T(r))$, where the last equality follows from Lemma \ref{lem:ns_fil_pi_cap} and Lemma \ref{lem:fil_ns_k_twist}.
Define a subcomplex of the Koszul complex $\Kos\big(\Gamma_S\prm, \Fil^0 \Npi^{[u, v]}(T(r))\big)$ (see Definition \ref{defi:koszul_complex_gammaR_prime}), as follows:
\begin{align*}
	\pazk\big(\Gamma_S\prm, \Fil^0 \Npi^{[u, v]}(T(r))\big) := \Fil^0 \Npi^{[u, v]}(T(r)) &\xrightarrow{\hspace{1mm}(\tau_i)\hspace{1mm}} \big(\pi\Fil^{-1} \Npi^{[u, v]}(T(r))\big)^{I_1\prm} \longrightarrow\\
	&\longrightarrow \big(\pi^2\Fil^{-2} \Npi^{[u, v]}(T(r))\big)^{I_2\prm} \longrightarrow \cdots.
\end{align*}
Similarly, we can define a complex $\pazk^c\big(\Gamma_S\prm, \pi\Fil^{-1} \Npi^{[u, v]}(T(r))\big)$ as a subcomplex of the Koszul complex $\Kos^c\big(\Gamma_S\prm, \Fil^0 \Npi^{[u, v]}(T(r))\big)$ (see Definition \ref{defi:koszul_complex_gammaR_prime}), where $c = \chi(\gamma_0) = \exp(p^m)$.
Consider a morphism of complexes $\tau_0 : \pazk\big(\Gamma_S\prm, \Fil^0 \Npi^{[u, v]}(T(r))\big) \rightarrow \pazk^c\big(\Gamma_S\prm, \pi \Fil^{-1} \Npi^{[u, v]}(T(r))\big)$, which is given as $\gamma_0-1$ in degree 0 and as $\tau_0^k : \big(\pi^k \Fil^{-k} \Npi^{[u, v]}(T(r))\big)^{I_k\prm} \rightarrow \big(\pi^{k+1} \Fil^{-k-1} \Npi^{[u, v]}(T(r))\big)^{I_k\prm}$, on the $k\textrm{-th}$ term of the definition above, for $1 \leqslant k \leqslant d$ (see Definition \ref{defi:tau_0} and Definition \ref{defi:koszul_complex_gamma}).
Denote the total complex of the diagram thus obtained as $\pazk\big(\Gamma_S, \Fil^0 \Npi^{[u, v]}(T(r))\big)$, which is a subcomplex of the Koszul complex $\Kos\big(\Gamma_S, \Fil^0 \Npi^{[u, v]}(T(r))\big)$.
In a similar manner, we can define complexes $\pazk\big(\Gamma_S\prm, \Npi^{[u, v/p]}(T(r))\big)$ and $\pazk^c\big(\Gamma_S\prm, \pi \Npi^{[u, v/p]}(T(r))\big)$ and a map $\tau_0$ from the former to the latter complex.

Recall that $t/\pi$ is a unit in $\ARpi^{[u, v]}$ (see Lemma \ref{lem:t_over_pi_unit}), therefore, we see that $t^k\Fil^{-k} \Npi^{[u, v]}(T(r)) = \pi^k\Fil^{-k} \Npi^{[u, v]}(T(r))$, for all $k \in \ZZ$.
Now, define a morphism of complexes $\beta : \pazk\big(\Gamma_S\prm, \Fil^0 \Npi^{[u, v]}(T(r))\big) \rightarrow \pazk\big(\Lie \Gamma_S\prm, \Fil^0 \Npi^{[u, v]}(T(r))\big)$, which is the identity in degree 0 and given as 
\begin{align*}
	\beta_k : \big(t^k\Fil^{-k} \Npi^{[u, v]}(T(r))\big)^{I_k\prm} &\longrightarrow \big(t^k\Fil^{-k} \Npi^{[u, v]}(T(r))\big)^{I_k\prm}\\
	(a_{i_1 \cdots i_k}) &\longmapsto \big(\nabla_{i_k} \cdots \nabla_{i_1} \tau_{i_1}^{-1} \cdots \tau_{i_k}^{-1}(a_{i_1 \cdots i_k})\big),
\end{align*}
on the $k\textrm{-th}$ term of the definition above, for $1 \leqslant k \leqslant d$.
Similarly, define a morphism of complexes $\beta^c : \pazk^c\big(\Gamma_S\prm, t\Fil^{-1} \Npi^{[u, v]}(T(r))\big) \rightarrow \pazk^c\big(\Lie \Gamma_S\prm, t\Fil^{-1} \Npi^{[u, v]}(T(r))\big)$ which is given as $\beta_0^c = \nabla_0 \tau_0^{-1}$ in degree 0 and as 
\begin{align*}
	\beta_k^c : \big(t^{k+1}\Fil^{-k-1} \Npi^{[u, v]}(T(r))\big)^{I_k\prm} &\longrightarrow \big(t^{k+1}\Fil^{-k-1} \Npi^{[u, v]}(T(r))\big)^{I_k\prm}\\
	(a_{i_1 \cdots i_k}) &\longmapsto \big(\nabla_{i_k} \cdots \nabla_{i_1} \nabla_0 \tau_0^{-1} \tau_{i_1}^{c,-1} \cdots \tau_{i_k}^{c,-1}(a_{i_1 \cdots i_k})\big),
\end{align*}
on the $k\textrm{-th}$ term of the definition above, for $1 \leqslant k \leqslant d$.
Similarly, one can define the maps $\beta$ and $\beta^c$ for the $\ARpi^{[u, v/p]}\textrm{-module}$ $\Npi^{[u, v/p]}$, giving morphisms of complexes $\beta : \pazk\big(\Gamma_S\prm, \Npi^{[u, v/p]}(T(r))\big) \rightarrow \pazk\big(\Lie \Gamma_S\prm, \Npi^{[u, v/p]}(T(r))\big)$ and $\beta^c : \pazk^c\big(\Gamma_S\prm, t\Npi^{[u, v/p]}(T(r))\big) \rightarrow \pazk^c\big(\Lie \Gamma_S\prm, t \Npi^{[u, v/p]}(T(r))\big)$.

For each $j \in \NN$, we have that $t^j \Fil^{-j} \Npi^{[u, v]}(T(r)) \subset \Npi^{[u, v]}(T(r))$ and the induced Frobenius gives $\varphi(t^j \Fil^{-j} \Npi^{[u, v]}(T(r))) = \varphi(t^{j-r} \Fil^{r-j} \Npi^{[u, v]}(T)(r)) \subset t^j \Npi^{[u, v/p]}(T(r))$, where we have used Lemma \ref{lem:fil_ns_k_twist} and the fact that $t/\pi$ is a unit in $\ARpi^{[u, v]}$ (see Lemma \ref{lem:t_over_pi_unit}).
Using the Frobenius morphism and the morphism of complexes described above, we obtain an induced morphism of complexes
\begin{align*}
	\left[
		\vcenter
		{
			\xymatrix
			{
				\pazk\big(\Gamma_S\prm, \Fil^0 \Npi^{[u, v]}(T(r))\big) \ar[r]^{1-\varphi} \ar[d]^{\tau_0} & \pazk\big(\Gamma_S\prm, \Npi^{[u, v/p]}(T(r))\big) \ar[d]^{\tau_0} \\
				\pazk^c\big(\Gamma_S\prm, t\Fil^{-1} \Npi^{[u, v]}(T(r))\big) \ar[r]^{1-\varphi} & \pazk^c\big(\Gamma_S\prm, t \Npi^{[u, v/p]}(T(r))\big)
			}
		}
	\right]
	\xrightarrow{(\beta, \beta^c)}
	\pazk\big(\varphi, \Lie \Gamma_S, \Npi^{[u, v]}(T(r))\big)
\end{align*}
We denote the complex on the left as $\pazk\big(\varphi, \Gamma_S, \Npi^{[u, v]}(T(r))\big)$ and write the map as
\begin{equation*}
		\call = (\beta, \beta^c) : \pazk\big(\varphi, \Gamma_S, \Npi^{[u, v]}(T(r))\big) \longrightarrow \pazk\big(\varphi, \Lie \Gamma_S, \Npi^{[u, v]}(T(r))\big),
\end{equation*}

\begin{prop}\label{prop:quasi_iso_Lie_relative}
	The morphism of complexes $\call$ described above is an isomorphism.
\end{prop}
\begin{proof}
	The proof follows in essentially the same manner as \cite[Lemma 4.6]{colmez-niziol-nearby-cycles}.
	One needs to use Lemma \ref{lem:gamma_minus_1_oc}, Lemma \ref{lem:nabla_quotient_tau_invertible} and Corollary \ref{lem:nabla_converges_relative} instead of \cite[Lemma 2.34]{colmez-niziol-nearby-cycles} in the proof.
	We omit the details.
\end{proof}

\subsection{Change of the annulus of convergence : Part 1}\label{subsec:change_annulus_1}

In this subsection, we will pass from the analytic ring $\ARpi^{[u, v]}$ to the overconvergent ring $\ARpi^{(0, v]+}$ and also twist our module by $\ZZ_p(r)$.
Let us set $\Npi^{(0, v]+}(T(r)) := \ARpi^{(0, v]+} \otimes_{\AR^+} \mbfn(T(r))$ equipped with the natural action of $\Gamma_R$ and a $\Gamma_R\textrm{-stable}$ filtration as in \eqref{eq:fil_ns_k}.
Define a subcomplex of the Koszul complex $\Kos\big(\Gamma_S\prm, \Fil^0 \Npi^{(0, v]+}(T(r))\big)$ (see Definition \ref{defi:koszul_complex_gammaR_prime}), as follows:
\begin{align*}
	\pazk\big(\Gamma_S\prm, \Fil^0 \Npi^{(0, v]+}(T(r))\big) := \Fil^0 \Npi^{(0, v]+}(T(r)) &\xrightarrow{\hspace{1mm}(\tau_i)\hspace{1mm}} \big(\pi\Fil^{-1} \Npi^{(0, v]+}(T(r))\big)^{I_1\prm} \longrightarrow\\
	&\longrightarrow \big(\pi^2\Fil^{-2} \Npi^{(0, v]+}(T(r))\big)^{I_2\prm} \longrightarrow \cdots.
\end{align*}
Similarly, we can define a complex $\pazk^c\big(\Gamma_S\prm, \pi\Fil^{-1} \Npi^{(0, v]+}(T(r))\big)$ as a subcomplex of the Koszul complex $\Kos^c\big(\Gamma_S\prm, \Fil^0 \Npi^{(0, v]+}(T(r))\big)$ (see Definition \ref{defi:koszul_complex_gammaR_prime}).
Now, consider a morphism of complexes $\tau_0 : \pazk\big(\Gamma_S\prm, \Fil^0 \Npi^{(0, v]+}(T(r))\big) \rightarrow \pazk^c\big(\Gamma_S\prm, \pi \Fil^{-1} \Npi^{(0, v]+}(T(r))\big)$ which is given as $\gamma_0-1$ in degree 0 and as $\tau_0^k : \big(\pi^k \Fil^{-k} \Npi^{(0, v]+}(T(r))\big)^{I_2\prm} \rightarrow \big(\pi^k \Fil^{-k-1} \Npi^{(0, v]+}(T(r))\big)^{I_2\prm}$, on the $k\textrm{-th}$ term of the definition above, for $1 \leqslant k \leqslant d$ (see Definition \ref{defi:tau_0} and Definition \ref{defi:koszul_complex_gamma}).
Write the total complex of the diagram thus obtained as $\pazk\big(\Gamma_S, \Fil^0 \Npi^{(0, v]+}(T(r))\big)$, a subcomplex of the Koszul complex $\Kos\big(\Gamma_S, \Fil^0 \Npi^{(0, v]+}(T(r))\big)$.
In a similar manner, we can define the complexes $\pazk\big(\Gamma_S\prm, \Npi^{(0, v/p]+}(T(r))\big)$ and $\pazk^c\big(\Gamma_S\prm, \pi \Npi^{(0, v/p]+}(T(r))\big)$ and a map $\tau_0$ from the former to the latter complex.

For each $j \in \NN$, we have that $\pi^j \Fil^{-j} \Npi^{(0, v]+}(T(r)) \subset \Npi^{(0, v]+}(T(r))$ and the induced Frobenius gives $\varphi(\pi^j \Fil^{-j} \Npi^{(0, v]+}(T(r))) = \varphi(\pi^{j-r} \Fil^{r-j} \Npi^{(0, v]+}(T)(r)) \subset \pi^j \Npi^{(0, v/p]+}(T(r))$, where the equality follows from Lemma \ref{lem:fil_ns_k_twist}.
So we define the complex,
\begin{displaymath}
	\pazk\big(\varphi, \Gamma_S, \Npi^{(0, v]+}(T(r))\big) :=
	\left[
		\vcenter
		{
			\xymatrix
			{
				\pazk\big(\Gamma_S\prm, \Fil^0 \Npi^{(0, v]+}(T(r))\big) \ar[r]^{1-\varphi} \ar[d]_{\tau_0} & \pazk\big(\Gamma_S\prm, \Npi^{(0, v/p]+}(T(r))\big) \ar[d]^{\tau_0} \\
				\pazk^c\big(\Gamma_S\prm, \pi \Fil^{-1} \Npi^{(0, v]+}(T(r))\big) \ar[r]^{1-\varphi} & \pazk^c\big(\Gamma_S\prm, \pi \Npi^{(0, v/p]+}(T(r))\big)
			}
		}
	\right].
\end{displaymath}

\begin{prop}\label{prop:quasi_iso_oc_robba_relative}
	The natural morphism of complexes $\pazk\big(\varphi, \Gamma_S, \Npi^{(0, v]+}(T(r))\big) \rightarrow \pazk\big(\varphi, \Gamma_S, \Npi^{[u, v]}(T(r))\big)$, induced by the inclusion $\Npi^{(0, v]+}(T(r)) \subset \Npi^{[u, v]}(T(r))$, is a $p^{3r}\textrm{-quasi-isomorphism}$.
\end{prop}
\begin{proof}
	The map in the claim is injective on each term, so we need to show that the cokernel complex is killed by $p^{3r}$.
	In the cokernel complex, for $k \in \NN$, we have maps
	\begin{equation}\label{eq:id_minus_frob_robba_over_oc1}
		1-\varphi : \pi^k \Fil^{-k} \Npi^{[u, v]}(T(r)) / \pi^k \Fil^{-k} \Npi^{(0, v]+}(T(r)) \longrightarrow \pi^k \Npi^{[u, v/p]}(T(r)) / \pi^k \Npi^{(0, v/p]+}(T(r)),
	\end{equation}
	and it is enough to show that these are $p^{3r}\textrm{-bijective}$.
	Let us set $\Npi^{(0, v]+}(T) := \ARpi^{(0, v]+} \otimes_{\AR^+} \mbfn(T)$, $\Npi^{(0, v]+}(T)(r) := \Npi^{(0, v]+}(T) \otimes_{\ZZ_p} \ZZ_p(r)$ and $\Npi^{[u, v]}(T)(r) := \Npi^{[u, v]}(T) \otimes_{\ZZ_p} \ZZ_p(r)$, equipped with the filtration as in \eqref{eq:fil_ns} (upto twisting the filtered pieces by $\ZZ_p(r)$ in the latter cases).
	Moreover, for any $k \in \NN$, by Lemma \ref{lem:fil_ns_k_twist} we have that $\pi^k \Fil^{-k} \Npi^{(0, v]+}(T(r)) = \pi^{k-r} \Fil^{r-k} \Npi^{(0, v]+}(T)(r)$ and $\pi^k \Fil^{-k} \Npi^{[u, v]}(T(r)) = \pi^{k-r} \Fil^{r-k} \Npi^{[u, v]}(T)(r)$.
	So, for $n = r-k$, we can rewrite \eqref{eq:id_minus_frob_robba_over_oc1} as,
	\begin{equation}\label{eq:id_minus_frob_robba_over_oc2}
		1 - \varphi : \pi^{-n} \Fil^{n} \Npi^{[u, v]}(T) / \pi^{-n} \Fil^{n} \Npi^{(0, v]+}(T) \longrightarrow \pi^{-n} \Npi^{[u, v/p]}(T) / \pi^{-n} \Npi^{(0, v/p]+}(T).
	\end{equation}
	Note that the twist has disappeared since $\varphi$ acts trivially on it.
	For $n \leqslant 0$, the claim follows from Lemma \ref{lem:id_minus_frob_bijective_robba_by_oc}.
	For $n > 0$, we first claim that the following natural map is $p^{n}\textrm{-bijective}$:
	\begin{equation}\label{eq:id_robba_over_oc2}
		\pi_1^{-n} \Npi^{[u, v]}(T) / \pi_1^{-n} \Npi^{(0, v]+}(T) \longrightarrow \pi^{-n} \Fil^{n} \Npi^{[u, v]}(T) / \pi^{-n} \Fil^{n} \Npi^{(0, v]+}(T),
	\end{equation}
	Indeed, recall that $\xi = \pi/\pi_1$ and from \eqref{eq:fil_ns} and Lemma \ref{lem:fil_as_prod}, it is clear that $\xi^n \Npi^{(0, v]+}(T) \subset \Fil^n \Npi^{(0, v]+}(T)$, in particular, we have $\Npi^{(0, v]+}(T) \subset \Npi^{[u, v]}(T) \cap \xi^{-n} \Fil^n \Npi^{(0, v]+}(T) = (\ARpi^{[u, v]} \cap \xi^{-n} \ARpi^{(0, v]+}) \otimes_{\AR^+} \mbfn(T) = \Npi^{(0, v]+}(T)$, where the first equality follows because $\mbfn(T)$ is free over $\AR^+$ and the second equality follows because $\xi^n \ARpi^{[u, v]} \cap \ARpi^{(0, v]+} \subset \Fil^n \ARpi^{(0, v]+} = \xi^n \ARpi^{(0, v]+}$ (see Definition \ref{defi:filtration_vanishing_varpi} and Remark \ref{rem:fil_pd_u}).
	In particular, we see that $\pi_1^{-n} \Npi^{[u, v]}(T) \cap \pi^{-n} \Fil^n \Npi^{(0, v]+}(T) = \pi_1^{-n} \Npi^{(0, v]+}(T)$, i.e.\ \eqref{eq:id_robba_over_oc2} is injective.
	Next, to show the $p^n\textrm{-surjectivity}$ of \eqref{eq:id_robba_over_oc2}, write $\ARpi^{[u, v]} = \ARpi^{[u]} + \ARpi^{(0, v]+}$ and set $\Npi^{[u]}(T) := \ARpi^{[u]} \otimes_{\AR^+} \mbfn(T)$ and $\Npi^+(T) := \ARpi^+ \otimes_{\AR^+} \mbfn(T)$, equipped with the induced filtration as in \eqref{eq:fil_ns}.
	Then, to obtain the $p^n\textrm{-surjectivity}$ of \eqref{eq:id_robba_over_oc2}, it is enough to show that the natural map $\pi_1^{-n} \Npi^{[u]}(T) + \pi^{-n} \Fil^n \Npi^+(T) \rightarrow \pi^{-n} \Fil^n \Npi^{[u]}(T)$ is $p^n\textrm{-surjective}$, or equivalently, that the natural map $\xi^n \Npi^{[u]}(T) + \Fil^n \Npi^+(T) \rightarrow \Fil^n \Npi^{[u]}(T)$ is $p^n\textrm{-surjective}$.
	To show the latter claim, let $\{e_1, \ldots, e_h\}$ be an $\AR^+\textrm{-basis}$ of $\mbfn(T)$, and take $x \in \Fil^n \Npi^{[u]}(T)$ and write $x = \sum_{i=1}^h a_i e_i$, with $a_i \in \ARpi^{[u]}$.
	Note that from Lemma \ref{lem:split_pd_elements} we can write $a_i = a_{i1} + a_{i2}$, with $a_{i1} \in \Fil^n \ARpi^{[u]} \subset p^{-n} \xi^n \ARpi^{[u]}$ (see Remark \ref{rem:fil_pd_u}) and $a_{i2} \in p^{-\lfloor nu \rfloor} \ARpi^+$.
	So we see that $x_1 = \sum_{i=1}^h a_{i1} e_i$ is in $p^{-n} \xi^n \Npi^{[u]}(T)$ and $x_2 = \sum_{i=1}^h a_{i2} e_i = x - x_1$ is in $p^{-\lfloor nu \rfloor} \Npi^+(T) \cap \Fil^n \Npi^{[u]}(T) \subset \Npi^{[u]}(T)[1/p]$.
	Now, as we have $u = (p-1)/p < 1$, therefore, it follows that $p^n x_2$ is in $\Npi^+(T) \cap \Fil^n \Npi^{[u]}(T) = \Fil^n \Npi^+(T)$ (see Lemma \ref{lem:ns_fil_pi_cap}), i.e.\ $p^n x = p^n x_1 + p^n x_2$ is in $\xi^n \Npi^{[u]}(T) + \Fil^n \Npi^+(T)$.
	In particular, we get that \eqref{eq:id_robba_over_oc2} is $p^n\textrm{-bijective}$, and therefore, \eqref{eq:id_minus_frob_robba_over_oc2} is $p^n\textrm{-isomorphic}$ to 
	\begin{equation*}
		1 - \varphi : \pi_1^{-n} \Npi^{[u, v]}(T) / \pi_1^{-n} \Npi^{(0, v]+}(T) \longrightarrow \pi^{-n} \Npi^{[u, v/p]}(T) / \pi^{-n} \Npi^{(0, v/p]+}(T).
	\end{equation*}
	Recall that we have $v = p-1$, so by Lemma \ref{lem:pi_m_divides_p} (iii) it follows that $\pi$ divides $p$ in $\ARpi^{(0, v/p]+}$ and $\pi_1$ divides $p$ in $\ARpi^{(0, v]+}$, therefore, \eqref{eq:id_minus_frob_robba_over_oc2} is $p^{2n}\textrm{-isomorphic}$ to the following map:
	\begin{equation*}
		1 - \varphi : \Npi^{[u, v]}(T) / \Npi^{(0, v]+}(T) \longrightarrow \Npi^{[u, v/p]}(T) / \Npi^{(0, v/p]+}(T).
	\end{equation*}
	Now, from Lemma \ref{lem:id_minus_frob_bijective_robba_by_oc}, the map above is bijective (note that Frobenius has no effect on twist).
	Therefore, we conclude that \eqref{eq:id_minus_frob_robba_over_oc1} is $p^{3n}\textrm{-bijective}$.
	As $n = r-k \leqslant r$, it follows that the cokernel complex of the map in the claim of the lemma is killed by $p^{3r}$.
	This allows us to conclude.
\end{proof}

\begin{lem}\label{lem:id_minus_frob_bijective_robba_by_oc}
	For each $k \in \NN$, the following natural map is bijective
	\begin{equation*}
		1 - \varphi : \pi^k \Npi^{[u, v]}(T) / \pi^k \Npi^{(0, v]+}(T) \isomorphic \pi^k \Npi^{[u, v/p]}(T) / \pi^k \Npi^{(0, v/p]+}(T),
	\end{equation*}
\end{lem}
\begin{proof}
	For $k = 0$, using a basis of $\mbfn(T)$, one first shows that the natural map $\Npi^{[u, v]}(T) / \Npi^{(0, v]+}(T) \rightarrow \Npi^{[u, v/p]}(T) / \Npi^{(0, v/p]+}(T)$ is bijective, in particular, $1-\varphi$ is an endomorphism of $\Npi^{[u, v]}(T) / \Npi^{(0, v]+}(T)$.
	Then, following the strategy of \cite[Lemma 4.8]{colmez-niziol-nearby-cycles} one shows that on the preceding quotient, $1 + \varphi + \varphi^2 + \cdots$ converges as an inverse to $1-\varphi$.
	We omit the details.
	For $k > 0$, note that $\varphi$ preserves the quotient $\pi^k \Npi^{[u, v]}(T) / \pi^k \Npi^{(0, v]+}(T)$.
	So, from the case $k=0$, it follows that $1 + \varphi + \varphi^2 + \cdots$ converges on the preceding quotient as well.
\end{proof}

\subsection{Change of the annulus of convergence : Part 2}\label{subsec:change_annulus_2}

In this subsection, we will change the ring of coefficients from $\ARpi^{(0, v]+}$ to $\ARpi^{(0, v/p]+}$ by replacing $\varphi$ with its left inverse $\psi$ (under the asssumption that $m \geqslant 2$).

\subsubsection{From \texorpdfstring{$(\varphi, \Gamma_S)\textrm{-complexes}$}{-} to \texorpdfstring{$(\psi, \Gamma_S)\textrm{-complexes}$}{-}}\label{subsubsec:phigamma_to_psigamma}

From Proposition \ref{prop:psi_oper}, recall that we have the left inverse $\psi$ of the Frobenius endomorphism on $\mbfa$, satisfying $\psi(\mbfa) \subset \mbfa$.
This induces an operator $\psi : \ARpi^{(0, v/p]+} \rightarrow \ARpi^{(0, v]+}$, which commutes with the action of $\Gamma_R$, in particular, we have $\psi(\ARpi^{(0, v]+}) \subset \ARpi^{(0, v]+}$.
Equivalently, one can also define the operator $\psi$ by first identifying $\iota_{\cycl} : R_{\varpi}^{(0, v/p]+} \isomorphic \ARpi^{(0, v/p]+}$ and then considering the left inverse of the cyclotomic Frobenius over $R_{\varpi}^{(0, v/p]+}$ (see \S \ref{subsec:cyclotomic_frob} and \S \ref{subsec:cyclotomic_embeddings}).

Next, from Lemma \ref{lem:psi_wach_nonpositive} recall that the operator $\psi$ extends to $\mbfn(T(r))$ and we have $\psi(\mbfn(T(r))) \subset \mbfn(T(r))$.
By extending scalars to $\ARpi^{(0, v]+}$ and from the discussion above we see that $\psi(\Npi^{(0, v]+}(T(r))) \subset \psi(\Npi^{(0, v/p]+}(T(r))) \subset \Npi^{(0, v]+}(T(r))$.
Moreover, using the description of the filtration on $\Npi^{(0, v]+}(T)$ from Lemma \ref{lem:fil_as_prod}, it follows that for $0 \leqslant k \leqslant r$, we have $\varphi(\Fil^{r-k} \Npi^{(0, v]+}(T)) \subset q^{r-k} \Npi^{(0, v/p]+}(T)$.
Upon multiplying the terms of the preceding inclusion by $\varphi(\pi^{k-r})$ and twisting by $\ZZ_p(r)$, we get that $\varphi(\pi^{k-r} \Fil^{r-k} \Npi^{(0, v]+}(T)(r)) \subset \pi^{k-r} \Npi^{(0, v/p]+}(T)(r)$.
In particular, by using Lemma \ref{lem:fil_ns_k_twist}, we note that $\pi^k \Fil^{-k} \Npi^{(0, v]+}(T(r)) \subset \psi(\pi^k \Npi^{(0, v/p]+}(T(r)))$ and since $\Fil^{-k} \Npi^{(0, v]+}(T(r)) \subset \Npi^{(0, v/p]+}(T(r))$, therefore, it follows that $(\psi - 1) (\pi^k \Fil^{-k} \Npi^{(0, v]+}(T(r))) \subset \psi(\pi^k \Npi^{(0, v/p]+}(T(r)))$.

Set $\pazk(\Gamma_S', N_{\psi}) := \psi(\pazk(\Gamma_S', \Npi^{(0, v/p]+}(T(r))))$ and $\pazk^c(\Gamma_S', N_{\psi}) := \psi(\pazk^c(\Gamma_S', \Npi^{(0, v/p]+}(T(r))))$.
From \S \ref{subsec:change_annulus_1}, recall that we defined maps $\tau_0: \pazk\big(\Gamma_S\prm, \Fil^0 \Npi^{(0, v]+}(T(r))\big) \rightarrow \pazk^c\big(\Gamma_S\prm, \pi \Fil^{-1} \Npi^{(0, v]+}(T(r))\big)$ and $\tau_0 : \psi(\pazk(\Gamma_S', \Npi^{(0, v/p]+}(T(r)))) \rightarrow \psi(\pazk^c(\Gamma_S', \Npi^{(0, v/p]+}(T(r))))$.
As $\psi$ commutes with the action of $\Gamma_S$, therefore, from the latter map, we obtain an induced morphism $\tau_0: \pazk(\Gamma_S', N_{\psi}) \rightarrow \pazk^c(\Gamma_S', N_{\psi})$.
Now, using the discussion above, note that we have a well-defined map between source complexes of the maps $\tau_0$ above, given as $\psi - 1 : \pazk\big(\Gamma_S\prm, \Fil^0 \Npi^{(0, v]+}(T(r))\big) \rightarrow \pazk(\Gamma_S', N_{\psi})$, and similarly for the target complexes of $\tau_0$.
Therefore, similar to the complex $\pazk\big(\varphi, \Gamma_S, N^{(0,v]+}(T(r))\big)$ in \S \ref{subsec:change_annulus_1}, we define the following complex:
\begin{displaymath}
	\pazk\big(\psi, \Gamma_S, \Npi^{(0, v]+}(T(r))\big) :=
	\left[
		\vcenter
		{
			\xymatrix
			{
				\pazk\big(\Gamma_S\prm, \Fil^0 \Npi^{(0, v]+}(T(r))\big) \ar[r]^{\hspace{10mm}\psi-1} \ar[d]_{\tau_0} & \pazk\big(\Gamma_S', N_{\psi}\big) \ar[d]^{\tau_0} \\
				\pazk^c\big(\Gamma_S\prm, \pi \Fil^{-1} \Npi^{(0, v]+}(T(r))\big) \ar[r]^{\hspace{12mm}\psi-1} & \pazk^c\big(\Gamma_S', N_{\psi}\big)\big)
			}
		}
	\right].
\end{displaymath}

\begin{prop}\label{prop:oc_psi_gamma_phi_gamma}
	The morphism $\tau_{\leqslant r} \pazk\big(\varphi, \Gamma_S, \Npi^{(0, v]+}(T(r))\big) \longrightarrow \tau_{\leqslant r} \pazk\big(\psi, \Gamma_S, \Npi^{(0, v]+}(T(r))\big)$, induced by the identity in the first column and $\psi$ in the second column is a $p^{r+2}\textrm{-quasi-isomorphism}$.
\end{prop}
\begin{proof}
	By definition, note that the map is surjective on each term, so we need to show that the kernel complex is $p^{r+2}\textrm{-acyclic}$.
	As the map in the claim is identity on the first column, therefore, the kernel complex can be written as
	\begin{equation*}
		\tau_{\leqslant r} \big[\pazk\big(\Gamma_S\prm, \big(\Npi^{(0, v/p]+}(T(r))\big)^{\psi=0}\big) \xrightarrow{\hspace{1mm}\tau_0\hspace{1mm}} \pazk^c\big(\Gamma_S\prm, \big(\pi \Npi^{(0, v/p]+}(T(r))\big)^{\psi=0}\big)\big].
	\end{equation*}
	Clearly the terms of the complex above are $\varphi(\ARpi^{(0, v]+})\textrm{-modules}$.
	Recall that $p/\pi \in \varphi(\ARpi^{(0, v]+})$ (since $\pi_1$ divides $p$ in $\ARpi^{(0, v]+}$, see Lemma \ref{lem:pi_m_divides_p} (ii) for $v = p-1$), so we obtain that $(\pi^k \Npi^{(0, v/p]+}(T(r)))^{\psi=0}$ is $p^{r-k}\textrm{-isomorphic}$ to $(\Npi^{(0, v/p]+}(T)(r))^{\psi=0}$, for $k \leqslant r$.
	In particular, the complex above is $p^r\textrm{-quasi-isomorphic}$ to the following complex:
	\begin{equation}\label{eq:kercomplex}
		\tau_{\leqslant r} \big[\Kos\big(\Gamma_S\prm, \big(\Npi^{(0, v/p]+}(T)(r)\big)^{\psi=0}\big) \xrightarrow{\hspace{1mm}\tau_0\hspace{1mm}} \Kos^c\big(\Gamma_S\prm, \big(\Npi^{(0, v/p]+}(T)(r)\big)^{\psi=0}\big)\big].
	\end{equation}
	We will show that the complex in \eqref{eq:kercomplex} is $p^2\textrm{-acyclic}$, but to prove our claim we will need a simpler description of the $\varphi(\ARpi^{(0, v]+})\textrm{-module}$ $\big(\Npi^{(0, v/p]+}(T)\big)^{\psi=0}$.

	Let $\{e_1, \ldots, e_h\}$ denote an $\AR^+\textrm{-basis}$ of $\mbfn(T)$.
	As the attached $(\varphi, \Gamma_S)\textrm{-module}$ $\Dpi(T) = \ARpi \otimes_{\AR} \mbfd(T)$ over $\ARpi$ is \'etale, so we see that $\{\varphi(e_1), \ldots \varphi(e_h)\}$ is an $\ARpi\textrm{-basis}$ of $\Dpi(T)$.
	Now, let us note that $z = \sum_{j=1}^h z_j \varphi(e_j)$ is in $\Dpi(T)^{\psi=0}$ if and only if $z_j \in \big(\ARpi\big)^{\psi=0}$, for each $1 \leqslant j \leqslant h$.
	Indeed, $\psi(z) = 0$ if and only if $\sum_{j=1}^h \psi(z_j) e_j = 0$, and since $e_j$ are linearly independent over $\ARpi$, therefore, we see that $\psi(z) = 0$ if and only if $\psi(z_j) = 0$, for all $1 \leqslant j \leqslant h$.
	Next, using Lemma \ref{lem:analytic_rings_psi_action} (ii), note that we have a decomposition $\ARpi^{\psi=0} = \oplus_{\alpha \neq 0} \varphi\big(\ARpi\big)[X^{\flat}]^{\alpha}$, where $[X^{\flat}]^{\alpha} = (1+\pi_m)^{\alpha_0}[X_1^{\flat}]^{\alpha_0} \cdots [X_d^{\flat}]^{\alpha_d}$ and $\alpha =(\alpha_0, \ldots, \alpha_d)$ is a $(d+1)\textrm{-tuple}$ with $\alpha_i \in \{0, \ldots, p-1\}$.
	Therefore, we see that $\Dpi(T)^{\psi=0} = \big(\sum_{j=1}^h \ARpi \varphi(e_j)\big)^{\psi=0} = \oplus_{\alpha \neq 0} \sum_{j=1}^h \varphi\big(\ARpi e_j\big)[X^{\flat}]^{\alpha} = \oplus_{\alpha \neq 0} \varphi\big(\Dpi(T)\big)[X^{\flat}]^{\alpha}$.
	Note that inside $\Dpi(T)$ we have $\big(\Npi^{(0, v/p]+}(T)\big)^{\psi=0} = \Dpi(T)^{\psi=0} \cap \Npi^{(0, v/p]+}(T)$.
	So using the decomposition above, we set $N[X^{\flat}]^{\alpha} := \varphi\big(\Dpi(T)\big)[X^{\flat}]^{\alpha} \cap \Npi^{(0, v/p]+}(T)$, for $\alpha \neq 0 $, where the intersection is taken inside $\Dpi(T)^{\psi=0}$.
	Note that we have $\varphi(\ARpi^{(0, v]+}) \subset \varphi(\ARpi) \cap \ARpi^{(0, v/p]+}$.
	Therefore, it follows that $N := N[X^{\flat}]^{\alpha} [X^{\flat}]^{-\alpha}$ is a $\varphi(\ARpi^{(0, v]+})\textrm{-module}$ contained in $\Npi^{(0, v/p]+}(T)$, stable under the action of $\Gamma_S$ and independent of $\alpha$.
	Indeed, for the last part note that for $\alpha \neq \alpha\prm$, we have $\sum_{i=1}^h \varphi(x_i e_i)[X^{\flat}]^{\alpha} \in N[X^{\flat}]^{\alpha}$ if and only if $\sum_{i=1}^h \varphi(x_i e_i) [X^{\flat}]^{\alpha\prm} \in N[X^{\flat}]^{\alpha\prm}$.
	In conclusion, we get the equalities $\big(\Npi^{(0, v/p]+}(T)\big)^{\psi=0} = \oplus_{\alpha \neq 0} N [X^{\flat}]^{\alpha} = \oplus_{\alpha \neq 0} \varphi(\Npi^{(0, v]+}) [X^{\flat}]^{\alpha}$, where the last equality follows from the following:

\begin{lem}\label{lem:Npivp_psi0}
	For $v = p-1$, let $x \in \Dpi(T)$ such that $\varphi(x) \in \Npi^{(0, v/p]+}(T)$, then $x \in \Npi^{(0, v]+}(T)$.
	In particular, we have $N = \varphi(\Npi^{(0, v]+}(T))$.
\end{lem}
\begin{proof}
	Let $\Npi^+(T) = \ARpi^+ \otimes_{\AR^+} \mbfn(T)$ and note that $\Dpi(T)/p = (\Npi^+(T)/p)[1/\pi_m]$ and $\Npi^{(0, v]+}(T) = \sum_{n \in \NN} p^n \pi_m^{-\lfloor ne/v \rfloor} \Npi^+(T)$ (since $\mbfn(T)$ is finite free over $\mbfa_R^+$).
	Then the proof of \cite[Lemma 2.14]{colmez-niziol-nearby-cycles} can easily be adapted to obtain the claim.
	We omit the details.
\end{proof}

\begin{rem}\label{rem:Npivp_psi0_gamma}
	From Lemma \ref{lem:Npivp_psi0}, we have $N = \varphi(\Npi^{(0, v]+}(T))$.
	Then for any $i \in \{0, \ldots, d\}$, using Lemma \ref{lem:gamma_minus_1_oc} (i), note that $(\gamma_i-1) \ARpi^{(0, v]+} \subset \pi \ARpi^{(0, v]+}$ and from Definition \ref{defi:wach_reps} note that $(\gamma_i-1) \mbfn(T) \subset \pi \mbfn(T)$.
	As $\varphi$ commutes with the action of $\Gamma_S$, therefore, we conclude that $(\gamma_i-1) N \subset \varphi(\pi) N$.
\end{rem}

	From the discussion above, it follows that the complex in \eqref{eq:kercomplex} is isomorphic to the complex
	\begin{equation}\label{eq:phigamma_complex_psi0_oc}
		\tau_{\leqslant r} \bigoplus_{\alpha \neq 0}
		\left[
			\vcenter
			{
				\xymatrix
				{
					\Kos\big(\Gamma_S\prm, N(r)[X^{\flat}]^{\alpha}\big) \ar[r]^{\tau_0} & \Kos^c\big(\Gamma_S\prm, N(r)[X^{\flat}]^{\alpha}\big)
				}
			}
		\right].
	\end{equation}

\begin{lem}\label{lem:psi0_kernel_complex_oc}
	The complex described in \eqref{eq:phigamma_complex_psi0_oc} is $p^2\textrm{-acyclic}$.
\end{lem}
\begin{proof}
	Our proof is motivated by the proof of \cite[Lemma 4.10]{colmez-niziol-nearby-cycles}.
	One can treat the terms of \eqref{eq:phigamma_complex_psi0_oc} corresponding to each $\alpha$ separately.
	The case of $\alpha_k \neq 0$, for some $k \neq 0$, follows similar to the proof of \cite[Lemma 4.10]{colmez-niziol-nearby-cycles}, where one shows that both the complexes $\Kos\big(\Gamma_S\prm, N(r)[X^{\flat}]^{\alpha})$ and $\Kos^c\big(\Gamma_S\prm, N(r)[X^{\flat}]^{\alpha})$ are $p\textrm{-acyclic}$, by using the facts that $(\gamma_k-1)N \subset \varphi(\pi) N$ (see Remark \ref{rem:Npivp_psi0_gamma}) and $\pi$ divides $p$ in $\varphi(\ARpi^{(0, v]+})$ (since $\pi_1$ divides $p$ in $\ARpi^{(0, v]+}$, see Lemma \ref{lem:pi_m_divides_p} (ii) for $v = p-1$).
	We omit the details.

	Now, let $\alpha_k = 0$, for all $k \neq 0$, and $\alpha_0 \neq 0$.
	To prove that the complex in \eqref{eq:phigamma_complex_psi0_oc} is $p\textrm{-acyclic}$, we will show that $\tau_0 : \Kos \rightarrow \Kos^c$ is injective and the cokernel complex is killed by $p$.
	This amounts to showing the same statement for the following map:
	\begin{equation}\label{eq:gamma_0_midus_deltas}
		\gamma_0 - \delta_{i_1} \cdots \delta_{i_q} : N[X^{\flat}]^{\alpha}(r) \longrightarrow N[X^{\flat}]^{\alpha}(r), \hspace{2mm} \delta_{i_j} = \tfrac{\gamma_{i_j}^c-1}{\gamma_{i_j}-1}.
	\end{equation}
	Let $n = p^{-m}(c-1)\alpha_0 \in \ZZ_p^{\times}$, $F = c^{r}(1+\pi)^n\gamma_0 - \delta_{i_1} \cdots \delta_{i_q}$ and $\epsilon^{\otimes r}$ a $\ZZ_p\textrm{-basis}$ of $\ZZ_p(r)$.
	Then we note that $(\gamma_0 - \delta_{i_1} \cdots \delta_{i_q}) \big(x [X^{\flat}]^{\alpha} \otimes \epsilon^{\otimes r}\big) = F(x) \cdot [X^{\flat}]^{\alpha} \otimes \epsilon^{\otimes r}$, for any $x \in N$.
	Moreover, we have that $c^r-1$ is divisible by $p^m$, $(1+\pi)^n = 1 + n\pi \hspace{1mm} \textrm{mod} \hspace{1mm} \pi^2$ and $\delta_{i_j} - 1 \in (\gamma_{i_j} - 1)\ZZ_p \llbracket \gamma_{i_j}-1 \rrbracket$.
	Therefore, we can write $\pi^{-1} F$ in the form $\pi^{-1}F = n + \pi^{-1}F\prm$, with $F\prm \in \big(p^m, \pi^2, \gamma_0-1, \ldots, \gamma_d-1\big)\ZZ_p \llbracket \pi, \Gamma_S \rrbracket$.
	Now, let $f = p/\pi \in \varphi(\mbfa_R^{(0, v]+})$ and note that $\pi^{-1} p^m x = \pi^{m-1} f^m x$ is in $\pi^{m-1} N$.
	Moreover, we have that $(\gamma_j-1) N \subset \varphi(\pi)N$, for $0 \leqslant j \leqslant d$ (see Remark \ref{rem:Npivp_psi0_gamma}) and $\varphi(\pi)/\pi^2 \in \varphi(\ARpi^{(0, v]+})$ (since $\pi_1$ divides $p$ in $\ARpi^{(0, v]+}$, see Lemma \ref{lem:pi_m_divides_p} (ii) for $v = p-1$).
	Furthermore, $\pi_m^{p^m}$ divides $\pi$ and $p$ in $\varphi(\ARpi^{(0, v]+})$ (see Lemma \ref{lem:pi_m_divides_p} (ii) for $v = p-1$).
	So we get that $\pi^{-1}F'(x) \in \pi_m^{p^m} N$ (since we assumed $m \geqslant 2$).
	In particular, we see that $\pi^{-1}F' = 0$ on $\pi_m^a N / \pi_m^{a+b} N$, for all $a \in \NN$ and $b = p^m$.
	Hence, $\pi^{-1} F$ induces multiplication by $n$ on $\pi_m^a N / \pi_m^{a+b} N$, for all $a \in \NN$, which implies that it is an isomorphism on $N$.
	From the preceding discussion, we conclude that the map in \eqref{eq:gamma_0_midus_deltas} is injective and its image is contained in $\pi N[X^{\flat}]^{\alpha}(r)$.
	But, as $\pi$ divides $p$ in $\varphi(\ARpi^{(0, v]+})$, therefore, we get that the cokernel of \eqref{eq:gamma_0_midus_deltas} is killed by $p$, as claimed.
\end{proof}

	Using Lemma \ref{lem:psi0_kernel_complex_oc}, we conclude that the natural morphism of complexes, in the claim of Proposition \ref{prop:oc_psi_gamma_phi_gamma}, is a $p^{r+2}\textrm{-quasi-isomorphism}$.
\end{proof}

\subsubsection{Changing the overconvergence radius}

Recall that $m \geqslant 2$ and let $\ell = p^{m-1}$.
Then from Proposition \ref{prop:ring_psi_eigenspace_decomp} (i), we have inclusions $\psi\big(\pi_m^{-\ell} \ARpi^{(0, v]+}\big) \subset \psi\big(\pi_m^{-\ell} \ARpi^{(0, v/p]+}\big) \subset \pi_m^{-p^{m-2}} \ARpi^{(0, v]+} \subset \pi_m^{-\ell} \ARpi^{(0, v/p]+}$.
In other words, $\pi_m^{-\ell} \ARpi^{(0, v]+}$ is stable under $\psi$.
Set $\Dpi^{(0, v]+}(T(r)) := \ARpi^{(0, v]+} \otimes_{\AR^+} \mbfd^+(T(r))$ and note that it is stable under the action of $\Gamma_S$.
Next, from Lemma \ref{lem:analytic_rings_psi_action} we have that $\psi\big(\ARpi^{(0, v/p]+}\big) = \ARpi^{(0, v]+}$, and for $v = p-1$, using Lemma \ref{lem:pi_m_divides_p} (iii), we have that $\pi_m^{-p\ell} \pi$ is a unit in $\ARpi^{(0, v/p]+}$.
So from Proposition \ref{prop:ring_psi_eigenspace_decomp}, it follows that $\psi\big(\pi^{-r} \ARpi^{(0, v/p]+}\big) = \pi_1^{-r}\ARpi^{(0, v]+}$, and therefore, $\psi\big(\pi^{-r}\Dpi^{(0, v/p]+}(T(r))\big) \subset \pi_1^{-r}\Dpi^{(0, v]+}(T(r))$.
Moreover, as we have $\psi(\mbfn(T)) \subset \mbfd^+(T)$, so from the discussion above we see that $\psi\big(\Npi^{(0, v/p]+}(T(r))\big) \subset \psi\big(\pi^{-r}\Dpi^{(0, v/p]+}(T(r))\big) \subset \pi_1^{-r}\Dpi^{(0, v]+}(T(r))$.
Furthermore, for $k \in \NN$ and $k \leqslant r$, it follows that we have $\pi^k \Npi^{(0, v/p]+}(T(r)) \subset \pi^{k-r} \Dpi^{(0, v/p]+}(T(r))$ and $\psi\big(\pi^k \Npi^{(0, v/p]+}(T(r))\big) \subset \pi_1^{k-r} \Dpi^{(0, v]+}(T(r)) \subset \pi^{k-r} \Dpi^{(0, v/p]+}(T(r))$.

By replacing $v$ by $v/p$ in \S \ref{subsec:change_annulus_1}, we define a complex $\pazk\big(\Gamma_S\prm, \Npi^{(0, v/p]+}(T(r))\big)$ as follows:
\begin{equation*}
	\Npi^{(0, v/p]+}(T(r)) \xrightarrow{\hspace{1mm}(\tau_i)\hspace{1mm}} \big(\pi \Npi^{(0, v/p]+}(T(r))\big)^{I_1\prm} \longrightarrow \big(\pi^2 \Npi^{(0, v/p]+}(T(r))\big)^{I_2\prm} \longrightarrow \cdots.
\end{equation*}
Similarly, we define a complex $\pazk^c\big(\Gamma_S\prm, \Npi^{(0, v/p]+}(T(r))\big)$ and a map $\tau_0$ from the former to the latter complex.
Note that from the discussion above and the inclusion $\Npi^{(0, v/p]+}(T(r)) \subset \pi^{-r} \Dpi^{(0, v/p]+}(T(r))$, we have that $(\psi - 1)\big(\pi^k \Npi^{(0, v/p]+}(T(r))\big) \subset \pi^{-r}\Dpi^{(0, v/p]+}(T(r))$.
So we define the following complex:
\begin{displaymath}
	\pazk\big(\psi, \Gamma_S, \Npi^{(0, v/p]+}(T(r))\big) :=
	\left[
		\vcenter
		{
			\xymatrix
			{
				\pazk\big(\Gamma_S\prm, \Npi^{(0, v/p]+}(T(r))\big) \ar[r]^{\psi-1\hspace{5mm}} \ar[d]_{\tau_0} & \Kos\big(\Gamma_S\prm, \pi^{-r}\Dpi^{(0, v/p]+}(T(r))\big) \ar[d]^{\tau_0} \\
				\pazk^c\big(\Gamma_S\prm, \pi \Npi^{(0, v/p]+}(T(r))\big) \ar[r]^{\psi-1\hspace{4mm}} & \Kos^c\big(\Gamma_S\prm, \pi^{-r} \Dpi^{(0, v/p]+}(T(r))\big)
			}
		}
	\right].
\end{displaymath}

\begin{lem}\label{lem:psi_gamma_oc_radius_wach}
	The morphism of complexes $\tau_{\leqslant r} \pazk\big(\psi, \Gamma_S, \Npi^{(0, v]+}(T(r))\big) \rightarrow \tau_{\leqslant r} \pazk\big(\psi, \Gamma_S, \Npi^{(0, v/p]+}(T(r))\big)$, induced by the inclusions $\Npi^{(0, v]+}(T(r)) \subset \Npi^{(0, v/p]+}(T(r))$ and $\psi(\Npi^{(0, v/p]+}(T(r))) \subset \pi^{-r} \Dpi^{(0, v/p]+}(T(r))$, is a $p^{r+2s}\textrm{-quasi-isomorphism}$.
\end{lem}
\begin{proof}
	As the map in the claim is injective on each term, we need to show that the cokernel complex is killed by $p^{r+2s}$.
	For $k \in \NN$ and $k \leqslant r$, in the cokernel complex, we have maps
	\begin{equation}\label{eq:psi_minus_id_oc}
		\psi - 1 : \pi^{k-r} \Npi^{(0, v/p]+}(T) / \pi^{k-r} \Fil^{r-k} \Npi^{(0, v]+}(T) \rightarrow \pi^{-r}\Dpi^{(0, v/p]+}(T) / \psi(\pi^{k-r} \Npi^{(0, v/p]+}(T)),
	\end{equation}
	and to prove the claim it is enough to show that \eqref{eq:psi_minus_id_oc} is $p^{r+2s}\textrm{-bijective}$ (the twist $(r)$ has disappeared because $\psi$ acts trivially on it).
	First, we will show the $p^{r+s}\textrm{-surjectivity}$.
	Recall that we have $\pi^s \mbfd^+(T) \subset \mbfn(T) \subset \mbfd^+(T)$ (see \cite[Corollary 4.11]{abhinandan-crystalline-wach}), and by extending scalars to $\ARpi^{(0, v/p]+}$ and dividing out by $\pi^r$, we see that $\pi^{s-r} \Dpi^{(0, v/p]+}(T) \subset \pi^{-r} \Npi^{(0, v/p]+}(T)$.
	So, it follows that $\pi^{-r}\Dpi^{(0, v/p]+}(T) / \pi^{k-r} \Npi^{(0, v/p]+}(T)$ is killed by $\pi^{k+s}$, and since $\pi$ divides $p$ in $\ARpi^{(0, v/p]+}$ (see Lemma \ref{lem:pi_m_divides_p} for $v = p-1$), therefore, we get that the preceding quotient is killed by $p^{k+s}$.
	Note that the quotient $\pi^{-r}\Dpi^{(0, v/p]+}(T) / \pi^{k-r} \Npi^{(0, v/p]+}(T)$ surjects onto the cokernel of \eqref{eq:psi_minus_id_oc}.
	Hence, for $k \leqslant r$, we see that the cokernel of \eqref{eq:psi_minus_id_oc} is killed by $p^{r+s}$ (this also shows that the truncation in degree $\leqslant r$ is necessary in order to bound the power of $p$).

	Next, to show the $p^s\textrm{-injectivity}$ of \eqref{eq:psi_minus_id_oc}, let $x \in \Npi^{(0, v/p]+}(T)$ such that there is a $y \in \Npi^{(0, v/p]+}(T)$ satisfying $(\psi-1)(\pi^{k-r} x) = \psi(\pi^{k-r} y)$, or equivalently, we have that $x = \xi^{r-k} \psi(x-y)$ belongs to $\xi^{r-k} \psi(\Npi^{(0, v/p]+}(T))$.
	Note that $\psi(\Npi^{(0, v/p]+}(T)) \subset \psi(\Dpi^{(0, v/p]+}(T)) \subset \Dpi^{(0, v]+}$, so we see that $\varphi(x) \in \Dpi^{(0, v/p]+}$.
	Moreover, from the discusson above, we know that the natural inclusion $\Npi^{(0, v/p]+}(T) \subset \Dpi^{(0, v/p]+}(T)$ is $p^s\textrm{-surjective}$.
	Therefore, it follows that $\varphi(p^s x) = p^s\varphi(x)$ is in $\Npi^{(0, v/p]+}(T)$, in particular, we see that $\psi(\varphi(p^s x)) = \psi(p^s q^{r-k}(x - y))$, i.e.\ $\varphi(p^s x) - q^{r-k} p^s (x - y)$ is in $(\Npi^{(0, v/p]+}(T))^{\psi=0}$.
	From the description of $(\Npi^{(0, v/p]+}(T))^{\psi=0}$ before Lemma \ref{lem:Npivp_psi0}, we can write $\varphi(p^s x) =  p^s q^{r-k}(x - y) + \sum_{\alpha \neq 0} \varphi(x_{\alpha}) [X^{\flat}]^{\alpha}$, for some $x_{\alpha} \in \Npi^{(0, v]+}(T)$.
	In particular, we see that $\varphi(p^s x)$ is in $\Npi^{(0, v/p]+}(T)$ and from Lemma \ref{lem:Npivp_psi0} we get that $p^s x$ is in $\Npi^{(0, v]+}(T)$.
	Furthermore, as we have that $\psi(\Npi^{(0, v/p]+}(T)) \subset \Dpi^{(0, v]+}(T)$, therefore, we see that $p^s x$ is in $\Npi^{(0, v]+}(T) \cap \xi^{r-k} \Dpi^{(0, v]+}(T) \subset \Npi^{(0, v]+}(T) \cap \big(\Fil^{r-k} \ARbar^{(0, v]+} \otimes_{\ZZ_p} V\big) \subset \Fil^{r-k} \Npi^{(0, v]+}(T)$, where the last inclusion follows from the definition of the filtration on $\Npi^{(0, v]+}(T)$ in \eqref{eq:fil_ns}.
	In particular, we have shown that $p^s \pi^{k-r}x$ belongs to $\pi^{k-r} \Fil^{k-r} \Npi^{(0, v]+}(T)$, and hence, \eqref{eq:psi_minus_id_oc} is $p^s\textrm{-injective}$.
	This allows us to conclude.
\end{proof}

From the discussion before Lemma \ref{lem:psi_gamma_oc_radius_wach}, recall that we have inclusions $\psi\big(\pi^{-r} \Dpi^{(0, v/p]+}(T(r))\big) \subset \pi_1^{-r} \Dpi^{(0, v]+}(T(r)) \subset \pi^{-r} \Dpi^{(0, v/p]+}(T(r))$.
So using the constuctions in \S \ref{sec:galois_cohomology}, we define the complex:
\begin{displaymath}
	\Kos\big(\psi, \Gamma_S, \Dpi^{(0, v/p]+}(T(r))\big) :=
	\left[
		\vcenter
		{
			\xymatrix
			{
				\Kos\big(\Gamma_S\prm, \pi^{-r}\Dpi^{(0, v/p]+}(T(r))\big) \ar[r]^{\psi-1\hspace{2mm}} \ar[d]_{\tau_0} & \Kos\big(\Gamma_S\prm, \pi^{-r}\Dpi^{(0, v/p]+}(T(r))\big) \ar[d]^{\tau_0} \\
				\Kos^c\big(\Gamma_S\prm, \pi^{-r}\Dpi^{(0, v/p]+}(T(r))\big) \ar[r]^{\psi-1\hspace{2mm}} & \Kos^c\big(\Gamma_S\prm, \pi^{-r}\Dpi^{(0, v/p]+}(T(r))\big)
			}
		}
	\right].
\end{displaymath}

\begin{lem}\label{lem:psi_gamma_oc_radius}
	The morphism of complexes $\tau_{\leqslant r} \pazk\big(\psi, \Gamma_S, \Npi^{(0, v/p]+}(T(r))\big) \rightarrow \tau_{\leqslant r} \Kos\big(\psi, \Gamma_S, \Dpi^{(0, v/p]+}(T(r))\big)$, induced by the inclusion $\Npi^{(0, v/p]+}(T(r))\subset \pi^{-r} \Dpi^{(0, v/p]+}(T(r))$, is a $p^{r+s}\textrm{-quasi-isomorphism}$.
\end{lem}
\begin{proof}
	Note that for the map of truncated complexes, the cokernel complex consists of $\ARpi^{(0, v/p]+}\textrm{-modules}$, given as $\pi^{-r} \Dpi^{(0, v/p]+}(T(r)) / \pi^k \Npi^{(0, v/p]+}(T(r))$, for $k \leqslant r$.
	Recall that we have $\pi^s \mbfd^+(T) \subset \mbfn(T) \subset \mbfd^+(T)$ (see \cite[Corollary 4.11]{abhinandan-crystalline-wach}), and by extending scalars to $\ARpi^{(0, v/p]+}$, dividing out by $\pi^r$ and twisting by $\ZZ_p(r)$, we see that $\pi^{s-r} \Dpi^{(0, v/p]+}(T(r)) \subset \Npi^{(0, v/p]+}(T(r))$.
	So, it follows that the quotient $\pi^{-r}\Dpi^{(0, v/p]+}(T(r)) / \pi^k \Npi^{(0, v/p]+}(T(r))$ is killed by $\pi^{k+s}$, and since $\pi$ divides $p$ in $\ARpi^{(0, v/p]+}$ (see Lemma \ref{lem:pi_m_divides_p} for $v = p-1$), therefore, we get that the preceding quotient is killed by $p^{k+s}$.
	As $k \leqslant r$, hence, we conclude that the cokernel complex is $p^{r+s}\textrm{-acyclic}$.
\end{proof}

\subsection{Change of the disk of convergence}\label{subsec:change_disk}

In this subsection, we will relate complexes in previous subsections to the Koszul complex computing continuous $G_S\textrm{-cohomology}$ of $T(r)$.
Recall that in \S \ref{sec:operator_psi_phi_gamma}, we defined an operator $\psi : \Dpi(T(r)) \rightarrow \Dpi(T(r))$ as a left inverse of $\varphi$.
Using this operator, we define the following complex:
\begin{displaymath}
	\Kos\big(\psi, \Gamma_S, \Dpi(T(r))\big) :=
	\left[
		\vcenter
		{
			\xymatrix
			{
				\Kos\big(\Gamma_S\prm, \Dpi(T(r))\big) \ar[r]^{\psi-1\hspace{1mm}} \ar[d]_{\tau_0} & \Kos\big(\Gamma_S\prm, \Dpi(T(r))\big) \ar[d]^{\tau_0} \\
				\Kos^c\big(\Gamma_S\prm, \Dpi(T(r))\big) \ar[r]^{\psi-1\hspace{1mm}} & \Kos^c\big(\Gamma_S\prm, \Dpi(T(r))\big)
			}
		}
	\right].
\end{displaymath}

\begin{lem}\label{lem:psi_gamma_oc}
	The natural morphism of complexes $\Kos\big(\psi, \Gamma_S, \Dpi^{(0, v/p]+}(T(r))\big) \rightarrow \Kos\big(\psi, \Gamma_S, \Dpi(T(r))\big)$, induced by the inclusion $\pi^{-r} \Dpi^{(0, v/p]+}(T(r)) \subset \Dpi(T(r))$, is a quasi-isomorphism.
\end{lem}
\begin{proof}
	The map in the claim is injective on each term, so we examine the cokernel complex.
	Write $\Dpi(T(r)) = \Dpi^{(0, v/p]+}(T(r))[1/\pi_m]^{\wedge}$, where ${ }^{\wedge}$ denotes the $p\textrm{-adic}$ completion.
	By Lemma \ref{lem:analytic_rings_psi_action}, we have that $\psi\big(\ARpi^{(0, v/p]+}\big) = \ARpi^{(0, v]+} \subset \ARpi^{(0, v/p]+}$, and for $\ell = p^{m-1}$ note that by Lemma \ref{lem:pi_m_divides_p} (iii), we have that $\pi_m^{-p\ell} \pi$ is a unit in $\ARpi^{(0, v/p]+}$.
	So, for $k \geqslant 1$, we get that $\psi\big(\pi_m^{-p^k\ell r}\ARpi^{(0, v/p]+}\big) \subset \pi_m^{-p^{k-1}\ell r}\ARpi^{(0, v/p]+}$ (see Proposition \ref{prop:ring_psi_eigenspace_decomp}).
	Moreover, recall that we have $\psi\big(\Dpi^{(0, v/p]+}(T(r))\big) \subset \Dpi^{(0, v/p]+}(T(r))$.
	Coupling this with the observation above, we get that $\psi\big(\pi_m^{-p^k\ell r}\Dpi^{(0, v/p]+}(T(r))\big) \subset \pi_m^{-p^{k-1}\ell r}\Dpi^{(0, v/p]+}(T(r))$.
	Therefore, it follows that the natural map
	\begin{equation*}
		\psi : \Dpi(T(r)) / \pi^{-r}\Dpi^{(0, v/p]+}(T(r)) \longrightarrow \Dpi(T(r)) / \pi^{-r}\Dpi^{(0, v/p]+}(T(r)),
	\end{equation*}
	is (pointwise) topologically nilpotent and $1-\psi$ is bijective over this quotient.
	So, we obtain that the following complexes are acyclic:
	\begin{align*}
		\big[\Kos\big(\Gamma_S\prm, \Dpi(T(r)) / \pi^{-r}\Dpi^{(0, v/p]+}(T(r))\big) &\xrightarrow{\hspace{1mm} \psi-1 \hspace{1mm}} \Kos\big(\Gamma_S\prm, \Dpi(T(r)) / \pi^{-r}\Dpi^{(0, v/p]+}(T(r))\big)\big],\\
		\big[\Kos^c\big(\Gamma_S\prm, \Dpi(T(r)) / \pi^{-r}\Dpi^{(0, v/p]+}(T(r))\big) &\xrightarrow{\hspace{1mm} \psi-1 \hspace{1mm}} \Kos^c\big(\Gamma_S\prm, \Dpi(T(r)) / \pi^{-r}\Dpi^{(0, v/p]+}(T(r))\big)\big].
	\end{align*}
	Hence, we conclude that the cokernel complex of the map in the claim is acyclic.
\end{proof}

Recall that we have the complex $\Kos\big(\varphi, \Gamma_S, \Dpi(T(r))\big)$ from Definition \ref{defi:koszul_phigammaD} and we make the following claim:
\begin{prop}\label{prop:phi_gamma_psi_gamma_iso}
	The natural morphism of complexes $\Kos\big(\varphi, \Gamma_S, \Dpi(T(r))\big) \longrightarrow \Kos\big(\psi, \Gamma_S, \Dpi(T(r))\big)$, induced by the identity on the first column and $\psi$ on the second column, is a quasi-isomorphism.
\end{prop}
\begin{proof}
	Notice that the map $\psi$ is surjective on $\Dpi(T(r))$, so the cokernel complex is 0.	
	To obtain the acylicity of the kernel complex, we need to show that the complex $\big[ \Kos\big(\Gamma_S\prm, \Dpi(T(r))^{\psi=0}\big) \xrightarrow{\hspace{1mm} \tau_0 \hspace{1mm}} \Kos\big(\Gamma_S\prm, \Dpi(T(r))^{\psi=0}\big) \big]$ is acyclic.
	To show our claim, we will analyze the module $\Dpi(T(r))^{\psi=0}$.
	Let $\{e_1, \ldots, e_h\}$ denote an $\AR^+\textrm{-basis}$ $\mbfn(T)$ and set $f_i = e_i \otimes \epsilon^{\otimes r}$, for each $1 \leqslant i \leqslant h$ and where $\epsilon^{\otimes r}$ is a $\ZZ_p\textrm{-basis}$ of $\ZZ_p(r)$.
	Since we have that $\AR \otimes_{\AR^+} \mbfn(T)(r) \isomorphic \mbfd(T)(r) = \mbfd(T(r))$, therefore, it follows that $\{f_1, \ldots, f_h\}$ is an $\AR\textrm{-basis}$ of $\mbfd(T(r))$.
	Furthermore, as $\Dpi(T(r)) = \ARpi \otimes_{\AR} \mbfd(T(r))$ is an \'etale $(\varphi, \Gamma_R)\textrm{-module}$ over $\ARpi$, so we see that $\{\varphi(f_1), \ldots, \varphi(f_h)\}$ is an $\ARpi\textrm{-basis}$ of $\Dpi(T(r))$.
	In this basis, we have that $z = \sum_{j=1}^h z_j \varphi(f_j)$ is in $\Dpi(T(r))^{\psi=0}$ if and only if $z_j$ is in $\ARpi^{\psi=0}$, for each $1 \leqslant j \leqslant h$.
	Indeed, $\psi(z) = 0$ if and only if $\sum_{j=1}^h \psi(z_j \varphi(f_j)) = \sum_{j=1}^h \psi(z_j) f_j = 0$, and since $f_j$ are linearly independent over $\ARpi$, therefore, we see that $\psi(z) = 0$ if and only if $\psi(z_j) = 0$, for all $1 \leqslant j \leqslant h$.

	Next, from Proposition \ref{prop:ring_psi_eigenspace_decomp}, we have a decomposition $\ARpi^{\psi=0} = \oplus_{\alpha} \varphi\big(\ARpi\big)[X^{\flat}]^{\alpha}$, where $[X^{\flat}]^{\alpha} = (1+\pi_m)^{\alpha_0}[X_1^{\flat}]^{\alpha_0} \cdots [X_d^{\flat}]^{\alpha_d}$ and $\alpha =(\alpha_0, \ldots, \alpha_d)$ is a $(d+1)\textrm{-tuple}$ with $\alpha_i \in \{0, \ldots, p-1\}$.
	Therefore, we get that $\big(\Dpi(T(r))\big)^{\psi=0} = \big(\sum_{i=1}^h \ARpi f_j\big)^{\psi=0} = \oplus_{\alpha \neq 0} \sum_{i=1}^h \varphi\big(\ARpi f_j\big)[X^{\flat}]^{\alpha}$.
	Note that the last term identifies with $\oplus_{\alpha \neq 0} \sum_{i=1}^h \varphi(\Dpi(T))(r)[X^{\flat}]^{\alpha}$.
	So, we obtain that the kernel complex of the map in the claim is isomorphic to the following complex:
	\begin{align}\label{eq:phigamma_complex_psi0}
		\bigoplus_{\alpha \neq 0}
		\left[
			\vcenter
			{
				\xymatrix
				{
					\Kos\big(\Gamma_S\prm, \varphi\big(\Dpi(T)\big)(r)[X^{\flat}]^{\alpha}\big) \ar[r]^{\tau_0} & \Kos^c\big(\Gamma_S\prm, \varphi\big(\Dpi(T)\big)(r)[X^{\flat}]^{\alpha}\big)
				}
			}
		\right].
	\end{align}
\begin{lem}\label{lem:psi0_kernel_complex}
	The complex described in \eqref{eq:phigamma_complex_psi0} is acyclic.
\end{lem}
\begin{proof}
	The proof follows in a manner similar to Lemma \ref{lem:psi0_kernel_complex_oc}, where one notes that it is enough to show the claim modulo $p$, and for the latter, one uses the fact that $\Dpi(T)/p = (\Npi^+(T)/p)[1/\pi_m]$, for $\Npi^+(T) = \ARpi^+ \otimes_{\AR^+} \mbfn(T)$.
	We omit the details to avoid repetition.
\end{proof}

	Using Lemma \ref{lem:psi0_kernel_complex}, we conclude that the natural morphism of complexes, in the claim of Proposition \ref{prop:phi_gamma_psi_gamma_iso}, is a quasi-isomorphism.
\end{proof}

\begin{proof}[Proof of Proposition \ref{prop:differential_koszul_complex_galois_cohomology}]
	Recall that $s$ is the height of the representation $T$ and $r$ is the twist (see Assumption \ref{assum:relative_crystalline_wach_free}).
	Note that from Proposition \ref{prop:diff_to_lie}, we have a natural $p^{4r}\textrm{-quasi-isomorphism}$ of complexes $\Kos\big(\varphi, \partial_A, \Fil^r \Npi^{[u, v]}(T)\big) \simeq \pazk\big(\varphi, \Lie \Gamma_S, \Npi^{[u, v]}(T(r))\big)$.
	Then, in Proposition \ref{prop:quasi_iso_Lie_relative}, we replace the infinitesimal action of $\Gamma_S$ with the continuous action of $\Gamma_S$ and obtain a natural isomorphism of complexes $\pazk\big(\varphi, \Lie \Gamma_S, \Npi^{[u, v]}(T(r))\big) \simeq \pazk\big(\varphi, \Gamma_S, \Npi^{[u, v]}(T(r))\big)$.
	Furthermore, in Proposition \ref{prop:quasi_iso_oc_robba_relative}, we switch from analytic coefficients rings to overconvergent coefficient rings to obtain a natural $p^{3r}\textrm{-quasi-isomorphism}$ of complexes $\pazk\big(\varphi, \Gamma_S, \Npi^{[u, v]}(T(r))\big) \simeq \pazk\big(\varphi, \Gamma_S, \Npi^{(0, v]+}(T(r))\big)$.
	Next, in Proposition \ref{prop:oc_psi_gamma_phi_gamma} and Lemma \ref{lem:psi_gamma_oc_radius_wach} and Lemma \ref{lem:psi_gamma_oc_radius}, we change the overconvergence radius to obtain a $p^{3r+3s+2}\textrm{-quasi-isomorphism}$ of complexes $\tau_{\leqslant r} \pazk\big(\varphi, \Gamma_S, \Npi^{(0, v]+}(T(r))\big) \simeq \tau_{\leqslant r} \Kos\big(\psi, \Gamma_S, \Dpi^{(0, v/p]+}(T(r))\big)$, where $\tau_{\leqslant}$ denotes the canonical truncation.
	Finally, in Lemma \ref{lem:psi_gamma_oc} and Proposition \ref{prop:phi_gamma_psi_gamma_iso} we change the disk of convergence to obtain natural quasi-isomorphisms of complexes $\Kos\big(\psi, \Gamma_S, \Dpi^{(0, v/p]+}(T(r))\big) \simeq \Kos\big(\psi, \Gamma_S, \Dpi(T(r))\big) \simeq \Kos\big(\varphi, \Gamma_S, \Dpi(T(r))\big)$.
	Combining these statements, we get the claim of Proposition \ref{prop:differential_koszul_complex_galois_cohomology} with $N = 10r+3s+2$.
\end{proof}

\subsection{Comparison with the Fontaine--Messing period map}\label{subsec:fm_maps_comparison}

The aim of this subsection is to show that the comparison map from $\Syn(S, M, r)$ to $\RGamma_{\cont}(G_S, (T(r)))$, in Theorem \ref{thm:syntomic_complex_galois_cohomology}, coincides with the Fontaine--Messing period map.
We will follow the strategy of Colmez--Niziol{\l} (see \cite[\S 4.7]{colmez-niziol-nearby-cycles}).
Recall that we have $S = R[\varpi]$, $\Sbar = \Rbar \subset \overline{\Fr R}$ and $S_{\infty} = R_{\infty} \subset \overline{\Fr R}$.
Note that by Definition \ref{defi:obese_rings}, we have rings $\fESbar^{\bmstar} := \fERbar^{\bmstar}$, for $\smstar \in \{\textpd, [u], [u, v]\}$, equipped with a Frobenius, a filtration and an action of $G_S \triangleleft G_R$.

Let us recall that $T$ is a positive finite $q\textrm{-height}$ $\ZZ_p\textrm{-representation}$ of $G_R$ as in Assumption \ref{assum:relative_crystalline_wach_free} and $V = T[1/p]$.
Note that by tensoring the fundamental exact sequence in \eqref{eq:fes_acrys} with $T$, we get the following $p^r\textrm{-exact}$ sequence,
\begin{equation}\label{rem:fes_acrys_rep}
	0 \longrightarrow T(r)\prm \longrightarrow \Fil^r \Acrys(\Sbar) \otimes_{\ZZ_p} T \xrightarrow{\hspace{1mm} p^r-\varphi \hspace{1mm}} \Acrys(\Sbar) \otimes_{\ZZ_p} T \longrightarrow 0.
\end{equation}
Next, from Assumption \ref{assum:relative_crystalline_wach_free} we have a finite free $R\textrm{-module}$ $M \subset \ODcrys(V)$ such that $M[1/p] = \ODcrys(V)$.
Moreover, we have a natural injective map $\OARpi^{\textpd} \otimes_R M \rightarrow \OARpi^{\textpd} \otimes_{\AR^+} \mbfn(T)$, compatible with the respective Frobenii, filtrations, $\ARpi^{\textpd}\textrm{-linear}$ connections and actions of $\Gamma_R$.
Additionally, by definition, we have a natural inclusion $\mbfa^+ \otimes_{\AR^+} \mbfn(T) \subset \mbfa^+ \otimes_{\ZZ_p} T$, compatible with the respective Frobenii and actions of $G_R$.
Extending scalars to $\OAcrys(\Sbar)$ in both the maps and composing them, we obtain the top horizontal arrow in the following diagram:
\begin{equation}\label{eq:oacris_M_T}
	\begin{tikzcd}
		\OAcrys(\Sbar) \otimes_R M \arrow[r] \arrow[d] & \OAcrys(\Sbar) \otimes_{\ZZ_p} T \arrow[d] \\
		\OBcrys(\Sbar) \otimes_R \ODcrys(V) \arrow[r, "\sim"] & \OBcrys(\Sbar) \otimes_{\QQ_p} V,
	\end{tikzcd}
\end{equation}
where the vertical arrows are natural inclusions and the lower horizontal arrow is a natural isomorphism (since $V$ is crystalline), compatible with the respective Frobenii, filtrations, actions of $G_R$ and $\Bcrys(\Sbar)\textrm{-linear}$ connections satisfying Griffiths transversality with respect to the filtrations (see \cite[Proposition 8.4.3]{brinon-padicrep-relatif}).
The diagram commutes by definition (see \cite[\S 4.5]{abhinandan-crystalline-wach} for a similar diagram), and it follows that the top horizontal arrow is injective.
Now, recall that the filtration on the bottom left object is given by the tensor product filtration (see Lemma \ref{lem:fil_gr_ds} and Remark \ref{rem:fil_gr_ds_rat}) and the filtration on the bottom right object is induced by the natural filtration on $\OBcrys(\Sbar)$.
As the filtration on the objects in the top row are induced from the filtration on the objects in the bottom row of their respective columns (see the discussion before Lemma \ref{lem:filr_induced} for the top left corner), therefore, it follows that the filtration on $\OAcrys(\Sbar) \otimes_R M$ matches with the indued filtration from $\OAcrys(\Sbar) \otimes_{\ZZ_p} T$.

Now, we consider the following commutative diagram:
\begin{center}
	\begin{tikzcd}[row sep=tiny]
		& \fESbarn^{\textpd} \arrow[rrd]\\
		\Acrys(\Sbar)_n \otimes_{O_{F, n}} R_{\varpi, \hspace{0.1mm} n}^+ \arrow[ru] \arrow[rrr] & & & \Sbar_n\\
		& \Rpin^{\PD} \arrow[rrd] \arrow[uu]\\
		\Rpin^+ \arrow[ru] \arrow[rrr] \arrow[uu] & & & S_n \arrow[uu],
	\end{tikzcd}
\end{center}
where the subscript $n$ denotes the reduction modulo $p^n$, the bottom horizontal arrow is induced by $X_0 \mapsto \varpi$ and the top horizontal arrow is the extension of the $\theta\textrm{-map}$ by the bottom horizontal arrow.

Using the rings discussed above, we will define the local Fontaine--Messing period map.
Set $\Omega_{\fESbarn^{\textpd}} := \fESbarn^{\textpd} \otimes_{\Rpin^+} \Omega_{\Rpin^+}$, $\Delta^{\textpd} := \fESbar^{\textpd} \otimes_R M$ and $\Delta^{\textpd}_n = \Delta^{\textpd}/p^n$ equipped with the induced filtration, Frobenius, $G_S\textrm{-action}$ and $\Acrys(\Sbar)_n\textrm{-linear}$ integrable connection $\partial$ satisfying Griffiths transversailty with respect to the filtration.
In particular, for $r \in \ZZ$, we have the following filtered de Rham complex,
\begin{equation*}
	\Fil^r \cald_{\Sbar, M, n}^{\bullet} : \Fil^r \Delta^{\textpd}_n \rightarrow \Fil^{r-1} \Delta^{\textpd}_n \otimes_{\Rpin^+} \Omega^1_{\Rpin^+} \rightarrow \Fil^{r-2} \Delta^{\textpd}_n \otimes_{\Rpin^+} \Omega^2_{\Rpin^+} \rightarrow \cdots.
\end{equation*}
Let us note that by extending the diagram \eqref{eq:oacris_M_T} along the natural inclusion $\OAcrys(\Sbar) \subset \fESbar^{\textpd}$ (see Remark \ref{rem:oarpd_erpd}), we obtain an $\fESbar^{\textpd}\textrm{-linear}$ injective map $\fESbar^{\textpd} \otimes_R M \rightarrow \fESbar^{\textpd} \otimes_{\ZZ_p} T$ compatible with the respective Frobenii, filtrations, $\Acrys(\Sbar)\textrm{-linear}$ connections and actions of $G_R$.
Then, for each $r \in \ZZ$, by reducing the induced map on the $r\textrm{-th}$ filtered part, modulo $p^n$, and taking horizontal sections for the $\Acrys(\Sbar)_n\textrm{-linear}$ connections, we obtain a natural map,
\begin{equation}\label{eq:pdpartil0_acrist}
	(\Fil^r \Delta^{\textpd}_n)^{\partial=0} = (\Fil^r(\fESbarn^{\textpd} \otimes_R M))^{\partial=0} \longrightarrow (\Fil^r \fESbarn^{\textpd} \otimes_{\ZZ_p} T)^{\partial=0} = \Fil^r \Acrys(\Sbar)_n \otimes_{\ZZ_p} T.
\end{equation}
In particular, from the discussion above and the filtered Poincar\'e Lemma \ref{lem:fil_poincare_lem_na}, we get a natural map,
\begin{equation}\label{eq:poincare_lemma_acrys_rep}
	\Fil^r \cald_{\Sbar, M, n}^{\bullet} \lisomorphic (\Fil^r \Delta^{\textpd}_n)^{\partial=0} \longrightarrow \Fil^r \mbfa_{\crys}(\Sbar)_n \otimes_{\ZZ_p} T.
\end{equation}

\begin{nota}
	For a $G_S\textrm{-module}$ $D$, let $C(G_S, D)$ denote the complex of continuous cochains of $G_S$ with values in $D$.
\end{nota}

\begin{defi}\label{defi:fontaine_messing_map}
	Define the syntomic complex with coefficients in $M$ as,
	\begin{equation}\label{eq:bar_syntomic_complex}
		\Syn(\Sbar, M, r)_n := \big[\Fil^r \cald_{\Sbar, M, n}^{\bullet} \xrightarrow{p^r-p^{\bullet}\varphi} \cald_{\Sbar, M, n}^{\bullet}\big].
	\end{equation}
	Define the Fontaine--Messing period map,
	\begin{equation}\label{eq:fm_period_map}
		\tilde{\alpha}^{\FM}_{r, n, S}: \Syn(S, M, r)_n \longrightarrow C(G_S, T/p^n(r)\prm),
	\end{equation}
	as the composition
	\begin{align*}
		\Syn(S, M, r)_n &= \big[\Fil^r \cald_{S, M, n}^{\bullet} \xrightarrow{p^r-p^{\bullet}\varphi} \cald_{S, M, n}^{\bullet}\big] \longrightarrow C\big(G_S, \big[\Fil^r \cald_{\Sbar, M, n}^{\bullet} \xrightarrow{p^r-p^{\bullet}\varphi} \cald_{\Sbar, M, n}^{\bullet}\big]\big) \longrightarrow\\
		& \longrightarrow C\big(G_S, \big[\Fil^r \Acrys(\Sbar)_n \otimes T \xrightarrow{p^r - \varphi} \Acrys(\Sbar)_n \otimes T \big]\big) \lisomorphic C\big(G_S, T/p^n(r)\prm),
	\end{align*}
	where the second right arrow is induced by \eqref{eq:poincare_lemma_acrys_rep} and the only left arrow is a $p^r\textrm{-quasi-isomorphism}$ as noted in \eqref{rem:fes_acrys_rep}.
\end{defi}

\begin{rem}\label{rem:fm_period_map_R}
	The definition of the Fontaine--Messing period map in \eqref{eq:fm_period_map} can also be given for $R$: 
	we use the ring $\OAcrys(\Rbar)$ instead of $\fESbar^{\textpd}$ and set $\Delta^{\textpd} = \OAcrys(\Rbar) \otimes_R M$.
	Then the map in \eqref{eq:poincare_lemma_acrys_rep} gets replaced by $\Fil^r \cald^{\bullet}_{\Rbar, M, n} \isomorphic \Fil^r \Acrys(\Rbar)_n \otimes T$ (where the filtered de Rham complex is obtained similar to modulo $p^n$ version of the complex $\Fil^r \cald^{\bullet}_{R, M}$ in \eqref{eq:fildeRham_R}).
	The definition of $\Syn(\Rbar, M, r)_n$ follows naturally and since the fundamental exact sequence is $G_R\textrm{-equivariant}$, we obtain the Fontaine--Messing period map,
	\begin{equation*}
		\tilde{\alpha}^{\FM}_{r, n, R}: \Syn(R, M, r)_n \longrightarrow C(G_R, T/p^n(r)\prm).
	\end{equation*}
\end{rem}

\begin{thm}\label{thm:lazard_fmlocal_comparison}
	The map $\tilde{\alpha}_{r, n, S}^{\FM}$ in \eqref{eq:fm_period_map} is $p^{N(T, e, r)}\textrm{-equal}$ to $\alpha_{r, n}^{\Laz}$ from Theorem \ref{thm:syntomic_complex_galois_cohomology}.
\end{thm}
\begin{proof}
	The $p\textrm{-power}$ equality of the two maps follows from the diagram below (where we only show the $\padic$ version to simplify notations).
	The objects and morphisms are described after the diagram.
	Note that we have $\Kdphi(\Ftor \Mpi^{\textpd}) = \Syn(S, M, r)$ and the map $\tilde{\alpha}_{r, S}^{\FM}$ is obtained by composing the arrows in the top row (note that $C_G(T(r))$ is $p^r\textrm{-isomorphic}$ to $C_G(T(r)\prm))$.
	Furthermore, the map $\alpha_r^{\Laz}$ is obtained by composing the maps in the outer left vertical, bottom horizontal and right vertical boundary.
	The isomorphisms in the diagram indicate a $p\textrm{-power}$ quasi-isomorphism between complexes.
	Finally, a simple diagram chase gives us the claim.
\end{proof}

\begin{center}
	\begin{tikzcd}[column sep=small]
		\Kdphi(\Ftor \Mpi^{\textpd}) \arrow[r] \arrow[d, "\tau_{\leqslant r}", "\wr"'] & C_G(\Kdphi(\Ftor \Delta^{\textpd})) \arrow[d] & C_G(\Kphi(\Ftor \Delta^{\textpd, \partial})) \arrow[l, "\sim", "\textup{PL}"'] \arrow[d] \arrow[r] & C_G(\Kphi(\Ftor TA_{\crys}))\\
		\Kdphi(\Ftor \Mpi^{[u, v]}) \arrow[d, "\wr"', "\textup{PL}"] \arrow[r] & C_G(\Kdphi(\Ftor \Delta^{[u, v]})) & C_G(\Kphi(\Ftor \Delta^{[u, v], \partial})) \arrow[l, "\sim", "\textup{PL}"'] \arrow[d] & C_G(T(r)) \arrow[ld, bend right=10, "\sim"', "\textup{FES}"] \arrow[u, "\wr", "\textup{FES}"']\\
		\Kdphid(\Ftor \Delta_{\varpi}^{[u, v]}) \arrow[ru, bend right=20] & & C_G(\Kphi(\Ftor TA^{[u, v]})) & C_G(\Kphi(TA_{\Sbar}(r))) \arrow[u, "\wr", "\textup{AS}"']\\
		\Kphid(\Ftor \Npi^{[u, v]}) \arrow[u, "\wr", "\textup{PL}"'] \arrow[d, "t^{\bullet}", "\tau_{\leqslant r} \hspace{0.2mm} \wr"'] & & & C_{\Gamma}(\Kphi(D_{R_{\infty}}(r))) \arrow[u, "\wr"]\\
		\pazk_{\varphi, \Lie \Gamma}(\Ftor \Npi^{[u, v]}) & \pazk_{\varphi, \Gamma}(\Ftor \Npi^{[u, v]}) \arrow[l, "\sim"', "\Laz"] \arrow[uur, bend right=40] & & C_{\Gamma}(\Kphi(\Dpi(r))) \arrow[u, "\wr"]\\
		\pazk_{\varphi, \Lie \Gamma}(\Npi^{[u, v]}(r)) \arrow[u, "\wr", "t^r"'] & \pazk_{\varphi, \Gamma}(\Npi^{[u, v]}(r)) \arrow[l, "\sim"', "\Laz"] \arrow[u, "t^r"] & \pazk_{\varphi, \Gamma}(\Npi^{(0, v]+}(r)) \arrow[l, "\sim"', "\textup{can}"] \arrow[r, "\sim"] & \textup{K}_{\varphi, \Gamma}(D_{\varpi}(r)) \arrow[u, "\wr"].
	\end{tikzcd}
\end{center}

In the diagram, we take $\Delta^{\textpd} = \fESbar^{\textpd} \otimes_R M$, $\Delta^{\textpd, \partial} = (\Delta^{\textpd})^{\partial=0}$, $TA_{\crys} = \Acrys(\Sbar) \otimes_{\ZZ_p} T$, $\Delta^{[u, v]} = \fESbar^{[u, v]} \otimes_R M$, $\Delta^{[u, v], \partial} = (\Delta^{[u, v]})^{\partial=0}$, $TA^{[u, v]} = \ASbar^{[u, v]} \otimes_{\ZZ_p} T$, $\Delta_{\varpi}^{[u, v]} = \fERpi^{[u, v]} \otimes_R M$ (see Definition \ref{defi:obese_rings}), $TA_{\Sbar}(r) = \ASbar \otimes_{\ZZ_p} T(r)$, $\Dpi(r) = \Dpi(T(r))$, $\Npi^{\bmstar}(r) = \Npi^{\bmstar}(T(r))$ and $D_{R_{\infty}}(r) = \ASinfty \otimes_{\ARpi} \Dpi(r)$.
Moreover, $G = G_S$, $\Gamma = \Gamma_S$ with $C_G$ and $C_{\Gamma}$ denoting the complex of continuous cochains for $G$ and $\Gamma$, respectively.
The letter ``K'' denotes the Koszul complex with subscripts: $\partial$ denotes the operators $((1+X_0) \frac{\partial}{\partial X_0}, \ldots, X_d \frac{\partial}{\partial X_d})$, the subscript $\Gamma$ denotes the operators $(\gamma_0-1, \ldots, \gamma_d-1)$ for our choice of topological generators of $\Gamma$, the subscript $\Lie \Gamma$ denotes the operators $(\nabla_0, \ldots, \nabla_d)$, with $\nabla_i = \log \gamma_i$ and the subscript $\partial_A$ denotes $((1+X_0)\frac{\partial}{\partial X_0}, X_1\frac{\partial}{\partial X_1}, \ldots, X_d \frac{\partial}{\partial X_d})$ as operators on $\AR^{[u, v]}$ and $E_R^{[u, v]}$ via the isomorphism $\iota_{\cycl} : \Rpi^{[u, v]} \isomorphic \ARpi^{[u, v]}$.
The letter ``$\pazk$'' denotes a certain subcomplex of the Koszul complex (see \S \ref{subsec:diff_to_lie}, \S \ref{subsec:lie_to_gamma}, \S \ref{subsec:change_annulus_1}, \S \ref{subsec:change_annulus_2}).

Next, let us describe the maps between the rows.
FES denotes a map coming from the fundamental exact sequences in \eqref{eq:fes_acrys} and \eqref{eq:fes_aruv}.
AS denotes a map originating from the Artin-Schreier theory in \eqref{eq:artin_schreier_arbar}.
PL denotes maps coming from the filtered Poincar\'e Lemma of \S \ref{subsec:fat_period_rings}.
In the first column, going from the first row to the second row is induced by the inclusion $\Rpi^{\textpd} \subset \Rpi^{[u, v]}$.
The leftmost slanted vertical map from the third to the second row is induced by the inclusion $\fERpi^{[u, v]} \subset \fESbar^{[u, v]}$.
From the second to the third row, the map in the third column is induced similar to \eqref{eq:pdpartil0_acrist}.
The leftmost vertical map from the second to the third row is the content of Lemma \ref{lem:poincare_lemma_2} and the leftmost vertical map from the fourth to the third row is the content of Lemma \ref{lem:poincare_lemma_1}; the composition being the content of Proposition \ref{prop:syntomic_to_phi_gamma}.
The rightmost vertical map from the fourth to the third row is the inflation map from $\Gamma_S$ to $G_S$, using the inclusion $\ASinfty \subset \ASbar$ (one could use almost \'etale descent to obtain the quasi-isomorphism) and the rightmost vertical map from the fifth to the fourth row uses the inclusion $\ARpi \subset \ASinfty$ (the quasi-isomorphism is obtained by decompletion techniques).
The leftmost vertical arrow from the fourth to the fifth row is given by multplication by suitable powers of $t$ as in Lemma \ref{lem:diff_to_lie} and the rightmost vertical arrow from the sixth to the fifth row is the comparison between the complex computing the continuous cohomology of $\Gamma_S$ and the Koszul complex as in \S \ref{subsec:gal_coho_kos_complex}.
The inclusions $\ARpi^+ \subset \Ainf(\Sbar) \subset \ASbar^{[u, v]}$ and $\Ainf(\Sbar) \otimes_{\AR^+} \mbfn(T) \subset \Ainf(\Sbar) \otimes_{\ZZ_p} T$ induce the slanted vertical arrow from the fifth to the third row.

Finally, let us describe the maps between the columns.
The top two maps from the first to the second column are induced by the respective inclusions $\Rpi^{\textpd} \subset \fESbar^{\textpd}$ and $\Rpi^{[u, v]} \subset \fESbar^{[u, v]}$.
The bottom two maps $\Laz$ between the first and the second column are Lazard isomorphisms discussed in \S \ref{subsec:diff_to_lie}.
The bottom map from the third to the second column is induced canonically from the inclusion $\ARpi^{(0, v]+} \subset \ARpi^{[u, v]}$.
From the third to the fourth column, the top horizontal map is induced similar to \eqref{eq:pdpartil0_acrist} and the bottom horizontal map is induced by the inclusion $\ARpi^{(0, v]+} \subset \ARpi$ (see \S \ref{subsec:change_annulus_2} and \S \ref{subsec:change_disk}).

\begin{cor}\label{cor:lazard_fmlocal_comparison_R}
	The morphism of complexes $\tilde{\alpha}_{r, n, R}^{\FM}$ in Remark \ref{rem:fm_period_map_R} is a $p^{N(p, r, s)}\textrm{-quasi-isomorphism}$.
\end{cor}
\begin{proof}
	Let $m = 2$, i.e.\ $K = F(\zeta_{p^2}-1)$ and $e = p(p-1)$.
	Then, over $S = O_K \otimes_{O_F} R$ we know that the local Fontaine--Messing period map $\tilde{\alpha}_{r, n, S}^{\FM}$ is $p^N\textrm{-isomorphic}$ to the Lazard map $\alpha_{r, n}^{\Laz}$ from Theorem \ref{thm:lazard_fmlocal_comparison}.
	Moreover, the Lazard map $\alpha_{r, n}^{\Laz}$ is a $p^N\textrm{-quasi-isomorphism}$ by Theorem \ref{thm:syntomic_complex_galois_cohomology}.
	As we fixed $m$, therefore, it follows that $N = 2n(T, e) + 14r + 7s + 2$ only depends on $p$, $r$ and $s$ (see \S \ref{subsec:proof_lazard_comp} for the explicit constant).
	Next, to descend to $R$, we note that the Fontaine--Messing period map is $G = \Gal(F(\zeta_{p^2})/F)\textrm{-equivariant}$, i.e.\ the following diagram commutes:
	\begin{center}
		\begin{tikzcd}
			\Syn(R, M, r)_n \arrow[d] \arrow[r, "\tilde{\alpha}_{r, n, R}^{\FM}"] & C(G_R, T/p^n(r)') \arrow[d, "\wr"]\\
			\RGamma(G, \Syn(S, M, r)_n) \arrow[r, "\tilde{\alpha}_{r, n, S}^{\FM}"] & \RGamma(G, C(G_S, T/p^n(r)')),
		\end{tikzcd}
	\end{center}
	where the right vertical map is a quasi-isomorphism.
	So, from the Galois descent argument in Lemma \ref{lem:syn_galois_descent} (for $e=p(p-1)$), it follows that the left vertical arrow is a $p^{4r+3p(p-1)}\textrm{-quasi-isomorphism}$.
	Hence, we obtain that the morphism of complexes $\tilde{\alpha}_{r, n, R}^{\FM}$ in Remark \ref{rem:fm_period_map_R} is a $p^{N(p,r,s)}\textrm{-quasi-isomorphism}$, for $N(p,r,s) = 2N+4r+3p(p-1)$.
\end{proof}

\subsection{Galois descent}\label{subsec:galois_descent}

Let $e = [K:F] = p^{m-1}(p-1)$, $G = \Gal(K/F)$ and $S = O_K \otimes_{O_F} R$.
For notational convenience, we will use crystalline and syntomic complexes as in \S \ref{subsec:syntomic_complex}.
We view the $R\textrm{-module}$ $M$ in Assumption \ref{assum:relative_crystalline_wach_free} as an object in $\CR(R/O_F, \Fil, \varphi)$, i.e.\ a filtered crystal equipped with Frobenius (see Remark \ref{rem:mic_conv_fil} and Definition \ref{defi:crystal_frob}).

\begin{lem}\label{lem:syn_galois_descent}
	The following natural map is a $p^{4r+3e}\textrm{-quasi-isomorphism}$
	\begin{equation*}
		\RGamma_{\syn}(R, M, r) \longrightarrow \RGamma(G, \RGamma_{\syn}(S, M, r)).
	\end{equation*}
\end{lem}
\begin{proof}
	The claim can be shown by closely following the proof of \cite[Lemma 5.9]{colmez-niziol-nearby-cycles}.
	We omit the details.
\end{proof}

\section{Crystals and syntomic cohomology}\label{sec:crystals_syntomic_cohomology}

\subsection{Filtered Frobenius crystals}

Let $\kappa$ be a perfect field of characteristic $p$, set $O_F = W(\kappa)$ and $F = \Fr O_F$.
Furthermore, let $K$ be a finite extension of $F$ such that $K \cap F^{\unrami} = F$ and let $O_K$ denote its ring of integers.

\begin{nota}
	In \S \ref{sec:crystals_syntomic_cohomology} and \S \ref{sec:padic_nearby_cycles} we will use letters $\frakX, \frakY, \frakZ$, etc.\ to denote schemes as well as $\padic$ formal schemes.
\end{nota}

Let $\frakX$ be a ($\padic$ formal) scheme over $O_K$ with $X$ its (rigid) generic fibre and $\frakX_{\kappa}$ its special fibre.
Set $\Sigma = \Spec O_F$ (resp. $\Sigma = \Spf O_F$) and for $n \in \NN$, let $\frakX_n = \frakX \otimes_{\ZZ_p} \ZZ/p^n$ and $\Sigma_n = \Spec(O_F/p^n)$.
Consider the big (\'etale) crystalline site $\CRYS(\frakX_n / \Sigma_n)$ with the PD-ideal $(p(O_F/p^n), [\hspace{1mm}])$ and the category of crystals of $\pazo_{\frakX_n/\Sigma_n}\textrm{-modules}$ (see \cite[\S III.4.2]{berthelot-cristalline}, \cite[\S 1.1.18, \S 1.1.19]{berthelot-breen-messing-dieudonne}, \cite[Corollary 1.15, Proposition 1.17]{bauer-bsd}).
Set $\CR(\frakX_n / \Sigma_n)$ to be the full subcategory of finite locally free crystals.
The homomorphisms $\frakX_n \rightarrow \frakX_{n+1}$ and $\Sigma_n \rightarrow \Sigma_{n+1}$ induce a pullback functor $i_{n, n+1}^{\ast} : \CR(\frakX_{n+1} / \Sigma_{n+1}) \rightarrow \CR(\frakX_n / \Sigma_n)$.
Similarly, define the big crystalline site $\CRYS(\frakX_1 / \Sigma_n)$ and the category of finite locally free crystals $\CR(\frakX_1 / \Sigma_n)$.
Note that the natural pullback functor $i_n^{\ast} : \CR(\frakX_n / \Sigma_n) \rightarrow \CR(\frakX_1 / \Sigma_n)$ induces an equivalence of categories by \cite[Chapitre IV, Th\'eor\`em 1.4.1]{berthelot-cristalline}.

\begin{defi}
	A finite locally free crystal on $\CRYS(\frakX / \Sigma)$ is the data $\pazf = (\pazf_n)_{n \geqslant 1}$, where $\pazf_n$ is an object of $\CR(\frakX_n / \Sigma_n)$ and we have isomorphisms $i_{n, n+1}^{\ast}(\pazf_{n+1}) \isomorphic \pazf_n$.
	A morphism between two crystals $\pazf$ and $\pazg$ on $\CRYS(\frakX / \Sigma)$ is a collection of morphisms $\pazf_n \rightarrow \pazg_n$, for each $n \geqslant 1$, compatible with the pullback isomorphisms.
	Denote the category of such objects by $\CR(\frakX / \Sigma)$.
	A finite locally free crystal on $\CRYS(\frakX_1 / \Sigma)$ is defined similarly and the pullback functor $i^{\ast} : \CR(\frakX / \Sigma) \rightarrow \CR(\frakX_1 / \Sigma)$ induces an equivalence of categories.
\end{defi}

Consider the category of filtered crystals on $\CRYS(\frakX/\Sigma)$ in the sense of \cite[Definition 16]{tsuji-ainf-genrep} (for relation between this category and Ogus' book \cite{ogus-ftcrystals}, see \cite[Remark 19]{tsuji-ainf-genrep}).
Take $\CR(\frakX_n / \Sigma_n, \Fil)$ to be the full subcategory of finite locally free filtered crystals on $\CRYS(\frakX_n / \Sigma_n)$.
We have the natural pullback functor $i_{n, n+1}^{\ast} : \CR(\frakX_{n+1} / \Sigma_{n+1}, \Fil) \rightarrow \CR(\frakX_n / \Sigma_n, \Fil)$.

\begin{defi}\label{defi:filtered_crys}
	A finite locally free filtered crystal on $\CRYS(\frakX / \Sigma)$ is the data $(\pazf_n)_{n \geqslant 1}$ in $\CR(\frakX / \Sigma, \Fil)$ such that the isomorphisms $i_{n, n+1}^{\ast}(\pazf_{n+1}) \isomorphic \pazf_n$ are compatible with filtration.
	A morphism between two filtered crystals is defined in an obvious way and we denote this category by $\CR(\frakX / \Sigma, \Fil)$.
\end{defi}

\begin{rem}\label{rem:mic_conv_fil}
	Let $R = \padic \textrm{ completion of an \'etale algebra over } O_F[X_1^{\pm 1}, \ldots, X_d^{\pm 1}]$ and let $\MIC(R)$ be the category of finite projective $R\textrm{-modules}$ equipped with an integrable connection and let $\MIC_{\converge}(R) \subset \MIC(R)$ denote the full subcategory of modules whose connection is $p\textrm{-adically}$ quasi-nilpotent.
	Let $\frakX = \Spf R$, then from \cite[Chapitre IV, Th\'eor\`em 1.6.5]{berthelot-cristalline} and \cite[Lemma 1.9]{morrow-tsuji-coeff} we obtain an equivalence of categories $\CR(\frakX/\Sigma) \isomorphic \MIC_{\converge}(R)$.
	This equivalence restricts to an equivalence $\CR(\frakX/\Sigma, \Fil) \isomorphic \MIC_{\converge}(R, \Fil)$.
\end{rem}

Finally, we will consider crystals equipped with a Frobenius structure.
The Frobenius endomorphism of $O_F$ and the absolute Frobenius on $\frakX_1$ induce Frobenius pullbacks $F^{\ast}_{\frakX_1} : \CR(\frakX_1 / \Sigma_n) \rightarrow \CR(\frakX_1 / \Sigma_n)$ and $F^{\ast}_{\frakX_1} : \CR(\frakX_1 / \Sigma) \rightarrow \CR(\frakX_1 / \Sigma)$.
Recall that we have the natural pullback functor $i^{\ast} : \CR(\frakX / \Sigma) \rightarrow \CR(\frakX_1 / \Sigma)$.

\begin{defi}\label{defi:crystal_frob}
	A \textit{Frobenius structure} on a finite locally free crystal $\pazf$ on $\CRYS(\frakX/\Sigma)$ is a morphism $\varphi_{\pazf} : F^{\ast}_{\frakX_1} i^{\ast} \pazf \rightarrow i^{\ast} \pazf$ such that it becomes an isomorphism in the isogeny category $\CR(\frakX/\Sigma)_{\QQ}$.
	A morphism between two crystals with Frobenius structure is taken to be a morphism in $\CR(\frakX/\Sigma)$ compatible with respective Frobenius structures.
	Denote the category of finite locally free crystals (resp.\ filtered crystals) equipped with a Frobenius structure as $\CR(\frakX/\Sigma, \varphi)$ (resp.\ $\CR(\frakX/\Sigma, \Fil, \varphi)$).
\end{defi}

\subsection{Syntomic complex}\label{subsec:syntomic_complex}

Let $\frakX$ be a smooth ($\padic$ formal) scheme over $O_K$, let $\Sigma = \Spec O_F$ (resp. $\Sigma = \Spf O_F$) and let $\pazf$ be an object of $\CR(\frakX/\Sigma, \Fil, \varphi)$, i.e.\ a finite locally free filtered crystal on $\CRYS(\frakX/\Sigma)$ equipped with a Frobenius structure.
In this subsection, we will define the syntomic cohomology of $\frakX$ with coefficients in $\pazf$.

Let $u_{\frakX_n/\Sigma_n} : (\frakX_n/\Sigma_n)_{\crys} \rightarrow \frakX_{n, \etale}$ denote the projection from the crystalline topos to the \'etale topos.
In the following, we regard sheaves on $\frakX_{n, \etale}$ as sheaves on $\frakX_{\kappa, \etale}$.
For $r \geqslant 0$, we have filtered crystalline cohomology complexes of $\pazf$:
\begin{align*}
	\RGamma_{\crys}(\frakX, \Fil^r \pazf)_n := \RGamma\big(\frakX_{n, \etale}, \Rup u_{\frakX_n/\Sigma_n \ast} \Fil^r \pazf_n\big), \hspace{2mm} \RGamma_{\crys}(\frakX, \Fil^r \pazf) := \holim \RGamma_{\crys}(\frakX, \Fil^r \pazf)_n.
\end{align*}

\begin{defi}
	Define the modulo $p^n$ and the completed syntomic complex with coefficients as,
	\begin{align*}
		\RGamma_{\syn}(\frakX, \pazf, r)_n &:= \big[\RGamma_{\crys}(\frakX, \Fil^r \pazf)_n \xrightarrow{p^r-\varphi} \RGamma_{\crys}(\frakX, \pazf)_n\big], \\
		\RGamma_{\syn}(\frakX, \pazf, r) &:= \holim \RGamma_{\syn}(\frakX, \pazf, r)_n.
	\end{align*}
	The mapping fibres are taken in the derived $\infty\textrm{-category}$ of abelian groups.
\end{defi}

\begin{rem}\label{rem:syn_complex_alt}
	In the derived category $D^+(\frakX_{\kappa, \etale}, \ZZ/p^n)$, we have quasi-isomorphisms $\RGamma_{\syn}(\frakX, \pazf, r)_n \simeq \RGamma_{\syn}(\frakX, \pazf, r) \otimes_{\ZZ_p}^{L} \ZZ/p^n$ and $\RGamma_{\syn}(\frakX, \pazf, r)_n \simeq [\RGamma_{\crys}(\frakX, \pazf)_n \xrightarrow{(p^r-\varphi, \textup{can})} \RGamma_{\crys}(\frakX, \pazf)_n \oplus \RGamma_{\crys}(\frakX, \pazf/\Fil^r \pazf)_n]$.
\end{rem}

\begin{defi}\label{defi:syntomic_complex_etale_site}
	Define $\calf_{n, \etale, \frakX}$ to be \'etale sheafification of $(\frakU \rightarrow \frakX) \mapsto \RGamma_{\crys}(\frakU, \pazf)_n$ and $\Fil^r \calf_{n, \etale, \frakX}$ to be \'etale sheafification of $(\frakU \rightarrow \frakX) \mapsto \RGamma_{\crys}(\frakU, \Fil^r \pazf)_n$, for $\frakU \rightarrow \frakX$ any \'etale map.
	Similarly, define $\cals_{n, \etale}(\pazf, r)_\frakX$ to be the \'etale sheafification of $(\frakU \rightarrow \frakX) \mapsto \RGamma_{\syn}(\frakU, \pazf, r)_n$.
\end{defi}

\begin{lem}
	In the setting above, we have $\cals_{n, \etale}(\pazf, r)_\frakX = \big[\Fil^r \calf_{n, \etale, \frakX} \xrightarrow{p^r-\varphi} \calf_{n, \etale, \frakX}\big]$ and $\RGamma_{\syn}(\frakX, \pazf, r)_n = \RGamma(\frakX_{\kappa, \etale}, \cals_{n, \etale}(\pazf, r)_\frakX)$.
\end{lem}

\begin{rem}\label{rem:syntomic_complex_hypercovering}
	The syntomic cohomology with coefficients can also be described using hypercoverings, for example, see \cite[\S 2.6]{tsuji-syntomic-complex} and \cite[\S 2.1]{tsuji-semistable-comparison}.
\end{rem}

\begin{nota}
	In the rest of this article we will denote mod $p^n$ (resp.\ completed) syntomic complex with coefficients in $\pazf$ as $\cals_n(\pazf, r)_\frakX$ (resp.\ $\cals(\pazf, r)_\frakX$).
\end{nota}

\section{\texorpdfstring{$p$}{-}-adic nearby cycles}\label{sec:padic_nearby_cycles}

In this section, we give some global applications of the computations done in previous sections.

\subsection{Fontaine--Laffaille modules}\label{subsec:global_fm_mods}

Let $R$ denote the $\padic$ completion of an \'etale algebra over $O_F[X_1^{\pm 1}, \ldots, X_d^{\pm 1}]$, for some $d \in \NN$, satisfying Assumption \ref{assum:small_algebra}, and let $s \in \NN$ such that $s \leqslant p-2$.
In \S \ref{subsec:fontaine_laffaile_to_wach} we defined the category $\MF_{[0, s], \free}(R, \Phi, \partial)$ of \textit{free relative Fontaine--Laffaille} modules of level $[0, s]$.

Let us now globalise the definition above.
Let $\frakX$ be a smooth ($\padic$ formal) scheme defined over $O_F$.
Consider a covering $\{\frakU_i\}_{i \in I}$ of $\frakX$ with $\frakU_i = \Spec A_i$ (resp.\ $\frakU_i = \Spf A_i$) such that the $\padic$ completions $\widehat{A}_i$ satisfy Assumption \ref{assum:small_algebra}, for each $i \in I$.
We fix lifts of Frobenius modulo $p$ as $\varphi_i : \widehat{A}_i \rightarrow \widehat{A}_i$.

\begin{rem}\label{rem:alpha_phi}
	In \S \ref{subsec:fontaine_laffaile_to_wach} we fixed a lifting $\varphi$ of the absolute Frobenius on $R/p$.
	However, for another lift $\varphi'$ the categories $\MF_{[0, s], \free}(R, \Phi, \partial)$ and $\MF_{[0, s], \free}(R, \Phi\prm, \partial)$ are naturally equivalent (\cite[Theorem 2.3]{faltings-crystalline} and \cite[Remark 33]{tsuji-ainf-genrep}).
	In particular, there is a well-defined isomorphism $\alpha_{\varphi, \varphi\prm} : \varphi^{\ast} M \isomorphic \varphi^{\prime \ast} M$ compatible with connections.
\end{rem}

\begin{defi}\label{defi:mffree}
	Define $\MF_{[0, s], \free}(\frakX, \Phi, \partial)$ to be the category of finite locally free filtered $\pazo_\frakX\textrm{-modules}$ $\pazm$ equipped with a $p\textrm{-adically}$ quasi-nilpotent integrable connection satisfying Griffiths transverality with respect to filtration, and such that there exists a covering $\{\frakU_i\}_{i \in I}$ of $\frakX$ as above with $\pazm_{\frakU_i} \in \MF_{[0, s], \free}(\widehat{A}_i, \Phi, \partial)$ for all $i \in I$ and on $\frakU_{ij}$ the two structures glue well under $\alpha_{\varphi_i, \varphi_j}$.
\end{defi}

\begin{rem}
	Let $\Sigma = \Spec O_F$ (resp.\ $\Sigma = \Spf O_F$), then the category $\MF_{[0, s], \free}(\frakX, \Phi, \partial)$ is a full subcategory of $\CR(\frakX/\Sigma, \Fil, \varphi)$ described in Definition \ref{defi:crystal_frob}.
\end{rem}

\begin{rem}\label{rem:locsys_from_fl}
	To any object of $\MF_{[0, s], \free}(\frakX, \Phi, \partial)$, in \cite[Theorem 2.6*]{faltings-crystalline}, Faltings associated a compatible system of \'etale sheaves on $\textup{Sp}(\widehat{A}_i[1/p])$ (see the functor $T_{\crys}$ in \S \ref{subsec:fontaine_laffaile_to_wach}).
	These sheaves can be expressed in terms of certain finite \'etale coverings of $\textup{Sp}(\widehat{A}_i[1/p])$.
	Extending these by normalisation to $\Spec(\widehat{A}_i)$, the resulting coverings glue to give a covering of the smooth formal $O_F\textrm{-scheme}$ $\frakX'$ associated to $\frakX$.
	For $\frakX$ a smooth $\padic$ formal scheme, note that $\frakX = \frakX'$ and this gives us an \'etale $\ZZ_p\textrm{-local system}$ on the rigid generic fibre $X$ of $\frakX$, which we denote by $\LL$.
	On the other hand, for $\frakX$ a smooth scheme, if $\frakX$ is proper then this covering is finite and algebraic and we obtain an \'etale $\ZZ_p\textrm{-local system}$ $\LL$ on $X = \frakX \otimes_{O_F} F$, or if $\frakX$ is an open subscheme of a proper semistable scheme $\frakY$ over $O_F$ to which $\pazm$ extends, i.e.\ there exists a (log) Fontaine--Laffaille module $\pazn$ over $\frakY$ (in the sense of \cite[Definition 2.3.6 \& Remark 2.3.13]{tsuji-syntomic-complex}) such that $\pazm = \pazn|_{\frakX}$, then the \'etale local system $\LL$ on $X = \frakX \otimes_{O_F} F$ is again well defined (see \cite[p.\ 63 \& Appendix]{tsuji-syntomic-complex}).
	By loc.\ cit., note that in the case of schemes, the preceding assumptions are necessary to obtain the \'etale local system $\LL$ on $X$.
\end{rem}

\subsection{Fontaine--Messing period map}\label{subsec:fm_period_map}

Let $\Sigma = \Spec O_F$ (resp.\ $\Sigma = \Spf O_F$) and $K$ a finite extension of $F$ such that $K \cap F^{\unrami} = F$.
Take $0 \leqslant s \leqslant p-2$ and $r \geqslant s+1$.

\subsubsection{The case of schemes}

Let $\frakX$ be a smooth scheme over $O_F$ with $i : \frakX_{\kappa, \etale} \rightarrow \frakX_{\etale}$ and $j : X_{\etale} \rightarrow \frakX_{\etale}$ the natural morphism of sites.
Take $\pazm$ in $\MF_{[0, s], \free}(\frakX, \Phi, \partial)$ and let $\LL$ denote the associated $\ZZ_p\textrm{-local system}$ on the generic fibre $X$ (note that from Remark \ref{rem:locsys_from_fl}, we need to additionally assume that $\frakX$ is a proper scheme or an open subscheme of a proper scheme to which $\pazm$ extends, however, constructions in this subsection are independent of these assumptions).
From \cite[\S 5.3]{abhinandan-crystalline-wach}, the $\pazo_\frakX\textrm{-module}$ $\pazm$ corresponds to a finite locally free filtered crystal in $\CR(\frakX/\Sigma, \Fil, \varphi)$ equipped with a Frobenius structure and (by abuse of notations) we denote this crystal again by $\pazm$.

To describe the Fontaine--Messing period map one can almost verbatim adapt the methods from \cite[\S 5]{tsuji-syntomic-complex} and \cite[\S 3.1]{tsuji-semistable-comparison}.
One first constructs a local version of the map and then uses hypercoverings to globalise.
Below we will describe the technical inputs needed for the construction of Fontaine--Messing map; for actual construction the reader should refer to loc.\ cit.
We focus on the local setup first, i.e.\ let $\frakX$ be an affine smooth scheme over $O_F$.
Let $\frakY = \frakX \otimes_{O_F} O_K$ and choose an embedding $\frakY \hookrightarrow \frakZ$ such that $\frakZ$ is an affine smooth scheme over $O_F$.
Then $\frakY$ can be covered by affine \'etale $\frakY\textrm{-schemes}$ $\frak\frakU = \Spec A$ with $A = O_K \otimes_{O_F} B$ and $B$ an \'etale algebra over $O_F[X_1^{\pm 1}, \ldots, X_d^{\pm 1}]$ such that its $\padic$ completion $\widehat{B}$ satisfies Assumption \ref{assum:small_algebra}.
Let $Y$ (resp.\ $U$) denote the generic fibre of $\frakY$ (resp.\ $\frakU$), i.e.\ $Y = \frakY \otimes_{O_K} K$ (resp.\ $U = \frakU \otimes_{O_K} K$).

\begin{rem}\label{rem:elkik_henselian}
	Take $A$ as above, let $A^h$ denote the $\padic$ henselisation of $A$, let $\overline{A^h}$ denote the union of finite $A^h\textrm{-subalgebras}$ $S \subset \overline{\textrm{Fr } A^h}$ such that $S[1/p]$ is \'etale over $A^h[1/p]$ and set $G_{A^h} = \Gal\big(\overline{A^h}[1/p] / A^h[1/p]\big)$.
	Then, by Elkik's approximation theorem \cite[Corollary p.\ 579]{elkik-solutions}, we have a natural isomorphism of Galois groups $G_{A^h} \simeq G_{\widehat{A}}$.
	Therefore, any discrete $G_{\widehat{A}}\textrm{-module}$ can be regardeed as a locally constant sheaf on the \'etale site of the generic fibre $U^h = \frakU^h \otimes_{O_K} K$, where $\frakU^h = \Spec A^h$.
\end{rem}

\begin{rem}\label{rem:elkik_syntomic}
	Note that we have henselian versions of the fundamental exact sequences in \eqref{eq:fes_acrys} and \eqref{rem:fes_acrys_rep}, where one replaces $\overline{\widehat{A}}$ by $\overline{A^h}$ and $G_{\widehat{A}}$ with $G_{A^h}$.
	In particular, similar to \eqref{eq:bar_syntomic_complex} one obtains a syntomic complex $\Syn(\overline{A^h}, \pazm_\frakU, r)_n$ of discrete $G_{A^h}\textrm{-modules}$ which we denote as $\overline{\cals}_n(\pazm, r)_\frakU$.
	Note that from Remark \ref{rem:elkik_henselian} the complex of $G_{A^h}\textrm{-modules}$ $\overline{\cals}_n(\pazm, r)_\frakU$ can be regarded as a complex of locally constant sheaves on $U^h_{\etale}$ and we obtain a morphism $\Gamma\big(\frakU, i_{\ast}\cals_n(\pazm, r)_\frakY\big) \rightarrow \Gamma\big(U^h, \overline{\cals}_n(\pazm, r)_\frakU\big)$ and a natural map 
	\begin{equation}\label{eq:rep_etale_system}
		\RGamma(G_{\widehat{A}}, T_{\crys}(\pazm_\frakU)/p^n(r)) \longrightarrow \RGamma_{\etale}(U^h, \LL/p^n(r)_{U}).
	\end{equation}
\end{rem}

Using Remark \ref{rem:elkik_henselian} and Remark \ref{rem:elkik_syntomic} together with the Poincar\'e Lemma \ref{lem:fil_poincare_lem_na}, the fundamental exact sequence (see \eqref{eq:fes_acrys}, \eqref{rem:fes_acrys_rep} and \eqref{eq:poincare_lemma_acrys_rep}) and \eqref{eq:rep_etale_system}, note that from the construction in \cite[\S 5]{tsuji-syntomic-complex} and \cite[\S 3.1]{tsuji-semistable-comparison}, one otains a natural morphism in $D^+(\frakY_{\etale}, \ZZ/p^n)$:
\begin{equation}\label{eq:local_fm_period_map}
	\cals_n(\pazm, r)_\frakY \longrightarrow i^{\ast} \Rup j_{\ast}\LL/p^n(r)\prm_Y.
\end{equation}

Next, let $\frakX$ be a smooth scheme over $O_F$, set $\frakY = \frakX \otimes_{O_F} O_K$ and let $Y$ denote its generic fibre.
To globalise the construction above, one considers an \'etale hypercovering $\frakU^{\bullet}$ of $\frakX$ and chooses a morphism of simplicial schemes $i^{\bullet} : \frakU^{\bullet} \rightarrow \frakZ^{\bullet}$, such that for each $s \in \NN$, the morphism $i^s$ is an immersion of schemes, $\frakZ^s$ is smooth over $O_F$ and there exist compatible liftings of Frobenius $F_{\frakZ^{\bullet}} := \{F_{\frakZ_n^{\bullet}} : \frakZ_n^{\bullet} \rightarrow \frakZ_n^{\bullet}\}$.
Then using the local description above and the theory of hypercoverings, from the construction in \cite[\S 5]{tsuji-syntomic-complex} and \cite[\S 3.1]{tsuji-semistable-comparison}, we obtain a natural morphism in $D^+(\frakY_{\etale}, \ZZ/p^n\ZZ)$ (independent of choices by loc.\ cit.):
\begin{equation*}
	\alpha_{r, n, \frakY}^{\FM} : \cals_n(\pazm, r)_\frakY \longrightarrow i^{\ast} \Rup j_{\ast} \LL/p^n(r)\prm_Y.
\end{equation*}

\subsubsection{The case of formal schemes}

The definition of the Fontaine--Messing period map for $\padic$ formal schemes follows in manner similar to that of schemes, with certain key differences which we point out below.
Let $\frakX$ be a smooth $\padic$ formal scheme over $O_F$ and set $\frakY = \frakX \otimes_{O_F} O_K$.
In this case, an affine \'etale formal scheme over $\frakY$ can be covered by affine formal schemes $\frakU = \Spf S$, with $S = O_K \otimes_{O_F} R$ and $R$ as in Assumption \ref{assum:small_algebra}.
For such local models, we consider the $p\textrm{-adically}$ completed version of Fontaine--Messing period map described in \eqref{eq:local_fm_period_map}.
Finally, to obtain the global version, one proceeds in exactly the same manner as in the case of schemes (with a hypercovering $(\frakU^{\bullet}, \frakZ^{\bullet}, F_{\frakZ^{\bullet}})$, where each $\frakU^s$ is of the form described above).

\begin{rem}\label{rem:fm_map_cyclotomic}
	Note that in the cyclotomic case, i.e.\ $K = F(\zeta_{p^m})$, for $m \in \NN$, the map described in \eqref{eq:local_fm_period_map} coincides with composition of the map $\tilde{\alpha}_{r, n, S}^{\FM}$ described in \S \ref{subsec:fm_maps_comparison} with the quasi-isomorphism $C(G_S, T/p^n(r)') \isomorphic \RGamma_{\etale}(U, \LL/p^n(r)')$ obtained by applying $K(\pi, 1)\textrm{-Lemma}$ for $p\textrm{-coefficients}$ (see \cite[Theorem 4.9]{scholze-rigid} and \cite[\S 5.4.1]{colmez-niziol-nearby-cycles}).
\end{rem}

\subsection{A global result}\label{subsec:global_result}

To state the main global result, let $\frakX$ be a smooth ($\padic$ formal) scheme defined over $O_F$ (for $\frakX$ a scheme, assume that it is proper or an open subscheme of a proper semistable scheme defined over $O_F$).
Let $\pazm$ be an object of the category $\MF_{[0, s], \free}(\frakX, \Phi, \partial)$, i.e.\ a relative Fontaine--Laffaille module of level $[0, s]$ for $0 \leqslant s \leqslant p-2$ (for $\frakX$ an open scheme, further assume that $\pazm$ extends to the compatification of $\frakX$, see Remark \ref{rem:locsys_from_fl}).
Let $\LL$ denote the associated $\ZZ_p\textrm{-local system}$ on the (rigid) generic fibre $X$ of $\frakX$.
Then, we show the following:
\begin{thm}\label{thm:syntomic_nearby_comparison}
	For $r \geqslant s+1$ and $0 \leqslant k \leqslant r-s-1$, the Fontaine--Messing period map,
	\begin{equation}\label{eq:global_fm_period_map}
		\alpha_{r, n, \frakX}^{\FM} : \pazh^k\big(\cals_n(\pazm, r)_\frakX\big) \longrightarrow i^{\ast} \Rup^k j_{\ast} \LL/p^n(r)\prm_X,
	\end{equation}
	is a $p^N\textrm{-isomorphism}$, where $N = N(p, r, s) \in \NN$ depends on $p$, $r$ and $s$ but not on $\frakX$ or $n$.
\end{thm}
\begin{proof}[Proof for schemes]
	By the definition of the Fontaine--Messing period map in \S \ref{subsec:fm_period_map}, we see that it is enough to show the $p\textrm{-power}$ quasi-isomorphism locally (provided the power of $p$ does not depend on the local model).
	Let $A$ be an $O_F\textrm{-algebra}$ such that its $\padic$ completion $\widehat{A}$ satisfies Assumption \ref{assum:small_algebra}, $\frakU = \Spec A$ and $M := \pazm_\frakU$.
	Note that we have $\RGamma_{\syn}(\frakU, \pazm_\frakU, r)_n = \Syn(\widehat{A}, M, r)_n$ and $\RGamma_{\syn}(\frakU, \pazm_\frakU, r) = \Syn(\widehat{A}, M, r)$.
	The Fontaine--Messing period map,
	\begin{equation*}
		\alpha_{r, n, \frakU}^{\FM}: \RGamma_{\syn}(\frakU, \pazm_\frakU, r)_n \longrightarrow \RGamma_{\etale}(U^h, \LL/p^n(r)\prm_{U^h}),
	\end{equation*}
	is the same as the composition of the henselian version of the map $\tilde{\alpha}_{r, n}^{\FM}$ with the natural map in \eqref{eq:rep_etale_system}, $C(G_{A^h}, T/p^n(r)') \rightarrow \RGamma_{\etale}(U^h, \LL/p^n(r)\prm_{U^h})$ (see Remarks \ref{rem:fm_period_map_R} and \ref{rem:fm_map_cyclotomic} for the $p\textrm{-adically}$ completed version).
	Note that henselian version of the map $\tilde{\alpha}_{r, n}^{\FM}$ is obtained by replacing $\overline{\widehat{A}}$ by $\overline{A^h}$ and $G_{\widehat{A}}$ with $G_{A^h}$.
	We set $\Syn(A, M, r) := \RGamma_{\syn}(\frakU, \pazm_\frakU, r)$.
	Let $k \leqslant r-s-1$ and our claim is that the map,
	\begin{equation*}
		\alpha_{r, n, A}^{\FM} : H^k(\Syn(A, M, r)_n) \xrightarrow{\tilde{\alpha}_{r, n}^{\FM}} H^k(G_{A^h}, T/p^n(r)') \longrightarrow H^k(U^h_{\etale}, \LL/p^n(r)\prm_{U^h}),
	\end{equation*}
	is an isomorphism (up to some power of $p$).
	To show the claim, we will pass to the $\padic$ completion of $A$.
	Let $\calu := \textup{Sp}\big(\widehat{A}\big[\frac{1}{p}\big]\big)$ and consider the following commutative diagram:
	\begin{center}
		\begin{tikzcd}
			H^k(\Syn(A, M, r)_n) \arrow[r,"\tilde{\alpha}_{r, n, A}^{\FM}"] \arrow[d, equal] & H^k(G_{A^h}, T/p^n(r)') \arrow[r] \arrow[d, "\wr"] & H^k(U^h_{\etale}, \LL/p^n(r)\prm_{U^h}) \arrow[d, "\wr"] \\
			H^k(\Syn(\widehat{A}, M, r)_n) \arrow[r,"\tilde{\alpha}_{r, n, \widehat{A}}^{\FM}", "\sim"'] & H^k(G_{\widehat{A}}, T/p^n(r)') \arrow[r, "\sim"] & H^k(\calu_{\etale}, \LL/p^n(r)\prm_\calu).
		\end{tikzcd}
	\end{center}
	The middle vertical arrow is an isomorphism because the two Galois groups are equal by Elkik's approximation theorem \cite[Corollary p. 579]{elkik-solutions} (see Remark \ref{rem:elkik_henselian}).
	The right vertical arrow is an isomorphism due to Gabber \cite[Theorem 1]{gabber-affine-pbc}.
	The bottom left horizontal arrow is a $p^N\textrm{-isomorphism}$, for $N = N(p, r, s) \in \NN$, as shown in the case of formal schemes below (for $R = \widehat{A}$), in particular, the top left horizontal arrow is also a $p^N\textrm{-isomorphism}$.
	The bottom right horizontal arrow is an isomorphism by a $K(\pi, 1)\textrm{-Lemma}$ due to Scholze \cite[Theorem 4.9]{scholze-rigid}, and therefore, the top right horizontal arrow is also an isomorphism.
	Hence, it follows that the composition of the top two horizontal arrows, i.e.\ $\alpha_{r, n, A}^{\FM}$ is a $p^N\textrm{-isomorphism}$.
\end{proof}

\begin{proof}[Proof for formal schemes]
	By the definition of the Fontaine--Messing period map in \S \ref{subsec:fm_period_map}, we see that it is enough to show the $p\textrm{-power}$ quasi-isomorphism locally (provided the power of $p$ does not depend on the local model).
	Let $R$ be an $O_F\textrm{-algebra}$ satisfying Assumption \ref{assum:small_algebra}, $\frakU = \Spf R$ and $M := \pazm_\frakU$.
	We have that the Fontaine--Messing period map
	\begin{equation*}
		\alpha_{r, n, R}^{\FM}: H^k(\Syn(R, M, r)_n) \longrightarrow H^k(G_R, T/p^n(r)') \isomorphic H^k(U_{\etale}, \LL/p^n(r)\prm_U),
	\end{equation*}
	is the same as the composition of the map $\tilde{\alpha}_{r, n, R}^{\FM}$ (see Remark \ref{rem:fm_period_map_R} and Remark \ref{rem:fm_map_cyclotomic}) with the natural isomorphism $H^k(G_R, T/p^n(r)') \isomorphic H^k(U_{\etale}, \LL/p^n(r)\prm_U)$ (see the $K(\pi, 1)\textrm{-Lemma}$ of \cite[Theorem 4.9]{scholze-rigid}).

	Finally, to show the isomorphism in degrees $0 \leqslant k \leqslant r-s-1$, we use Corollary \ref{cor:lazard_fmlocal_comparison_R} with Example \ref{exm:choice_odcrist} (iii) for Fontaine--Laffaille modules.
	To compute $N = N(p, r, s) \in \NN$, we combine the constants obtained in the proof of Theorem \ref{thm:syntomic_complex_galois_cohomology}, Corollary \ref{cor:lazard_fmlocal_comparison_R} (i.e.\ Lemma \ref{lem:syn_galois_descent} for $e = p(p-1)$) and Example \ref{exm:choice_odcrist} (iii) to obtain that $N = 32r + 14s + 3p(p-1) + 4$.
	In particular, $N$ does not depend on $n$ or the local model $\frakU$.
	This allows us to conclude.
\end{proof}


\phantomsection
\printbibliography[heading=bibintoc, title={References}]

\Addresses

\end{document}